\documentclass[12pt]{report}

\usepackage{amsmath}
\usepackage{amsfonts,amsthm,amssymb,hyperref}

\usepackage{tikz}
\usetikzlibrary{positioning}

\usepackage{xcolor,enumerate,comment}
\usepackage{graphicx}
\usepackage{epsf}
\usepackage{fullpage}
\usepackage{multicol}
\usepackage[square,sort,comma,numbers]{natbib}

\usepackage{makeidx}
\makeindex

\definecolor{olive}{rgb}{0.4, 0.6, .1}

\newtheorem{theorem}{Theorem}[section]
\newtheorem{proposition}[theorem]{Proposition}
\newtheorem{corollary}[theorem]{Corollary}
\newtheorem{lemma}[theorem]{Lemma}
\theoremstyle{definition}
\newtheorem{definition}[theorem]{Definition}
\theoremstyle{remark}
\newtheorem{remark}[theorem]{Remark}
\newtheorem{example}[theorem]{Example}

\numberwithin{equation}{section}

\renewcommand{\P}{\mathbf{P}}
\newcommand{\Prob}[1]{\mathbf{P}\left\{#1\right\}}
\def\one{{\bf 1}}
\def\zero{{\bf 0}}
\newcommand{\E}{\mathbf{E}}
\newcommand{\modulus}{\tau}

\newcommand{\sA}{\mathcal{A}}
\newcommand{\sB}{\mathcal{B}}
\newcommand{\sBA}{\mathcal{B}a}
\newcommand{\sC}{\mathcal{C}}
\newcommand{\sD}{\mathcal{D}}
\newcommand{\sG}{\mathcal{G}}

\newcommand{\sJ}{\mathcal{J}}

\newcommand{\sF}{\mathcal{F}}
\newcommand{\sK}{\mathcal{K}}

\newcommand{\sM}{\mathcal{M}}
\newcommand{\sN}{\mathcal{N}}

\newcommand{\sR}{\mathcal{R}}
\newcommand{\sS}{\mathcal{X}}

\newcommand{\sU}{\mathcal{U}}
\newcommand{\sV}{\mathcal{V}}
\newcommand{\sX}{\mathcal{X}}
\newcommand{\sY}{\mathcal{Y}}
\newcommand{\sZ}{\mathcal{Z}}

\newcommand{\NN}{\mathbb{N}}

\newcommand{\DD}{\mathbb{D}}
\newcommand{\EE}{\mathbb{E}}

\newcommand{\II}{\mathbb{I}}
\newcommand{\JJ}{\mathbb{J}}
\renewcommand{\SS}{\mathbb{S}}
\newcommand{\LL}{\mathbb{L}}

\newcommand{\QQ}{\mathbb{Q}}
\newcommand{\R}{\mathbb{R}}
\newcommand{\UU}{\mathsf{U}}
\newcommand{\VV}{\mathsf{V}}
\newcommand{\XX}{\mathbb{X}}
\newcommand{\YY}{\mathbb{Y}}
\newcommand{\ZZ}{\mathbb{Z}}
\newcommand{\XXT}{\tilde{\XX}}

\newcommand{\sfD}{\mathsf{D}}

\newcommand{\fF}{\mathfrak{F}}
\newcommand{\fK}{\mathfrak{K}}

\newcommand{\Cont}{\mathbb{C}}
\newcommand{\Dfun}{\mathbb{D}}
\newcommand{\fC}{\Cont}
\newcommand{\Cp}{\Cont_{\mathrm{p}}}
\newcommand{\USC}{\mathrm{USC}}
\newcommand{\cone}{\mathsf{C}}
\newcommand{\Measures}{\mathsf{M}}
\newcommand{\Msigma}[1][\sS]{\Measures_\sigma(#1)}
\newcommand{\Mb}[1][\sS]{\Measures(#1)}
\newcommand{\Mp}[1][\sS]{\Measures_{\mathrm{p}}(#1)}
\newcommand{\Mpk}[1][k]{\Measures_{\mathrm{p},#1}(\sX)}
\newcommand{\Mpx}[1][k]{\sM_{\mathrm{p}}^{(#1)}(\sX)}
\newcommand{\Mpi}[1][\sX]{\sM_{\mathrm{p}}^{(1)}\big(#1)}
\newcommand{\Mpik}[1][\sX]{\sM_{\mathrm{p}}^{(k)}\big(#1\big)}
\newcommand{\Mpiko}[1][\sX]{\sM_{\mathrm{p}}^{(k+1)}\big(#1\big)}
\newcommand{\Mpikm}[1][\sX]{\sM_{\mathrm{p}}^{(k-1)}\big(#1\big)}
\newcommand{\Mone}[1][\XX]{\Measures_1(#1)}
\newcommand{\Lp}[1][p]{\mathbb{L}^{#1}}

\newcommand{\Rpp}{(0,\infty)}
\newcommand{\base}{U}

\newcommand{\dmet}{\mathsf{d}}

\newcommand{\shift}{\phi}

\newcommand{\me}{{\boldsymbol{m}}}
\newcommand{\Me}{{\boldsymbol{M}}}
\newcommand{\Be}{{\boldsymbol{B}}}
\newcommand{\De}{{\boldsymbol{D}}}
\newcommand{\Npp}{{\boldsymbol{N}}}

\newcommand{\etap}{{\boldsymbol{\eta}}}
\newcommand{\zetap}{{\boldsymbol{\zeta}}}
\newcommand{\xip}{{\boldsymbol{\xi}}}
\newcommand{\colons}{\colon}

\newcommand{\diff}{\,\mathrm{d}}
\newcommand{\bdiff}{\mathrm{d}}
\newcommand{\mydot}{{\raisebox{.3ex}{$\scriptscriptstyle{\,\bullet\,}$}}}

\newcommand{\eps}{\varepsilon}
\renewcommand{\emptyset}{\varnothing}
\newcommand{\ttau}{\hat\modulus}
\newcommand{\suptau}{\bar\modulus}
\newcommand{\steiner}{\mathsf{s}}

\newcommand{\wto}{\overset{\scriptscriptstyle{w}}{\longrightarrow}}
\newcommand{\vto}[1][\sS]{\xrightarrow{\scriptscriptstyle{#1}}}
\newcommand{\dto}{\overset{\scriptscriptstyle{d}}{\longrightarrow}}

\DeclareMathOperator{\diam}{diam}

\DeclareMathOperator{\pr}{pr}
\DeclareMathOperator{\cl}{cl\,}
\DeclareMathOperator{\Int}{Int}

\DeclareMathOperator{\RV}{RV}
\DeclareMathOperator{\supp}{supp}
\DeclareMathOperator{\card}{card}
\DeclareMathOperator{\conv}{conv}

\newlength{\querylen}
\usepackage{fancybox}

\renewcommand{\paragraph}{\subsubsection}

\begin{document}

\title{Foundations of regular variation on topological spaces}

\author{Bojan Basrak, Nikolina Milin\v{c}evi\'c, Ilya Molchanov\\
University of Zagreb, University of Bern}

\date{\today}
\maketitle

\begin{abstract}
  Since its introduction by J. Karamata, regular variation has evolved
  from a purely mathematical concept into a cornerstone of theoretical
  probability and statistical data analysis. It is extensively studied
  and applied 
  in different areas. Its significance lies in characterising large
  deviations, determining the limits of partial sums, and predicting
  the long-term behaviour of extreme values in stochastic
  processes. Motivated by various applications, the framework of
  regular variation has expanded over time to incorporate random
  observations in more general spaces, notably Banach
  spaces and Polish spaces.
	
  In this monograph, we identify three fundamental components of
  regular variation: scaling, boundedness, and the topology of the
  underlying space. We explore the role of each component in detail
  and extend a number of previous results to general
  topological spaces. Our abstract approach unifies various
  concepts appearing in the literature, streamlines existing proofs,
  and paves the way for novel contributions. These include a generalised
  theory of (hidden) regular variation for random measures and sets
  and an innovative treatment of regularly varying random functions and
  elements scaled by independent random quantities, among other
  advancements.  Throughout the text, key results and definitions are
  illustrated by instructive examples, including extensions of several
  established models. By bridging abstraction with
  practice, this work aims to deepen both the theoretical
  understanding and the methodological applicability of regular variation.
\end{abstract}

\newpage

\tableofcontents

\newpage

\chapter{Preface}
\label{sec:introduction}

\textit{Regular variation in probability theory.}
Regular variation is recognised as a crucial analytical tool in both
theoretical and applied probability. It often serves as the primary model
for the tail behaviour of probability measures.  In the classical setting
of random variables on the real line, this property essentially means that
the probability measure of $[t,\infty)$ (that is, the probability that
a random variable is at least $t$) for large $t$ behaves like a
negative power of $t$ multiplied by a slowly varying terms (e.g.,
of logarithmic order). The foundations of this theory were laid by
J.~Karamata in the 1930s and later refined by L.~de Haan
\cite{haa70} and E.~Seneta \cite{seneta72}. N.~Bingham, Ch.~Goldie and
J.~Teugels \cite{bin:gol:teu} provide a comprehensive account of
regular variation in mathematics that ranges well beyond its
role in probability theory and includes applications in number
theory, complex analysis, and differential equations. Further
generalisations can be found in \cite{bul:kle}.

To date, the concept of regularly varying random variables is well
understood, along with its multivariate extension for random vectors;
see, for example,
\cite{Haan:Ferreira2010,res87,resnick:2007}. Starting with regular
variation in the space of continuous functions on $[0,1]$ studied by
\cite{haan-lin2001}, these classical constructions have been
extended to random elements in complete separable
metric spaces in \cite{hul:lin06},
\cite{lin:res:roy14} and \cite{seg:zhao:mein17}.
A recent book \cite{MR4926839}
provides an overview of extensions of regular
variation for functions, proposing a new definition of generalised
regular variation. Using the algebraic framework of Popa groups, the
authors effectively connect this new theory with the classical forms
of regular variation used in mathematical analysis. Their work,
however, does not address the regular variation of
probability distributions or measures, which is 
central to this monograph. 

The regular variation property
plays a key role in the theory of extreme values;
see \cite{haa70,lea:lin:roo86}.
It is well known that the regular variation property characterises
probability distributions belonging to the domain of attraction of
non-Gaussian stable laws in Hilbert spaces (see \cite{MR0345155}) and
in Banach spaces (see \cite{MR0521695}).  It is also applicable to
random elements in general convex cones, which are semigroups equipped 
with a scaling operation; see \cite{davydov08}. Unlike the aforementioned
work, we do not equip our carrier space with any algebraic
structure, thereby permitting an investigation of
regular variation \emph{per se}.

\smallskip

\textit{Vague convergence.} 
It is well understood that the concept of regular variation is
naturally related to the idea of \emph{vague convergence}, which
adapts the classical concept of weak convergence to infinite
measures by restricting them to a 
family of sets that are bounded in some specific sense.

Various approaches to extending regular variation and vague
convergence to general metric spaces have appeared in the literature.  One
is based on the concept of a modulus and was developed with
so-called star-shaped metric spaces in mind; see
\cite{seg:zhao:mein17} and also \cite{bladt:hash22}, who pushed this
machinery beyond the star-shaped spaces framework. Another approach is
based on an abstract concept of boundedness (or bornology); see
\cite{bas:plan19} and \cite{kul:sol20}, who developed
an approach to 
vague convergence closely related to the ideas in \cite{kalle17}.  As one of our
auxiliary, but potentially useful results, we show that these
two approaches are equivalent under mild topological assumptions.

Observe also that there exist earlier generalisations of regular
variation made by mimicking the classical constructions in Euclidean
spaces --- presented, e.g., in
\cite{mikosch24:_extrem_value_theor_time_series} and
\cite{resnick24:_art_findin_hidden_risks} --- by excluding the zero
point or a closed cone from the underlying separable metric
space; see \cite{haan-lin2001}, \cite{hul:lin06} and
\cite{lin:res:roy14}. However, it is not often recognised that the
resulting theory depends not only on the excluded cone, but also on
the collection of bounded sets in the residual space
remaining after this exclusion. In other words, along with what is
excluded, it is important to know how the exclusion is carried out. We
illustrate this phenomenon  with several simple examples; see, for instance,
Example~\ref{ex:moduli-2}, which deals with the very classical
situation of the joint regular variation of two independent Pareto
random variables.

Traditionally, generalisations of regular variation
have been mostly derived by replicating definitions available in
Euclidean spaces; for an overview, see
\cite{kul:sol20}. Therefore, it has typically been
assumed that the underlying space is metric and separable, or 
Polish. In the sequel, we show that metrisability of the carrier
space is not necessary for the study of the tail behaviour of measures;
in this way, we disentangle the roles of the metric
(which is only necessary to determine the topology) and scaling in the 
study of regular variation. Examples of non-metrisable spaces are
abundant, for instance, the space of continuous functions with pointwise
convergence (the $\Cp$-space, see \cite{tkach11}), 
infinite-dimensional Banach spaces with weak topologies, or duals to
Banach spaces with a weak-star
topology. Further important examples of non-Polish spaces arise as
continuous images of Polish spaces, which are (not necessarily
metrisable) Souslin spaces. Even in the context of metric spaces,
separability is violated in many natural examples, such as the space
of bounded real-valued sequences with the uniform distance or the
space of functions of bounded variation.

There are many motivations for generalising regular variation and the
related notion of vague convergence: (i) eliminating irrelevant assumptions
is natural and appealing from the mathematical point of view; (ii)
relying on the general concepts of scaling, ideal (boundedness) and topology
simplifies applications of the theory in various settings; and (iii)
many intriguing examples do not fit the standard theory and some insights
are only revealed after such a generalisation.

\smallskip

\textit{Role of topology.} In this work we build the general theory of
regular variation on topological spaces, starting from first
principles and relying on the idea of a measurable group action on a
topological space.  For a treatise on the weak convergence of bounded
measures on topological spaces we refer to
\cite{bogachev07,bogachev18}.  In the following we extend these ideas
to unbounded measures and define the corresponding concept of
\emph{vague convergence}. This substantially generalises the setting
of \cite{MR2271177}. A central role is played by a general
definition of vague convergence with respect to an ideal of sets,
which extends the recently introduced concepts of
boundedness (or bornology) on the space; see \cite{bas:plan19} and
\cite{kul:sol20}. To minimise topological assumptions, we work
with measures on the Baire $\sigma$-algebra on the carrier space. If
this space is perfectly normal (which is always the case if the space is
metrisable), the Baire and Borel $\sigma$-algebras coincide and all
the results below directly apply to the Borel setting.

Since we do not
require that the underlying space is sequential, in order to ensure
broader applicability of the defined vague convergence, we work with
convergence of nets, even though the regular variation
property is defined by taking limits of measures parametrised by a
real number. In view of this, some results require continuity of
measures with respect to increasing nets; such measures are usually
called $\tau$-additive. This property makes it possible to work with
measures on non-separable spaces. As a replacement for separability,
we rely on the hereditary Lindel\"of property, which means that every
open covering of any open set admits a countable subcovering.

\smallskip

\textit{Scaling.} The essential ingredients in the study of regularly
varying measures are scaling and functions that are homogeneous under
that scaling --- referred to as moduli or gauge functions. Such functions
provide information on how ``extreme'' or ``unusual'' an element of the
space is. Homogeneous measures (also called tail measures) arise as
limiting objects characterising the tail behaviour of regularly varying
measures. It is known that homogeneous measures naturally arise as
L\'evy measures of stable laws on semigroups; see
\cite{davydov08,evan:mol18}.  Generic representation results for
homogeneous measures on measurable spaces (without topological
assumptions) were obtained in \cite{evan:mol18} and subsequently
rediscovered in \cite{dombry18:_tail}.  In contrast to
\cite{evan:mol18}, we consider measures that approach
a homogeneous (tail) measure only in the limit.  The rate of decay of
the corresponding tail probabilities is called
the tail index and is denoted by $\alpha$.  We
systematically work with the case of a positive index
$\alpha$. The case of a negative $\alpha$ is easily obtained by
redefining the scaling or by a continuous mapping argument.

While our standing assumption is the Baire measurability of the
scaling operation, additional properties -- such as the continuity of the
scaling and the existence of a continuous modulus -- make it possible to
provide equivalent characterisations of the regular variation
property, which then become easier to check. Such restrictions are common in
other works, notably in \cite{bladt:hash22}, where continuity
is tacitly added starting in Section~3 as a standing assumption to
handle vague convergence.

\smallskip

\textit{Ideals.}
As explained in \cite{bas:plan19}, the definition of vague
convergence refers to a family of sets designated as
bounded sets. These sets form a bornology (or boundedness) on the
carrier space. A very similar idea, under the name of localisation, is
used in \cite{kalle17} in the context of Polish spaces. This concept
is a special case of a collection of sets called
an ideal (the dual of a filter), which
becomes a bornology if it covers the entire space.  The systematic use
of ideals simplifies and generalises results concerning vague
convergence, leads to a useful formulation of a continuous mapping
theorem, and makes it possible to view the ideal as a variable
parameter in the definition of regular variation. It also makes it
possible to keep the background topological space fixed and, instead
of excluding some 
cones from it as is common in \cite{lin:res:roy14}, work with different
ideals on the same space. As our examples show, it is important not
only what is excluded from the space, but also how bounded sets
approach the excluded region. We show that this effect appears even
in the plane. In contrast to previous work, we
also consider ideals that might contain some
scaling-invariant elements. This is particularly essential when working
with regular variation on the space of sets, since scaling-invariant 
families of closed sets are abundant.

A constructive definition of an ideal employs
a homogeneous function (modulus) on the carrier space. 
The simplest case arises
when the space does not contain scaling-invariant elements and admits
a strictly positive continuous modulus. We refer to such a modulus as
\emph{proper}. Its upper level sets  
naturally generate a bornology (or ideal); see
\cite{seg:zhao:mein17}.  However, there exist natural examples (e.g.,
$\R^\infty$ with the topology of pointwise convergence)
where a proper modulus does not exist. In the context of
c\`adl\`ag functions (right-continuous functions with left limits),
this was already 
observed in \cite{bladt:hash22}, which utilised a countable family of
moduli that are homogeneous functions, possibly vanishing on some
non-trivial elements of the carrier space. We advance and generalise this approach
by identifying the main features that explain this situation and clarifying
the relations between ideals and homogeneous functions. In particular,
we  establish the equivalence between the existence of a
continuous modulus and the topological properties of the ideal.

Furthermore, we present a construction for the product of ideals and
apply it to describe the phenomenon of hidden regular variation
without relying on a metric on the underlying space.
In the case of a metric space, we show that this approach recovers the
conventional 
construction based on excluding a neighbourhood of a given cone.

The equivalence of Baire and Borel $\sigma$-algebras in metric spaces
allows for the significant simplification of many proofs, highlighting the
roles of the separability and completeness assumptions.  In this context,
we recover known results on regular variation in Polish spaces.
Even here, our abstract approach
simplifies many existing results --- most
notably, the equivalence between the regular variation property and
the existence of a spectral measure arising from the polar
decomposition.

\medskip

\textit{Organisation and main new results.}
This monograph is organised as follows. Chapter~\ref{sec:scaling-bornology}
introduces the main topological concepts and the definitions of a modulus
and an ideal. We identify properties of an ideal which ensure that it
is generated by a continuous modulus. Furthermore, we
introduce the concept of products of ideals, which is essential for
formalising results on hidden regular variation. It also
provides a general representation of homogeneous
measures. While our main setting is that of Baire measures, the results
become applicable to Borel measures on perfectly normal (most
importantly, metric) carrier spaces, where Borel and Baire
$\sigma$-algebras coincide. 

Chapter~\ref{sec:vague-convergence-rv} 
begins by formalising vague convergence in general topological spaces
and proceeds to develop the theory of regularly varying measures. The
key r\^{o}le is played by functionally open and functionally closed
sets, which replace open and closed sets appearing in the classical
Portmanteau theorem. 
We discuss continuous mappings of regularly varying measures, the hidden regular 
variation phenomenon, and extensions to product spaces. In relation to
the latter point, our principal contribution in
Section~\ref{sec:scal-prod-spac}
is a significant generalisation of Breiman’s
lemma dealing with the product of a regularly varying random
element and a functional thereof. Furthermore, in
Section~\ref{sec:quotient-spaces-under} we present a 
comprehensive characterisation of regular 
variation for equivalence classes in quotient spaces. It should be
noted that passing to quotient spaces may lead to the loss of many
important topological features, and here our general setting of
regular variation in topological spaces is particularly useful.

In the setting of metric spaces adopted in
Section~\ref{sec:regul-vari-polish}, we develop an approximation
framework tailored to the analysis of regularly varying distributions. This
approach also yields a deeper understanding of regular variation in Banach and
in Hilbert spaces, extending recent works  in the
Hilbert space setting \cite{MR4745556} and  in Banach
spaces \cite{kim:kok24,janssen23}.

A substantial part of our work is contained in
Section~\ref{sec:examples}, which is devoted
to the application of these general tools in specific spaces. 
\begin{itemize}
\item We explain how various kinds of regular variation on
  the line are related to each other and to the continuous mapping argument.
\item The role of different ideals on Euclidean spaces in relation to
  regular variation is highlighted.
\item We characterise regular variation on the space of sequences
  with various topologies. 
\item We provide a formal proof of the regular variation of continuous
  random functions in relation to their finite-dimensional
  distributions and discuss the role of various ideals on the space of
  functions. 
\item We address the case of non-metrisable functional spaces, namely,
  the space of continuous functions with the topology of pointwise
  convergence. 
\item We discuss regular variation in spaces of functions that are not
  necessarily continuous, namely, c\`ad\`ag functions, upper
  semicontinuous functions and Sobolev spaces. 
\item We systematically explore regular variation of point processes
  (random counting measures) and obtain new characterisation results
  for the regular variation property.
  Here we cover the full
  spectrum of hidden regular variation for very general
  scaling schemes. This also leads to  substantially simplified
  proofs of several known results from \cite{dtw22}
  for marked point processes with
  scaling applied to the marks.  
\item We systematically study regularly varying random closed sets and
  connect them to the properties of functions and point
  processes. We obtain new results on the regular variation of random
  compact sets, including a substantial range of hidden regular
  variation properties. Further results relate the hidden regular
  variation of point processes to the hidden regular variation of random closed
  sets obtained as their supports.  
\end{itemize}

\chapter{Scaling, ideals, and homogeneous measures}
\label{sec:scaling-bornology}

\section{Group action}
\label{sec:group-acti-invar}

\paragraph{Topological spaces}
Let $\XX$ be a \emph{topological space}, that is, $\XX$ is a set with a
specified family $\sG$ of open subsets (called the topology)
such that $\XX$ and $\emptyset$
are open, and the family $\sG$ is closed under arbitrary unions and finite
intersections. The complements of the sets in $\sG$ constitute the family
$\sF$ of closed sets. For any set $B$, the intersection of all closed
sets containing $B$ is called the \emph{closure} of $B$ and the
\index{closure} \index{interior}
union of all open subsets of $B$ is the \emph{interior} of $B$. We
denote
the closure and the interior of a set $B$ in $\XX$ by
$\cl B$ and $\Int B$, respectively.

The space $\XX$ is \emph{Hausdorff}
\index{space!Hausdorff}
if any two distinct points have disjoint open neighbourhoods.
A stronger property is that any two distinct points can
be separated by a continuous function; this defines a \emph{completely
  Hausdorff} space.
\index{space!completely Hausdorff}
A space is said to be \emph{completely regular} 
\index{space!completely regular}
if every closed set and any point not contained in it
can be separated by a continuous function.  

A \emph{functionally closed}
\index{functionally closed set}
set is the inverse image of
a closed subset of $\R$, equivalently of
$\{0\}$, under
a continuous real-valued function, and a \emph{functionally open}
\index{functionally open set}
set is its complement. The space $\XX$ is said to be \emph{perfectly
  normal}
\index{space!perfectly normal}
if any two disjoint closed sets are precisely separated by a
continuous function, that is, there exists a function
$f:\XX\to[0,1]$ that takes the value $0$ exactly on the first set and the
value $1$ exactly on the second.
In such a space, every open (closed) set is functionally open
(closed).

Recall that $\XX$ is called a \emph{Polish space}
\index{space!Polish} \index{metric}
if it is metrisable by a
metric $\dmet$ under which it becomes a complete separable metric space.  Another
frequently imposed assumption is that $\XX$ is a \emph{Souslin space},
\index{space!Souslin}
which is a continuous image of a Polish space under a mapping to a
Hausdorff space. Equivalently, it is a space homeomorphic to an analytic
subset of a compact metric space; see, e.g.,
\cite[Definition~6.6.1]{bogachev07}.
Souslin spaces arise naturally in descriptive set theory and analysis,
where they serve as a generalisation of Polish spaces in contexts
involving Borel and analytic sets. 

The space $\XX$ is said to be \emph{first countable}
\index{space!first countable}
if each point has
a countable base of its neighbourhoods. This property implies that
$\XX$ is \emph{sequential},
\index{space!sequential}
meaning that its topological properties can be
expressed in terms of sequences. It is important to note that the
reverse implication, i.e., that every sequential space is
first countable, does not hold in general. First countability is a
stronger condition.

In order to deal with a non-sequential topological space $\XX$,
  one can use the concept of a \emph{net}.
\index{net} \index{limit} A
non-empty set $\Gamma$ is called \emph{directed} \index{directed set} if it
is partially ordered and for all $s,t\in \Gamma$ there exists an
$r\in \Gamma$ such that $s\leq r$ and $t\leq r$. For example, the
family of all open neighbourhoods of a point forms a directed set
ordered by the reverse inclusion. A map $\gamma\mapsto x_\gamma$ from
$\Gamma$ to $\XX$ is called a \emph{net}; see
\cite[Section~1.6]{eng89}. A net $x_\gamma$ is said to converge to
$x\in\XX$ (notation $x_\gamma\to x$) if, for each open set $U$
containing $x$, there exists a $\gamma_0\in\Gamma$ such that
$x_\gamma\in U$ whenever $\gamma_0\leq \gamma$. In a Hausdorff space,
each convergent net has at most one limit, denoted by $\lim_\gamma x_\gamma$. It
is well known that a point $x$ belongs to the closure of $A\subset\XX$
if and only if $x$ is a limit of a net consisting of elements of $A$;
see, e.g., \cite[Proposition~1.6.3]{eng89}.

A topological space $\XX$ is \emph{compact} \index{space!compact} if,
for any family of open sets that covers $\XX$, there is a finite
subfamily that also covers $\XX$. If $\XX$ is Hausdorff, this is
equivalent to the finite intersection property, which states that if
$\{F_s\}_{s\in S}$ is any family of closed sets such that each finite
subfamily has a non-empty intersection, then the entire family has a 
non-empty intersection; see \cite[Theorem~3.1.1]{eng89}. In 
Hausdorff spaces this is also equivalent to the fact that every net in
$\XX$ has a cluster point; see \cite[Theorem~3.1.23]{eng89}. A set $A\subset
\XX$ is compact if it is a compact space under the topology induced
from $\XX$. In Hausdorff spaces, compact sets are necessarily
closed. Furthermore, $A$ is \emph{relatively compact} if its closure
in $\XX$ is compact. A space $\XX$ is locally compact if each point has a
relatively compact neighbourhood.
\index{relatively compact set}
\index{space!locally compact}

Denote by $\sB(\XX)$ the \emph{Borel} $\sigma$-algebra on $\XX$ and by
$\sBA(\XX)$ the \emph{Baire} $\sigma$-algebra, which is the smallest
\index{Borel sigma-algebra@Borel $\sigma$-algebra}
\index{Baire sigma-algebra@Baire $\sigma$-algebra}
$\sigma$-algebra under which all continuous real-valued functions
$f:\XX\to\R$ are measurable. Since every continuous function is Borel
measurable, $\sBA(\XX)\subset\sB(\XX)$. In a perfectly normal (e.g.,
metric) space each open set is functionally open; consequently, the Borel
and Baire $\sigma$-algebras coincide. In particular, this is the case
for all Polish spaces. 

\paragraph{Scaling}
We always assume that a topological space $\XX$ is equipped with a
\emph{Baire measurable action} of the multiplicative group
$\Rpp$, that is, with a
$\sBA(\XX)\otimes\sB(\Rpp)/\sBA(\XX)$-measurable map
$(x, t)\mapsto T_t x$ from $\XX\times \Rpp$ to $\XX$ that
satisfies the following two conditions:
\index{scaling}
\begin{enumerate}[({A}1)]
\item for all $x\in \XX$ and $t, s>0$, $T_t(T_sx)= T_{ts} x$;
\item for all $x\in \XX$, $T_1 x = x$.
\end{enumerate}
The action of $T_t$ is called \emph{scaling} by $t$.  Henceforth, we
fix a generic scaling $T_t$ on the space $\XX$. For any $A\subset\XX$,
we denote the set $\{T_tx\colons x\in A\}$ by $T_tA$.

The assumptions on the scaling imply that each map $T_t$ is a
bijection. The scaling is
said to be \emph{continuous} if the map $(x,t)\mapsto T_tx$
is continuous on $\XX\times\Rpp$
with the product topology.
\index{scaling!continuous}
The continuity assumption immediately
implies the Baire measurability property, since the Baire
$\sigma$-algebra makes all continuous functions measurable. We use the term
\emph{space-continuous} for scalings that are continuous in the variable $x$
for every fixed $t\in\Rpp$.
\index{scaling!space-continuous}

A Borel measurable scaling is defined similarly by replacing the Baire
$\sigma$-algebra with the Borel one.  The space $\XX$ equipped with a Borel
measurable scaling is called a \emph{measurable cone};
\index{measurable cone} 
see \cite{bladt:hash22}.  On a Polish space $\XX$, the Baire measurability
and Borel measurability properties coincide. Furthermore, the scaling is
continuous if it is separately continuous in its two arguments, that
is, the map $x\mapsto T_tx$ is continuous on $\XX$ for all
$t\in\Rpp$ and the map $t\mapsto T_tx$ is continuous on $\Rpp$
for all $x\in\XX$; see \cite[Theorem~9.14]{MR1321597}.

The main example of scaling is the \emph{linear scaling}. In this case, 
$\XX$ is a topological vector space over the real numbers and the
scaling is defined as the multiplication by a positive scalar.
In this case, we write $tx$ instead
of $T_tx$ and denote this scaling using the usual product notation
$\mydot$.
\index{linear scaling}

\begin{lemma}
  \label{lemma:scaling}
  Fix an arbitrary $t>0$. If the scaling on a topological space
  is space-continuous, then the map
  $x \mapsto T_t x$ is both open and closed, and 
  $\cl(T_t B)=T_t\cl B$, $\Int T_tB=T_t\Int B$ for all $B\subset\XX$.
  Furthermore, $T_tG$ is functionally open whenever $G$ is functionally open.
\end{lemma}
\begin{proof}
  If $A$ is open in $\XX$, then $T_tA=T_{t^{-1}}^{-1} A$ is also
  open, being the inverse image of $A$ under the continuous map
  $T_{t^{-1}}$. The same argument shows that $T_tA$ is closed whenever $A$
  is closed.
	
  The set $T_t\Int B$ is an open subset of $T_tB$, so
  $T_t\Int B\subset\Int T_tB$. Conversely, if
  $x\in \Int T_tB$, then an open neighbourhood $U$ of $x$ is a subset
  of $T_t B$. Then $T_{t^{-1}}U\subset B$ is an open neighbourhood of
  $T_{t^{-1}}x\in B$, implying that $T_{t^{-1}}x\in \Int B$, or equivalently,
  $x\in T_t\Int B$.
	
  Regarding the closure, $T_t\cl B$ is a closed set containing $T_tB$,
  hence, $T_t\cl B\supset\cl(T_tB)$. To show the reverse inclusion,
  assume that $x\in 
  T_t\cl B$ and consider an open neighbourhood $U$ of $x$. Since
  $T_{t^{-1}}U$ is an open neighbourhood of $T_{t^{-1}}x\in\cl B$, it
  must intersect $B$. Consequently, $U$
  intersects $T_t B$, which implies that $x\in\cl(T_tB)$.

  If $G=\{x\colons f(x)>0\}$ for a continuous real-valued function $f$, then
  $T_tG=\{x\colons f(T_{t^{-1}}x)>0\}$ is also functionally open. 
\end{proof}

\paragraph{Cones, semicones, and scaling-invariant elements}
For $x\in\XX$, $A\subset\XX$, and $I\subset \Rpp$, we denote
$T_I x = \{T_t x\colons t\in I\}$, $T_t A = \{T_t x: x\in A\}$ and
$T_I A = \{T_t x\colons t\in I, x\in A\}$.  The set
$T_{(0,\infty)}x=\{T_tx\colons t\in(0,\infty)\}$
is called the \emph{orbit} of $x\in\XX$, representing
the set of scaled values of $x$ under the scaling action for 
all $t>0$.
\index{orbit}
If $\XX$ is Polish, every
orbit is Borel; see \cite[Section~15.14]{MR1321597}.
Due to the measurability of the scaling, $T_tB$ is a Baire set for every Baire
set $B$ and all $t>0$. This follows from the fact that $T_tB=T^{-1}_{t^{-1}}B$. 

\begin{definition}
  \label{def:cone}
  \index{cone} \index{semicone}
  A \emph{cone} in $\XX$ is a set $\cone\subset\XX$ such that
  $T_t\cone=\cone$ for all $t>0$.  A \emph{semicone} $\VV\subseteq\XX$ is a
  set such that $T_t \VV \subseteq \VV$ for all $t\geq 1$, equivalently,
  $T_{[1,\infty)}\VV=\VV$.
\end{definition}

Note that the entire space 
$\XX$ is trivially a cone, as scaling the entire space by any positive
factor will result in the space itself. 
If $T_tx\neq x$ for all $x\in\XX$ and $t\neq1$, then the action of the
Polish group $(\Rpp,\times)$ on $\XX$ is said to be \emph{free},
meaning that no element remains fixed under any non-trivial scaling; 
see \cite[Sec.~15.D]{MR1321597}.
\index{free action}
If this is not the case and $\XX$
possesses elements which are $T_t$-invariant for certain $t\neq 1$, then
these elements are said to be \emph{scaling-invariant} and are called
\emph{zeros}.
\index{zeros}
\index{scaling-invariant elements}
While in many cases all zeros are invariant under $T_t$
for all $t>0$, Remark~\ref{rem:sets-scaling} describes the situation
when there are elements invariant under $T_t$ only for some $t\neq1$
and then for all powers of this $t$. We denote by $\zero$ the set of zero
elements; in particular, $\zero$ is a singleton if the zero element is
unique or is empty if there are no zero elements. If $\XX$ is a
topological vector space with linear scaling, then $\zero=\{0\}$
is the singleton, being the zero vector in $\XX$.

\begin{lemma}
  \label{lemma:invariant}
  If the scaling on a topological space is continuous, then $\zero$ is
  an $F_\sigma$ set (that is, a countable union of closed sets).  If,
  additionally, $T_tx=x$ for at least one $t\neq1$ implies that
  $T_tx=x$ for all $t>0$, then $\zero$ is a closed set.
\end{lemma}
\begin{proof}
  For $\eps\in(0,1/2)$, let $A_\eps$ be the set of $x\in\zero$ such
  that $T_tx=x$ for some $t\in[\eps,1-\eps]$. Assume that $(x_\gamma)$
  is a net in  $A_\eps$ such that $x_\gamma\to x$. Then there exist
  $t_\gamma\in[\eps,1-\eps]$ such that
  $T_{t_\gamma}x_\gamma=x_\gamma$. By
  passing to a subnet, we may assume that $t_\gamma\to 
  t\in[\eps,1-\eps]$. The continuity of the scaling implies that
  $T_{t_\gamma}x_\gamma\to T_tx$; since we also have
  $T_{t_\gamma}x_\gamma=x_\gamma\to x$, it follows that 
  $T_tx=x$, and thus $x\in A_\eps$. Therefore,
  $A_\eps$ is closed.  Finally, $\zero$ is an $F_\sigma$ set because it
  can be expressed as the union of the sets 
  $A_{1/m}$ for $m\ge3$.

  For the second statement, let $(x_\gamma)$ be a
  net of zero elements that converges to some $x\in \XX$. By the
  additional assumption, for each $\gamma$ and every $t>0$, we have
  $x_\gamma=T_tx_\gamma$. Taking the limit in $\gamma$, we obtain $T_t
  x=x$, which implies $x\in\zero$.
\end{proof}

The standing assumption in all previous work (see, e.g.,
\cite{bladt:hash22,dombry18:_tail,kul:sol20}) is that $\zero$ is
either a singleton or empty.
If $\XX$ is a Polish space with metric $\dmet$
and a continuous scaling extended to $t\in[0,\infty)$ such that
$\zero=\{0\}$ is a singleton and $T_0x=0$ for all $x\in\XX$, and if
$\dmet(0, x)<\dmet(0,T_tx)$ for all $t>1$ and $x\neq0$, then $\XX$
with this scaling is called a \emph{star-shaped} metric space;
see \cite{seg:zhao:mein17}.
\index{star-shaped metric space}
\index{space!metric star-shaped}

\section{Modulus}
\label{sec:modulus}

\paragraph{Homogeneous functions and proper moduli}

A function $f:\XX\to[-\infty,\infty]$ is said to be \emph{homogeneous}
if
\index{homogeneous function}
\begin{displaymath}
  f(T_tx)=t f(x),\quad t>0,\; x\in\XX.
\end{displaymath}
If $\XX$ contains scaling-invariant elements, then any homogeneous
function must either vanish or be infinite on them.

\begin{definition}
  \label{def:modulus}
  A \emph{modulus} (also called gauge or pseudonorm) on $\XX$ is a
  homogeneous Baire measurable function $\modulus:\XX\to [0,\infty]$,
  where $[0,\infty]$ is endowed with its Borel $\sigma$-algebra. A modulus
  is said to be \emph{proper} if $\modulus(x)\in(0,\infty)$ whenever
  $x$ is not scaling invariant. For $A\subset\XX$, we define
  \index{modulus} \index{proper modulus}
  \begin{equation}
    \label{eq:12}
    \ttau(A)=\inf\{\modulus(x)\colons x\in A\},
  \end{equation}
  and we let $\sS_\modulus$ be the family of all sets $A\subset\XX$ such that
  $\ttau(A)>0$. By convention, we set $\ttau(\emptyset)=\infty$.
\end{definition}

Due to the homogeneity property,
the concept of a modulus is intrinsically tied to the
scaling operation on the carrier space.
Note that Baire measurability of the modulus implies its Borel
measurability. If a modulus is continuous (with respect to the standard
topology on $[0,\infty]$), then it is obviously Baire measurable. If
$\XX$ is a subset of a normed linear space, then the norm is clearly a
continuous modulus.  In the following, we usually write $\{\modulus>t\}$
instead of $\{x\colons \modulus(x)>t\}$, with analogous interpretations of
$\{\modulus\in(0,\infty)\}$, $\{\modulus\geq t\}$ and
$\{\modulus=t\}$.  Unlike \cite{seg:zhao:mein17}, we allow a
modulus to vanish.  In certain cases, it is useful to allow the
modulus to take infinite
values; see \cite{bladt:hash22} in the context of 
Polish spaces.

\begin{lemma}
  \label{lemma:continuity-homogeneous}
  Assume that the scaling on a topological space
  is space-continuous. Then 
  a modulus $\modulus$ is continuous if and only if $\{\modulus<t\}$
  is open and $\{\modulus\leq t\}$ is closed for some
  $t\in(0,\infty)$, in which case these properties hold for all $t\in(0,\infty)$. 
\end{lemma}
\begin{proof}
  \textsl{Necessity} is immediate. For \textsl{sufficiency}, suppose
  the condition holds for some $t$. Since $\modulus$ is homogeneous,
  we have $\{\modulus<st\}=T_s\{\modulus<t\}$ and $\{\modulus\leq
  st\}=T_s\{\modulus\leq t\}$ for any $s>0$. By 
  Lemma~\ref{lemma:scaling}, the space continuity of the scaling
  implies that $T_s$ is an open and closed map. Thus,
  $\{\modulus<st\}$ is open and $\{\modulus\leq
  st\}$ is closed for all $s>0$, which establishes the continuity of
  $\modulus$. 
\end{proof}

\paragraph{Transversal and polar decomposition}

Given a modulus $\modulus$, we define
\begin{equation}
  \label{eq:sphere}
  \SS_\modulus=\{x\in\XX\colons \modulus(x)=1\}.
\end{equation}
\index{transversal}
Note that every $x\in\XX$ with $\modulus(x)\in(0,\infty)$ can be
uniquely represented as
$x=T_tu$ for a unique $u\in\SS_\modulus$ and $t\in(0,\infty)$.
A Borel set $\SS$ in $\XX$ is said to be a
\emph{transversal} if every point $x\in\XX$ can be uniquely represented as
$x=T_tu$ for some $t\in\Rpp$ and $u\in\SS$. Consequently, $\SS_\modulus$ is
a transversal on $\{\modulus\in(0,\infty)\}$.

\begin{definition}
  \label{def:polar-decomposition}
  Assume that 
  $\modulus$ is a modulus on $\XX$.  The \emph{polar decomposition} is
  the map \index{polar decomposition}
  $\rho:\{\modulus\in (0,\infty)\}\to\SS_\modulus\times\Rpp$
  given by
  \begin{displaymath}
    \rho(x) = \big(T_{\modulus(x)^{-1}}x, \modulus(x)\big).
  \end{displaymath}
\end{definition}

The polar decomposition is unique once the modulus, and consequently the
corresponding transversal $\SS_\modulus$, is fixed. 
If both the modulus and the scaling are continuous, then the
polar decomposition is a homeomorphism between
$\{\modulus \in(0,\infty)\}$ and $\SS_\modulus\times\Rpp$ with the
product of the topology induced on $\SS_\modulus$ and the Euclidean
topology on $\Rpp$.  In the
absence of continuity, one has to impose some stronger topological
properties on the spaces. For instance, if $\XX$ is Polish and the
modulus is not necessarily continuous, then the polar decomposition is
a measurable bijection between $\{\modulus\in(0,\infty)\}$ and
$\SS_\modulus\times\Rpp$ endowed with their respective Borel
$\sigma$-algebras.  For the sake of completeness, we provide the proof
of the following result from \cite[Lemma~3.3(iii)]{evan:mol18} for the case
of a Souslin space $\XX$. 

\begin{lemma}
  Assume that the scaling acts freely on a completely regular Souslin
  space $\XX$. Then a transversal exists if and only if there exists a
  proper modulus $\modulus$ on $\XX$ such that \eqref{eq:sphere} holds.  The
  map $x\mapsto(T_{\modulus(x)^{-1}}x,\modulus(x))$ is
  $\sB(\XX)/\sB(\SS)\otimes\sB(\Rpp)$-measurable, where $\sB(\SS)$ is
  the Borel $\sigma$-algebra induced on $\SS$.
\end{lemma}
\begin{proof}
  The topological assumptions on $\XX$ ensure that $\XX$ is perfectly
  normal, see \cite[Theorem~6.7.7]{bogachev07}, so that the Baire and
  Borel $\sigma$-algebras coincide.
  
  Assume that a transversal $\SS$ exists. If $x=T_tu$ for $x\in\XX$
  and $u\in\SS$, we set $\modulus(x)=t$. Since $\SS$ is Borel, the space
  $\SS\times\Rpp$ is Souslin. The map $(x,t)\mapsto T_tx$ is an
  injective Borel map between two Souslin spaces, hence, its inverse is
  Borel; see \cite[Theorem~6.7.3]{bogachev07}.

  By the definition of the transversal and the bijectivity of the
  scaling, for any $x\in\XX$, we have
  \begin{align*}
    \{\modulus(T_sx)\}
    &=\{t\colons T_sx=T_tu\; \text{for some}\; u\in\SS\}\\
    &=\{t\colons x=T_{s^{-1}}T_tu\; \text{for some}\; u\in\SS\}\\
    &=\{st'\colons x=T_{t'}u\; \text{for some}\; u\in\SS\}\\
    &=\{s\modulus(x)\}, \quad s\in\Rpp.
  \end{align*}
  Since no $x\in\SS$ is scaling invariant, 
  \begin{displaymath}
    \{\modulus(x)\}=\{t\colons x=T_tu\; \text{for some}\; u\in\SS\}=\{1\}. 
  \end{displaymath}
  Thus, $\modulus$ is indeed a modulus, and the polar representation map is
  measurable. 

  Now assume the existence of a modulus $\modulus$. Since $\modulus$ is Borel,
  $\SS$ is Borel. For each $x\in\XX$, letting $u=T_{\modulus(x)^{-1}}x$
  shows that $x=T_{\modulus(x)}u$ with $u\in\SS$. If $x=T_tu=T_sv$ for
  $u,v\in\SS$ and $t,s\in\Rpp$, then $\modulus(T_tu)=\modulus(T_sv)$,
  which implies $t=s$. It then follows that $u=T_{t^{-1}}x=v$.
\end{proof}

The following result collects several properties of measurable
scaling on Souslin spaces. Recall that all results concerning Souslin spaces
apply to Polish spaces as well. 

\begin{lemma}
  Assume that the scaling acts freely on a Souslin space $\XX$ and
  that $\modulus$ is a modulus on $\XX$ with transversal $\SS$.
  \begin{enumerate}[i)]
  \item For each Borel set $A$ in $\XX$, its image $T_tA$ is Borel
    for all $t\in\Rpp$.
  \item For all $I\in\sB(\Rpp)$ and $A\in\sB(\SS)$, we have
    $T_IA\in\sB(\XX)$.
  \item For each $B\in\sB(\XX)$ and $I\in\sB(\Rpp)$, the set $T_IB$ is Souslin.
  \end{enumerate}
\end{lemma}
\begin{proof}
  i) By \cite[Theorem~6.7.3]{bogachev07}, $T_tA$ is Souslin. Let 
  $A^c=\XX\setminus A$ be the complement of $A$, which is also a Borel
  set. Since $T_t$ is a
  bijection on $\XX$, the set $T_tA^c$ is Souslin and constitutes the complement
  of $T_tA$. Since both $T_tA$ and its complement are Souslin, it
  follows that $T_tA$ is Borel; see \cite[Corollary~6.6.10]{bogachev07}.

  ii) It suffices to observe that
  $T_IA=\{x\colons (T_{\modulus(x)^{-1}}x,\modulus(x))\in A\times I\}$ and
  invoke the measurability of the polar decomposition.

  iii) Following the approach in
  \cite[Lemma~2]{evan:mol18}, let $I^{-1}=\{t^{-1}\colons t\in I\}$. Then,
  \begin{displaymath}
    T_IB=\{y\in\XX\colons T_ty\in B\;\text{for some}\; t\in I^{-1}\}
  \end{displaymath}
  is the projection on $\XX$ of the set $\{(y,t)\colons T_ty\in
  B\}\cap (\XX\times I^{-1})$. 
  Since the scaling is jointly measurable, the set of $(y,t)$ such
  that $T_ty\in B$ is Baire and therefore Borel. It suffices to note that
  $I^{-1}$ is Borel, so that $T_IB$ is a projection of a
  Borel 
  set in $\XX\times\Rpp$. By \cite{rog:wil96} and
  \cite[Theorem~6.7.2]{bogachev07}, such a projection is necessarily Souslin.
\end{proof}

For a point $x\in\XX$ with $\modulus(x)\in(0,\infty)$, we call 
$T_{\modulus(x)^{-1}}x$ the directional component  and
$\modulus(x)$ the radial component of $x$. 
If $A\subset\SS_\modulus$, then $T_{(t,s)}A$ is the set of
all points in $\XX$ whose directional component lies in $A$ and radial
component belongs to the interval $(t,s)$. 
The following result implies that every open set in $\sS_\modulus$
(as defined in Definition~\ref{def:modulus}) 
can be represented as a union of sets of the form
$T_{(t,s)}(U\cap \SS_\modulus)$ for
functionally open sets $U$. Throughout the remainder of this section,
the scaling is assumed to be continuous. 

\begin{lemma}
  \label{lemma:sector}
  Assume that $\XX$ is a completely Hausdorff topological space 
  equipped with a continuous scaling and that $\modulus$ is a
  continuous modulus. Then, for each open set $G$ and each $x\in G$
  with $\modulus(x)\in(0,\infty)$, there exist a functionally open set
  $U$ and $0<t<s$ such that
  \begin{displaymath}
    x\in
    \Big\{y\colons \modulus(y)\in(t,s), T_{\modulus(y)^{-1}}y\in U\Big\}\subset G.
  \end{displaymath}
\end{lemma}
\begin{proof}
  Let $\Gamma$ be the directed set of all functionally open
  sets $U$ containing $x$ partially ordered by the reverse inclusion.  The
  complete Hausdorff property implies that $x$ is the intersection
  of all $U\in\Gamma$. Without loss of generality, assume that
  $\modulus(x)=1$.  Since the scaling is continuous, the set
  $\{(y,t)\colons T_ty\in G\}$ is open in the product space
  $\XX\times\Rpp$ and contains $(x,1)$. Consequently, it also contains
  a product set of the form
  $\{x\}\times [1-\eps,1+\eps]$ for some $\eps\in(0,1)$, implying that
  $T_s x\in G$ for all $s\in[1-\eps,1+\eps]$. Fix this $\eps$.

  Arguing by contradiction, assume that for all functionally open
  sets $U\in\Gamma$, we have $T_{[1-\eps,1+\eps]}U\not\subset
  G$. Then there exists a net $(y_U)_{U\in\Gamma}$ such that
  $y_U\in U$ and $T_{s_U}y_U\notin G$ for some $s_U\in[1-\eps,1+\eps]$. By
  passing to a subnet, we may assume that $s_U\to s\in[1-\eps,1+\eps]$. Since
  $x$ is the intersection of its neighbourhoods, it follows that
  $y_U\to x$. The continuity of the scaling implies that
  $T_{s_U}y_U\to T_sx$. Since $G^c$ is closed, we must have $T_sx\notin G$, which is a
  contradiction.  Thus, $T_{[1-\eps,1+\eps]}U\subset G$ for some
  functionally open set $U$. Note that $x\in U$ by the construction of
  $\Gamma$.
  Therefore, 
  \begin{displaymath}
    x\in U=T_{(1-\eps,1+\eps)}\big(U\cap\SS_\modulus\big)\subset G. \qedhere
  \end{displaymath}
\end{proof}

If the scaling is continuous, the saturation $T_{(0,\infty)}$ of any
open set $G$ is open.  \index{saturation} Additionally, if the orbit
$T_{(0,\infty)}x$ of every $x$ is closed and if the scaling acts freely
on a Polish space $\XX$, then a (Borel) transversal exists; see
\cite[Theorem~12.16]{MR1321597}.

\paragraph{Modulus on a star-shaped metric space}

Let $\XX$ be a star-shaped metric space with $\zero=\{0\}$.  Recall
that, in this case, the scaling extended to $[0,\infty)$ is assumed to
be continuous.  The following result provides an explicit
construction of a transversal from a metric ball centred at $0$.

\begin{lemma}
  \label{lemma:tau-from-ball}
  Let $\XX$ be a star-shaped Polish space with metric $\dmet$.  Assume
  that there exists an $\eps>0$ such that
  \begin{displaymath}
    \sup\{\dmet(0,T_tx)\colons t\in\Rpp\}\geq \eps
  \end{displaymath}
  for all $x\neq 0$. Then a transversal exists. If $0$ has a
  relatively compact neighbourhood, then the corresponding modulus is
  continuous.
\end{lemma}
\begin{proof}
  Let $\SS$ be the boundary of the ball $B_r(0)$ centred at $0$ with
  radius $r\in(0,\eps)$. If the orbit $\{T_tx\colons t\in\Rpp\}$ of some
  $x\neq 0$ were to intersect $\SS$ at two distinct points, say $T_tx$ and
  $T_sx$ with $t<s$, then $\dmet(0,T_tx)<\dmet(0,T_sx)$ by the
  definition of a star-shaped space, which is a contradiction.

  It remains to show that the orbit of any $x\neq0$ intersects
  $\SS$. If this were not the case, the continuity of the function
  $\dmet(0,T_tx)$ in $t$ and the fact that $\dmet(0,T_tx)\to0$ as
  $t\downarrow0$ would imply that
  \begin{displaymath}
    \sup\{\dmet(0,T_tx)\colons t\in\Rpp\}<r,
  \end{displaymath}
  contrary to the choice of $\eps$.  Note that $\SS$ is closed and, consequently,
  is compact for a sufficiently small $r>0$ if $0$ has a relatively 
  compact neighbourhood.

  Assume that $x_n=T_{t_n}u_n\neq 0$ and $x_n\to x\neq0$ as
  $n\to\infty$. Since $\SS$ is compact, we may assume, by passing to a subsequence,
  that $u_n\to u\in\SS$ as $n\to\infty$. If $t_{n_k}\to 0$ as
  $k\to\infty$, then
  \begin{displaymath}
    T_{t_{n_k}}u_{n_k}\to 0\quad \text{as}\; k\to\infty
  \end{displaymath}
  by the continuity of the scaling; thus, no subsequence of $\{t_n\}$
  converges to zero. If $t_{n_k}\to\infty$ as $k\to\infty$, then
  $u=T_{t_{n_k}^{-1}}x_{n_k}\to 0$, which is also a
  contradiction. Finally, if $t_{n_k}\to t$, then
  \begin{displaymath}
    T_{t_{n_k}}u_{n_k}\to T_t u=x,
  \end{displaymath}
  implying $t=\modulus(x)$. It follows that
  $t_n\to\modulus(x)$, so that the modulus is continuous.

  It remains to consider the case $x_n\to0$. Write $x_n=T_{t_n}u_n$
  with $u_n\in\SS$. If $t_n$ did not converge to zero, then, along a
  subsequence, either $t_n\to t\in(0,\infty)$ or $t_n\to\infty$. In
  the first case, compactness of $\SS$ and continuity of the scaling
  would give a non-zero limit. In the second case, the star-shaped
  property gives $\dmet(0,T_{t_n}u_n)\geq r$ for all sufficiently
  large $n$, again a contradiction. Hence $t_n\to0$, which proves
  continuity at zero.
\end{proof}

\section{Ideals and bornologies}
\label{sec:bornology}

\paragraph{Definition and basic properties of ideals}

An \emph{ideal} $\sS$ is a family of subsets of $\XX$ such that
\index{ideal}
\begin{enumerate}[1)]
\item $\sS$ is closed under finite unions; and 
\item $\sS$ is hereditary under inclusion, that is, for each
  $B\in\sS$, the family $\sS$ contains all subsets of $B$;
\end{enumerate}
see \cite{gier:hof:keim:law:mis:scot80}. Note that an ideal might, and
usually does, contain non-Borel sets. The family of complements of
the sets in $\sS$ is a \emph{filter}.
\index{filter}
If $\sS$ contains a countable
subfamily $\{\base_n,n\in\NN\}$ that is cofinal in $\sS$ (meaning that each
element of $\sS$ is a subset of a certain member of this subfamily),
then $\sS$ is said to have a \emph{countable base}.
\index{countable base}
It is said to have a \emph{countable open base} if all sets $\base_n$ are functionally
open.  Without loss of generality, we may assume that
$\base_n\subset \base_{n+1}$ for all $n$. A collection of sets
generates the ideal that is the intersection of all ideals which
contain this collection. 

An ideal $\sS$ is called \emph{scaling consistent} if, for every $B\in\sS$,
\index{ideal!scaling consistent}
we have $T_tB\in\sS$ for all $t>0$, where $T_t$ denotes the
scaling map on $\XX$.
If $\sS$ is scaling consistent, the union $\UU=\cup_{B\in\sS} B$ of
all sets in $\sS$ is a cone in $\XX$.  An ideal is
\emph{topologically consistent} if, for every $B\in\sS$, the ideal $\sS$
\index{ideal!topologically consistent}
also contains a functionally closed set $F\supset B$ and $\sS$
contains a functionally open neighbourhood of each $x\in \UU$.

\begin{lemma}
  \label{lemma:cont-base}
  Assume that $\sS$ is a topologically consistent ideal with a
  countable open base $\{\base_n,n\in\NN\}$. Then there exists a
  family $\{G_t,t\geq 1\}$ of functionally open sets such that, for all
  $t<s$, there exists a functionally closed set $F$ such that 
  $G_t\subset F\subset G_s$ and $G_n=\base_n$ for all $n\in\NN$. 
\end{lemma}
\begin{proof}
  Since $\sS$ is topologically consistent, for each $n\in\NN$, the set
  $\base_n$ is a subset of a functionally closed set in $\sS$, which
  in turn is a subset of $\base_m$ for some $m>n$. Thus, it is possible
  to remove some sets from the base to ensure that
  $\base_n\subset F_n\subset\base_{n+1}$ with a functionally closed
  $F_n$ for all $n\in\NN$.  For each $n\in\NN$, let $G_n=\base_n$,
  $G_{n+1}=\base_{n+1}$. We construct $G_t$ for $t\in(n,n+1)$ as
  follows. The functionally closed sets $F_n$ and $\base_{n+1}^c$ are
  disjoint and so can be separated by a continuous function
  $f:\XX\to[0,1]$, which takes value $0$ exactly on $F_n$ and the
  value $1$ exactly on $\base_{n+1}^c$; see \cite[Theorem~1.5.13]{eng89}. If
  $t=n+\eps$ for $\eps\in(0,1)$, let $G_t=\{f<\eps\}$.
  The condition then holds with $F=\{f\leq \eps\}$.
\end{proof}

\begin{remark}
  \label{rem:top-consistent}
  If an ideal $\sS$ covers $\XX$, it is called a \emph{bornology} on
  \index{bornology} \index{boundedness}
  $\XX$; see \cite{hog77}. Bornologies are predominantly used in
  the context of topological vector spaces. 
  The sets in $\sS$ are said to be
  \emph{bounded}, and so a bornology is sometimes called a
  \emph{boundedness}; see \cite{bas:plan19}, \cite[Section~V.5]{hu66}
  and \cite{kul:sol20}. Every ideal $\sS$ is a bornology on its union
  $\UU=\cup_{B\in\sS} B$.  A \emph{countable base} of a bornology on $\XX$
  is called a \emph{localising sequence}; see, e.g.,
  \index{localising sequence}
  \cite[Section~1.1]{kalle17}.  If a topologically consistent
  bornology admits a countable open base $(\base_n)_{n\in\NN}$, then,
  without loss of generality, we can assume that the closure of
  $\base_n$ is a subset of $\base_{n+1}$. In this case, the authors of
  \cite{bas:plan19} say that $\sS$ admits a \emph{properly localising
    sequence} and the space $\XX$ is \emph{localised}.
  \index{space!localised}
  Note that the
  functionally open/closed sets do not appear in the cited works,
  which adopt the setting of Polish spaces. It is easy to see that the
  localisation condition (the existence of a properly localising
  sequence) implies that the bornology is topologically consistent if
  the underlying space is perfectly normal.
\end{remark}

The names of the following two bornologies stem from
\cite[Definition~2.5]{bor:fit95}. Further examples of bornologies can
be found in \cite{hog77}. 

\begin{example}[Fr\'{e}chet bornology]
  The family of all metrically bounded sets in a metric space
  $(\XX,\dmet)$ is called the \emph{Fr\'{e}chet bornology}.
  \index{Frechet bornology@Fr\'{e}chet bornology}
  A bornology on a Polish space can be realised as the Fr\'echet
  bornology via a metric that preserves the topology on the
  underlying space if and only if $\sS$ has a countable open base and
  is topologically consistent; see \cite[Corollary~5.12]{hu66} and
  \cite[Theorem~2.1]{bas:plan19}. 
  While this fact is used in \cite{bas:plan19} and
  \cite[App.~B]{kul:sol20}, we do not rely on it, since our carrier
  space is not necessarily metrisable. Note that there are uncountably
  many ways to realise $\sS$ as the Fr\'echet bornology; see
  \cite{beer99}. Furthermore, the Fr\'echet bornology is not
  necessarily scaling consistent. For this, we need to ensure that
  any scaling of a bounded set remains bounded. For instance, this holds
  if the scaling $x\mapsto T_tx$ is Lipschitz in $x$ for all $t$, that is,
  \begin{displaymath}
    \dmet(T_tx,T_ty)\leq c_t\dmet(x,y), \quad x,y\in\XX,
  \end{displaymath}
  with a constant $c_t$ that may depend on $t$. 
\end{example}

\begin{example}[Hadamard bornology]
  \label{rem:Hadamard}
  It is easy to see that each bornology with a countable open base
  contains all compact subsets of the underlying space.
  \index{Hadamard bornology} The \emph{Hadamard bornology} on $\XX$ is
  the smallest ideal containing all relatively compact sets. It is
  also known as the compact bornology. If the scaling is
  space-continuous and $\XX$ is a locally compact separable space, then the
  Hadamard bornology is topologically and scaling consistent and has a
  countable open base. For instance, in $\R^d\setminus\{0\}$ with
  the Euclidean topology the Hadamard bornology consists of all
  bounded sets whose closure does not intersect a neighbourhood of the
  origin.
\end{example}

\begin{example}[Bornologies on topological vector spaces]
  Assume that $\XX$ is a topological vector space.
  \index{von Neumann bornology}
  A set $A$ belongs
  to the \emph{von Neumann bornology} if each neighbourhood $U$ of $0$
  absorbs $A$, that is, there exists a $c>0$ such that $T_tU\supset A$
  for all $t\geq c$. A set $A$ belongs to the \emph{precompact
    bornology} if for each neighbourhood $V$ of $0$ there exist
  finitely many points $x_1,\dots,x_k\in\XX$ such that $A$ is a subset
  of the union of $(x_i+V)$, $i=1,\dots,k$; see \cite[Section~1.3]{hog77}.
\end{example}

The bornologies defined above follow the intuitive idea of including
in the bornology sets that are bounded in one of the classical senses,
which may be also expressed as the fact that sets are bounded away
from infinity. The following important example defines the idea of sets
that are bounded away from a given point or set in $\XX$.

\begin{example}[Exclusion ideals]
  \label{example:exclusion}
  Let $\cone$ be a closed cone in $\XX$.  All subsets of $\XX$ that
  do not intersect a functionally open neighbourhood of $\cone$ form 
  an ideal called a \emph{(topological) exclusion ideal}.
  \index{exclusion ideal!topological} The union of its sets is
  $\XX\setminus \cone$.  The exclusion ideal is topologically
  consistent if $\XX$ is a completely regular space (so that a point
  $x\notin\cone$ and $\cone$ can be separated); it is scaling
  consistent if the scaling is space-continuous (see
  Lemma~\ref{lemma:scaling}). The exclusion ideal has a countable open
  base if $\cone$ equals the intersection of a decreasing countable family of
  functionally open sets. It is often assumed that the excluded cone
  contains the set $\zero$ of scaling-invariant elements.
\end{example}

\paragraph{Ideals and moduli}
We often impose the following condition on an ideal $\sS$. For
metrisable spaces, this condition is equivalent to Item~B4 in
\cite{bladt:hash22}.

\index{Condition (B)}
\begin{itemize} 
\item[(B)] \label{condB} The ideal $\sS$ has a countable base
  $(T_{n^{-1}}\VV)_{n\in\NN}$, where $\VV$ is a functionally open
  semicone $\VV\subset\XX$ such that $\cap_{t>1} T_t \VV = \emptyset$,
  and there exists a functionally closed set $F$ such
  that, for every $s\in(0,1)$,
  \begin{equation}
    \label{eq:10}
    \VV\subset F \subset T_s\VV.
  \end{equation}
\end{itemize}

An important ideal is $\sS_\modulus$ from
Definition~\ref{def:modulus}, which is said to be generated by the
modulus $\modulus$. The union of all sets in $\sS_\modulus$  is
$\{\modulus>0\}$. The following result shows that this generic
construction corresponds to ideals satisfying \hyperref[condB]{(B)}.

\begin{lemma}
  \label{lemma:stau}
  An ideal $\sS$ on a topological space $\XX$ satisfies Condition~\hyperref[condB]{(B)}
  if $\sS=\sS_\modulus$ for 
  a finite-valued continuous modulus $\modulus$. The reverse implication
  holds if the scaling is space-continuous.
\end{lemma}
\begin{proof}
  By changing the space $\XX$ to the union of all sets in $\sS$, it
  is always possible to assume that $\sS$ is a bornology on $\XX$,
  which we do for the remainder of this proof.  \smallskip

  \noindent \textsl{Sufficiency.}  It is easy to see that
  $\sS_\modulus$ is closed under finite unions and contains all
  subsets of any of its member sets. Let $\VV=\{\modulus>1\}$, which is
  a functionally open 
  semicone.  Then
  \begin{displaymath}
    F=\{\modulus\geq 1\}\subset\{\modulus>s\}=T_s\VV
  \end{displaymath}
  for $s<1$, and $F$ is functionally closed.
  Finally, $T_t\VV=\{\modulus>t\}$. Therefore, the intersection of $T_t\VV$
  over all $t>1$ is empty (since $\modulus$ is finite), and the sequence
  $(T_{n^{-1}}\VV)_{n\in\NN}$ is a base of $\sS_\modulus$, so that
  Condition~\hyperref[condB]{(B)} holds.
  \smallskip
		
  \noindent \textsl{Necessity.}  Assume that Condition~\hyperref[condB]{(B)} holds. By applying
  $T_t$ to all sides of \eqref{eq:10}, we see that for each $t>0$
  there exists a functionally closed set $F_t=T_tF$ such that
  \begin{equation}
    \label{eq:10bis}
    T_t\VV\subset F_t \subset T_s\VV, \quad 0<s<t.
  \end{equation}
  Define 
  \begin{equation}\label{modulus}
    \modulus(x)= \sup \{t>0\colons x\in T_t \VV\},\quad x\in \XX.
  \end{equation}
  Clearly, $\modulus$ is homogeneous. It is strictly
  positive and finite, since $T_{n^{-1}}\VV$, $n\in\NN$, is a base of
  $\sS$ and the intersection of $T_n\VV$ over $n\in\NN$ is empty.  If
  $\modulus(x)>t$ for $t>0$, then $x\in T_t \VV$. Since $T_t\VV$ is open
  (owing to the continuity of the
  scaling; see Lemma~\ref{lemma:scaling}), a neighbourhood of
  $x$ is contained in $T_t\VV$, implying that the set
  $\{\modulus>t\}$ is open.
  Next, we observe that 
  \begin{displaymath}
    \{\modulus<t\} = \cup_{s<t} \{\modulus>s\}^c
    = \left(\cap_{s<t} T_s\VV\right)^c
    = \left(\cap_{s<t} F_s\right)^c.
  \end{displaymath}
  Indeed, if $x\in\cap_{s<t}F_s$, then, by \eqref{eq:10bis},
  $x\in T_r\VV$ for all $r<t$, and hence $\modulus(x)\geq t$.
  Conversely, if $x\notin F_s$ for some $s<t$, then $x\notin T_s\VV$,
  and therefore $x\notin T_r\VV$ for all $r\geq s$; hence
  $\modulus(x)\leq s<t$.
  Therefore,
  $\{\modulus<t\}$ is open as the complement of a closed set.
  As every open set $G\subset\Rpp$ is the countable union of open intervals
  $(a,b)$, and since
  \begin{displaymath}
    \{\modulus\in(a,b)\}=\{\modulus>a\}\cap \{\modulus<b\}
  \end{displaymath}
  is open, it follows that
  $\modulus^{-1}(G)$ is open. Thus, $\modulus$ is continuous on $\XX$.
		
  Finally, the inclusions
  \[
    \{\modulus>1\}\subset\VV\subset\{\modulus\geq1\}
  \]
  and their scaled versions imply that the bases
  $(T_{n^{-1}}\VV)_{n\in\NN}$ and
  $(\{\modulus>1/n\})_{n\in\NN}$ generate the same ideal.
  Thus, $\sS=\sS_\modulus$.
\end{proof}
  %


\begin{remark}
  Note that the authors of
  \cite[Appendix~B]{kul:sol20} impose three conditions on the
  bornology on a Polish space $\XX$ (where the scaling acts freely);
  they call them (B1)--(B3). It turns out that, if $\sS$ is a bornology on
  $\XX$ and the scaling is Borel measurable, their condition (B2) is equivalent to
  our Condition~\hyperref[condB]{(B)} which, in turn, implies (B1) and (B3). 
\end{remark}

\begin{remark}
  Note that the modulus $\modulus$ from Lemma~\ref{lemma:stau} is constructed so
  that $\{\modulus>0\}$ is the union of all sets in $\sS$. Thus,
  this modulus is proper if this union equals $\XX\setminus\zero$.
\end{remark}

\begin{example}[Domination of moduli]
  We say that a modulus $\modulus_1$ dominates a modulus $\modulus_2$,
  if there exists a $c>0$ such that
  $\modulus_2(x) \leq c\, \modulus_1(x)$ for all $x\in\XX$. In this case, the
  corresponding ideals satisfy
  $\sS_{\modulus_2} \subset \sS_{\modulus_1}$. When
  $\modulus_1, \modulus_2$ dominate each other, they are considered
  \emph{equivalent}
  \index{equivalent moduli}
  and, of course, the corresponding ideals
  coincide. For example, if $\XX=\R_+^2$, then
  $\modulus_1(x)=\max(x_1,x_2)$ and $\modulus_2(x)=\max(x_1,2x_2)$ are
  equivalent, as are the moduli $\modulus_1(x)=\min(x_1,x_2)$ and
  $\modulus_2(x)=(1/x_1+1/x_2)^{-1}$.
\end{example}

In many cases Condition~\hyperref[condB]{(B)} does not hold and is replaced by the following
weaker condition.
\begin{itemize}
\item[(B$_0$)] \label{condB0} $\sS$ is generated by a
  family of ideals, each of which satisfies Condition~\hyperref[condB]{(B)}.
\end{itemize}
\index{Condition (B0)@Condition (B$_0$)}
Condition~\hyperref[condB0]{(B$_0$)} means that $\sS$ is the smallest ideal 
containing a family of ideals, each satisfying Condition~\hyperref[condB]{(B)}.  By
Lemma~\ref{lemma:stau}, $\sS$ satisfies \hyperref[condB0]{(B$_0$)} if and only if it is
generated by a (possibly uncountable) family $\{\modulus_i\colons i\in \JJ\}$
of continuous moduli, that is, $A\in\sS$ if and only if there exists a
finite set $I\subset \JJ$ such that
\begin{displaymath}
  \inf_{x\in A} \,\max_{i\in I}\, \modulus_i(x)>0.
\end{displaymath}
Equivalently, each $A\in\sS$ is a subset of a finite union of
semicones, each of which satisfies Condition~\hyperref[condB]{(B)}. 

\begin{lemma}
  \label{lemma:free}
  If an ideal $\sS$ satisfies Condition~\hyperref[condB0]{(B$_0$)}, then the scaling acts freely on
  the union of sets in $\sS$.
\end{lemma}
\begin{proof}
  Let $x_0$ be a point in the union of all sets from $\sS$, such
  that $T_sx_0=x_0$ for some $s>1$. Then $x_0$ belongs to an ideal
  that satisfies Condition~\hyperref[condB]{(B)}; thus, $\modulus(x_0)=t>0$ for some modulus
  $\modulus$. It follows that $t=\modulus(T_sx_0)=st$, which is a contradiction.
\end{proof}

Sometimes (see, e.g.,
Remark~\ref{rem:sets-scaling}), it may be convenient to work with
ideals whose union includes some scaling-invariant elements, and 
which therefore do not satisfy Condition~\hyperref[condB0]{(B$_0$)}.

\paragraph{Metric exclusion ideals}
Recall the topological exclusion ideal from
Example~\ref{example:exclusion}.  A variant of this construction is
common in metric spaces.

\begin{definition}
  \label{def:exclusion}
  Let $\cone$ be a set in a metric space $\XX$ with metric $\dmet$. The
  \emph{metric exclusion ideal} $\sS_\cone$ is the smallest ideal
  \index{exclusion ideal!metric}
  \index{metric exclusion ideal}
  containing the sets $\{x\colons \dmet(x,\cone)\geq r\}$ (consisting of points 
  at distance of at least $r$ from $\cone$) for all $r>0$.
\end{definition}

The Polish space version of Definition~\ref{def:exclusion} appears in
\cite{lin:res:roy14}.  This metric exclusion ideal was used in
\cite{seg:zhao:mein17} for star-shaped metric spaces, where an additional
imposed condition -- namely that $\ttau(B)>0$ for every set $B$ that
is the complement of an open ball centred at zero -- ensures that the
metric exclusion ideal obtained by excluding $\cone=\zero$ coincides
with $\sS_\modulus$. Lemma~\ref{lemma:stau} together with the continuity of the
modulus implies that Condition~(B) holds in this case.

Note that the metric exclusion ideal does not change when passing
from $\cone$ to its closure. Thus, $A\in\sS_\cone$ means that
\begin{displaymath}
  \dmet(A,\cone)=\inf\{\dmet(x,y)\colons x\in A,y\in\cone\}>0.
\end{displaymath}
We write $\dmet(x,\cone)$ as a shorthand for $\dmet(\{x\},\cone)$.  It is
typical to assume that the excluded set $\cone$ is a cone in $\XX$.

If the metric exclusion ideal $\sS_\cone$ is scaling consistent, and
there exists $x\in\cone$ such that $T_tx\notin\cone$ for some $t>0$, then
$\{T_tx\}\in\sS_\cone$, so that $x=T_{t^{-1}}\{T_tx\}\in\sS_\cone$,
which is a contradiction. Thus, $\cone$ is necessarily a cone.
Conversely, the metric exclusion ideal $\sS_\cone$ obtained by
excluding a cone $\cone$ is scaling consistent if
\begin{displaymath}
  \dmet(T_tx,\cone)\geq a_t\, \dmet(x,\cone), \quad x\in\XX,
\end{displaymath}
with $a_t>0$ for all $t>0$.

\begin{example}[Metric exclusion ideal in $\R^d$]
  \label{ex:ideals-Rd}
  \index{metric exclusion ideal!in Euclidean space}
  The simplest metric exclusion ideal in the Euclidean space $\XX=\R^d$ is
  obtained by excluding the origin, that is, with $\cone=\{\zero\}$.  This ideal is
  generated by the Euclidean norm (taken as a modulus) or, equivalently,
  by any norm on $\R^d$.  We denote the corresponding metric exclusion
  ideal by $\sR^d_0$ and write $\sR_0$ for $\sR^1_0$. Furthermore, the
  same notation is used for the ideals on $\R_+^d$ or on
  $(0,\infty)^d$ generated by the norm, which are thus induced by the ideal
  $\sR_0^d$ restricted to these subsets.
\end{example}

\begin{example}[Exclusion of coordinate axes]
  Let $\XX=\R_+^2$ with the Euclidean metric, and let $\cone$ be the union
  of the two coordinate axes. The open metric $r$-neighbourhood of $\cone$ is
  the union of the two strips $[0,r)\times\R_+$ and $\R_+\times[0,r)$. It
  should be noted that the complements of open neighbourhoods of $\cone$
  form a richer family; consequently, the topological exclusion ideal is
  strictly larger than the metric exclusion ideal. Indeed, the set
  $\{(x_1,x_2)\in\R_+^2\colons x_1x_2>1\}$ intersects every metric
  neighbourhood of $\cone$, yet it does not intersect the open
  neighbourhood of $\cone$ given by
  $\{(x_1,x_2)\in\R_+^2\colons x_1x_2<1/2\}$.  It is easy to see that the
  metric exclusion ideal coincides with the topological exclusion
  ideal whenever the excluded cone $\cone$ is compact.
\end{example}

\begin{lemma}
  \label{lemma:metric-exclusion-Rd}
  Let $(\XX,\dmet)$ be a metric space equipped with a 
  continuous scaling, and let $\cone$ be a cone in $\XX$. Then the
  metric exclusion ideal $\sS_\cone$ satisfies Condition~\hyperref[condB]{(B)} if one of
  the following two conditions holds.
  \begin{enumerate}[(a)]
  \item The metric $\dmet$ is homogeneous, that is,
    $\dmet(T_tx,T_ty)=t\dmet(x,y)$ for all $t>0$.
  \item For all $x\notin\cone$, we have 
    \begin{equation}
      \label{eq:mei2}
      \dmet(T_sx,\cone)>\dmet(T_tx,\cone)>0,\quad 0<t<s<\infty,
    \end{equation}
    and the functions
    \begin{equation}
      \label{eq:mei}
      a_t=\inf_{x\in\XX\setminus\cone} \frac{\dmet(T_tx,\cone)}{\dmet(x,\cone)},
      \quad b_t=\sup_{x\in\XX\setminus\cone}
      \frac{\dmet(T_tx,\cone)}{\dmet(x,\cone)},
      \quad t>0,
    \end{equation}
    satisfy $\inf_{t>0} b_t =0$ and $\sup_{t>0} a_t = \infty$.
  \end{enumerate}  
\end{lemma}
\begin{proof}
  (a) Let $\cone$ be the excluded cone. Then
  $\modulus(x)=\dmet(x,\cone)$ is a modulus, since $\cone=T_t\cone$
  and 
  \begin{displaymath}
    \modulus(tx)=\dmet(T_tx,\cone)=\dmet(T_tx,T_t\cone)=t\dmet(x,\cone). 
  \end{displaymath}
  This function is evidently continuous and $\sS_\modulus$ is exactly
  the metric exclusion ideal $\sS_\cone$. 

  (b) Define $\VV=\{x\colons \dmet(x,\cone)>1\}$. By \eqref{eq:mei2}, $\VV$ is a
  semicone. Define
  \begin{displaymath}
    \modulus(x)=\sup\{t>0\colons T_{t^{-1}}x\in \VV\}
    = \Big(\inf\big\{t>0\colons \dmet(T_tx,\cone)>1\big\}\Big)^{-1}, \quad x\notin\cone,
  \end{displaymath}
  and let $\modulus(x)=0$ for all $x\in\cone$, where \eqref{eq:mei2}
  is used to obtain the second equality. We first show that the
  set under the infimum is not empty, that is, for all $x\notin\cone$
  there exists a $t>0$ such that $\dmet(T_t x,\cone)>1$. If this
  were not the case, then
  \begin{displaymath}
    a_t\dmet(x,\cone)\leq \dmet(T_t x,\cone)\leq 1
  \end{displaymath}
  for all $t>0$, which contradicts the imposed condition that the
  supremum of $a_t$ is infinite.  Similarly, if $\modulus(x)=\infty$, then
  \begin{displaymath}
    b_t\dmet(x,\cone)\geq \dmet(T_t x,\cone)>1
  \end{displaymath}
  for all sufficiently small $t>0$, which is impossible by the
  imposed condition on $b_t$. The function $\modulus$ is evidently homogeneous.
  Next, we verify that $\modulus$ is continuous. From \eqref{eq:mei2}
  it follows that
  \begin{displaymath}
    \{\modulus >s\}=\{x\colons \dmet(T_{s^{-1}}x,\cone)>1\},\quad 
    \{\modulus <s\}=\{x\colons \dmet(T_{s^{-1}}x,\cone)<1\}.
  \end{displaymath}
  Both sets are open due to continuity of the distance
  function and the scaling. Furthermore, these equations and
  \eqref{eq:mei} imply that
  \begin{displaymath}
    \{x\colons a_{s^{-1}}\dmet(x,\cone)>1\}\subset
    \{\modulus >s\}\subset \{x\colons b_{s^{-1}}\dmet(x,\cone)>1\},
  \end{displaymath}
  meaning that the ideal $\sS_\modulus$ generated by $\modulus$ coincides
  with $\sS_\cone$.
\end{proof}

A modulus $\modulus$ on $\XX$ is said to be \emph{inf-compact} if the
set $\{\modulus\leq t\}$ is compact for all $t\in\R_+$. If the scaling
is space-continuous, it suffices to impose this for a single $t>0$,
since homogeneity yields that $\{\modulus\leq t\}$ is compact for all
$t\in\R_+$. 
\index{modulus!inf-compact}
In particular, the cone $\{\modulus=0\}$ is a compact set. 

\begin{lemma}
  \label{lemma:inf-compact-excluded}
  If $\modulus$ is an inf-compact continuous modulus on a metric space
  $(\XX,\dmet)$, then the ideal $\sS_\modulus$ is the metric exclusion ideal 
  obtained by excluding the cone $\{\modulus=0\}$.
\end{lemma}
\begin{proof}
  By the inf-compactness assumption, $\cone=\{\modulus=0\}$ is a
  compact set. Fix $t>0$.  Consider the function
  $x\mapsto \dmet(x,\{\modulus\geq t\})$. This function is continuous on
  the compact set $\cone$ and thus attains a minimum $r$. Since
  $\cone$ and $\{\modulus\geq t\}$ are disjoint, $r$ must be strictly
  positive. It follows that
  \begin{displaymath}
    \{y\colons \dmet(y,\cone)\geq r\}\supset \{\modulus\geq t\}.
  \end{displaymath}
  Since $t>0$ was arbitrary, every set in $\sS_\modulus$ is a subset
  of a set in $\sS_\cone$.

  Conversely, let
  \begin{displaymath}
    R_t=\sup_{y\colons \modulus(y)\leq t}\, \dmet(y,\cone).
  \end{displaymath}
  This is finite by the compactness of $\cone$ and
  $\{\modulus\leq t\}$.  If $\dmet(y,\cone)> R_t$, then , by the
  definition of $R_t$, it is impossible for $\modulus(y)\leq t$ to
  hold. Hence,
  \begin{displaymath}
    \{y\colons \dmet(y,\cone)> R_t\}\subset \{\modulus> t\}.
  \end{displaymath}
  Since $R_t\to0$ as $t\to0$ due to the nesting of compact sets
  $\{\modulus\le t\}$,
  every set in the ideal $\sS_\cone$ is a
  subset of a set in $\sS_\modulus$. 
\end{proof}

\paragraph{Products of ideals}
It is possible to introduce products of ideals as follows. Let
$\XX$ be a space with an ideal $\sS$, and let $\JJ$ be an index set. Denote
by $\XX^\JJ$ the set of functions from $\JJ$ to $\XX$, which
corresponds to the
Cartesian power of $\XX$ if $\JJ$ is at most countable. For
$I\subset \JJ$, let $\pr_I A$ denote the projection of $A\subset\XX^\JJ$
onto the coordinates in $I$. 

\begin{definition}
  \label{def:product}
  For a finite $k$, which is at most the cardinality of $\JJ$, the $k$-th
  product ideal $\sS^\JJ(k)$ of $\sS$ is the family of sets
  $A\subset \XX^\JJ$ such that
  \index{ideal!product of} \index{product of ideals}
  \begin{equation}
    \label{eq:41}
    A \subset \bigcup_{r=1}^l \Big\{y\in \XX^\JJ\colons \pr_{I_r} y\in
    \underbrace{B\times \cdots \times B}_{k}\Big\}
   \end{equation}
   for some $B\in\sS$ and a finite collection of sets
   $I_1,\dots,I_l\subset \JJ$, each of cardinality $k$.
\end{definition}

Note that $\sS^\JJ(k)$ decreases as $k$ increases.  If
$\JJ=\{1,\dots,m\}$, then the product ideal is denoted by $\sS^m(k)$,
and we write $\sS^\infty(k)$ if $\JJ=\NN$.

\begin{example}[Products of three ideals]
  Let $\JJ=\{1,2,3\}$ and $A\subset \XX^3=\XX\times\XX\times\XX$. If
  $k=1$, we may take $I_1=\{1\}$, $I_2=\{2\}$, and $I_3=\{3\}$ to
  obtain the largest right-hand side in \eqref{eq:41}; thus,
  $A\in\sS^3(1)$ if and only if $A$ is a subset of the union of sets
  whose projection onto one of the three components is bounded. If $k=2$, then we
  may take $I_1=\{1,2\}$, $I_2=\{1,3\}$, and $I_3=\{2,3\}$, showing
  that $A\in\sS^3(2)$ if $A$ is a subset of the union of sets whose
  projections on some pair of components is a subset of $B\times B$
  for $B\in\sS$. Finally, if $k=3$, then $I_1=\{1,2,3\}$ and
  $A\in\sS^3(3)$ if $A\subset B\times B\times B$ for $B\in\sS$.
\end{example}

The product of non-identical ideals $\sS_j$, $j\in \JJ$, is defined by
replacing the Cartesian product of $r$ copies of $B$ on the right-hand
side of \eqref{eq:41} 
with the product $\times_{j\in I_r} B_j$
of sets $B_j\in\sS_j$, $j\in I_r$. In particular, for
two ideals $\sX$ and $\sY$ on spaces $\XX$ and $\YY$, we define
\begin{equation}
  \label{eq:product-two-ideals}
  \sX\times\sY=\{A\subset \XX\times\YY\colons  
  A\subset B_1\times B_2 \;\text{for some}\; B_1\in\sX,\;B_2\in\sY\},
\end{equation}
and
\begin{equation}
  \label{eq:product-two-ideals-1}
  \sX\otimes\sY=\{A\subset \XX\times\YY\colons  
  A\subset (B_1\times \YY)\cup(\XX\times B_2)
  \;\text{for some}\; B_1\in\sX,B_2\in\sY\}. 
\end{equation}
It is obvious that $\sX\otimes\sY$ is richer than $\sX\times\sY$. For
identical ideals, we have $\sX\times\sX=\sX^2(2)$ and
$\sX\otimes\sX=\sX^2(1)$. The product of ideals given by
$\sX_1\times\cdots\times\sX_m$ appears in \cite[Section~2.2]{hog77}.

In the following we often encounter products involving trivial
ideals. Let $\sX$ be an ideal on $\XX$. If $\sY$ is a trivial ideal on
$\YY$, meaning that $\YY\in\sY$ and so $\sY$ consists of all subsets
of $\YY$, then we write $\sX\times\YY$ for the ideal consisting of
all subsets of $A\times\YY$ for all $A\in\sX$. 

If $\sS$ is generated by a modulus $\modulus$, then $A\in\sS^\JJ(k)$
means that there exists a finite collection of sets
$I_1,\dots,I_l\subset \JJ$, each of cardinality $k$, such that 
\begin{displaymath}
  \inf_{x=(x_j)_{j\in\JJ}\in A}\;
  \max_{r\in\{1,\dots,l\}} \; \min_{i\in I_r}\; \modulus(x_i)>0.
\end{displaymath}
In this case, the product ideal satisfies
Condition~\hyperref[condB0]{(B$_0$)}, namely, it is 
generated by the family of moduli $x\mapsto \min_{i\in I}\modulus(x_i)$
where $I\subset\JJ$ ranges over all sets 
of cardinality $k$. 

\begin{example}[Products of ideals on the line]
  \label{example:product}
  Let $\XX=\R_+$ with the ideal $\sR_0$ generated by $\modulus(x)=x$,
  $x\in\R_+$. Then $A\in\sR_0^d(1)$ means that there exists $\eps>0$
  such that $A$ is contained in the union of sets $\{x\in\R^d\colons 
  x_i\geq \eps\}$, $i=1\dots,d$. 
  This ideal coincides with the metric
  exclusion ideal $\sR^d_0$ on $\R_+^d$ obtained by excluding the
  origin (see Example~\ref{ex:ideals-Rd}). Furthermore,
  $A\in\sR^d_0(d)$ means that there exists $\eps>0$ such that
  $\min(x_1,\dots,x_d)\geq \eps$ for all $x\in A$. The latter ideal coincides with
  the metric exclusion ideal on $\R_+^d$ obtained by excluding the
  union of all coordinate hyperplanes in $\R_+^d$. More generally,
  $\sR_0^d(k)$ with $k\in\{2,\dots,d-1\}$ is obtained by excluding
  convex hulls of unions of at most $k-1$ coordinate semiaxes. 
\end{example}

\paragraph{Graph product of ideals}

The following construction is motivated by recent developments in
graphical models for extremes; see, e.g., \cite{MR4136498} and
\cite{MR1419991}.  Consider an undirected graph with node set $\JJ$ and
values $x_j$ at the nodes $j\in\JJ$. We
write $i\sim j$ if the nodes $i$ and $j$ are connected by an edge, and we
call a set $I\subset\JJ$ a clique (or fully connected) if $i\sim j$
for all $i,j\in I$. The \emph{graph product}
\index{graph product of ideals}
of the ideals $\sX_i$ associated to each node $i\in\JJ$ is defined as the
family of subsets $A$ of $\XX^{\JJ}$ such that there exists an
$i\in\JJ$ for which the projection of $A\subset \XX^{\JJ}$ onto the
components in $I=\{j\colons j\sim i\}$ is a subset of $\times_{l\in I} B_l$
for $B_l\in \sX_l$, $l\in I$. If the graph is complete and all $\sX_i$
are the same $\sX$, then
the graph product coincides with $\sX^{\JJ}(k)$, where $k$ is the
cardinality of $\JJ$.  In general, the graph product is larger than 
$\sX^{\JJ}(k)$ and smaller than $\sX^{\JJ}(1)$.

If the ideals $\sX_i$ are generated by moduli $\modulus_i$, $i\in\JJ$,
then the graph product of the ideals is generated by the modulus
\begin{displaymath}
  \modulus(x)=\sup_{i\in\JJ}\min_{j\sim i} \modulus_j(x_j),
  \quad x\in\XX^{\JJ} 
\end{displaymath}
with the convention that $i\sim i$, so that each minimum is taken over
a non-empty set. 

An ideal $\sX$ on $\XX^{\JJ}$ generates the corresponding graph. In
this construction,
the nodes $i$ and $j$ are connected by an edge if every
$A\in\sX$ has the property that its projection onto the
$(i,j)$ components is contained in $B_i\times B_j$ for some
$B_i\in\sX_i$ and $B_j\in\sX_j$.

\begin{example}[Graph product]
  Consider the first graph in Figure~\ref{fig:1}(a), where each node
  carries the ideal $\sR_0$. Then sets in the graph product ideal
  are exactly those that are subsets of the union of the sets
  $[s,\infty)\times[s,\infty)\times\R_+\times\R_+$ or
  $\R_+\times [s,\infty)\times[s,\infty)\times\R_+$ or
  $\R_+\times\R_+\times\R_+\times[s,\infty)$ for some $s>0$. 
  
  \begin{figure}
    \centering
    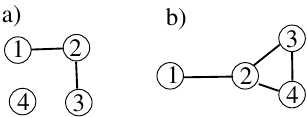
    \caption{Two graphs \label{fig:1}}
  \end{figure}
 
  For the second graph in Figure~\ref{fig:1}(b), the graph product
  consists of sets that are subsets of the union of the sets
  $[s,\infty)\times[s,\infty)\times\R_+\times\R_+$ or
  $\R_+\times [s,\infty)\times[s,\infty)\times[s,\infty)$.
\end{example}

\paragraph{Bornologically consistent maps}
In the literature on bornological spaces, a map between two spaces
equipped with bornologies is said to be bounded if the image of every
bounded set is bounded; see \cite[Section~1.2]{hog77}. The following
definition relies on the boundedness of inverse images and is formulated
for ideals. 

\begin{definition}
  \label{def:morphism}
  Let $\XX$ and $\YY$ be two spaces with ideals $\sX$
  and $\sY$, respectively.  A function $\psi:\XX\to\YY$ is said to
  be \emph{bornologically consistent} (also called
  $(\sX,\sY)$-bornologically consistent) if $\psi^{-1}(B)\in\sX$ for
  all $B\in\sY$.
  \index{bornologically consistent map}
\end{definition}

\begin{lemma}
  \label{lemma:map-zero}
  Let $\psi:\XX\to\YY$ be a continuous map between two metric spaces
  equipped with metric exclusion ideals obtained by excluding the singletons
  $\{x_0\}$ and $\{y_0\}$, respectively, such that
  $\psi^{-1}(y_0)=x_0$ (see
  Definition~\ref{def:exclusion}). Then $\psi$ is bornologically
  consistent. 
\end{lemma}
\begin{proof}
  Each element of the ideal on $\YY$ is a subset of the complement of
  some open ball $B$ centred at $y_0$. By the properties of inverse
  images, 
  $\psi^{-1}(B^c)=\psi^{-1}(B)^c$.  Since $\psi$ is continuous and
  $y_0\in B$, the set $\psi^{-1}(B)$ is an open
  neighbourhood of $x_0$. Consequently, it contains an open ball of positive radius
  centred at $x_0$. It follows that its complement $\psi^{-1}(B^c)$
  belongs to the metric 
  exclusion ideal on $\XX$.
\end{proof}

\section{Homogeneous measures}
\label{sec:homogeneous-measures-1}

\paragraph{Baire, Borel and Radon measures}

A \emph{Baire measure} on $\XX$ is a countably additive function
$\mu: \sBA(\XX) \to[0,\infty]$ defined on the
Baire $\sigma$-algebra $\sBA(\XX)$.
\index{Baire measure} \index{Borel measure} \index{Radon measure}
\index{measure!Borel}\index{measure!Baire}\index{measure!Radon}
A \emph{Borel measure} is a countably additive function
$\mu: \sB(\XX) \to[0,\infty]$ on the Borel $\sigma$-algebra
$\sB(\XX)$. A Borel measure $\mu$ is said to be \emph{Radon} if it is
inner regular: for
each measurable $A$, the value  $\mu(A)$
equals the supremum of $\mu(K)$ over all
compact sets $K\subset A$; see \cite[Definition~7.1.1]{bogachev07}.

\index{measure!boundedly finite}
A Baire measure $\mu$ is said to be \emph{boundedly finite}
with respect to an
ideal $\sS$ (for short,
finite on $\sS$) if $\mu(B)<\infty$ for all $B\in \sS\cap \sBA(\XX)$ and
$\mu$ is supported by the union of all sets in $\sS$, that is, $\mu$
vanishes on the complement of $\cup_{B\in\sS} B$. Note that the
union of all sets in $\sS$ is a Baire set if $\sS$ has a countable
open base. The family of measures that are boundedly finite on $\sS$
is denoted by $\Msigma$, while
$\Mb$ denotes the family of Borel measures that are finite on $\sS$
and supported by the union of all sets in $\sS$.  Recall that
the Baire and Borel $\sigma$-algebras coincide, and thus  
$\Msigma=\Mb$ if $\XX$ is perfectly normal, most importantly, if $\XX$ is
Polish. If $\sX$ contains $\XX$, then all boundedly finite measures
are finite; we denote the family of all finite (respectively, probability)
Borel measures on $\XX$ by $\Mb[\XX]$ (respectively, $\Mone$), using the
subscript $\sigma$ in the case of Baire measures. 

\begin{lemma}
  If $\sS$ is a topologically consistent ideal on a Polish space
  $\XX$, then every measure in $\Mb$ is Radon.
\end{lemma}
\begin{proof}
  Let $\mu\in\Mb$.  By the assumed topological consistency,
  if $A\in\sS\cap\sB(\XX)$, then $\sS$ contains a
  closed set $F\supset A$. Since $\mu(F)$ is finite, the
  restriction of $\mu$ to $F$ is a Radon measure; see
  \cite[Theorem~4.1.11(iii)]{bogachev18}. Thus, $\mu(A)$ equals the
  supremum of $\mu(K)$ over all compact sets $K\subset A\subset F$.
\end{proof}

Note the following easy fact.

\begin{lemma}
  \label{lemma:enclosed}
  Let $\sS$ be an ideal with a countable open base
  $(\base_m)_{m\in\NN}$ on a topological space $\XX$. Assume that
  $\mu_m$ is a finite Baire measure on $\base_m$ for each $m\in\NN$ such that
  the restriction of $\mu_n$ to $\base_m$ equals $\mu_m$ for all
  $m\leq n$. Then there exists a unique boundedly finite Baire measure
  $\mu\in\Msigma$ such that the
  restriction of $\mu$ onto $\base_m$ coincides with $\mu_m$ for all
  $m\in\NN$.
\end{lemma}
\begin{proof}
  Without loss of generality assume that $(\base_m)$ is increasing.
  For each $A\in\sS$ we have $A\subset\base_m$ for some $m$ and we set
  $\mu(A)=\mu_m(A)$. 
  By assumption, this definition is consistent,
  and the required property of restrictions holds. The function $\mu$
  is countably additive on the algebra of sets obtained as the union of
  $\sigma$-algebras $\sBA(\base_i)$
  of Baire subsets of $\base_i$, $i\in\NN$. Indeed, if $A=\cup A_i$,
  then $A\subset \base_m$ for some $m$ and so we may invoke the
  $\sigma$-additivity of $\mu_m$. By the standard measure
  extension theorem, $\mu$ extends to a $\sigma$-additive measure
  on $\sBA(\XX)$; this also confirms uniqueness. 
\end{proof}

\paragraph{Homogeneity property of Baire measures}
A Baire measurable scaling on $\XX$ is lifted to act on Baire
measures as follows \index{scaling!of measures}
\begin{equation}
  \label{eq:23}
  (T_t\mu)(B)=\mu(T_{t^{-1}}B),\quad B\in\sBA(\XX),\; t>0.
\end{equation}
Assume that $\sS$ is a scaling-consistent ideal.  Then
$T_t\mu\in\Msigma$ whenever $\mu\in\Msigma$.

\begin{definition}
  A measure $\mu\in\Msigma$ is said to be \emph{$\alpha$-homogeneous}
  if \index{measure!homogeneous} \index{homogeneous measure}
  \begin{displaymath}
    T_t \mu= t^{\alpha} \mu, \quad t>0,
  \end{displaymath}
  for some $\alpha\in\R$.
\end{definition}

If $\mu$ is $\alpha$-homogeneous with $\alpha\neq0$ and $\sS$ contains
a set $B$ such that $B\subset \zero$, then $\mu(B)=0$. A measure that
is homogeneous with $\alpha=0$ is said to be scaling invariant. 

\begin{lemma}
  \label{lemma:alpha}
  Let $\mu\in\Msigma$ be an $\alpha$-homogeneous Baire measure, and let
  $\modulus$ be a continuous modulus.  If $B=\{\modulus> 1\}\in\sS$
  and $\mu(B)>0$, then $\alpha\geq 0$. If $B=\{\modulus\leq 1\}\in\sS$ and
  $\mu(B)>0$, then $\alpha\leq 0$. Furthermore, if $\mu$ is finite on a
  semicone $\VV$, then
  \begin{equation}
    \label{eq:24}
    \mu\big(\{x\in \VV\colons \modulus(x)=t\}\big)=0
  \end{equation}
  for all $t>0$.
\end{lemma}
\begin{proof}
  If $B=\{\modulus> 1\}$, then
  \begin{displaymath}
    \mu(T_{t^{-1}}B)=\mu(\{\modulus>t^{-1}\})=t^\alpha\mu(B).
  \end{displaymath}
  Since the left-hand side is nondecreasing in $t$ and converges to
  zero as $t\to0$, we have $\alpha\geq0$. The case where
  $B=\{\modulus\leq 1\}$ and $\alpha\leq0$ is handled by a similar
  argument.

  If \eqref{eq:24} does not hold for some $t>0$, then
  \begin{displaymath}
    \mu(\{x\in \VV\colons \modulus(x)=st\})
    \geq s^{-\alpha}\mu(\{x\in \VV\colons \modulus(x)=t\})>0
  \end{displaymath}
  for all $s\in(1,\infty)$. This is impossible if
  $\mu(\{x\in \VV\colons \modulus(x)=t\})>0$,
  since, by choosing countably many distinct values of $s$ in a compact
  interval of $(1,\infty)$, one obtains countably many disjoint subsets
  of $\VV$, each with measure bounded below by a positive constant,
  contradicting the finiteness of $\mu(\VV)$.
\end{proof}

Each $\alpha$-homogeneous Borel measure on $\Rpp$ with $\alpha>0$
is proportional to the measure
\begin{equation}
  \label{eq:29}
  \theta_\alpha(\bdiff t)=\alpha t^{-\alpha-1}dt.
\end{equation}
The following result shows that the measures $\theta_\alpha$
arise from disintegration of homogeneous measures. It
specialises the representation of general homogeneous measures from
\cite{evan:mol18}, where it is proved for measures on a generic
$\sigma$-algebra (which includes the case of Baire measures).
	
\begin{theorem}\label{spectral-measure}
  Assume that a scaling acts freely on a topological space
  $\XX$, and $\sS$ is a bornology
  on $\XX$ that admits a countable base $(\base_n)_{n\in\NN}$ such
  that each $\base_n$ is a Baire measurable semicone.  If
  $\mu\in\Msigma$ is an $\alpha$-homogeneous Baire measure with
  $\alpha>0$, then $\mu$ is the pushforward of
  $\sigma\otimes \theta_\alpha$ by the map $(u,t)\mapsto T_tu$, where
  $\sigma$ is a finite Baire measure on $\XX$ such that
  \begin{displaymath}
    \int_0^\infty \sigma(T_t\base_n) t^{\alpha-1}\diff t<\infty,\quad n\in\NN.
  \end{displaymath}
  Equivalently,
  \begin{equation}
    \label{eq:4}
    \mu(B) = \alpha \int_0^\infty \sigma(T_t B) t^{\alpha-1}\diff t,
    \quad B\in \sS.
  \end{equation}
\end{theorem}

The measure $\sigma$ is called the \emph{spectral measure} of $\mu$.
\index{spectral measure}
Theorem~\ref{spectral-measure} applies
if $\sS$ satisfies Condition~\hyperref[condB]{(B)}, and then $\sigma$ can be chosen to be a finite
measure on $\SS_\modulus$ as given in \eqref{eq:sphere}. In this form,
\eqref{eq:4} on the space $\XX'=\XX\setminus\zero$ also appears in
\cite{kul:sol20}, under the additional assumption that $\zero$ is a
singleton. Note that the spectral measure is not necessarily unique
unless its support is a specified transversal in $\XX$, e.g., the
transversal $\SS_\modulus$ generated by a continuous modulus $\modulus$
(which ensures that $\SS_\modulus$ is a Baire set).

\begin{corollary}
  Let $\mu\in\Msigma$ be an $\alpha$-homogeneous Baire
  measure with $\alpha>0$, where the ideal $\sS$ is generated by an at
  most countable family of continuous moduli. Then \eqref{eq:4} holds
  on the union of all sets in $\sS$.
\end{corollary}
\begin{proof}
  By Lemma~\ref{lemma:free} the scaling acts freely. Consequently,
  Theorem~\ref{spectral-measure} applies with
  \begin{displaymath}
    \base_n=\cup_{i=1,\dots,n}\{\modulus_i>1/n\},\quad n\in\NN. \qedhere
  \end{displaymath}
\end{proof}

If the spectral measure $\sigma=\sigma_\modulus$ in \eqref{eq:4} is
supported by $\SS_\modulus$ generated by a continuous modulus
$\modulus$, then \eqref{eq:4} can be expressed as
\begin{equation}
  \label{eq:mu-sector}
  \mu\big(\{x\colons
  \rho(x)\in A\times[s,t]\}\big)=\sigma(A)(t^\alpha-s^\alpha),
\end{equation}
where $A$ is a Baire subset of $\SS_\modulus$ and $\rho$ is the polar
decomposition map. 

\begin{lemma}
  \label{lemma:change-modulus}
  Let $\modulus$ and $\ell$ be two continuous
  moduli on a topological space
  $\XX$ such that the ideal $\sS_{\ell}$ generated by $\ell$
  is a subset of the ideal generated by $\modulus$. If
  $\sigma_\modulus$ is a finite measure on $\SS_\modulus$, then
  \begin{displaymath}
    \sigma_{\ell}(A)=\int_{\SS_\modulus}
      \ell(x)^\alpha \one_{\ell(x)>0} \one_{T_{\ell(x)^{-1}}x\in A}
      \sigma_\modulus(\bdiff x)
  \end{displaymath}
  is a finite measure on $\SS_{\ell}=\{x\colons \ell(x)=1\}$ and the corresponding
  homogeneous measure $\mu_\ell$ on $\XX$ defined in \eqref{eq:mu-sector}
  is the restriction of $\mu$ given in \eqref{eq:4} onto $\sS_{\ell}$. 
\end{lemma}
\begin{proof}
  Denote by $\rho_\ell$ the polar decomposition map corresponding to
  $\ell$.  For a Baire subset $A$ of $\SS_{\ell}$ we have
  \begin{align*}
    \sigma_{\ell}(A)=\mu\big((\rho_\ell)^{-1}&(A\times[1,\infty))\big)
    =\int_0^\infty \int_{\SS_\modulus}
      \one_{\rho_\ell(T_tv)\in A\times[1,\infty)}\one_{\ell(v)>0}
      \sigma_\modulus(\bdiff v) \theta_\alpha(\bdiff t)\\
    &=\int_0^\infty \int_{\SS_\modulus}
      \one_{T_{\ell(v)^{-1}}v\in A}\one_{\ell(T_tv)\geq 1}\one_{\ell(v)>0}
      \sigma_\modulus(\bdiff v) \theta_\alpha(\bdiff t)\\
    &=\int_{\SS_\modulus}
      \one_{T_{\ell(v)^{-1}}v\in A} \one_{\ell(v)>0}\ell(v)^\alpha 
      \sigma_\modulus(\bdiff v). \qedhere
  \end{align*}
\end{proof}

The continuity assumption on $\modulus$ and $\ell$ in
Lemma~\ref{lemma:change-modulus} ensures the Baire measurability of
$\SS_\modulus$ and $\SS_\ell$. A variant of this lemma holds for Borel measures
and Borel measurable functions $\modulus$ and $\ell$.

\paragraph{Extensions to product spaces}

Let $\XX$ be a perfectly normal topological space, so that the Borel
and Baire $\sigma$-algebras coincide. For each $n\in\NN$, let $\XX^n$
be the Cartesian power of $\XX$ equipped with the product topology and the
corresponding Borel $\sigma$-algebra $\sB(\XX^n)$.
\index{product space}
Since perfect normality is not preserved under Cartesian products, we
also assume that the Borel and Baire $\sigma$-algebras on $\XX^n$
coincide for all $n$.

Let $\pi_{k,n}\colons \XX^n\to\XX^k$ for $k\leq n$ be the natural projection maps,
defined by $\pi_{k,n}(x_1,\dots,x_n)=(x_1,\dots,x_k)$.
Define the projection maps
$\pi_k=\pi_{k,\infty}\colons \XX^\infty\to\XX^k$ by mapping a sequence
$(x_1,x_2,\dots)$ to $(x_1,\dots,x_k)$.  Endow the space $\XX^\infty$
with the smallest $\sigma$-algebra that makes the projection maps
$\pi_k$ measurable for all $k$. This
$\sigma$-algebra is generated by the algebra $\bigcup_k
\pi_k^{-1}(\sB(\XX^k))$. 

Let $\sS_n$ be ideals on $\XX^n$ for $n\in\NN$, such that the
projection maps $\pi_{k,n}$ are bornologically consistent for all 
$k\geq 1$ and any $n\geq k$ (see Definition~\ref{def:morphism}), that
is, 
\begin{equation}
  \label{eq:proj-S}
  \pi_{k,n}^{-1}(A)\in\sS_n \quad\text{for all}\;\; A\in\sS_k.
\end{equation}
Furthermore, let $\sS_\infty$ be the smallest ideal on $\XX^\infty$ such that
the maps $\pi_k$ are bornologically consistent for all $k\geq 1$. The
following result is a version of \cite[Theorem~B.1.33]{kul:sol20},
specialised to infinite products and extended to the non-Polish
setting. The scheme of the proof is similar to that in \cite{MR0404570}. The Radon
assumption makes it possible to approximate a measure with its values
on compact sets, which is the key step in the usual proof of
Kolmogorov's extension theorem for finite-dimensional distributions. 

\begin{theorem}
  \label{thr:fidi-extension}
  Let $\XX$ be a perfectly normal topological space. 
  For 
  each $n\geq 1$, let $\mu_n\in\Mb[\sS_n]$
  be a Radon (and thus Borel) measure on $\sB(\XX^n)$. Assume that
  $\sS_n$ has a countable open base for all $n\in\NN$ and that
  $\pi_{k,n}\mu_n=\mu_k$ for all $1\leq k\leq n$. Then there exists
  a unique measure $\mu$ on this cylinder $\sigma$-algebra, boundedly
  finite with respect to $\sS_\infty$, such that $\pi_k\mu=\mu_k$ for all
  $k\geq 1$.
  \index{extension of measures to product spaces}
\end{theorem}
\begin{proof}
  Define a sequence of ideals on $\XX^\infty$ by letting $\sV_k$ be
  the ideal generated by the sets $\pi_k^{-1}(A)$, $A\in\sS_k$.
  Equivalently, $\sV_k$ consists of all subsets of such cylinders.
  By \eqref{eq:proj-S}, $\sV_n$ is generated by
  $\big\{\pi_n^{-1}(A)\colons A\in\sS_n\big\}$, which includes the
  family $\big\{\pi_n^{-1}(\pi_{k,n}^{-1}(A))\colons A\in\sS_k\big\}$
  generating $\sV_k$ for $k\leq n$. Thus, $\sV_n\supset\sV_k$. 

  Let $(\base_{m,k})_{m\in\NN}$ be a countable open base of
  $\sS_k$.
  By applying the extension theorem for Radon measures
  \cite[Theorem~7.7.1]{bogachev07} to the measures $\mu_n$, restricted to
  subsets of $\pi_{k,n}^{-1}(\base_{m,k})$, we obtain a family of
  measures on $\{A\cap \base_{m,k}\colons A\in\sV_k\}$, which extend to a
  unique Radon measure $\mu^{(k)}\in\Mb[\sV_k]$, see
  Lemma~\ref{lemma:enclosed}. Since $\pi_n\mu^{(k)}=\mu_n$ for all
  $n\geq k$ and $\sV_k\subset\sV_n$, Lemma~\ref{lemma:enclosed} yields
  that there exists a unique measure $\mu$ that is finite on every set
  in the
  union of all ideals $\sV_k$, which is exactly the ideal $\sS_\infty$.
\end{proof}

\newpage

\chapter{Vague convergence and regular variation}
\label{sec:vague-convergence-rv}

\section{Vague convergence}
\label{sec:vague-convergence-1}

\paragraph{Definition of vague convergence}
Let $\sS$ be an ideal on $\XX$ with the union set
$\UU=\cup_{B\in\sS} B$. Denote by $\fC(\XX,\sS)$ the family of bounded
continuous functions $f:\UU\to\R$ such that the support of $f$ (that
is, the closure in $\XX$ of the set $\{x\in\UU\colons f(x)\neq0\}$)
belongs to $\sS$.  
\index{bounded continuous functions}

\begin{example}[Poor family of continuous functions]
  In a pathological situation $\fC(\XX,\sS)$ can be empty, for
  instance, when the sets in $\sS$ cover only a lower-dimensional part
  of the space (e.g., one of the axes in $\R^2$). In such cases, there
  are no bounded continuous functions $f:\XX\to\R$ supported only on
  the $x$-axis.  However, this would be impossible if $\XX$ is
  completely Hausdorff and the ideal contains a functionally open
  neighbourhood $G$ of any point $x$ such that $\{x\}\in\sS$. Under
  this condition, the complement of $G$ and the point $x\in G$ are
  separated by a continuous function and so $\{x\}$ is the
  intersection of the sets $\{f>0\}$ for functions $f$ in
  $\fC(\XX,\sS)$. A stronger assumption --- the existence of a
  countable open base --- also ensures the non-triviality of
  $\fC(\XX,\sS)$.
\end{example}

This example shows that some assumptions on $\sS$ are needed to
ensure that $\fC(\XX,\sS)$ is non-trivial. In the following, our
standing assumption is the topological consistency of the ideal, which
guarantees that, for each point $\{x\}$ in the union of all sets from the
ideal, the ideal contains
a functionally open neighbourhood of $x$. We will also often assume
that the ideal has a countable open base. 

\begin{definition}
  \label{def:vague}
  Let $\sS$ be a topologically consistent ideal on a topological space
  $\XX$. A net 
  $(\mu_\gamma)_{\gamma\in\Gamma}\subset \Msigma$ of Baire measures is
  said to \emph{vaguely converge} to $\mu\in \Msigma$ on $\sS$
  (denoted by $\mu_\gamma\vto\mu$) if
  \index{vague convergence}
  \begin{equation}
     \label{eq:vague-f} 
     \lim_\gamma\int f \diff \mu_\gamma= \int f \diff \mu
    \end{equation}
  for all $f\in \fC(\XX,\sS)$. 
\end{definition}

The definition of vague convergence crucially depends on the
availability of sufficiently many continuous real-valued functions on
$\XX$. To ensure that there are non-trivial continuous
functions, it is sensible to assume that $\XX$ is completely
Hausdorff. The stronger property of perfect normality makes it
possible to work with Borel measures instead of Baire ones.

It should be noted that all measures $\mu_\gamma$ and the vague limit
$\mu$ are elements of $\Msigma$,
which means that they are supported on the union of all sets in
$\sS$.
If $\sS$ is trivial (that is, $\sS$ contains $\XX$ and so all subsets
of $\XX$), then all involved measures are finite and
Definition~\ref{def:vague} reduces to the definition of \emph{weak
  convergence} denoted by $\mu_\gamma\wto\mu$ (and $\xi_\gamma\dto\xi$
for the corresponding random elements).
\index{weak convergence}
\index{convergence in distribution}
Letting $f$ be identically one, we see that in this case
$\mu_\gamma(\XX)\to\mu(\XX)$. 

Note that the standard definition deals with vague convergence on
the ideal taken as the Fr\'echet bornology (see
\cite{daley:verejones:2008, kalle17}) or on a metric exclusion ideal
(see \cite{hul:lin06,lin:res:roy14,seg:zhao:mein17}).
Definition~\ref{def:vague} is an adaptation of the definition of vague
convergence on Polish spaces for sequences of measures; see
\cite{kalle17} and \cite{bas:plan19}, where it is formulated under the
assumption that $\XX$ is localised. Special cases of vague convergence are
given various names in the literature: e.g., weak$^\#$-convergence in
\cite{daley:verejones:2008}, $M$-convergence in
\cite{hul:lin06,lin:res:roy14}, vague$^\#$-convergence in
\cite[Definition~B.1.16]{kul:sol20}, or $\sB$-vague convergence in
\cite[Definition~4.1]{bladt:hash22}.

\paragraph{Mapping theorem}

If a Baire measurable map $\psi:\XX\to\YY$ maps the union set of
$\sX$ into the union set of $\sY$ and is bornologically
consistent (see Definition~\ref{def:morphism}) and
$\mu\in\Msigma[\sX]$, then the \emph{pushforward} measure $\psi\mu$
\index{pushforward}
(defined by $\psi\mu(A)=\mu(\psi^{-1}(A))$ for Baire measurable sets
$A$ in the ideal $\sY$ on $\YY$) belongs to $\Msigma[\sY]$. The same
conclusion holds for Borel measures if $\psi$ is Borel measurable. It should be
noted that it is possible for $\psi\mu$ to be trivial, that is, it
may vanish on all sets in $\sY$.  The following result merges the
mapping theorems provided in \cite[Theorem~2.5]{hul:lin06},
\cite[Theorem~2.3]{lin:res:roy14} and \cite[Theorem~B.1.21]{kul:sol20}
for (bounded on a bornology) Borel measures on Polish spaces with the
continuous mapping theorem for finite Baire measures available in
\cite[Theorem~4.3.12]{bogachev18}.

\begin{lemma}
  \label{lemma:mapping}
  Let $\psi:\XX\to\YY$ be a bornologically consistent continuous map
  between topological spaces $\XX$ and $\YY$, equipped with 
  topologically consistent ideals $\sX$ and $\sY$.
  If $\mu_\gamma \vto[\sX]\mu$ for $\mu_\gamma,\mu\in\Msigma$,
  then $\psi\mu_\gamma\vto[\sY]\psi\mu$.
\end{lemma}
\begin{proof}
  Consider $f\in \fC(\YY,\sY)$. Assume that $f$ is supported on
  $A\in\sY$. Since $\psi$ is continuous, $f\circ \psi$ is continuous
  with support contained in the set $\psi^{-1}(A)\in\sX$, so that
  $f\circ \psi\in\fC(\XX,\sX)$. Since $\psi\mu\in\Msigma[\sY]$, we have
  \begin{displaymath}
    \lim_\gamma\int f\diff (\psi\mu_\gamma)
    =\lim_\gamma\int (f\circ \psi)\diff \mu_\gamma
    = \int f\diff (\psi\mu). \qedhere
  \end{displaymath}
\end{proof}

\begin{remark}
  \label{rem:cont-maps}
  The statement of Lemma~\ref{lemma:mapping} holds also if $\psi$ is
  continuous on any set $A\in\sX$ or, equivalently, on the union of all
  sets in $\sX$ if the ideal $\sX$ has a countable base.  On Polish
  spaces, the statement of Lemma~\ref{lemma:mapping} can be extended
  to maps $\psi$ that are discontinuous on a set of vanishing
  $\mu$-measure. 
  In the non-Polish setting, one may assume that the space $\XX$
  is completely regular, the map $\psi$ is Borel measurable, the
  measures $\mu_\gamma$ are Borel, and the limiting measure $\mu$ is
  tau-additive; see \cite[Theorem~4.3.12]{bogachev18} in the
  setting of weak convergence. The definition of a tau-additive
  measure can be found on page~\pageref{def:tau-additive}. 
\end{remark}

\paragraph{The vague topology}
The \emph{vague topology} associated with an ideal $\sS$
has a base consisting of the sets
\index{vague topology}
\begin{equation}
  \label{eq:22}
  \sU=\Big\{\nu\in \Msigma\colons \Big|\int f_i\diff \nu-\int
  f_i\diff \mu\Big|<\eps, \; i=1,\dots,n\Big\}
\end{equation}
for all $f_1,\dots,f_n\in \fC(\XX,\sS)$, $n\geq1$, $\eps>0$ and
$\mu\in\Msigma$.  By construction, this topology corresponds to the
definition of vague convergence. It is evident that the vague
topology becomes finer (stronger) if the ideal $\sS$ is replaced by a richer
one.

\begin{lemma}
  \label{lemma:topology}
  Let $\sS$ be a topologically consistent ideal on a topological space
  $\XX$.  The vague
  topology on $\Msigma$ has an open base consisting of
  intersections of the sets
  \begin{align*}
    \big\{\nu\in \Msigma &\colons \nu(F_i)<\mu(F_i)+\eps,i=1,\dots,n \big\},\\
    \big\{\nu\in \Msigma &\colons \nu(G_i)>\mu(G_i)-\eps,i=1,\dots,n \big\},
  \end{align*}
  for all functionally closed sets $F_1,\dots,F_n\in\sS$, 
  functionally open sets $G_1,\dots,G_n\in\sS$, all $n\geq1$,
  $\eps>0$, and all $\mu\in \Msigma$.
\end{lemma}
\begin{proof}
  Fix a $\mu\in \Msigma$. Consider $\sU$ from
  \eqref{eq:22}.  By the definition of an ideal and topological
  consistency, there exists a functionally closed set $B\in\sS$, which
  contains the supports of all functions $f_1,\dots,f_n$. Note that
  $\mu(B)>0$. By
  decomposing functions into the differences of their positive and
  negative parts, it is possible to assume that $f_i(x)\in[0,c)$ for
  all $x$ and $i$. Consider a partition of $[0,c]$ by points
  $0=c_0<c_1<\cdots<c_m=c$ such that $c_{j+1}-c_j<\eps/(6\mu(B))$ and
  $\mu(f^{-1}_i(c_j))=0$ for all $i=1,\dots,n$ and
  $j=0,\dots,m$. This partition exists, since $\mu(B)$ is finite and
  the inverse images $f_i^{-1}(x)$ are disjoint functionally closed (and
  so Baire measurable sets) for different $x$. 

  Denote $A_{ij}=f^{-1}_i((c_{j-1},c_j])$,
  $j=1,\dots,m$. Choose $\delta=\eps/(2m)$, and let
  \begin{multline*}
    V=\bigcap_{j=1}^m\bigcap_{i=1}^n \Big\{\nu\in \Msigma\colons 
    \nu(F_{ij})<\mu(F_{ij})+\delta, \nu(G_{ij})>\mu(G_{ij})-\delta\Big\}\\
    \cap \Big\{\nu\in \Msigma\colons |\nu(B)-\mu(B)|<\mu(B)\Big\},
  \end{multline*}
  where $F_{ij}=f_i^{-1}([c_{j-1},c_j])$ and
  $G_{ij}=f_i^{-1}((c_{j-1},c_j))$. The only exception is $F_{i1}$
  which is a functionally closed set from $\sS$ which contains
  $A_{i1}$.  If $\nu\in V$, then
  \begin{displaymath}
    \nu(A_{ij})\geq \nu(G_{ij})>\mu(G_{ij})-\delta=\mu(A_{ij})-\delta
  \end{displaymath}
  and, similarly, $\nu(A_{ij})<\mu(A_{ij})+\delta$, so that
  \begin{displaymath}
    |\nu(A_{ij})-\mu(A_{ij})|< \delta.
  \end{displaymath}
  Since $|f_i-\sum_j c_j\one_{A_{ij}}|<\eps/(6\mu(B))$ and $\nu(B)$ is
  finite, we have
  \begin{align*}
    \Big|\int f_i\diff \mu-\int f_i\diff \nu\Big|
    \leq \frac{\eps}{6\mu(B)}\left(\mu(B)+\nu(B)\right)+m\delta<\eps,
  \end{align*}
  meaning that $V\subset \sU$.

  Consider the converse. Let
  \begin{displaymath}
    V=\Big\{\nu\in \Msigma\colons 
    \nu(F)<\mu(F)+\eps, \nu(G)>\mu(G)-\eps\Big\}.
  \end{displaymath}
  Since $F$ is a functionally closed set, $F=\{x\colons g(x)=0\}$ for a
  continuous non-negative function $g$. The functionally closed sets
  $F$ and $F_1=\{x\colons g(x)\geq\eps\}$ are disjoint and so can be
  precisely separated by a continuous function $f_1:\XX\to[0,1]$ that takes
  the value $1$ on $F$ and the value $0$ on $F_1$; see
  \cite[Theorem~1.5.13]{eng89}. By choosing an appropriate $\eps>0$,
  this yields a function $f_1$ such that $f_1(x)=1$ for all $x\in F$
  and $\mu(f_1^{-1}((0,1)))\leq \eps/2$.

  Since $G$ is functionally open, write $G=\{h>0\}$ for a continuous
  function $h$. Choose $\eps>0$ such that
  $\mu(G\setminus\{h\geq\eps\})$ is small, and choose a continuous
  function $f_2:\XX\to[0,1]$ with $0\leq f_2\leq\one_G$ and $f_2=1$ on
  $\{h\geq\eps\}$. In particular, $f_2(x)=1$ for all $x\notin G$ and
  it is possible to find $\eps$ so that
  $\mu(f_2^{-1}((0,1)))\leq \eps/2$. If
  \begin{displaymath}
    \nu\in\sU=\Big\{\nu \in \Msigma\colons \Big|\int f_i\diff \nu-\int
    f_i\diff \mu\Big|<\delta, i=1,2\Big\}
  \end{displaymath}
  for $\delta=\eps/2$, then
  \begin{displaymath}
    \int f_1\diff \nu<\int f_1\diff \mu+\delta\leq \mu(F)+\eps
  \end{displaymath}
  and
  \begin{displaymath}
    \int f_2\diff \nu>\int f_2\diff \mu -\delta\geq \mu(G)-\eps,
  \end{displaymath}
  so that $\sU\subset V$. 
\end{proof}

\begin{definition}
  \label{def:mu-cont}
  Let $\mu\in\Msigma$. 
  A Baire set $B$ is said to be a \emph{$\mu$-continuity set} if there
  exist a functionally open set $G$ and a functionally closed set $F$ in
  $\sX$ such
  that $G\subset B\subset F$ and $\mu(F\setminus G)=0$.
  \index{continuity set} \index{mu-continuity set@$\mu$-continuity set}
\end{definition}

In perfectly normal spaces, the $\mu$-continuity property of $B$ is
equivalent to the boundary of $B$ having vanishing measure.

\begin{lemma}
  \label{lemma:indicator}
  Let $B$ be a $\mu$-continuity set for $\mu\in\Msigma$, with the ideal
  $\sS$ having a countable open base. Then, for
  each $\eps>0$, there exist $f_1,f_2\in\fC(\XX,\sS)$ such that
  $f_1\leq \one_B\leq f_2$ and $\int (f_2-f_1)\diff \mu\leq \eps$.
\end{lemma}
\begin{proof}
  Assume that $G$ and $F$ are the sets from
  Definition~\ref{def:mu-cont} that enclose $B$.
  There exists a set $\base_n$ from the
  base of $\sS$ such that $F\subset\base_n$. The functionally closed
  sets $F$ and $\base_n^c$ can be separated by a continuous function
  $h_1$ that takes the value $1$ on $F$ and the value $0$ on $\base_n^c$;
  see \cite[Theorem~1.5.13]{eng89}.
  Choose a sufficiently small $\delta'>0$ such that
  $\mu(\{h_1>1-\delta'\}\setminus F)<\eps/2$.  Let $f_2$ be
  the function that separates $F$ and $\{h_1\leq 1-\delta'\}$. Then
  $\int f_2\diff \mu-\mu(F)\leq \eps/2$. Since $f_2$ takes the value $1$ on
  $F$, we have $f_2\geq \one_B$. 

  Since $G$ is functionally open, we have $G=\{h_2>0\}$ for a
  continuous function $h_2$. Choose $\delta''$ such that $\mu(\{h_2\geq
  \delta''\}\setminus G)<\eps/2$. The two functionally closed sets
  $\{h_2\geq \delta''\}$ and $G^c$ can be separated by a continuous
  function $f_1$, which then satisfies $\mu(G)-\int f_1\diff \mu\leq
  \eps/2$. It remains to note that $f_1\leq \one_B$. 
\end{proof}

The following result states that it is possible to amend the base of an
ideal to ensure that it consists of $\mu$-continuity sets.

\begin{lemma}
  \label{lemma:continuity-base}
  Let $(\base_m)_{m\in\NN}$ be an open base of a topologically
  consistent ideal $\sX$ on a topological space $\XX$. For each
  $\mu\in\Msigma[\sX]$, there exists another open base
  $(\base'_m)_{m\in\NN}$ for $\sX$ consisting of $\mu$-continuity
  sets.
\end{lemma}
\begin{proof}
  Fix an $m\in\NN$ and argue as in Lemma~\ref{lemma:cont-base}. By the 
  topological consistency of $\sX$, there exists a functionally closed
  set $F\in\sX$ such that $\base_m\subset F$. Since
  $(\base_m)_{m\in\NN}$ is a base, we have likewise $F\subset \base_n$
  with $n>m$. By eliminating some sets from the base, we can ensure
  that $n=m+1$, yielding $\base_m\subset F\subset\base_{m+1}$. The two
  functionally closed sets $F$ and $\base_{m+1}^c$ are separated by a
  continuous function $f_m:\XX\to[0,1]$. Since $\mu(\base_{m+1})<\infty$, we
  have $\mu(\{f_m=t\})=0$ for all but at most a countable number of
  $t\in(0,1)$. Fix a $t_m$ such that $\mu(\{f_m=t_m\})=0$ and let
  $\base'_m=\{f_m<t_m\}$.
\end{proof}

The following result is an analogue of the Portmanteau theorem for 
vague convergence. Similar results for ideals on metric spaces
obtained by excluding a point can be found in \cite{MR2271177} and
\cite{hul:lin06}, and results for ideals with a
countable base on Polish spaces are provided in \cite[Lemma~4.1]{kalle17}.

\begin{theorem}[Portmanteau theorem for vague convergence]
  \label{lemma:vVSw}
  Let $\sS$ be a topologically consistent ideal on
  a topological space $\XX$. For a net
  $(\mu_\gamma)\subset \Msigma$ and a measure
  $\mu\in \Msigma$, the following statements are
  equivalent.
  \index{vague convergence!Portmanteau theorem}
  \index{Portmanteau theorem}
  \begin{enumerate}[(i)]
  \item $\mu_\gamma \vto \mu$.
  \item We have $\lim_\gamma \mu_\gamma(B)=\mu(B)$ for all
    $\mu$-continuity sets $B\in\sS\cap\sBA(\XX)$.
  \item For all functionally open sets $G\in\sS$ and for all functionally
    closed sets $F\in\sS$,
    \begin{displaymath}
      \limsup_\gamma \mu_\gamma(F)\leq \mu(F),\quad
      \liminf_\gamma \mu_\gamma(G)\geq \mu(G). 
    \end{displaymath}
  \end{enumerate}
  If $\sS$ has a countable open base, then the above conditions are
  equivalent to the following condition. 
  \begin{enumerate}
  \item[(iv)] There exists a countable open base $(D_m)_{m\in\NN}$ of $\sS$
    such that $\mu_\gamma|_{D_m} \wto \mu|_{D_m}$ for all $m\in\NN$,
    where $\mu|_{D_m}$ is the restriction of $\mu$ to $D_m$.
  \end{enumerate}
\end{theorem}
\begin{proof}
  Vague convergence is equivalent to (iii) by
  Lemma~\ref{lemma:topology}, which shows that the vague topology has
  a base defined in terms of functionally open/closed sets.
  The equivalence of (ii) and (iii) is
  standard.

  \smallskip
  \noindent (i)$\Rightarrow$(iv) Let $(\base_m)$ be an open base for
  $\sS$. By Lemma~\ref{lemma:continuity-base}, we can construct a new
  countable open base $(D_m)_{m\in\NN}$ for $\sS$ 
  consisting of $\mu$-continuity sets.  Let $f$ be a continuous
  non-negative bounded function on $\XX$, bounded by
  $c>0$. Fix $m\in\NN$ and $\eps>0$. By Lemma~\ref{lemma:indicator},
  there exist $f_1,f_2\in\fC(\XX,\sS)$ such that
  $f_1\leq \one_{D_m}\leq f_2$ and
  \begin{displaymath}
    \int (f_2-f_1)\diff \mu<\eps/c. 
  \end{displaymath}
  Then
  $f_1f\leq f\one_{D_m}\leq f_2f$ and $\int (ff_2-ff_1)\diff \mu<\eps$.  By
  (i),
  \begin{displaymath}
    \int f_if\diff \mu_\gamma\to\int f_1f\diff \mu, \quad i=1,2.
  \end{displaymath}
  The arbitrary choice of $\eps>0$ ensures that
  \begin{displaymath}
    \int f\one_{D_m}\diff \mu_\gamma\to \int f\one_{D_m}\diff \mu.
  \end{displaymath}
  Thus, $\mu_\gamma|_{D_m}\wto\mu|_{D_m}$ as finite measures on $\XX$.
  
  \smallskip
  \noindent
  (iv)$\Rightarrow$(i) If $f\in\fC(\XX,\sS)$, then the support of
  $f$ is a subset of $D_m$ for some $m$. Hence,
  \begin{displaymath}
    \int f \diff \mu_\gamma = \int f\one_{D_m}\diff \mu_\gamma 
    = \int f \diff \mu_\gamma|_{D_m} \to \int f \diff \mu|_{D_m} = \int f
    \diff \mu. \qedhere
  \end{displaymath}
\end{proof}

The following result addresses relative compactness in the vague
topology. Similar results in the context of Polish spaces can be found
in \cite[Theorem~4.2]{kalle17} and
\cite[Theorem~2.6]{hul:lin06}. Since we work with general topological
spaces, the following result does not involve the tightness condition.

\begin{lemma}
  \label{lemma:rsc}
  A subset $M$ of $\Msigma$ is relatively sequentially compact in the
  vague topology on $\sS$ if there exists a countable open base
  $(\base_m)_{m\in\NN}$ for $\sS$, such that the sets
  $M_i=\{\mu|_{\base_i}\colons \mu\in M\}$ built of restrictions of $\mu\in
  M$ to $\base_i$ are relatively compact in the
  weak topology for all $i\in\NN$.
  \index{relative compactness of measures}
\end{lemma}
\begin{proof}
  We present a proof that very closely follows the proof of
  Theorem~2.6 in \cite{hul:lin06}. Let $(\mu_n)_{n\in\NN}$ be a
  sequence in $M$. By assumption, there is a weakly convergent
  subsequence $(\mu_{n_1(k)})_{k\in\NN}$ such that
  $\mu_{n_1(k)}|_{\base_1}\wto\nu_1$ for $\nu_1\in\Msigma$. Taking
  further subsequences and applying the diagonal argument we obtain a
  subsequence $(\mu_{n_k(k)})_{k\in\NN}$ such that
  $\mu_{n_k(k)}|_{\base_i}\wto\nu_i$ as $k\to\infty$ for all
  $i\in\NN$.

  Let $f\in\fC(\XX,\sS)$. The support of $f$ is a subset of some
  $\base_i$. Because of this, for all $j\geq i$, we have
  \begin{displaymath}
    \int f \diff \nu_j=\lim_k \int f \diff \mu_{n_k(k)}|_{\base_j}
    =\lim_k \int f \diff \mu_{n_k(k)}|_{\base_i}=\int f\diff \nu_i.
  \end{displaymath}
  Therefore, $\nu_j|_{\base_i}=\nu_i$.  By Lemma~\ref{lemma:enclosed},
  there exists a measure $\nu\in\Msigma$ such that
  $\nu|_{\base_i}=\nu_i$ for all $i\in\NN$. The above argument shows that
  $\mu_{n_k(k)}$ vaguely converges to $\nu\in\Msigma$ on the ideal $\sS$. 
\end{proof}

Recall that the uniform boundedness and the tightness of a family of
Radon measures on a completely regular space imply relative
compactness in the weak topology, and also the weak sequential compactness
if all compact sets in $\XX$ are metrisable; see
\cite[Theorem~4.5.3]{bogachev18}. The inverse implication (and thus the
equivalence of tightness and relative compactness) is the essence of
the Prokhorov theorem, which holds in Polish spaces.

\paragraph{Characterisation of vague convergence using the polar decomposition} 
A Hausdorff topological space is \emph{Lindel\"of} if each of its open
\index{space!Lindeloef@Lindel{\"o}f}
covers has a countable subcover (in metric spaces this is equivalent to
separability); see \cite[Definition~6.1.4]{bogachev07} and
\cite[Section~3.8]{eng89}, where the regularity property is
additionally imposed. The space is \emph{hereditary Lindel\"of} if the
Lindel\"of property holds for any subspace or, equivalently, if all open
subsets have this property; see \cite[Exercise~3.8.A]{eng89}.  Every
perfectly normal Lindel\"of space is hereditary Lindel\"of; see
\cite[Exercise~3.8.A]{eng89}.  Furthermore, every Souslin space is
hereditary Lindel\"of; see \cite[Lemma~6.6.4]{bogachev07}.
\index{space!hereditary Lindeloef@hereditary Lindel{\"o}f}

A non-negative Baire measure is said to be \emph{tau-additive} (usually
written as $\tau$-additive -- note
that this has nothing to do with the letter $\tau$ used for a modulus)
if \label{def:tau-additive}
\index{measure!tau-additive}
\begin{displaymath}
  \mu\Big(\bigcup_{\gamma\in\Gamma} G_\gamma\Big)
  =\lim_\gamma \mu(G_\gamma)
\end{displaymath}
for each increasing net $(G_\gamma)_{\gamma\in\Gamma}$ of functionally
open sets such that their union is also functionally open. This
definition is traditionally applied to Borel measures (see
\cite[Definition~7.2.1]{bogachev07}), and its Baire variant is
mentioned in \cite[Corollary~7.3.3]{bogachev07}, which states that every
tau-additive measure on a completely regular space admits a unique
extension to a tau-additive Borel measure.  Every locally finite
Borel measure
on a second-countable topological space is tau-additive; see
\cite[Proposition~3.1(a)]{MR1435288}.  A Borel measure on a metric
space is tau-additive if and only if its topological support is
separable; see \cite[page~49]{bogachev18}. Furthermore, every
Baire measure on a hereditary Lindel\"of space is tau-additive and
every Radon measure is tau-additive; see
\cite[Proposition~7.2.2]{bogachev07} for the Borel version. 

The following result aims to simplify the process of checking vague
convergence.  If Condition~(B) from Section~\ref{sec:bornology}
(see page~\pageref{eq:10}) holds and $\sS_\modulus$ is generated by
a continuous modulus $\modulus$, let $\sA_\mu$ denote the family of
sets of the form
\begin{equation}
  \label{eq:34}
  A=\Big\{x\colons \modulus(x)\in(a,b), T_{\modulus(x)^{-1}}x\in V\Big\}
  =\rho^{-1}\big(V\times(a,b)\big)
\end{equation}
for all $0<a<b\leq\infty$, all functionally open sets $V$, and such that
$A$ is a $\mu$-continuity set. Recall that $\rho$ is the polar
decomposition map from Definition~\ref{def:polar-decomposition} and
note that $\sA_\mu$ is a subfamily of $\sS_\modulus$.

\begin{theorem}
  \label{thr:basis}
  Assume that a completely Hausdorff space $\XX$ is equip\-ped with
  a continuous scaling and $\sS$ is an ideal that satisfies
  Condition~\hyperref[condB]{(B)}. Let $\mu$ be a tau-additive Baire measure.
  Then $\mu_\gamma\vto\mu$ is equivalent to the following condition.
  \begin{enumerate}
  \item[(v)] There exists a countable open base $(D_m)_{m\in\NN}$ for
    $\sS$, consisting of $\mu$-continuity sets, such that
    $\mu_\gamma(D_m)\to\mu(D_m)$ for all $m\in\NN$ and 
    $\mu_\gamma(A)\to \mu(A)$ for all $A\in\sA_\mu$.
  \end{enumerate}
\end{theorem}
\begin{proof}
  \textsl{Necessity} follows from
  Theorem~\ref{lemma:vVSw}(ii), since $A\in\sA_\mu$ is a subset of
  $D_m$ for sufficiently large $m$.

  \smallskip
  \noindent
  \textsl{Sufficiency.} Let $G$ be a functionally open set in
  $\sS$. For each $x\in G$ we have $\modulus(x)\in(0,\infty)$.  By
  Lemma~\ref{lemma:sector},
  \begin{displaymath}
    x\in A_x=\Big\{y\colons \modulus(y)\in[t,s], T_{\modulus(y)^{-1}}y\in V\Big\}\subset G
  \end{displaymath}
  for some $0<t<s\leq\infty$ and a functionally open set
  $V=\{z\colons f(z)>0\}$,
  where $f$ is a continuous function. Note that
  $\eps=f(T_{\modulus(x)^{-1}}x)>0$ and
  \begin{multline*}
    x\in U=\Big\{y\colons \modulus(y)\in(t,s),
    f(T_{\modulus(y)^{-1}}y)>\delta\Big\}\\
    \subset 
    F=\Big\{y\colons \modulus(y)\in[t,s], f(T_{\modulus(y)^{-1}}y)\geq
    \delta\Big\}\subset G.
  \end{multline*}
  Since $\mu(G)$ is finite, by varying $t$ and $s$, and
  $\delta\in(0,\eps)$, we can ensure that $U$ is a functionally open
  $\mu$-continuity set.
  Thus, each $x\in G$ belongs to some functionally open $U\in\sA_\mu$
  with $U\subset G$, implying that each functionally open set $G$ is the
  union of open sets in $\sA_\mu$.


  
  Convergence $\mu_\gamma(A)\to \mu(A)$ on all $A\in\sA_\mu$ and
  the closedness of the family $\sA_\mu$ under finite intersections imply that
  $\mu_\gamma$ converges to $\mu$ on the family $\sU$ of all finite
  unions of sets in $\sA_\mu$. For every functionally open $G$ and
  functionally open $U\subset G$ such that $U\in\sU$, we have
  \begin{displaymath}
    \mu(U)=\lim_\gamma\mu_\gamma(U)\leq \liminf_\gamma \mu_\gamma(G). 
  \end{displaymath}
  Since $G$ is the union of all sets in $\sA_\mu$ that are
  contained in $G$, the tau-additivity property yields that
  \begin{displaymath}
    \mu(G)=\sup\big\{\mu(U)\colons U\subset G,U\in\sU\big\}
    \leq \liminf_\gamma \mu_\gamma(G).
  \end{displaymath}
  Let $F$ be a functionally closed set. Then $F\subset D_m$ for some
  $m$ and $G=D_m\setminus F$ is functionally open.
  Since $\mu_\gamma(D_m)\to\mu(D_m)$, and applying the above to
  $G$, we have that
  \begin{displaymath}
    \limsup_\gamma \mu_\gamma(F)\leq \mu(F).
  \end{displaymath}
  The statement follows from Theorem~\ref{lemma:vVSw}(iii).
\end{proof}

\paragraph{Vague convergence on product spaces}
If $\XX$ is a topological space with the ideal $\sX$, then
a family $\sfD$ of bounded Baire measurable functions $f:\XX\to\R$
is said to determine vague convergence (on the
ideal $\sX$) if for all nets
\index{convergence determining class}
$(\mu_\gamma)_{\gamma\in\Gamma}\subset \Msigma$ of Baire measures and
a measure $\mu\in \Msigma$, the convergence
\begin{equation}
  \label{eq:intf-gamma}
  \lim_\gamma \int f\diff \mu_\gamma=\int f \diff \mu, \quad f\in\sfD,
\end{equation}
implies that $\mu_\gamma\vto\mu$. For a recent discussion of
convergence determining classes in the setting of weak convergence;
see \cite{MR2673979}. By definition, the family of functions
from $\fC(\XX,\sX)$
determines vague convergence. By Theorem~\ref{lemma:vVSw}, the
family of indicators of all $\mu$-continuity Baire sets in $\sX$
also determines vague convergence (with the limit being $\mu$) if
$\sX$ is topologically consistent. Note that we do not require 
the functions $f\in\sfD$ to be integrable with respect to $\mu$ or
$\mu_\gamma$; for example,
it is possible for $\sfD$ to include the function identically equal to
one, in which case both sides of \eqref{eq:intf-gamma} may be
infinite, but the equality is understood in the extended sense.

The following result concerns the conditions for vague convergence on
product spaces.

\begin{lemma}
  \label{lemma:conv-det-product}
  Let $\XX$ and $\YY$ be two perfectly normal topological
  spaces. Assume that $\sX$ is equipped with a topologically
  consistent ideal $\sX$ having a countable open base. Let $\sV$ be
  the ideal on $\XX\times\YY$ generated by sets $A\times\YY$ for all
  $A\in\sX$. Assume that $\sfD_\sX$ and $\sfD_\sY$ are two families of
  functions that determine vague convergence on $\sX$ and 
  weak convergence on $\YY$, respectively, and are such that the family
  $\sfD_\sY$ contains the function that is identically one.  If
  $\nu\in\Mb[\sV]$ is a tau-additive Borel measure and
  $(\nu_\gamma)_{\gamma\in\Gamma}\subset\Mb[\sV]$ is a net of Borel
  measures on $\XX\times\YY$ such that
  \index{product space!vague convergence on}
  \begin{equation}
    \label{eq:product-two-elements}
    \lim_\gamma
    \int f(x)h(y) \nu_\gamma(\bdiff (x,y))=\int f(x)h(y) \nu(\bdiff (x,y)),
    \quad f\in\sfD_\sX, h\in\sfD_\sY,
  \end{equation}
  then $\nu_\gamma\vto[\sV]\nu$. 
\end{lemma}
\begin{proof}
  For a Borel measure $\nu$ on the product space $\XX\times\YY$,
  denote by $\nu^{(1)}$ its projection onto $\XX$.  Note that
  $\nu\in\Mb[\sV]$ ensures that $\nu$ is finite on sets $A\times\YY$
  for $A\in\sX$, and thus $\nu^{(1)}\in\Mb[\sX]$. Let
  $B_1\in\sX$ be a Borel $\nu^{(1)}$-continuity set. By
  \eqref{eq:product-two-elements}, setting $h$ to be identically one and
  applying Theorem~\ref{lemma:vVSw}, we have
  \begin{displaymath}
    \lim_\gamma\nu^{(1)}_\gamma(B_1)= \nu^{(1)}(B_1). 
  \end{displaymath} 
  Assume $\nu^{(1)}(B_1)>0$ (and is necessarily finite), and,
  without loss of generality, that $\nu^{(1)}_\gamma(B_1)>0$ for all
  $\gamma$.  Fix any $h\in\sfD_\sY$ and denote by $\mu^{(1)}_\gamma$
  the measure given by
  \begin{displaymath}
    \mu^{(1)}_\gamma(A)=\int \one_{x\in A}h(y)\nu_\gamma(\bdiff (x,y)),
    \quad A\in\sB(\XX),
  \end{displaymath}
  and, similarly, derive $\mu^{(1)}$ from $\nu$. For each
  $f\in\sfD_\sX$, \eqref{eq:product-two-elements} implies that
  \begin{displaymath}
    \int f\diff \mu^{(1)}_\gamma = \int f(x)h(y) \nu_\gamma(\bdiff (x,y))
    \to \int f(x)h(y) \nu(\bdiff (x,y)) =\int f\diff \mu^{(1)}.
  \end{displaymath}
  Hence, $\mu^{(1)}_\gamma\vto[\sX]\mu^{(1)}$. Since $B_1\in\sX$ is a
  $\nu^{(1)}$-continuity set,
  \begin{displaymath}
    \mu^{(1)}(\partial B_1)
    =\int \one_{x\in \partial B_1}h(y)\nu(\bdiff (x,y))
    \leq (\sup h) \nu^{(1)}(\partial B_1)=0,
  \end{displaymath}
  and thus $B_1$ is also a $\mu^{(1)}$-continuity set. By the
  Portmanteau theorem, 
  \begin{displaymath}
    \lim_\gamma \int \one_{x\in B_1} h(y) \nu_\gamma(\bdiff (x,y))
    = \int \one_{x\in B_1} h(y) \nu(\bdiff (x,y)).
  \end{displaymath}

  Furthermore, define the normalised measure 
  \begin{displaymath}
    \mu^{(2)}_\gamma(A)
    =\frac{\nu_\gamma(B_1\times A)}{\nu^{(1)}_\gamma(B_1)}, \quad
    A\in\sB(\YY),
  \end{displaymath}
  and $\mu^{(2)}$ likewise.  For each $h\in\sfD_\sY$, the ratio of
  limits yields  
  \begin{multline*}
    \int h\diff  \mu^{(2)}_\gamma
    =\frac{1}{\nu^{(1)}_\gamma(B_1)}
    \int\one_{x\in B_1} h(y)\nu_\gamma(\bdiff (x,y))\\
    \to \frac{1}{\nu^{(1)}(B_1)}\int\one_{x\in B_1} h(y)\nu(\bdiff (x,y))
    =\int h\diff  \mu^{(2)},
  \end{multline*}
  implying $\mu^{(2)}_\gamma\wto\mu^{(2)}$. If $B_2$ is a
  $\mu^{(2)}$-continuity set, then
  \begin{displaymath}
    \nu_\gamma(B_1\times B_2)=\mu^{(2)}_\gamma(B_2)\nu^{(1)}_\gamma(B_1)
    \to \mu^{(2)}(B_2)\nu^{(1)}(B_1)
    =\nu(B_1\times B_2).
  \end{displaymath}
  Letting $h$ be identically one, we obtain that
  $\nu_\gamma(B_1\times\YY)\to \nu(B_1\times\YY)$.

  If $\nu^{(1)}(B_1)=0$, then
  \begin{displaymath}
    \nu_\gamma(B_1\times B_2)
    \leq \nu^{(1)}_\gamma(B_1)\to 0.
  \end{displaymath}
  Consequently, the limit $\nu_\gamma(B_1\times B_2)\to 0$ holds.
  Thus, for all $\nu^{(1)}$-continuity
  sets $B_1$ and $\mu^{(2)}$-continuity sets $B_2$, we establish that
  \begin{equation}
    \label{eq:B1xB2}
    \nu_\gamma(B_1\times B_2)\to \nu(B_1\times B_2).
  \end{equation}

  In view of Lemma~\ref{lemma:continuity-base}, assume that
  $(\base_m)_{m\in\NN}$ is a countable open base for $\sX$ 
  consisting of $\nu^{(1)}$-continuity sets. Choosing
  $D_m=\base_m\times\YY$ yields $\nu_\gamma(D_m)\to\nu(D_m)$ for
  all $m$. Note that $(D_m)_{m\in\NN}$ forms an open base for $\sV$.


  Fix $m\in\NN$. Let $G$ be an open subset of
  $D_m=\base_m\times\YY$, and let $(x,y)\in G$.
  Since the sets $G_1\times G_2$ for all open sets $G_1$ in $\XX$ and
  $G_2$ in $\YY$ constitute a base for the product topology, we have
  $G\supset G_1\times G_2\ni(x,y)$ for some $G_1$ and $G_2$. By
  topological consistency of $\sX$, it is possible to assume that
  $G_1\in\sX$. The set $G_1^c$ and the point $x$ are disjoint and thus can
  be separated by a continuous function. Examining its level sets, we
  identify an open set $B_1$ such that $x\in B_1\subset G_1$ and $B_1$
  is a $\nu^{(1)}$-continuity set. Applying the same construction to $y$ and
  $G_2$ and the measure $\mu^{(2)}$ yields an open set $B_2\subset\YY$
  such that
  \begin{displaymath}
    (x,y)\in B_1\times B_2\subset G 
  \end{displaymath}
  and $B_2$ is a $\mu^{(2)}$-continuity set. Since sets of the type
  $B_1\times B_2$ form a family closed under finite intersections,
  \eqref{eq:B1xB2} holds for all sets $U$ in the family
  $\mathcal{U}$ of finite unions of $B_1\times B_2$ chosen as
  above. Since $G$ is the union of these sets $B_1\times B_2$, we have
  \begin{displaymath}
    \nu(U)=\lim_\gamma \nu_\gamma(U)
    \leq \liminf_\gamma \nu_\gamma(G), \quad U\in\mathcal{U}.
  \end{displaymath}
  The tau-additivity of $\nu$ implies 
  \begin{displaymath}
    \nu(G)=\sup_{U\in\mathcal{U}} \nu(U)
    \leq \liminf_\gamma \nu_\gamma(G).
  \end{displaymath}
  Since this holds for every open $G\subset D_m$, the Portmanteau
  theorem yields
  \[
    \nu_\gamma|_{D_m}\wto \nu|_{D_m}
  \]
  as finite measures on $\XX\times\YY$, for each $m$.  The proof
  concludes by applying Theorem~\ref{lemma:vVSw}(iv).
\end{proof}

\begin{remark}
  Recall the products of two ideals $\sX\times\sY$ and $\sX\otimes\sY$
  defined by \eqref{eq:product-two-ideals} and
  \eqref{eq:product-two-ideals-1}. The setting of
  Lemma~\ref{lemma:conv-det-product} corresponds to the case where
  $\sV=\sX\times\sY$ with $\sY$ containing $\YY$ and consequently all its
  subsets. A variant of Lemma~\ref{lemma:conv-det-product} holds for
  $\sV=\sX\otimes\sY$, assuming that $\sfD_\sX$ also contains a
  function identically equal to one.

  Note also that it is possible to stipulate 
  \eqref{eq:product-two-elements} only for non-negative functions $f$
  and $g$.
\end{remark}

\section{Regularly varying measures on topological spaces}
\label{sec:regul-vary-meas}

\paragraph{Definition of regular variation}
Let $\RV_\alpha$ denote the family of functions which are regularly
\index{regularly varying function}
\index{function!regularly varying}
varying at infinity with index $\alpha>0$, that is, the family of all
positive measurable functions $g:(a,\infty)\to\Rpp$ for some
$a\in \R$ with the property
\begin{displaymath}
  \lim_{t\to \infty} \frac{g(x t)}{g(t)} = x^{\alpha}
\end{displaymath}
for all $x>0$.

A measure $\mu\in\Msigma$ is said to be \emph{non-trivial} on an ideal
$\sS$ if $\mu(B)>0$ for some $B\in\sS$.
\index{measure!non-trivial}

\begin{definition}
  \label{def:RV}
  Let $\sS$ be a topologically and scaling consistent ideal on a
  topological space $\XX$ equipped with a scaling operation.  A Baire
  measure $\nu$ on $\sBA(\XX)$ is said to be \emph{regularly varying}
  on $\sS$ if there 
  exists a measurable function $g:\R_+\to\R_+$ with $g(t)\to\infty$ as
  $t\to\infty$ and such that \index{measure!regularly varying}
  \index{regularly varying measure} \index{regular variation}
  \begin{equation}
    \label{eq:24g}
    g(t)T_{t^{-1}}\nu\vto\mu\quad \text{as}\; t\to\infty
  \end{equation}
  where $\mu \in\Msigma$ is non-trivial on $\sS$ (in particular, $\mu$
  is finite on each set in $\sS$). The measure $\mu$
  is called the \emph{tail measure} of $\nu$ and
  \index{tail measure} 
  the function $g$ is called the \emph{normalising function}.
  \index{normalising function}
  We then write $\nu\in \RV(\XX,T,\sS,g,\mu)$ (or
  $\nu\in\RV(\sS)$ if the other components are clear
  from the context).
\end{definition}

Recall that linear scaling on $\XX$ is denoted by $\mydot$ and
then the regular variation of $\nu$ is denoted as
$\nu\in \RV(\XX,\mydot,\sS,g,\mu)$.
Note that the measure $\nu$ in Definition~\ref{def:RV} (unlike the
tail measure $\mu$) is not necessarily supported on the union of all
sets in $\sS$. While this definition accommodates possibly infinite
measures $\nu$, it is usually applied to probability measures.

The regular variation property is often formulated by assuming the
existence of an increasing sequence $(a_n)_{n\in\NN}$ of positive
constants such that $n(T_{a_n^{-1}}\nu)$ converges vaguely in $\sS$ to
a measure $\mu$ as $n\to\infty$. A variant of this definition is given in
\cite{hul:lin06},
where the sequence $(a_n)$ is assumed to be regularly varying. Another
variant appears in \cite[Definition~B.2.1]{kul:sol20}, where the
sequence $(a_n)$ is assumed to be monotone and converging to
infinity. Below we show that this definition is equivalent to ours
under Condition~(B), see Theorem~\ref{thr:ks}(iii). 
Furthermore, it follows from the proof of Theorem~\ref{thr:ks} that,
if $\sS$ contains the set $\{\modulus>1\}$ for a continuous
modulus $\modulus$, one may omit the assumption that
$g(t)\to\infty$ in Definition~\ref{def:RV}.

\index{random element!Baire measurable}
A Baire measurable random element $\xi$ is a measurable map from a
probability space to $\XX$ equipped with its Baire $\sigma$-algebra. If its
distribution $\nu$ (which is a Baire probability measure on $\XX$)
satisfies the conditions of Definition~\ref{def:RV}, that is,
\begin{equation}
  \label{eq:xi-RV}
  g(t)\Prob{T_{t^{-1}}\xi\in \cdot} \vto\mu(\cdot)\quad \text{as}\;
  t\to\infty
\end{equation}
with $g(t)\to\infty$ as $t\to\infty$, then $\xi$ is said to be
regularly varying on $\sS$, and we write
$\xi\in \RV(\XX,T, \sS,g,\mu)$.

\begin{remark}
  \label{rem:same-ideal}
  The definition of regular variation relies on the choice of the
  underlying ideal $\sS$ on $\XX$. In this context, it is irrelevant
  whether the ideal is generated by a single modulus or
  by several different moduli; see also
  Section~\ref{sec:polar-decomp-regul}. 
\end{remark}

\begin{remark}
  \label{rem:wide} 
  If the assumption $g(t)\to\infty$ as $t\to\infty$ in
  Definition~\ref{def:RV} is omitted, pathological situations may
  arise due to the flexibility in the  choice of the ideal
  $\sS$. For instance, if the scaling is  the identity map,
  that is, $T_tx=x$ for all $x$ and $t$, then every measure
  $\nu\in\Msigma$ is regularly varying with constant normalising
  function $g\equiv1$.
\end{remark}

The following result shows that in most interesting cases, namely, 
when $\sS$ satisfies Condition~\hyperref[condB0]{(B$_0$)}, the validity of \eqref{eq:24g} with any
measurable function $g$ implies that $g\in\RV_\alpha$.
The value of $\alpha$ is called the \emph{tail index} of $\nu$.
\index{tail index}
The following proposition appears as
\cite[Lemma~4.5(iii)]{bladt:hash22} in the Polish
space setting with a continuous scaling.  However, the proof there
tacitly presumes the existence of a semicone $\VV$ with
$\mu(\VV)\in(0,\infty)$, a condition which is not included in the
assumptions of 
the cited result. We impose this explicitly and do not require
the continuity of the scaling.

\begin{proposition}
  \label{prop:RV-g}
  Assume that $\XX$ is a topological space with a scaling and
  topologically and scaling
  consistent ideal $\sS$, such that $\sS$ contains the
  set $\{\modulus>1\}$ with $\mu(\{\modulus>1\})>0$ for a finite
  continuous modulus $\modulus$.  If $\nu\in \RV(\XX,T,\sS,g,\mu)$, then
  $g\in\RV_\alpha$ for some $\alpha>0$, and $\mu$ is
  $\alpha$-homogeneous. In particular, this holds if $\sS$ satisfies
  Condition~\hyperref[condB0]{(B$_0$)}. 
\end{proposition}
\begin{proof}
  Let $B=\{\modulus>1\}$, so that $B\in\sS$ is a semicone and
  $\mu(B)>0$. Then $T_tB$ fails to be a $\mu$-continuity set for at
  most countably many $t>0$. Fix a $t$ such that $T_tB$ is a
  $\mu$-continuity set. By replacing $\modulus$ with $t^{-1}\modulus$,
  it is possible to assume that $t=1$. Hence, for $s\in(0,1)$ such that $T_sB$
  is a $\mu$-continuity set, we have
  \begin{displaymath}
    \lim_{t\to\infty}\frac{g(ts)}{g(t)}
    =\lim_{t\to\infty}\frac{g(ts)\nu(T_{ts}B)}{g(t)\nu(T_t(T_sB))}
    =\frac{\mu(B)}{\mu(T_sB)}\in(0,\infty),
  \end{displaymath}
  where we have used the fact that $T_sB\supset B$ for $s\in(0,1)$,
  and so $\mu(T_sB)>0$.  By
  \cite[Theorem~1.3]{seneta72}, $g$ is regularly varying and
  $\mu(T_sB)=cs^{-\alpha}$ is a power function of $s$. By
  monotonicity, $\mu(T_sB)=\mu(\{\modulus>s\})$ is non-increasing in $s$ and
  converges to zero as $s\to\infty$ (since $\mu(\{\modulus>s\})$ is
  finite),
  so that $\mu(T_sB)<\mu(B)$ for
  some $s>1$, implying $g\in\RV_\alpha$ with $\alpha>0$. The homogeneity
  property of $\mu$ follows directly.
  If $\sS$ satisfies Condition~\hyperref[condB0]{(B$_0$)} and $\mu(A)\in(0,\infty)$ for some
  $A\in\sS$, then $A$ is a subset of $\{\modulus_I>t\}$ for
  $\modulus_I=\max_{i\in I}\modulus_i$ and some $t>0$, whence the above argument
  applies. 
\end{proof}

It is common practice to derive the normalising function $g$ from
$\nu$ itself.  It follows from the proof of
Proposition~\ref{prop:RV-g} that \eqref{eq:24g} holds with
$g(t)=1/\nu(T_tB)$ if $\nu(T_tB)>0$ for all $t$, and then the tail measure
becomes $\mu(\cdot)/\mu(B)$.

\begin{example}[Regular variation of random variables]
  \index{random variable!regularly varying}
  Consider $\XX=\R_+$ equipped with linear scaling and the ideal $\sR_0$
  obtained by excluding the origin. A random variable $\xi$ is in
  $\RV(\R_+,\cdot,\sR_0,g,\mu)$ if and only if the function
  $t\mapsto 1/\Prob{\xi>t}$ is in $\RV_{\alpha}$.
  In this case, $\mu$ is proportional to
  $\theta_\alpha$ defined in \eqref{eq:29}. If $g(t)=1/\Prob{\xi>t}$,
  then the tail measure is exactly $\theta_\alpha$. 
\end{example}

\begin{remark}
  A family $(\nu_t)_{t>0}$ of Baire probability measures on
  $\XX$ is said to be regularly varying in the triangular array (or the
  scheme of series) if 
  \eqref{eq:24g} is replaced by
  \index{regular variation!scheme of series}
  \begin{displaymath}
    g(t)T_{t^{-1}}\nu_t\vto\mu\quad \text{as}\; t\to\infty.
  \end{displaymath}
  A particularly important setting arises when $\nu_t$ results from
  applying a transformation to $\nu$ that differs from the basic
  scaling, for example, translations on the real line with the multiplicative
  scaling. We do not pursue this direction in the present work. 
\end{remark}

\paragraph{Characterisation of regular variation}

\begin{theorem}
  \label{thr:ks}
  Let $\XX$ be a topological space with a scaling and an ideal $\sS$
  satisfying 
  Condition~\hyperref[condB0]{(B$_0$)}. For a Baire measure $\nu$ on
  $\XX$, the following 
  statements are equivalent.
  \index{regular variation!characterisation of}
  \begin{enumerate}[(i)]
  \item The measure $\nu$ is regularly varying, that is,
    $\nu\in \RV(\XX,T,\sS,g,\mu)$.
  \item  We have
    \begin{equation}	
      \label{eq:vague}
      \frac{1}{\nu(T_t B)} T_{t^{-1}}\nu \vto \mu_b \quad
      \text{as}\; t\to\infty
    \end{equation}
    for some $B\in\sS$ and a non-trivial  $\mu_b\in\Msigma$. 
  \end{enumerate}
  If either condition holds, then \eqref{eq:vague} holds with
  $B=\{\modulus>1\}$ for a continuous modulus $\modulus$, where $\mu$ is
  non-trivial on $B$, and it is possible to let $g(t)=1/\nu(T_tB)$ in
  (i). Moreover, the measures $\mu$ and $\mu_b$ agree up to a
  multiplicative constant, and one may choose $g$ or $B$
  such that $\mu = \mu_b$.

  Each of conditions (i) and (ii) implies the following.
  \begin{enumerate}
  \item[(iii)] There exists a nondecreasing sequence $(a_n)_{n\in\NN}$
    of positive constants such that
    \begin{displaymath}
      nT_{a_n^{-1}}\nu\vto\mu_a\quad \text{as}\; n\to\infty
  \end{displaymath}
  for a non-trivial measure $\mu_a \in\Msigma$.
  \end{enumerate}
  Then  $(a_n)_{n\in\NN}$ is regularly varying (as a sequence) and it is
  possible to choose this sequence such that $\mu = \mu_a$. If $\XX$
  is a completely Hausdorff space with a continuous scaling operation 
  and an ideal satisfying Condition~\hyperref[condB]{(B)}, and $\mu_a$ is
  tau-additive, then (iii) implies (i) and (ii).
\end{theorem}
\begin{proof}
  (i)$\Rightarrow$(ii) Condition~\hyperref[condB0]{(B$_0$)} implies that each $A\in\sS$ satisfies
  \begin{displaymath}
    A\subset \bigcup_{j=1,\dots,m} \{\modulus_{i_j}>t_i\}
  \end{displaymath}
  for continuous moduli $\modulus_{i_1},\dots,\modulus_{i_m}$ 
  and $t_1,\dots,t_m>0$. Hence, a non-trivial measure
  $\mu\in\Msigma$ always satisfies $\mu(\{\modulus>1\})\in(0,\infty)$
  for some continuous modulus $\modulus$, which we fix in the sequel.

  Since $\nu\in \RV(\XX,T,\sS,g,\mu)$, Proposition~\ref{prop:RV-g}
  applies, so that $\mu$ is homogeneous. By Lemma~\ref{lemma:alpha}
  and the continuity of $\modulus$,
  $B=\{\modulus>1\}$ is a $\mu$-continuity set. By \eqref{eq:24g},
  $g(t)\nu(T_tB)\to \mu(B)\in(0,\infty)$ as $t\to\infty$. Thus,
  \begin{displaymath}
    \frac{T_{t^{-1}}\nu}{\nu(T_t
      B)}\vto\frac{\mu}{\mu(B)}\quad\text{as}\; t\to\infty,
  \end{displaymath}
  so that \eqref{eq:vague} holds and $\mu_b=\mu/\mu(B)$. It
  is possible to incorporate the multiplicative factor in $g$ or
  let $B=\{\modulus>s\}$ for a suitable positive $s$ to
  ensure that $\mu=\mu_b$.
  
  
  (ii)$\Rightarrow$(i) As above, fix a continuous modulus $\modulus$
  such that $\{\modulus>1\}\in\sS$ and $\mu_b(\{\modulus>1\})>0$. 
  Since the sets $\{\modulus=s\}$ are disjoint for all $s>1$ and are
  subsets of $\{\modulus>1\}$, sets $\{\modulus>s\}$ are
  $\mu_b$-continuity sets for all but at most countably many
  $s>0$. So assume that $D=\{\modulus>s\}$ is a $\mu_b$-continuity set
  and choose $s$ near $1$ to ensure that $\mu_b(D)>0$. Then
  \eqref{eq:vague} yields that
  \begin{displaymath}
    \frac{\nu(T_tD)}{\nu(T_tB)}
    \to \mu_b(D)\quad \text{as}\; t\to\infty. 
  \end{displaymath}
  Denote $g(t)=1/\nu(T_tD)$. Thus, 
  \begin{displaymath}
    g(t)T_{t^{-1}}\nu 
    =\frac{\nu(T_tB)}{\nu(T_tD)} \frac{1}{\nu(T_tB)}T_{t^{-1}}\nu 
    \vto \frac{1}{\mu_b(D)} \mu_b \quad \text{as}\; t\to\infty,
  \end{displaymath}
  so that (i) holds, $\mu_b$ is homogeneous and $g\in\RV_\alpha$ by
  Proposition~\ref{prop:RV-g}.



  
  The equivalence of (i) and (iii) in Polish spaces is largely the
  content of Theorem~B.2.2 in \cite{kul:sol20}. The following proof
  provides additional detail omitted therein and does not require the Polish
  assumption, which was used in \cite{kul:sol20} to invoke the
  Prokhorov theorem.
  
  (i),(ii)$\Rightarrow$(iii)
  Referring to (ii), choose the equivalent normalising function
  $g(t)=1/\nu(T_tB)$, where $B$ is a semicone of the form
  $\{\modulus>1\}$ and $\mu(B)\in(0,\infty)$. Then $g$ is non-decreasing
  and regularly varying with a positive index $\alpha$.
  Let $a_n=\inf\{t\colons g(t) \geq n\}$. Since 
  $g$ is monotone and grows to infinity,
  $a_n$ is a well-defined non-decreasing sequence. By
  \eqref{eq:24g}, $g(a_n)T_{a_n^{-1}}\nu\vto \mu$ as $n\to\infty$. The
  values $a_n$ are obtained by applying the generalised inverse
  function of $g$, which is also regularly varying; see
  \cite[Section~1.5(5$^o$)]{seneta72}. Thus, $(a_n)_{n\in\NN}$ is
  regularly varying and $g(a_n) \sim n$ as $n \to \infty$, which
  yields (iii) with $\mu_a= \mu$.
  
  (iii)$\Rightarrow$(ii) Let us first confirm that
  $\lim_n a_n = \infty$. Assume that $\sup_{n} a_{n} =
  s<\infty$. Since the additional Condition~(B) is imposed in this
  part of the theorem, the ideal $\sS$ is generated by a
  continuous modulus $\modulus$, and the union of the sets
  $\{\modulus>1/m\}$, $m\in\NN$, coincides with the union of all sets
  in $\sS$. Thus, there exists some $t>0$ such that
  $\nu(\{\modulus>st\})>0$. Since $\{\modulus>t\}$ is functionally
  open and $\{\modulus\geq t\}$ is functionally closed,
  $B=\{\modulus>t\}$ is a $\mu_a$-continuity set for all but at most 
  countably many $t>0$.
  Then 
  \begin{displaymath}
    n(T_{a_{n}^{-1}}\nu)(B)=n\nu(T_{a_{n}}B)
    \to \mu_a(B) \in (0,\infty)\quad \text{as}\; n \to \infty.
  \end{displaymath}
  However, since $a_n\leq s$, we have 
  $n\nu(T_{a_{n}}B)\geq n\nu(T_{s}B)\to \infty$ as
  $n\to\infty$, since we assumed $\nu(T_{s}B)>0$. This yields a
  contradiction.


  Let
  \begin{equation}
    \label{eq:25}
    A= \{x\colons \modulus(x)>t, T_{\modulus(x)^{-1}}x\in V\}
  \end{equation}
  for $t\in(0,\infty)$ and a functionally open set $V$. Arguing as in the
  proof of Theorem~\ref{thr:basis}, it is possible to ensure that $A$
  is a $\mu_a$-continuity set.
  Since $a_{n}$ is monotonically increasing to $\infty$, each $t>0$
  belongs to a unique interval $a_{n(t)} \leq t < a_{n(t)+1}$. Taking
  into account that $T_tA\subset T_sA$ whenever $t\geq s$, we have
  \begin{displaymath}
    \frac{\nu(T_t A) }{\nu(T_t B)} \leq
    \frac{n(t)+1}{n(t)}  \frac{n(t)\nu(T_{a_{n(t)}}
      A) }{(n(t)+1) \nu(T_{a_{n(t)+1}} B)}.
  \end{displaymath}
  where $B=\{\modulus>1\}$ is such that $\mu_a(B)\in(0,\infty)$.  The
  right-hand side converges to $\mu_a (A) / \mu_a(B)$ as
  $t\to \infty$. With a similar bound from below, it follows that the
  left-hand side possesses the same limit as $t\to\infty$. This
  establishes the convergence in \eqref{eq:vague}  for all
  $\mu_a$-continuity sets $A$, namely,
  \begin{displaymath}
     \frac{1}{\nu(T_t B)} \nu(T_tA) \to \mu_a(A) \quad
      \text{as}\; t\to\infty,
  \end{displaymath}
  and consequently for sets in $\sA_{\mu_a}$, which
  are obtained as differences of two sets of the form given in \eqref{eq:25}.

  Choose a countable open base for $\sS$ as $\base_{m}=\{\modulus>s_m\}$,
  $m\in\NN$, with $s_m\downarrow 0$ as $m\to\infty$.  By choosing
  $s_m>0$, it is possible to ensure that each $\base_{m}$ is a
  $\mu_a$-continuity set. By the above argument,
  $\nu(T_t \base_{m})/\nu(T_t B)\to \mu_a (\base_{m}) / \mu_a(B)$ as
  $t\to \infty$. The statement follows from Theorem~\ref{thr:basis}.
\end{proof}

\begin{proposition}
  \label{prop:base}
  Assume that $\sS_\modulus$ is an ideal on a topological space $\XX$
  generated by a continuous modulus $\modulus$. Let $\sA$ be a family
  of functionally open sets in $\sS_\modulus$ such that $\sA$ is
  closed under finite intersections and each functionally open set
  $B\in\sS_\modulus$ is the union of at most countably many
  sets in $\sA$. If there exists a $\mu\in\Msigma[\sS_\modulus]$
  such that $g(t)(T_{t^{-1}}\nu)(A)\to \mu(A)\in[0,\infty)$ for all
  $A\in\sA$ and the pushforward $\modulus\nu$ of $\nu$ under
  $\modulus$ is regularly varying, 
  namely, 
  $\modulus\nu\in\RV(\R_+,\mydot,\sR_0,g,c\theta_\alpha)$, then
  $\nu\in\RV(\XX,T,\sS_\modulus,g,\mu)$. Furthermore, the pushforward
  $\modulus\mu$ of $\mu$ under the map $\modulus$ is equal to 
  $c\theta_\alpha$.
\end{proposition}
\begin{proof}
  The statement follows from Theorem~\ref{lemma:vVSw}(iv), since the
  conditions guarantee that the restrictions of $g(t)T_{t^{-1}}\nu$
  to the sets of the countable open base $\base_m=\{\modulus> 1/m\}$
  of $\sS_\modulus$ weakly converge to the restrictions of $\mu$ onto
  these sets.
\end{proof}

The following is a well-known (within the Polish space setting) fact which
relates the regular variation property to the convergence of point
processes constructed from i.i.d.\ samples.

\begin{proposition}
  \label{prop:pp}
  Let $\XX$ be a topological space with a topologically and scaling 
  consistent ideal $\sS$.  Then $\xi\in\RV(\XX,T,\sS,g,\mu)$ for a
  random element $\xi$ implies that, for all $\mu$-continuity sets
  $A\in\sS$,
  \begin{equation}
    \label{eq:35}
    \Prob{T_{a_n^{-1}}\{\xi_1,\dots,\xi_n\} \cap A=\emptyset}
    \to e^{-\mu(A)}\quad \text{as}\; n\to\infty,
  \end{equation}
  where $(a_n)$ is a sequence of normalising constants and
  $\xi_1,\dots,\xi_n$ are independent copies of $\xi$.  The converse
  implication holds if $\XX$ is completely Hausdorff, the scaling is
  continuous, and the measure $\mu$ on the right-hand side of \eqref{eq:35}
  is tau-additive.
\end{proposition}
\begin{proof}
  The left-hand side of \eqref{eq:35} can be expressed as 
  \begin{displaymath}
    \Big(\Prob{\xi\notin T_{a_n}A}\Big)^n
    =\Big(1-\Prob{\xi\in T_{a_n}A}\Big)^n.
  \end{displaymath}
  The statement follows by taking the logarithm and using
  Theorem~\ref{thr:ks}(iii). 

  Conversely, if \eqref{eq:35} holds for all $\mu$-continuity sets
  $A\in\sS$, then
  \begin{displaymath}
    n\,\Prob{\xi\in T_{a_n}A} \to \mu(A)\quad \text{as}\; n\to\infty,
  \end{displaymath}
  since otherwise the limit of $((1-p_n)^n)$ with $p_n=\Prob{\xi\in
    T_{a_n}A}$ would differ from $e^{-\mu(A)}$. Thus, $n
  T_{a_n^{-1}}\nu \vto \mu$, and Theorem~\ref{thr:ks} yields regular
  variation.
\end{proof}

If $\XX$ is Polish, \eqref{eq:35} can be strengthened to show that the
scaled binomial point process $\{\xi_1,\dots,\xi_n\}$ converges in
distribution to the
Poisson point process with intensity measure $\mu$; see
\cite[Section~3.5]{res87}.
\index{Poisson process}

\begin{theorem}
  \label{Thm:PPconv}
  Let $\xi$ be a random element in a Polish space $\XX$ equipped with a 
  continuous scaling, and let $\sS$ be a topologically and scaling
  consistent ideal on $\XX$. Then $\xi\in\RV(\XX,T,\sS,g,\mu)$ if and only if
  the sequence of point processes
  \begin{displaymath}
    N_n = \sum_{i=1}^n \delta_{T_{a_n^{-1} }\xi_i},
  \end{displaymath}
  constructed from independent copies $\xi_1,\dots,\xi_n$ of $\xi$,
  converges in distribution in the space of counting
  measures to a Poisson point process $N$ with 
  intensity measure $\mu\in\Mb$ for a sequence $(a_n)$ of normalising
  constants as defined in Theorem~\ref{thr:ks}(iii).
\end{theorem}
\begin{proof}
  The proof is a straightforward adaptation of the proof of
  Proposition~3.21 in \cite{res87}.  According to Proposition~4.1 in
  \cite{bas:plan19}, it suffices to consider the convergence of the Laplace
  functionals of the point processes $N_n$. Consider an arbitrary
  bounded, non-negative, and continuous function $f$ with support in
  some set $B\in\sS$.  Then,
  \begin{displaymath}
    \EE \left[ e ^ {- \int f \diff N_n} \right]
    =\EE \left[ e ^ {-  \sum_{i=1}^n f(T_{a_n^{-1} }\xi_i)} \right]
    = \left( 1- \E h(T_{a_n^{-1}}\xi)\right)^n,
  \end{displaymath}
  where $h(x)=1-e^{-f(x)}\in\Cont(\XX,\sS)$.
  Hence, the left-hand side converges as $n\to\infty$ if and only if
  $n\E h(T_{a_n^{-1}}\xi)$ converges. The latter holds for all such
  $f$ if and
  only if $\xi\in\RV(\XX,T,\sS,g,\mu)$, since each bounded continuous
  function with support in $\sS$ can be written as $1-e^{-f(x)}$ after
  adding a constant and scaling.  Since
  $n\E h(T_{a_n^{-1}}\xi)\to\int h\diff \mu$, we obtain
  \begin{displaymath}
    \EE \Big[ e ^ {- \int f\diff N_n} \Big]
    \to \exp \Big( - \int h \diff \mu \Big)
    = \EE \big[ e ^ {- \int f \diff N} \big], \quad \text{as}\;
    n\to\infty, 
  \end{displaymath}
  meaning that the Laplace functional of $N_n$ converges to the
  Laplace functional of the Poisson process $N$ with intensity
  measure $\mu$.
\end{proof}

\section{Continuous mappings}
\label{sec:continuous-mappings}

\paragraph{General continuous mapping theorem}
Consider topological spa\-ces $\XX$ and $\YY$ equipped with Baire measurable
scalings $T^\XX$ and $T^\YY$.

\begin{definition}
  A Baire measurable map $\psi:\XX\to\YY$ is said to be a
  \emph{morphism} (also called an equivariant or homogeneous map)
  \index{morphism} \index{equivariant map}
  if
  \begin{equation}
    \label{eq:2}
    \psi(T^\XX_t x) = T^\YY_{t} \psi(x) 
  \end{equation}
  for all $x\in \XX$ and $t>0$.
\end{definition}

If $\psi:\XX\to\YY$ is a Baire measurable map, then the pushforward
\index{pushforward}
$\psi\mu$ of a Baire measure $\mu$ on $\XX$ is defined by
$\psi\mu(A)=\mu(\psi^{-1}(A))$ for $A\in\sBA(\YY)$. 
The following result establishes that the regular variation property
is preserved under continuous bornologically consistent morphisms, see
Definition~\ref{def:morphism}. This and various further results for
continuous (possibly random) maps of random vectors in $\R^d$ can be
found in \cite[Section~2.1.4]{kul:sol20}. 

\begin{theorem}\label{mapping-theorem}
  \index{continuous mapping theorem}
  Let $\XX$ and $\YY$ be topological spaces equipped with Baire
  measurable scalings and 
  topologically and scaling consistent
  ideals $\sX$ and $\sY$. Assume that $\psi: \XX \to \YY$ is a
  continuous bornologically consistent morphism.
  Let $\nu$ be a Baire measure on $\XX$ such that
  $\nu\in \RV(\XX,T^\XX,\sX,g,\mu)$. If $\psi\mu$ is non-trivial
  on $\sY$ in the sense that there exists $B\in\sY$ such that 
  $\psi\mu(B)\in(0,\infty)$, then
  $\psi\nu\in\RV(\YY,T^\YY,\sY,g,\psi\mu)$. 
\end{theorem}
\begin{proof}
  For any $A \in \sBA(\YY)$, we have
  \[
    (T_{t^{-1}}^\YY \psi\nu)(A) 
    = \psi\nu(T^\YY_t A)
    = \nu(\psi^{-1}(T^\YY_t A))
    = \nu(T^\XX_t \psi^{-1}(A))
    = (\psi T^\XX_{t^{-1}}\nu)(A),
  \]
  where the third equality follows from the morphism property \eqref{eq:2}.
  Therefore, $T_{t^{-1}}^\YY \psi\nu = \psi T^\XX_{t^{-1}}\nu$ as Baire measures.

  By the regular variation of $\nu$, $g(t) T^\XX_{t^{-1}}\nu \vto \mu$
  as $t\to\infty$.  Since $\psi$ is continuous and bornologically
  consistent, Lemma~\ref{lemma:mapping} implies
  \[
    g(t) T_{t^{-1}}^\YY \psi\nu
    = g(t) \psi T^\XX_{t^{-1}}\nu = \psi(g(t) T^\XX_{t^{-1}}\nu) \vto
    \psi\mu
    \quad \text{as}\; t\to\infty.
  \]
  Hence, $\psi\nu \in \RV(\YY,T^\YY,\sY,g,\psi\mu)$.
\end{proof}

\begin{remark}[Implied ideal]
  \label{ex:implied_born}
  \index{ideal!implied}
  A bijective map between two spaces can be used to transfer the
  scaling from the original space to the target space, so that this
  map becomes a morphism. Let $\XX$ and $\YY$ be two topological
  spaces, where $\XX$ is equipped with the scaling $T^\XX$, and let
  $\psi:\XX\to\YY$ be a homeomorphism, that is, a continuous bijective
  map with a continuous inverse. Define a scaling on $\YY$ by
  \begin{displaymath}
    T^\psi_{t}(y) = \psi \left( T^\XX_{t}(\psi^{-1}(y)) \right), \quad y\in\YY.
  \end{displaymath}
  Let $\sX$ be a topologically and scaling consistent ideal on $\XX$,
  and let $\sY$ be the induced ideal 
  on $\YY$ consisting of all sets $V\subset \YY$ such that
  $\psi^{-1}(V)\in\sX$. Note that $\sY$ inherits the topological and
  scaling consistency from $\sX$. By construction, $\psi$ is bornologically
  consistent. Furthermore, since 
  $\psi^{-1}(\psi(A))=A$ for all $A\in\sX$, it follows
  that $\psi^{-1}$ is also bornologically
  consistent. The measure $\mu$ is non-trivial on $\sX$ if and
  only if the pushforward $\psi\mu$ is non-trivial on $\sY$. Thus,
  $\xi\in\RV(\XX,T^\XX, \sX, g,\mu)$ if and only if
  $\psi(\xi)\in \RV(\YY,T^\psi, \sY,g, \psi \mu)$. This construction
  will be illustrated by examples in
  Section~\ref{sec:random-variables}. 
\end{remark}

\paragraph{Invariant measures and extension of ideals}
Assume that $\mathsf{G}$ is a group of bicontinuous bijective morphisms
$\mathsf{g}:\XX\to\XX$, where the group operation is the composition
and the inverse is given by the inverse mapping.

\begin{proposition}
  \label{prop:extension-invariant}
  \index{measure!invariant under a group}
  Assume that $\nu\in\RV(\XX,T,\sX,g,\mu)$, where $\XX$ is a
  topological space with an ideal possessing a countable open base and
  satisfying Condition~\hyperref[condB0]{(B$_0$)}. If $\nu$ is
  invariant under all transformations 
  in the group $\mathsf{G}$, then $\nu$ is also regularly varying
  on the ideal generated by the collection 
  $\mathsf{g}\sX=\{B\subset\XX\colons \mathsf{g}^{-1}B\in\sX\}$ for
  $\mathsf{g}\in\mathsf{G}$.
\end{proposition}
\begin{proof}
  Note that the tail measure $\mu$ is invariant under all
  $\mathsf{g}\in\mathsf{G}$.  Assume that $B$ is a $\mu$-continuity
  set, so that there exist a functionally open set $U$ and a
  functionally closed $F$ such that $U\subset B\subset F$ and
  $\mu(F\setminus U)=0$. Then $\mathsf{g}U$ is functionally open,
  since
  $\mathsf{g}U=\{\mathsf{g}x\colons f(x)>0\}=\{x\colons
  f(\mathsf{g}^{-1}x)>0\}$ for 
  a continuous real-valued function $f$. Similarly, $\mathsf{g}F$ is
  functionally closed. Since
  $\mu(\mathsf{g}F\setminus \mathsf{g}U)=\mu(F\setminus U)=0$, the set
  $\mathsf{g}B$ is also a $\mu$-continuity set for each
  $\mathsf{g}\in\mathsf{G}$. It should be noted as well that finite
  intersections of $\mu$-continuity sets are also $\mu$-continuity
  sets.
  
  Let $B$ be a $\mu$-continuity set in the ideal generated by
  $\mathsf{g}_1\sX\cup \mathsf{g}_2\sX$ for
  $\mathsf{g}_1,\mathsf{g}_2\in\mathsf{G}$. Then
  $B\subset \mathsf{g}_1A_1\cup \mathsf{g}_2A_2$ for some
  $A_1,A_2\in\sX$. Since $\sX$ has a countable open base, we may
  assume that $A_1$ and $A_2$ are $\mu$-continuity
  sets. By the inclusion-exclusion principle and the invariance of $\nu$,
  \begin{align*}
    g(t)&\nu(T_tB)
    =g(t)\nu\big(T_tB\cap (T_t\mathsf{g}_1A_1\cup T_t\mathsf{g}_2A_2)\big)\\
    &=g(t)\nu\big(T_t(B\cap \mathsf{g}_1A_1)\big)
      +g(t)\nu\big(T_t(B\cap \mathsf{g}_2A_2)\big)
      -g(t)\nu\big(T_t(B\cap \mathsf{g}_1A_1\cap \mathsf{g}_2A_2)\big)\\
    &=g(t)\nu\big(T_t(\mathsf{g}_1^{-1}B\cap A_1)\big)
      +g(t)\nu\big(T_t(\mathsf{g}_2^{-1}B\cap A_2)\big)\\
    &\qquad\qquad\qquad\qquad
      -g(t)\nu\big(T_t(\mathsf{g}_1^{-1}B\cap A_1\cap \mathsf{g}_1^{-1}\mathsf{g}_2A_2)\big).
  \end{align*}
  Note that $\mathsf{g}_1^{-1}B\cap A_1$, $\mathsf{g}_2^{-1}B\cap A_2$
  and
  $\mathsf{g}_1^{-1}B\cap A_1\cap \mathsf{g}_1^{-1}\mathsf{g}_2A_2$
  are $\mu$-continuity sets.  These sets belong to $\sX$, since they
  are subsets of $A_1$ or $A_2$, and are all $\mu$-continuity sets. Thus,
  \begin{align*}
    \lim_{t\to\infty} g(t)\nu(T_tB)
    &=\mu\big(\mathsf{g}_1^{-1}B\cap A_1\big)
      +\mu\big(\mathsf{g}_2^{-1}B\cap A_2\big)
      -\mu\big(\mathsf{g}_1^{-1}B\cap A_1\cap
      \mathsf{g}_1^{-1}\mathsf{g}_2A_2\big)\\
    &=\mu\big(B\cap (\mathsf{g}_1A_1\cup \mathsf{g}_2A_2)\big)=\mu(B). 
  \end{align*}
  The case where $B$ is a subset of a union of more than two
  transformed sets in $\sX$ follows analogously. The argument further
  extends to countable unions by monotone convergence of measures for
  increasing sequences of $\mu$-continuity sets. 
\end{proof}

\section{Polar decomposition and regular variation}
\label{sec:polar-decomp-regul}

\paragraph{Modulus of a regularly varying random element}
Assume that Condition~(B) (see page~\pageref{eq:10})
holds, meaning that the ideal $\sS=\sS_\modulus$ is
generated by a continuous modulus $\modulus$.  The modulus $\modulus$
is a bornologically consistent morphism between $\XX$ (equipped with the ideal
$\sS_\modulus$) and $\R_+$ (equipped with linear scaling and the
ideal $\sR_0$). Recall that $\theta_\alpha$ is a measure on $\Rpp$
with density $\theta_\alpha(\bdiff t)=\alpha
t^{-\alpha-1}dt$. This measure is homogeneous and so is regularly
varying on $(0,\infty)$ with linear scaling. 

Theorem~\ref{mapping-theorem} implies the following
result.

\begin{corollary}
  \label{cor:tau-map}
  Assume that $\sS_\modulus$ is an ideal on a topological space $\XX$
  generated by a continuous modulus $\modulus$.  If
  $\xi\in \RV(\XX,T, \sS,g, \mu)$, then its modulus satisfies
  \begin{displaymath}
    \modulus(\xi)\in \RV(\R_+,\mydot,\sR_0,g,c\theta_\alpha),
  \end{displaymath}
  where $\alpha>0$ is the tail index and
  the constant $c$ is given by $c=\mu(\{\modulus>1\})$.
\end{corollary}

\paragraph{Regular variation criterion using the polar decomposition}

The following result utilises the polar decomposition map $\rho$,
which is a homeomorphism between the set $\{\modulus>0\}$ and
the product space $\YY=\SS_\modulus\times\Rpp$, see
Definition~\ref{def:polar-decomposition}. A variant for a Banach
space, where the modulus is given by the norm, is provided in
\cite{meinguet10} and was later incorporated into the setting of star-shaped
metric spaces in \cite{seg:zhao:mein17}. A Hilbert space version can
be found in \cite{MR4745556}.

The scaling on $\YY$ acts only on the second component, that is,
$T^\YY_s(u,t)=(u,ts)$, and thus acts freely on $\YY$. Consider the
bornology $\sY$ on $\YY$ generated by the modulus
$\modulus_\YY(u,t)=t$.

\begin{proposition}
  \label{prop:polar-decomp}
  Assume that $\XX$ is a perfectly normal space equipped with a
  continuous scaling $T$ and an ideal $\sS$
  generated by a continuous finite modulus $\modulus$.  Let $\xi$ be a
  Borel measurable random element in $\XX$, and let $\sigma$ be a non-zero
  finite Borel
  measure on $\SS_\modulus$. Then the following statements are
  equivalent.
  \begin{enumerate}[(i)]
  \item $\xi\in \RV(\XX,T, \sS,g, \mu)$, where $\mu$ is a
    tau-additive measure such that \eqref{eq:4} holds.
  \item The random variable $\modulus(\xi)$ is regularly varying with
    normalising function $g\in\RV_\alpha$, and 
    \begin{equation}
      \label{eq:5}
      \Prob{T_{\modulus(\xi)^{-1}}\xi\in \cdot  \, |\, \modulus(\xi)>t} \wto
      \sigma(\cdot)/\sigma(\SS_\modulus)
      \quad \text{ as}\; t\to \infty,
    \end{equation}
    in the sense of the weak convergence of probability measures on
    $\SS_\modulus$, where $\sigma$ is tau-additive.
  \item The measure $\sigma$ is tau-additive, and for some $\alpha>0$,
    \begin{equation}
      \label{eq:6}
      \Prob{(T_{\modulus(\xi)^{-1}}\xi,t^{-1}\modulus(\xi))\in
          A\times(u,\infty)\big| \modulus(\xi)>t}
      \to \frac{(\sigma\otimes\theta_\alpha)(A\times(u,\infty)\,)}
      {\sigma(\SS_\modulus)}
    \end{equation}
    as $t\to \infty$ for all Borel $\sigma$-continuity sets
    $A\subset\SS_\modulus$ and all $u\geq1$. 
  \end{enumerate}
\end{proposition}
\begin{proof}
  The perfect normality assumption implies that the Baire and Borel
  $\sigma$-algebras coincide.  Without loss of generality, assume that
  $\modulus$ does not vanish, otherwise, we pass from $\XX$ to the set
  $\{\modulus>0\}$.  By Theorem~\ref{mapping-theorem} and
  Remark~\ref{ex:implied_born}, $\xi\in \RV(\XX,T, \sS,g, \mu)$ if and
  only if $\eta=\rho(\xi)=(T_{\modulus(\xi)^{-1}}\xi,\modulus(\xi))$
  is regularly varying in $\YY$ with the bornology $\sY$ and its tail
  measure is $\rho\mu=\sigma\otimes\theta_\alpha$.  In view of
  \cite[Theorem~7.6.5]{bogachev07}, representation \eqref{eq:4}
  implies that $\mu$ is tau-additive if and only if $\sigma$ is.

  By Corollary~\ref{cor:tau-map}, $\modulus(\xi)$ is regularly varying and
  \begin{displaymath}
    g(t)\Prob{\modulus(\xi)>t}\to \mu(\{\modulus>1\})=\sigma(\SS_\modulus)
  \end{displaymath}
  where the last equality follows directly from \eqref{eq:4}.

  If $\eta$ is regularly varying on $\sY$, then (ii) and (iii)
  hold by Theorem~\ref{lemma:vVSw}(ii), since
  \begin{align*}
    \lim_{t\to\infty} \Prob{T_{\modulus(\xi)^{-1}}\xi\in A\mid
      \modulus(\xi)>t} 
    &=\frac{1}{\sigma(\SS_\modulus)}
      \lim_{t\to\infty} g(t)\Prob{T_{\modulus(\xi)^{-1}}\xi\in A,
      \modulus(\xi)>t}\\
    &=\frac{(\sigma\otimes\theta_\alpha)(A\times (1,\infty))}
    {\sigma(\SS_\modulus)}
    =\frac{\sigma(A)}{\sigma(\SS_\modulus)}
  \end{align*}
  for all $\sigma$-continuity sets $A$.
  
  We then show that the regular variation of $\eta$ follows from (ii) and
  (iii). Let $B=\{\modulus>1\}$ and note that
  $(\rho\mu)(B)=\sigma(\SS_\modulus)$ for $\rho\mu$ being the
  pushforward of $\mu$ constructed using
  \eqref{eq:4}. Let $\nu$ denote  the distribution of 
  $\eta=\rho(\xi)$. The open base of $\sY$ consists of the sets
  $\{\modulus>s_m\}$ with $s_m\downarrow 0$, implying
  \begin{displaymath}
    \frac{\nu(T_tD_m)}{\nu(T_tB)}
    =\frac{\Prob{\modulus(\xi)>ts_m}}{\Prob{\modulus(\xi)>t}}
    \to \frac{\theta_\alpha((s_m,\infty))}{\theta_\alpha((1,\infty))}
    =\frac{(\rho\mu)(D_m)}{(\rho\mu)(B)}\quad \text{as}\; t\to\infty,
  \end{displaymath}
  which follows either from the regular variation of $\modulus(\xi)$
  imposed in (ii) or from \eqref{eq:6} with $A=\SS_\modulus$.

  It is enough to consider sets of the form $A\times(s_1,s_2)$, where
  $A\subset\SS_\modulus$ is functionally open and a
  $\sigma$-continuity set, and $0<s_1<s_2$.
  Let $\sA_\mu$ be the family of such sets $A\times (s_1,s_2)$. Then
  \begin{displaymath}
    \frac{\nu(T_t(A\times(s_1,s_2)))}{\nu(T_tB)}
    \to \frac{\sigma(A)\theta_\alpha((s_1,s_2))}{\sigma(\SS_\modulus)}
    =\frac{(\rho\mu)(A\times(s_1,s_2))}{(\rho\mu)(B)}
  \end{displaymath}
  by \eqref{eq:5}, using the same convergence with the threshold
  $s_1t$ and $s_2t$, or \eqref{eq:6}. Hence, the convergence in
  \eqref{eq:24g} follows from Theorem~\ref{thr:basis}, and the
  corresponding convergence for $\xi$ follows by applying the inverse
  homeomorphism $\rho^{-1}$.
\end{proof}

In part (iii) of Proposition~\ref{prop:polar-decomp} it is possible to
restrict $A$ to a convergence determining class on $\SS_\modulus$,
which may be smaller than the family of all Borel sets in
$\SS_\modulus$. The tau-additivity assumption is redundant in
hereditary Lindel\"of spaces, for example, in separable metric spaces, where
all Borel measures are automatically tau-additive.

\paragraph{Convergence of conditional distributions} 
The following useful characterisation result holds in all topological
spaces.

\begin{theorem}
  \label{thr:polar-decomp-4}
  \index{conditional distribution!convergence of}
  Assume that $\XX$ is a topological space equipped with a continuous
  scaling $T$ and an ideal $\sS$ generated by a continuous 
  modulus $\modulus$.  Let $\xi$ be a Baire measurable random element in
  $\XX$. Then $\xi\in \RV(\XX,T, \sS,g, \mu)$ if and only if
  \begin{equation}
    \label{eq:33}
    \Prob{T_{t^{-1}}\xi\in\cdot\mid\modulus(\xi)>t}
    \wto \frac{\mu(\cdot\cap \{\modulus>1\})}{\mu(\{\modulus>1\})}
    \quad \text{ as}\; t\to \infty
  \end{equation}
  for an $\alpha$-homogeneous measure $\mu\in\Msigma$ such that
  $\mu(\{\modulus>1\})\in(0,\infty)$, for some $\alpha>0$.
\end{theorem}
\begin{proof}
  \textsl{Necessity.} If $\xi$ is regularly varying with 
  normalising function $g$, then $\modulus(\xi)$ is regularly varying
  by Corollary~\ref{cor:tau-map}. Consequently,
  \begin{multline*}
    \Prob{T_{t^{-1}}\xi\in B\mid\modulus(\xi)>t}
    =\frac{g(t)\Prob{\xi\in T_t(B\cap \{\modulus>1\})}}
    {g(t)\Prob{\modulus(\xi)>t}}\\
    \to \frac{\mu(B\cap \{\modulus>1\})}{\mu(\{\modulus>1\})}
    \quad \text{as}\; t\to\infty
  \end{multline*}
  for all $\mu$-continuity sets $B$.

  \smallskip
  \noindent
  \textsl{Sufficiency.} If \eqref{eq:33} holds, then, for all but at
  most countably many $s>1$,
  \begin{displaymath}
    \frac{\Prob{\modulus(\xi)>st}}{\Prob{\modulus(\xi)>t}}
    =\Prob{T_{t^{-1}}\xi\in \{\modulus>s\}\mid \modulus(\xi)>t}
    \to \frac{\mu(\{\modulus>s\})}{\mu(\{\modulus>1\})}\in (0,1]
  \end{displaymath}
  as $t\to\infty$.

  Since $\mu(\{\modulus>1\})>0$, the right-hand side is positive for
  all $s>1$ sufficiently close to $1$, except possibly for countably
  many values. Hence \cite[Theorem~1.3]{seneta72} applies to
  $g(t)=\Prob{\modulus(\xi)>t}^{-1}$ and yields $g\in\RV_\alpha$ for
  some $\alpha>0$.
  Let $A$ be a $\mu$-continuity set in $\sS$. Then
  $A\subset\{\modulus>u\}$ for some $u>0$. The homogeneity of $\mu$
  and the continuity of the scaling imply that $T_{u^{-1}}A$ is also a
  $\mu$-continuity set. Now \eqref{eq:33} yields 
  \begin{multline*}
    g(t)\Prob{T_{t^{-1}}\xi\in A}=
    \Prob{T_{(ut)^{-1}}\xi \in T_{u^{-1}} A|\modulus(\xi)>ut}
    \frac{\Prob{\modulus(\xi)>ut}}{\Prob{\modulus(\xi)>t}}\\
    \to  \frac{u^{\alpha} \mu(A)}{\mu(\{\modulus>1\})}u^{-\alpha}
    =\frac{\mu(A)}{\mu(\{\modulus>1\})} 
    \quad \text{as}\; t\to\infty,
  \end{multline*}
  which establishes the vague convergence in \eqref{eq:24g} up to a
  constant factor in the tail measure, which can be absorbed into the
  normalising function if desired.
\end{proof}

\paragraph{Changing the reference modulus}
If two moduli generate the same ideal, then the regular variation
property remains unchanged and the tail measure remains the same, see
Remark~\ref{rem:same-ideal}. However, the spectral measures may differ
if one chooses polar representations generated by different moduli,
see Lemma~\ref{lemma:change-modulus}.  The following result identifies
conditions that ensure that \eqref{eq:33} holds after replacing
$\modulus$ with another homogeneous function $\ell$. This was
motivated by \cite{dom:rib15}, which interprets $\ell$ as a risk or
cost functional; see also \cite{fon:dav22,cot:di:opit24} for more
recent uses of such interpretations in the context of regular variation
on the space of continuous functions and \cite{dys:mik20} for
regularly varying random vectors.

\begin{theorem}
  \label{thr:Dombry-Ribatet}
  Assume that $\XX$ is a topological space equipped with a continuous
  scaling $T$ and the ideal $\sS_\modulus$ generated by a continuous
  proper modulus $\modulus$. Assume that $\ell:\XX\to[0,\infty)$ is
  another continuous modulus such that $\modulus(x_n)\to0$ implies
  that $\ell(x_n)\to 0$. Let $\xi$ be a Baire measurable random element in
  $\XX$. If $\xi\in \RV(\XX,T,\sS_\modulus,g,\mu)$ and $\ell$ does not vanish
  $\mu$-almost everywhere, then
  $\ell(\xi)\in \RV(\R_+,\mydot,\sR_0,g,c\theta_\alpha)$ where
  $c=\mu(\{\ell>1\})\in(0,\infty)$, and
  \begin{equation}
    \label{eq:33ell}
    \Prob{T_{t^{-1}}\xi\in\cdot\mid\ell(\xi)>t}
    \wto \frac{\mu(\cdot\cap \{\ell>1\})}{\mu(\{\ell>1\})}
    \quad \text{ as}\; t\to \infty.
  \end{equation}
\end{theorem}
\begin{proof}
  We first show that $\{\ell>s\}\in\sS_\modulus$ for all $s>0$. If this
  does not hold, then there exists a sequence
  $(x_n)_{n\in\NN}\subset \{\ell>s\}$ such that $\modulus(x_n)\to 0$ as
  $n\to\infty$. However, this would imply $\ell(x_n)\to0$, which
  contradicts the assumption that $\ell(x_n)>s$ for all $n$.
  
  The regular variation of $\ell(\xi)$ follows from the continuous
  mapping theorem, noting that $\mu$ is non-trivial on the ideal
  generated by $\ell$. Since $\{\ell>t\}$ for $t>0$ are continuity sets
  for the tail measure by the continuity of $\ell$ and
  Lemma~\ref{lemma:alpha}, the regular variation property of $\xi$
  implies
  \begin{displaymath}
    g(t)\Prob{\ell(\xi)>t}
    \to \mu(\{\ell>1\}).
  \end{displaymath}
  Thus, the tail measure of $\ell(\xi)$ is $c\theta_\alpha$ with
  $c=\mu(\{\ell>1\})$. Furthermore, for every $\mu$-continuity set $B$
  such that $B\cap\{\ell>1\}$ is also a $\mu$-continuity set,
  \begin{displaymath}
    \Prob{T_{t^{-1}}\xi\in B\mid\ell(\xi)>t}
    =\frac{g(t)\Prob{T_{t^{-1}}\xi\in B,\ell(\xi)>t}}
    {g(t)\Prob{\ell(\xi)>t}}
    \to \frac{\mu(B\cap\{\ell>1\})}{\mu(\{\ell>1\})}
  \end{displaymath}
  as $t\to\infty$. 
\end{proof}

\begin{example}
  \label{ex:DR}
  Note that the continuity of $\ell$ is essential in
  Theorem~\ref{thr:Dombry-Ribatet}. For example, let $\XX=\R_+^2$ with
  linear scaling, and let
  \begin{displaymath}
    \xi=\zeta(\eta,1)+(1-\zeta)(\sqrt{\eta},0),
  \end{displaymath}
  where $\zeta$ is the Bernoulli random variable with parameter $1/2$
  and
  \begin{displaymath}
    \eta\in\RV(\R_+,\mydot,\sR_0,g,\theta_\alpha).
  \end{displaymath}
  If $\XX$ is equipped
  with the ideal $\sR_0^2$ generated by
  $\modulus(x_1,x_2)=\max(x_1,x_2)$, then
  $\xi\in\RV(\R_+^2,\mydot,\sR_0^2,g,\mu)$ with $\mu$ being the
  pushforward of $\frac{1}{2}\theta_\alpha$ under the map $z\mapsto
  (z,0)$. If we define the discontinuous function
  $\ell(x_1,x_2)=x_1\one_{x_2=0}$, then, for $A$ being the
  complement of $[0,a]^2$,
  \begin{align*}
    \Prob{T_{t^{-1}}\xi\in A\mid \ell(\xi)>t}
    &=\Prob{T_{t^{-1}}\xi\in A\mid \zeta=0,\sqrt{\eta}>t}\\
    &=\Prob{\sqrt{\eta}>at\mid \sqrt{\eta}>t}
    \to a^{-2\alpha}.
  \end{align*}
  The limiting distribution is inconsistent with the index $\alpha$,
  and thus $\ell(\xi)$ is regularly varying with tail index $2\alpha$,
  not with the original tail index $\alpha$.  Note that the chosen
  function $\ell$ is continuous at the origin and does not vanish
  $\mu$-almost everywhere. This example shows that Theorem~3 in
  \cite{dom:rib15} requires an adjustment, namely, assuming the
  continuity of $\ell$.
\end{example}

\paragraph{Ideals generated by multiple moduli}
Now consider the case of an ideal $\sX$ satisfying Condition~\hyperref[condB0]{(B$_0$)},
and therefore 
generated by a family of continuous moduli. 

\begin{proposition}
  \label{prop:polar-many-tau}
  Assume that $\XX$ is a topological space with an ideal $\sX$
  generated by a family $\{\modulus_i,i\in I\}$ of continuous finite
  moduli, and let $\xi$ be a Baire random element in $\XX$. Then
  $\xi\in\RV(\XX,T,\sX,g,\mu)$ with $\mu\in\Msigma$ if and only if for each
  $i_1,\dots,i_k\in I$, $k\in\NN$, and $a>0$, we have
  \begin{equation}
    \label{eq:3}
    g(t)\Prob{T_{t^{-1}}\xi\in A, \max_{j=1,\dots,k}\modulus_{i_j}(\xi)> at}
    \to \mu\Big(A\cap\big(\bigcup_{j=1}^k\{\modulus_{i_j}> a\}\big)\Big)
  \end{equation}
  as $t\to\infty$ for each $\mu$-continuity set $A$.
\end{proposition}
\begin{proof}
  \textsl{Necessity.} Condition~\eqref{eq:3} follows from
  \eqref{eq:24g} and Theorem~\ref{lemma:vVSw} by applying it to the
  intersection of any $\mu$-continuity set with
  \[
    \big\{\max_{j=1,\dots,k}\modulus_{i_j}>a\big\}.
  \]
  and noticing that $\{\modulus_i> a\}$ is a
  $\mu$-continuity set for all $a>0$ by Proposition~\ref{prop:RV-g}
  and \eqref{eq:24} from Lemma~\ref{lemma:alpha}.

  \smallskip
  \noindent
  \textsl{Sufficiency.} Let $A$ be a $\mu$-continuity set in 
  $\sX$. Then by Condition~\hyperref[condB0]{(B$_0$)},
  $A\subset\cup_{j=1}^k\{\modulus_{i_j}> a\}$ for 
  some $i_1,\dots,i_k\in I$ and $a>0$. By
  \eqref{eq:3},
  \begin{multline*}
    g(t)\Prob{T_{t^{-1}}\xi\in A}\\
    =g(t)\Prob{T_{t^{-1}}\xi\in A,
      \max_{j=1,\dots,k}\modulus_{i_j}(\xi)>at}
    \to \mu(A\cap \cup_{j=1}^k\{\modulus_{i_j}>a\})=\mu(A), 
  \end{multline*}
  where it has been used that the maximum of the finitely many moduli
  $\modulus_{i_j}$ is
  continuous and bornologically consistent. 
\end{proof}

The necessity part of Proposition~\ref{prop:polar-many-tau} can be
formulated as a continuous mapping result. If $I=\{1,2,\ldots\}$ is
countable, and $\xi\in\RV(\XX,T,\sX,g,\mu)$ then the sequence
\begin{displaymath}
  (\modulus_1(\xi),\modulus_2(\xi),\ldots)
\end{displaymath}
is regularly varying in the space $\R_+^\infty$ with the
ideal $\sR_0^\infty(1)$, see Section~\ref{sec:infinite-sequences}. 

\paragraph{Reduction to a sub-ideal and hidden regular variation}
It is common in the theory of extreme values to consider, together with
the basic space $\XX$, its subcones with possibly different ideals,
which may result in different tail behaviours. This phenomenon in
$\R^d$ is well studied and called \emph{hidden regular variation};
see, for instance, Section~9.4 of \cite{resnick:2007}, the whole
  book \cite{resnick24:_art_findin_hidden_risks} devoted to this topic,  and
Example~\ref{ex:hidden-RV} below.
\index{hidden regular variation}
\index{regular variation!hidden}
Many examples of this phenomenon
arise in the context of various ideals on product spaces, see
Definition~\ref{def:product} and Theorem~\ref{thr:power-space}.

The following result describes the relationships between the regular
variation property on smaller and larger ideals.
For instance, the regular variation property may be lost when passing
to a smaller ideal; alternatively, it may happen that $\nu$ is regularly
varying on $\sS$ with tail index $\alpha$, while being regularly
varying on a smaller ideal $\sV\subset\sS$ with a different tail index
$\alpha_1>\alpha$.

\begin{proposition}
  \label{prop:sub-ideal}
  Let $\nu\in\RV(\XX,T,\sS,g,\mu)$, where the ideal $\sS$ satisfies
  Condition~\hyperref[condB0]{(B$_0$)} on a topological space $\XX$
  and the tail measure has tail 
  index $\alpha$, and let $\sV\subset\sS$ be
  another ideal that also satisfies Condition~\hyperref[condB0]{(B$_0$)}.
  \begin{enumerate}[(i)]
  \item If the tail measure $\mu$ is non-trivial and finite on all
    sets from $\sV$,
    then $\nu\in\RV(\XX,T,\sV,g,\mu)$.
  \item If $\nu\in\RV(\XX,T, \sV,g_1,\mu_1)$ with tail index $\alpha_1$,
    then either $\mu(\{\modulus>1\})\in(0,\infty)$ for a continuous
    modulus $\modulus$ such that $\{\modulus>1\}\in \sV$, in which
    case, after multiplying $g_1$ and $\mu_1$ by reciprocal constants
    if necessary, 
    it is possible to set $g_1=g$ and $\mu_1$ to be the restriction of
    $\mu$ onto the measurable sets in $\sV$, and then
    $\alpha_1=\alpha$; or, otherwise, $\mu(\{\modulus>1\})=0$ for any
    modulus $\modulus$ such that $\sS_\modulus\subset\sV$, and then
    $\alpha_1\geq \alpha$. 
  \end{enumerate}
\end{proposition}
\begin{proof}
  (i) This follows directly from the definition, since every test
  function for the smaller ideal $\sV$ is also a test function for
  $\sS$ or from Theorem~\ref{mapping-theorem} applied to the embedding
  map $\psi(x)=x$, $x\in\XX$.

  \noindent
  (ii) By Definition~\ref{def:RV}, $g_1(t)T_{t^{-1}}\nu$ converges to
  $\mu_1$ vaguely in $\sV$ with $g_1\in\RV_{\alpha_1}$.
  Condition~\hyperref[condB0]{(B$_0$)} ensures the existence of a set
  $B=\{\modulus> 1\}$ for some continuous 
  modulus $\modulus$ which is a continuity set for both $\mu$ and $\mu_1$,
  and such that $\mu_1(B)$ is non-trivial. Hence,
  \begin{align*}
    g_1(t)\nu(T_tB) &\to \mu_1(B)\in(0,\infty),\\
    g(t)\nu(T_tB) &\to \mu(B)\in[0,\infty)
  \end{align*}
  with $g\in\RV_\alpha$. If $\mu(B)>0$, then it is possible to let
  $g_1=g$, which necessitates
  $\alpha=\alpha_1$. If $\mu(B)=0$, then for any $\eps>0$,
  $g(t)\nu(T_tB)<\eps$ for large $t$. Since $g_1(t)\nu(T_tB)$ is
  bounded away from zero, we must have $g(t)=o(g_1(t))$ as
  $t\to\infty$, which implies $\alpha\leq\alpha_1$. 
\end{proof}

\section{Regular variation on quotient spaces}
\label{sec:quotient-spaces-under}

\paragraph{Scaling on quotient spaces}

Let $\sim$ be an \emph{equivalence relation} on a topological space
\index{equivalence relation} \index{saturation}
$\XX$. Let $[x]=\{y\in\XX\colons y\sim x\}$ denote the
\emph{equivalence class} of
$x\in\XX$.  For a set $B\subset\XX$, denote by $[B]$ the
\emph{saturation} of $B$, which is the set of all $y\in\XX$ such that
$y\sim x$ for some $x\in B$. Equivalently, $[B]$ is the union of the
equivalence classes $[x]$ for all $x\in B$.

The family of equivalence classes is called the \emph{quotient space}
\index{quotient space} \index{factor space} \index{quotient map} (or
the factor space). The quotient space is denoted by $\XXT$ and its
generic element by $\tilde{x}$. The \emph{quotient map}
$q:\XX \to \XXT$ associates each $x\in\XX$ with its equivalence class.
It is worth pointing out that $q q^{-1} \tilde{x} = \tilde{x}$ but
the composition $x\mapsto q^{-1} q x = [x]$ is a set-valued map.  The space $\XXT$ is
endowed with the finest topology under which $q$ is continuous; this
topology is called the quotient (or factor) topology; see
\cite[Section~2.4]{eng89}. Note that a general continuous surjection $q$
between two topological spaces is referred to as a quotient map if the
topology on the target space consists of all sets $G$ such that
$q^{-1}(G)$ is open in the original space $\XX$.
The Baire $\sigma$-algebra generated by the quotient topology is
denoted by $\sBA(\XXT)$.

The quotient space $\XXT$ is Hausdorff if (i) the quotient map is open
(equivalently, if $[G]$ is open for each open $G$ in $\XX$) and (ii) the
equivalence relation 
$R=\{(x,y)\colons x\sim y\}$ is closed in $\XX\times\XX$; see
\cite[Proposition~I.8.3.8]{bourbaki1}.  If $\XX$ is Polish and $\XXT$
is Hausdorff, then $\XXT$ is a Souslin space. Recall that Souslin
spaces are hereditary Lindel\"of; see
\cite[Lemma~6.6.4]{bogachev07}. However,
even in the case where $\XX$ is Polish, $\XXT$ may fail to
be metrisable or even completely regular;
see \cite{arh65} and \cite{him68}. If the space $\XX$ is sequential,
then the quotient space $\XXT$ is also sequential; see
\cite[Exercise~2.4.G]{eng89}.

Assume that the scaling and the equivalence relation on $\XX$ satisfy the
following condition.
\index{Condition (S)}
\begin{itemize}
\item[(S)] \label{condS} If $x\sim y$, then $T_tx\sim T_ty$ for all $t>0$.
\end{itemize}
If \hyperref[condS]{(S)} holds, it is possible to define a scaling on $\XXT$ by setting
\begin{equation}
  \label{scaling-on-XXT}
  \tilde{T}_t \tilde{x} = q T_t x
\end{equation}
for any representative $x\in q^{-1}\tilde{x}$. By \hyperref[condS]{(S)}, this definition is consistent:
if $z\in q^{-1}\tilde{x}$, then $x\sim z$ and $T_tx\sim T_tz$, whence
$qT_tx=qT_tz$.

\begin{lemma}
  If the quotient map $q$ is an open map
  and the scaling on a topological space $\XX$
  is continuous, then the induced scaling on $\XXT$ is continuous.
\end{lemma}
\begin{proof}
  Let $\tilde{G}$ be an open set in $\XXT$. Then $\tilde{G}=qG$, where
  $G=q^{-1}\tilde{G}$ is open in $\XX$, and
  \begin{align*}
    \big\{(t,\tilde{x})\in\Rpp\times\XXT\colons 
    \tilde{T}_t\tilde{x}\in \tilde{G}\big\}
    &=\big\{(t,qx)\in\Rpp\times\XXT\colons q(T_tx)\in qG\big\} \\
    &=\big\{(t,qx)\in\Rpp\times\XXT\colons T_tx\in [G]\big\}.
  \end{align*}
  The right-hand side is open in $\Rpp\times\XXT$. Indeed,
  $[G]=q^{-1}qG$ is open in $\XX$, the set $\{(t,x)\colons T_tx\in [G]\}$ is open in
  $\Rpp\times\XX$, and the map $(t,x)\mapsto (t,qx)$ is open.
\end{proof}

\paragraph{Ideals on quotient spaces}
Let $\sS$ be an ideal on $\XX$.  Denote
\begin{displaymath}
  [\sS]=\big\{B\in\sS\colons [B]\in\sS\big\}, 
\end{displaymath}
which is also an ideal on $\XX$ and clearly $[\sS]\subset\sS$.
Define an ideal $q\sS$ on the quotient space as the family of sets of
the form $qB$ for $B\in[\sS]$.  In other words, bounded sets in $\XXT$ are
images of sets whose saturations are bounded in $\XX$. Thus,
$\tilde{B}\in q\sS$ exactly if $q^{-1}\tilde{B}\in\sS$; equivalently,
$q^{-1}\tilde{B}\in[\sS]$. Consequently, $q\sS$ is the ideal induced
on $\XXT$ through the quotient map $q$.
If Condition~\hyperref[condS]{(S)} holds and $\sS$ is scaling
consistent, then $q\sS$ is likewise  
a scaling consistent ideal on $\XXT$.
\index{ideal!on quotient spaces}

The following result aims to show that, under certain conditions, $q\sS$
is generated by a continuous modulus, hence, satisfies
Condition~\hyperref[condB]{(B)}. This enables the application of 
results based on the polar decomposition in the
context of quotient spaces.

\begin{lemma}
  \label{lemma:q-B} 
  Assume that Condition~\hyperref[condS]{(S)} holds, the quotient map
  $q$ is an open map and that an ideal $\sS$ on the 
  topological space $\XX$ is generated by a continuous modulus
  $\modulus$. Assume further that the map $x\mapsto\ttau([x])$ is
  continuous on $\XX$,
  where $\ttau$ is defined in \eqref{eq:12}. Then $[\sS]$ is generated
  by $\ttau([x])$ and the ideal $q\sS$ is generated by the continuous
  modulus
  \begin{displaymath}
    \tilde\modulus(\tilde{x})=\ttau\big(q^{-1}\tilde{x}\big), \quad
    \tilde{x}\in\XXT. 
  \end{displaymath}
  \index{modulus!on quotient spaces}
\end{lemma}
\begin{proof}
  For each open interval $(a,b)$, 
  \begin{displaymath}
    \big\{\tilde{x}\colons \tilde\modulus(\tilde{x})\in (a,b)\big\}
    =\big\{qx\colons \tilde\modulus(qx)\in (a,b)\big\}
    =q\big\{x\colons \ttau([x])\in (a,b)\big\}.
  \end{displaymath}
  The continuity of $\tilde\modulus$ follows from the imposed
  condition on $\ttau$ and the fact that $q$ is an open map.
  By construction, $B\in[\sS]$ if and only if $[B]\in\sS$,
  equivalently, $\ttau([B])>0$. Since
  \begin{displaymath}
    \ttau([B])=\inf\big\{\ttau([x])\colons x\in B\big\},
  \end{displaymath}
  the ideal $[\sS]$ is generated by $\ttau([x])$.  For each $t>0$,
  \begin{displaymath}
    q^{-1}\big\{\tilde{x}\colons \tilde\modulus(\tilde{x})>t\big\}
    =\big\{x\colons \tilde\modulus(qx)>t\big\}
    =\big\{x\colons \ttau([x])>t\big\} \in [\sS]. 
  \end{displaymath}
  If $\tilde{B}\in q\sS$, then $q^{-1}\tilde{B}\in \sS$, equivalently,
  $q^{-1}\tilde{B}\in [\sS]$, so that
  \begin{displaymath}
    \inf_{\tilde{x}\in\tilde{B}} \tilde\modulus(\tilde{x})
    =\inf_{x\in q^{-1}\tilde{B}} \ttau([x])>0.
  \end{displaymath}
  Thus, the ideal $q\sS$ is generated by the continuous modulus
  $\tilde\modulus$. 
\end{proof}

Assume that the ideal $\sS$ on $\XX$ satisfies Condition~\hyperref[condB0]{(B$_0$)} and consequently is
generated by a family $(\modulus_j)_{j\in \JJ}$ of continuous moduli.
For a finite subset $I\subset \JJ$, define 
\begin{displaymath}
  \modulus_I(x)=\max_{i\in I} \modulus_i(x),\quad x\in\XX.
\end{displaymath}
While the families $(\modulus_j)_{j\in \JJ}$ and
$(\modulus_I)_{I\subset \JJ}$ generate the same ideal, the latter family
is essential when passing to quotient spaces.

\begin{example}
  Let $\XX$ be the space of continuous functions on $[0,1]$. Assume
  that functions $x$ and $y$ are equivalent if $x-y$ is a
  constant function. For each rational $u_i\in[0,1]$ define the modulus
  $\modulus_i(x)=|x(u_i)|$, $x\in\XX$. Then, for each $i$,
  $\ttau_i([x])=0$ because the equivalence
  class $[x]$ contains a function that is zero at $u_i$. However,
  $\ttau_I([x])$ does not vanish, if $I\supset\{ i,j\}$ with 
  $x(u_i)\neq x(u_j)$; for $I=\{i,j\}$ it equals
  $\frac12 |x(u_i)-x(u_j)|$.
\end{example}

The following result is proved analogously to Lemma~\ref{lemma:q-B}.

\begin{lemma}
  Assume that the quotient map $q$ is an open map on the topological
  space $\XX$. If $\sS$ satisfies Condition~\hyperref[condB0]{(B$_0$)}
  with a family 
  $(\modulus_j)_{j\in \JJ}$ of continuous moduli such that, for every
  finite subset $I\subset\JJ$, the map
  \[
    x\mapsto \ttau_I([x])
  \]
  is continuous, where $\ttau_I$ is defined using
  $\modulus_I=\max_{i\in I}\modulus_i$, then the pushforward ideal
  $q\sS$ also satisfies Condition~\hyperref[condB0]{(B$_0$)} and is
  generated by the family of continuous moduli
  $\tilde{\modulus}_I(\tilde{x})=\ttau_I(q^{-1}\tilde{x})$,
  $\tilde{x}\in\XXT$ for all finite subsets $I\subset \JJ$.
\end{lemma}

\paragraph{Regularly varying random elements in quotient spaces}
If $\xi$ is a Baire measurable random element in $\XX$, then $q\xi$ is
a Baire measurable random element in $\XXT$ owing to the continuity of
$q$. With the above-defined scaling and ideal on $\XXT$, we can consider
regular variation on the quotient space $\XXT$.
	
\begin{theorem}\label{rv-XXT}
  Let $\XXT$ be a quotient space of a topological space $\XX$ with the
  quotient map $q$, and assume that Condition~\hyperref[condS]{(S)}
  holds so that the scaling $\tilde T$ is well defined on $\XXT$.
  If $\xi\in\RV(\XX,T,\sS,g,\mu)$ for a scaling-consistent ideal $\sS$,
  and the restriction of $\mu$ to $[\sS]$ is non-trivial on $[\sS]$,
  that is, $\mu(B)>0$ for some 
  $B\in[\sS]$, then
  \begin{displaymath}
    q\xi\in\RV(\XXT,\tilde{T},q\sS,g,\tilde{\mu}),
  \end{displaymath}
  where $\tilde{\mu}$ is the pushforward under $q$
  of the restriction of
  $\mu$ to the union of the sets in $[\sS]$.
  \index{regular variation!on quotient spaces}
\end{theorem}
\begin{proof}
  The quotient map $q:\XX\to \XXT$ is continuous and
  $\tilde{T}_t qx = q T_t x$ for all $x\in \XX$; thus, $q$ is a
  continuous morphism. By the construction of $q\sS$, the map $q$ is
  bornologically consistent. Then, 
  \begin{displaymath}
    (q\mu)(qB)=\mu(q^{-1}(qB))=\mu([B])
  \end{displaymath}
  and $\mu([B])>0$ for some $B\in\sS$ by the assumed non-triviality of
  $\mu$ on $[\sS]$. Thus, $qB\in q\sS$, so that
  $q\mu$ is non-trivial on $q\sS$ and the result follows from
  Theorem~\ref{mapping-theorem}.
\end{proof}

\begin{remark}
  \label{rem:qS=S}
  If $[\sS]=\sS$, then the quotient map always preserves the regular
  variation property. This occurs, for instance, if $\sS$ is the
  metric exclusion ideal obtained by excluding the saturation of a set.
  Furthermore, if the modulus $\modulus$ is
  constant on each equivalence class, then
  $\tilde\modulus(qx)=\modulus(x)$ for all $x\in\XX$ and $[\sS]=\sS$.
\end{remark}

Under additional conditions, it is possible to infer the regular variation
property of a random element $\xi$ in $\XX$ from the regular variation
property of $q\xi=\tilde{\xi}$ in $\XXT$. To achieve this, one must identify a
\emph{selection map}
\index{selection map}
$\tilde{q}$, which is a continuous morphism from
$\XXT$ to $\XX$ such that $\tilde{q}(qx)\in [x]$ for all $x\in\XX$.
If $\tilde{q}$ is bornologically consistent, then
$\tilde{q}(\tilde{\xi})$ is a regularly varying random element in
$\XX$. For instance, the selection map is bornologically consistent,
if $\tilde{q}(\tilde{x})=x$ implies
$\modulus(x)=\ttau([x])$ for all $x\in\XX$, that is, 
$\tilde{q}(qx)$ selects an element from $[x]$ with the minimal
modulus.

\begin{example}[Modulus as the quotient map]
  Let $\XX$ be a topological space with the ideal $\sS_\modulus$
  generated by a continuous modulus $\modulus$. For $x,y\in\XX$, let
  $x\sim y$ if $\modulus(x)=\modulus(y)$.
  The quotient space $\XXT$ may be identified with the image of
  $\modulus$ equipped with the quotient topology induced by the map
  $q x=\modulus(x)$. If this quotient topology coincides with the usual
  topology on $\R_+$, then this recovers the statement of
  Corollary~\ref{cor:tau-map}. Condition~\hyperref[condS]{(S)} holds
  and $[\sS]=\sS$. If $\xi$ is 
  regularly varying on $\sS_\modulus$, then $q\xi=\modulus(\xi)$ is
  regularly varying on $q\sS=\sR_0$. This recovers the statement of
  Corollary~\ref{cor:tau-map}.
\end{example}

\begin{example}[Equivalence up to translations]
  If $\XX$ is a normed linear space, quotient spaces are usually
  defined using a linear subspace $\LL$ of $\XX$, whereby $x\sim y$
  if and only if $x-y\in\LL$. Then $[B]=B+\LL=\{x+y\colons x\in B,y\in\LL\}$. If
  $\XX$ is a Hilbert space, then the quotient space may be identified
  with the orthogonal complement $\LL^\perp$. The corresponding quotient
  map is clearly open. If the ideal $\sS_0$ is obtained by excluding
  the zero element in $\XX$, then $[\sS_0]$ is the metric exclusion
  ideal on $\XX$ obtained by excluding $\LL$.
  If $\xi$ is regularly varying on $\sS_0$ and its tail measure $\mu$
  is non-trivial outside $\LL$, then $q\xi$ is regularly varying in
  the quotient space. The tail measure is the pushforward of the
  restriction of $\mu$ to $\XX\setminus\LL$.

  It is known that each separable Banach space $\XX$ is a quotient of
  the space $\ell_1$ of absolutely summable sequences via the
  quotient map
  \begin{displaymath}
    qz=\sum_{i=1}^\infty z_ix_i,\quad z=(z_i)\in\ell_1,
  \end{displaymath}
  where $(x_i)$ is a countable dense set in the unit ball of $\XX$;
  see, for example, Theorem 5.1 in \cite{MR2766381}. The above argument
  applies with $\LL=\{z\in\ell_1\colons \sum z_ix_i=0\}$. Thus, regular
  variation in $\XX$ may be derived from regular variation
  on $\ell_1$ analysed in Section~\ref{sec:infinite-sequences}.

  Furthermore, each separable Banach space is isometric to a subspace
  of the quotient space of $\ell_\infty$ with respect to the space
  $c_0$ of sequences converging to zero; see
  \cite[Proposition~5.5]{MR2766381}. In this way, regular
  variation on $\XX$ follows from regular variation on
  $\ell_\infty$ if the tail measure is not concentrated on $c_0$.
\end{example}

\section{Product spaces}
\label{sec:scal-prod-spac}

\paragraph{Scaling along all components.}

Assume that $\XX$ is a Hausdorff topological space equipped with a
continuous scaling $T_t$. Fix $m\in\NN$. The scaling can be lifted to
the Cartesian power $\XX^m$ by applying it to every factor, that is,
\index{product space}
\begin{displaymath}
  T_t(x_1,\dots,x_m)=(T_tx_1,\dots,T_tx_m),\quad
  x_1,\dots,x_m\in\XX. 
\end{displaymath}
If $\nu$ is a Radon (and so Borel) measure on $\XX$, denote by
$\nu^{\otimes m}$ the $m$-th product measure on Borel subsets of the
Cartesian power $\XX^m$; see \cite[Theorem~7.6.2]{bogachev07}, which
relies on the Hausdorff assumption. It is known from
\cite[Theorem~4.1.11(iii)]{bogachev18} that every Borel measure on a
completely regular Souslin space is Radon.  The product measure is
also well defined and Fubini's theorem holds if all factors are
tau-additive Borel measures; see \cite[Theorem~7.6.5]{bogachev07}.

\begin{theorem}
  \label{thr:power-space}
  Let $\XX$ be a Hausdorff topological space equipped with a
  continuous scaling $T$ 
  and a topologically and scaling consistent ideal $\sX$. If
  \begin{displaymath}
    \nu\in\RV(\XX,T,\sX,g,\mu)
  \end{displaymath}
  is a Radon probability measure, and assume that
  $T_{t^{-1}}\nu\wto\nu_0$. 
  Then, for every $m\in\NN$ and
  $k\in\{1,\dots,m\}$, we have
  \begin{displaymath}
    \nu^{\otimes m}\in \RV(\XX^m,T,\sX^m(k),g^k,\mu^{(k)}),
  \end{displaymath}
  where
  \begin{displaymath}
    \mu^{(k)}(\bdiff(x_1,\dots,x_m))
    =\sum_{I\subset\{1,\dots,m\},\card(I)=k}
    \prod_{i\in I}\mu(\bdiff x_i)\prod_{j\notin I}\nu_0(\bdiff x_j).
  \end{displaymath}
  \index{regular variation!on product space}
\end{theorem}
\begin{proof}
  Recall that a set $A$ in $\XX^m$ belongs to the ideal $\sX^m(k)$ if
  and only if there exist $B\in\sX$ and sets 
  $I_1,\dots,I_l\subset \{1,\dots,m\}$, each of cardinality $k$, such
  that
  \begin{displaymath}
    A\subset \bigcup_{j=1}^l A_{1j}\times \cdots\times A_{mj},
  \end{displaymath}
  where $A_{ij}=B$ if $i\in I_j$ and $A_{ij}=\XX$ otherwise.  Fix 
  sets $B_1,\dots,B_k\in\sX$ and set
  $I=\{i_1,\dots,i_k\}\subset\{1,\dots,m\}$ of cardinality $k$.

  Denote by
  $(B_1,\dots,B_k)_I$ the set obtained as the Cartesian
  product $A_1\times\cdots\times A_m$ of the sets
  \begin{displaymath}
    A_j=
    \begin{cases}
      B_l, & \text{if}\; j=i_l\in I=\{i_1,\dots,i_k\},\\
      \XX, & \text{otherwise},
    \end{cases}
    \quad j=1,\dots,m. 
  \end{displaymath}
  The family of such sets $(B_1,\dots,B_k)_I$ for all Baire sets
  $B_1,\dots,B_k$ in $\sX$ and all sets $I$ of cardinality $k$
  generates the ideal $\sX^m(k)$. Together with rectangles in the
  remaining coordinates taken from a weak-convergence determining
  class for $\nu_0$, these sets determine the vague topology on
  $\sX^m(k)$. If all $B_1,\dots,B_k$ are
  $\mu$-continuity sets, then
  \begin{displaymath}
    g(t)^k\nu^{\otimes m}
    \big(T_t(B_1,\dots,B_k)_I\big)
    =\prod_{j=1}^k \Big(g(t)\nu(T_tB_j)\Big)
    \to \prod_{j=1}^k \mu(B_j)\quad \text{as}\; t\to\infty. 
  \end{displaymath}
  Note that the sets $(B_1,\dots,B_k)_I$ with $\card(I)\leq k-1$ do
  not belong to the ideal $\sX^m(k)$.

  To identify the
  vague limit, it is enough to consider rectangles in which the $k$
  large coordinates lie in sets from $\sX$, while the remaining
  coordinates lie in $\nu_0$-continuity sets. Thus, for
  $B_1,\dots,B_k\in\sX$ and Borel $\nu_0$-continuity sets $C_j$,
  $j\notin I$, we obtain
  \[
    g(t)^k\nu^{\otimes m}\big(T_t R\big)
    =
    \prod_{l=1}^k g(t)\nu(T_tB_l)
    \prod_{j\notin I}\nu(T_tC_j)
    \to
    \prod_{l=1}^k\mu(B_l)\prod_{j\notin I}\nu_0(C_j),
  \]
  where $R$ is the corresponding rectangle.
  If $I=\{i_1,\dots,i_l\}$ has cardinality $l\geq k+1$ and the
  coordinates in $I$ are are constrained to lie in sets
  $B_1,\dots,B_l\in\sX$, define $(B_1,\dots,B_l)_I$ analogously. Then
  \[
    g(t)^k\nu^{\otimes m}\big(T_t(B_1,\dots,B_l)_I\big)
    =
    \prod_{j=1}^k \big(g(t)\nu(T_tB_j)\big)
    \prod_{j=k+1}^l \nu(T_tB_j)
    \to0 .
  \]
  Thus no terms with more than $k$ large coordinates contribute to
  the limit.
 
  The limiting measure is exactly $\mu^{(k)}$ from the statement.
\end{proof}

Consider now the product of $m$ possibly different topological spaces
$\XX_1,\dots,\XX_m$ with ideals $\sX_1,\dots,\sX_m$. The scaling on
the space $\XX_1\times\cdots\times\XX_m$ is defined by applying the
scalings $T^{(i)}$ on each of the factors, so that
$T_t x=(T_t^{(1)}x_1,\dots,T_t^{(m)}x_m)$ for
$x=(x_1,\dots,x_m)\in \XX_1\times\cdots\times\XX_m$. Assume that
$\nu_i\in\RV(\XX_i,T^{(i)},\sX_i,g_i,\mu_i)$, $i=1,\dots,m$, are Radon
probability measures and that $T_{t^{-1}}^{(i)}\nu_i\wto\nu_{0,i}$ as
$t\to\infty$. 

The regular variation property of the measure
$\nu=\nu_1\otimes\cdots\otimes\nu_m$ depends on the construction of
the ideal of the product space. For instance, assume that
$\sX=\sX_1\times\cdots\times\sX_m$, meaning that $B\in\sX$ whenever at least
one of its projections belongs to the corresponding ideal $\sX_i$.
Assume that there exists
$i\in\{1,\dots,m\}$ such that $g_i(t)/g_j(t)\to 0$ as $t\to\infty$
for all $j\neq i$. In this case only the $i$-th component determines
the regular variation of $\nu$, namely, $\nu$ is regularly varying
with the normalising function $g_i$ and the tail measure being the
product of $\mu_i$ and $\nu_{0,j}$ in the remaining coordinates.
If several of the slowest-growing normalising functions are
asymptotically equivalent, then the tail measure on the product ideal
$\sX_1\times\cdots\times\sX_m$ is the sum of the corresponding
one-large-component contributions, with the remaining coordinates
carrying their weak limits
$\nu_{0,j}$, in particular $\delta_0$ when
$T_{t^{-1}}^{(j)}\nu_j\wto\delta_0$.

Assume that $\sX=\sX_1\otimes\cdots\otimes\sX_m$, meaning that
$B\in\sX$ if all its projections are bounded. Then $\nu$ is regularly
varying on $\sX$ with the normalising function $g_1\cdots g_m$ and the
tail measure being $\mu_1\otimes\cdots\otimes\mu_m$.

\begin{example}[Maps on product spaces]
  \label{ex:maps-products}
  Let $f$ be a continuous map from $\XX_1\times\cdots\times\XX_m$ into
  another topological space $\YY$ equipped with a continuous scaling
  and a topologically and scaling consistent ideal $\sY$. If $f$ is a
  bornologically consistent morphism and
  $\nu=\nu_1\otimes\cdots\otimes\nu_m$ (being the product of Radon
  measures) is regularly varying on the product space and the
  pushforward of the tail measure of $\nu$ under $f$ is non-trivial on
  $\sY$, then the pushforward of $\nu$ under $f$ is regularly varying
  on $\YY$, and its tail measure is the pushforward of the tail
  measure of $\nu$.

  For instance, consider the
  product of identical spaces $\R_+$ with the ideal
  $\sR_0$. Let $\xi_1,\dots,\xi_m$ be independent non-negative random
  variables such that 
  \begin{displaymath}
    \xi_i\in\RV(\R_+,\mydot,\sR_0,g,\mu_i), \quad i=1,\dots,m.
  \end{displaymath}
  The function $f(\xi_1,\dots,\xi_m)=\xi_1+\cdots+\xi_m$ is a
  bornologically consistent continuous morphism from
  $(\sR_0)^m(1)=\sR_0^m$ to $\R_+$ with the ideal $\sR_0$. By
  Theorem~\ref{thr:power-space} with $k=1$, we see that
  $\xi_1+\cdots+\xi_m$ is regularly varying on $\R_+$ with the tail
  measure being $\mu_1+\cdots+\mu_m$. If, more generally, $\xi_i$ is
  regularly varying with the normalising function $g_i$ and the
  normalising functions are not asymptotically equivalent, then only
  the components with the slowest growing normalising functions,
  equivalently the heaviest tails, contribute to the sum giving the
  tail measure.
\end{example}

\paragraph{Scaling along one component}

A scaling along \emph{one component} in a product space is often
useful for applications of the continuous mapping theorem. For
example, consider a pair $(\eta,\xi)\in\YY\times\XX$ where scaling
is applied only to the 
second component.
\index{scaling!along one component}
In this case, the ideal on the product space imposes no
restriction on the first component, and so the regular variation of
the pair means that $g(t)\E h(\eta,T_{t^{-1}}\xi)$ converges to a
finite limit for all
continuous bounded functions $h$ supported on $\YY\times B$ for $B$
from an ideal on $\XX$. Using the continuous mapping theorem, we can
often establish regular variation for random elements of the form
$f(\eta,\xi)$, where $f$ is a function homogeneous in its second
argument. This approach will be used in 
the next section to derive a generalisation of Breiman's lemma
and in Section~\ref{sec:continuous-functions} to establish the regular
variation property of random functions.

\begin{theorem}
  \label{thr:pair}
  Consider the product $\YY\times \XX$ of perfectly normal topological
  spaces $\YY$ and $\XX$, where $\XX$ is equipped with a continuous
  scaling $T_t$ and an ideal $\sX$ generated by the continuous finite
  modulus $\modulus^\XX$ on $\XX$. Define a scaling on the product
  space by $T_t(y,x)=(y,T_tx)$. Assume that
  $\xi\in\RV(\XX,T,\sX,g,\mu)$, and let $(\eta_x)_{x \in \XX}$ be a
  jointly measurable stochastic process taking values in $\YY$ that is independent of
  $\xi$. Furthermore, suppose that
  \begin{equation}
    \label{eq:hxW}
    \eta_x\dto  W \quad  \text{as}\;\;\; \modulus^\XX(x) \to \infty
  \end{equation}
  for some random element $W$ in $\YY$. Assume that $\mu$ and the
  distribution $P_W$ of $W$ are tau-additive measures. 
  Then $(\eta_\xi,\xi)\in\RV(\YY\times\XX,T,\sV,g,\mu')$, where the
  ideal $\sV$ is the product ideal generated by the sets $\YY\times A$,
  $A\in\sX$, and 
  $\mu'$ denotes the tau-additive extension of $P_W\otimes\mu$.
\end{theorem}
\begin{proof}
  Note that the ideal $\sV$ on $\YY\times\XX$ is generated by the
  modulus $\modulus(y,x)=\modulus^\XX(x)$.

  Let $f:\XX\to\R_+$ be a bounded continuous function supported on
  $\{\modulus^\XX>u\}$ for some $u>0$, and let
  $h:\YY\to\R_+$ be a bounded continuous function. By independence of
  the process $(\eta_x)$ and $\xi$,
  \begin{displaymath}
    \E\big[f(T_{t^{-1}}\xi)h(\eta_\xi)\big]
    =
    \E\big[f(T_{t^{-1}}\xi)m_h(\xi)\big],
  \end{displaymath}
  where $m_h(x)=\E h(\eta_x)$. By \eqref{eq:hxW}, for every
  $\eps>0$ there exists $R>0$ such that
  \begin{displaymath}
    |m_h(x)-\E h(W)|\leq\eps
    \quad\text{whenever}\quad \modulus^\XX(x)>R .
  \end{displaymath}
  If $f(T_{t^{-1}}\xi)\neq0$, then $\modulus^\XX(\xi)>ut$, and so,
  for all sufficiently large $t$,
  \begin{displaymath}
    |m_h(\xi)-\E h(W)|\leq\eps
  \end{displaymath}
  on the support of $f(T_{t^{-1}}\xi)$.
  Hence
  \begin{displaymath}
    \lim_{t\to\infty}
    g(t)\E\big[f(T_{t^{-1}}\xi)h(\eta_\xi)\big]
    =
    \E h(W)\int f\,\diff\mu .
  \end{displaymath}
  Equivalently,
  \begin{displaymath}
    \lim_{t\to\infty}
    g(t)\E\big[f(T_{t^{-1}}\xi)h(\eta_\xi)\big]
    =
    \int h(y)f(x)\,\mu'(\bdiff(y,x)),
  \end{displaymath}
  where $\mu'$ denotes the tau-additive extension of
  $P_W\otimes\mu$. By linearity, the same convergence holds for bounded continuous
  functions $h$ of arbitrary sign. By 
  Lemma~\ref{lemma:conv-det-product}, this convergence for product test
  functions implies
  \begin{displaymath}
    g(t)\Prob{T_{t^{-1}}(\eta_\xi,\xi)\in\cdot}
    \vto \mu'
  \end{displaymath}
  on the ideal $\sV$. Therefore
  $(\eta_\xi,\xi)\in\RV(\YY\times\XX,T,\sV,g,\mu')$.
\end{proof}

\paragraph{Generalised Breiman lemma}

The following result implies that, if
$\xi\in\RV(\R_+,\mydot,\sR_0,g,c\theta_\alpha)$ and $\eta$ is a
non-trivial  random element in a Banach space $\YY$ independent of $\xi$ 
and satisfying
$\E \|\eta\|^{\alpha+\delta}<\infty$ for some
$\delta>0$, then the random element $\xi \eta$ is regularly varying in
$\YY$ with the same tail index $\alpha$. For real-valued $\eta$ this
fact is often referred to as Breiman's lemma; however, it holds much
more generally as demonstrated by the following result.

\begin{lemma}[Generalised Breiman lemma]
  \label{lem:BreimanNew}
  Let $\YY$ be a perfectly normal Lindel\"of topological space
  equipped with a continuous scaling $T_t$ and the ideal $\sY$
  generated by a continuous finite modulus $\modulus$. Consider the
  space $\ZZ=\YY\times(0,\infty)$ with the ideal $\sZ=\YY\times\sR_0$ (that
  is, $\sZ$ is generated by $\YY\times B$ for $B\in\sR_0$) and the
  scaling $T^\ZZ$ obtained by applying linear scaling to the second
  component.
  \index{Breiman's lemma!generalisation of}

  Let $(\eta,\zeta)$ be a random element in $\YY\times(0,\infty)$. Assume
  that
  \begin{equation}
    \label{eq:eta-zeta}
    (\eta,\zeta)\in \RV(\ZZ,T^\ZZ,\sZ,g,\sigma\otimes\theta_\alpha),
  \end{equation}
  where $\sigma$ is a finite Borel measure on $\YY$ such that
  $\sigma(\{\modulus>0\})>0$.
  Furthermore, assume that, for some $\delta>0$ and $C<\infty$,
  \begin{equation}
    \label{eq:BreimanAss2}
    \E \big[\modulus(\eta)^{\alpha+\delta}\mid\zeta\big]
    \leq C \quad \text{a.s. on}\; \{\zeta>0\}.
  \end{equation}
  Then
  \begin{displaymath}
    T_\zeta\eta\in \RV(\YY,T,\sY,g,\mu),
  \end{displaymath}
  where $\mu$ is the pushforward of the product measure
  $\sigma \otimes \theta_\alpha$ under the map
  $(y,s)\mapsto T_s y$ from $\YY\times \R_+$ to $\YY$.
\end{lemma}
\begin{proof}
  Recall that a perfectly normal Lindel\"of space is also hereditary
  Lindel\"of; see \cite[Exercise~3.8.A]{eng89}. By
  \cite[Proposition~7.2.2(iv)]{bogachev07}, the topological 
  assumptions on $\YY$ ensure that $\sigma$ is tau-additive and we can
  work with the product measure $\sigma\otimes\theta_\alpha$.
  
  By \eqref{eq:eta-zeta}, 
  $\zeta\in\RV(\R_+,\mydot,\sR_0,g,c\theta_\alpha)$, where
  $c=\sigma(\YY)$.  Fix an arbitrary $\eps>0$ and a continuous
  function $h:\YY\to\R$ such that $|h|\leq M<\infty$ and $h$ is
  supported on $\{y\colons \modulus(y)>a\}$ for some $a>0$. Then
  \begin{align*}
    g(t) \E  h\big(T_{t^{-1}} T_\zeta\eta\big)
    &= g(t) \E \Big[h\big(T_{t^{-1}} T_\zeta\eta\big)
    \one_{\zeta>t\eps}  \Big]
      + g(t) \E \Big[ h \big(T_{t^{-1}}T_\zeta\eta\big)
      \one_{\zeta\leq t\eps}  \Big] \\
    &= I_1(t) + I_2(t)\,.
  \end{align*}
  By the assumption on $h$,
  \begin{align*}
    |I_2(t)|
    &=g(t)\E \big[ |h(T_{t^{-1}} T_\zeta\eta)|
      \one_{\zeta\leq t \eps}  \big]\\ 
    &\leq M g(t)
      \Prob{\modulus(\eta)\zeta\one_{\zeta\leq t \eps}>ta } \\
    & \leq 
      \frac{ M g(t)}{(ta)^{\alpha+\delta}}
      \E \Big[\zeta^{\alpha+\delta} \one_{\zeta\leq t\eps}  
      \E \big[\modulus(\eta)^{\alpha+\delta} \one_{\zeta>0}\mid
      \zeta\big] \Big]\\
    & \leq 
      \frac{ C M g(t)}{(ta)^{\alpha+\delta}}
      \E \big[\zeta^{\alpha+\delta} \one_{\zeta\leq t\eps}\big]
    \leq \frac{ C M }{a^{\alpha+\delta}}
      \frac{\E \big[\zeta^{\alpha+\delta} \one_{\zeta\leq t \eps}\big]}
      {(t\eps)^{\alpha+\delta}\Prob{\zeta>t \eps  }}\eps^{\alpha+\delta}
      g(t)\Prob{\zeta >t \eps}.
  \end{align*}
  By the regular variation of $\zeta$ and Karamata's theorem (see
  \cite[Section~B.4]{bdm}), we 
  obtain that $g(t)\Prob{\zeta >t \eps}\to c\eps^{-\alpha}$ and
  \begin{displaymath}
    \frac{\E \big[\zeta^{\alpha+\delta} \one_{\zeta\leq t \eps}\big]}
    {(t\eps)^{\alpha+\delta}\Prob{\zeta>t \eps  }}
    \to \frac{\alpha}{\delta}\quad \text{as}\; t\to\infty.
  \end{displaymath}
  Therefore,
  \begin{displaymath}
    \limsup_{t\to\infty} |I_2(t)|\leq
    \frac{ C M }{a^{\alpha+\delta}} \frac{\alpha c}{\delta}
    \eps^{\delta}.
  \end{displaymath}
  By \eqref{eq:eta-zeta}, applied to bounded continuous
  approximations of the function
  $(y,s)\mapsto h(T_sy)\one_{\{s>\eps\}}$, and since
  $\sigma\otimes\theta_\alpha$ gives no mass to
  $\YY\times\{\eps\}$, we obtain
  \begin{displaymath}
    I_1(t)=g(t)\E\big[h(T_{t^{-1}}T_\zeta\eta)\one_{\zeta>t\eps}\big]
    \to \int_\eps^\infty \int_{\YY} h(T_s y)\sigma(\bdiff
    y)\theta_\alpha(\bdiff s) \quad \text{as}\; t\to\infty. 
  \end{displaymath}
  We now show that the limit of $I_1(t)$ as $\eps\to0$ is finite.  By
  \eqref{eq:33} and \eqref{eq:eta-zeta}, there exists a random vector
  $(Y,U)$ in $\R_+\times[1,\infty)$ with independent components,
  satisfying
  \begin{equation}
    \label{eq:dtoWU}
    \Prob{(\modulus(\eta)^\alpha, t^{-1} \zeta) \in\cdot\,\mid \zeta>t}
    \wto \Prob{(Y,U) \in \cdot\, }.
  \end{equation}
  Note that $\Prob{U>s}=s^{-\alpha}$ for $s\geq 1$ and $Y$ is
  distributed as the pushforward of $c^{-1}\sigma$ under the map $y\to
  \modulus(y)^\alpha$. 
  Take an arbitrary non-decreasing sequence $(t_n)$, $t_n \to \infty$,
  and consider nonnegative random variables $Y_n$ satisfying
  \begin{displaymath}
    \Prob{Y_n\in \cdot} = \Prob{\modulus(\eta)^\alpha \in\cdot\,\mid
      \zeta>t_n}.
  \end{displaymath}
  Denote $p=(\alpha+ \delta)/\alpha>1$.  For each $n$,
  \eqref{eq:BreimanAss2} implies that
  \begin{displaymath}
    \E Y_n^p  = \E \big[\modulus(\eta)^{\alpha +\delta}
    \mid \zeta>t_n \big]
    =\frac{1}{\Prob{\zeta>t_n}}
    \E \Big[\one_{\zeta>t_n}
    \big[\modulus(\eta)^{\alpha+\delta} \one_{\zeta>0} \mid\zeta\big]
    \Big] \leq C .
  \end{displaymath}
  Therefore, the sequence $(Y_n)_{n\geq1}$ is uniformly integrable.
  The continuous mapping theorem and \eqref{eq:dtoWU} yield that
  $Y_n\dto Y$. Hence, the uniform integrability of
  the random variables $(Y_n)$ implies that 
  $\E Y_n \to \E Y<\infty$, and thus
  \begin{displaymath}
    \E Y=c^{-1}\int \modulus(y)^\alpha\sigma(\bdiff y)<\infty.
  \end{displaymath}
  If $h$ is bounded by a constant $M$ and supported on
  $\{\modulus>a\}$, then
  \begin{align*}
    \lim_{t\to\infty} |I_1(t)|
    &\leq M \int_{\YY} \min\{(\modulus(y)/a)^\alpha,\eps^{-\alpha}\}
    \sigma(\bdiff y)\\
    &\leq Ma^{-\alpha} \int_{\YY} \modulus(y)^\alpha\,\sigma(\bdiff y)
    =Ma^{-\alpha} c \E Y<\infty.
  \end{align*}
  Since $\E Y<\infty$, we let $\eps \to 0$ in the
  expression for $I_1(t)$ and  
  use the bound on $I_2(t)$ to obtain
  \begin{displaymath}
    \lim_{t\to\infty} g(t) \E  h \big(T_{t^{-1}} T_\zeta\eta\big)
    =\int_0^\infty \int_{\YY} h(T_{s}y)\,\sigma(\bdiff y)\,
    \theta_\alpha(\bdiff s)\,.
  \end{displaymath}
  It remains to recognise that the limit coincides with the
  integral with respect to the pushforward $\mu$ of
  $\sigma\otimes\theta_\alpha$ under the map
  $(y,s)\mapsto T_sy$. This pushforward is non-trivial on $\sY$,
  since 
  \begin{displaymath}
    \mu(\{y:\modulus(y)>\eps\})
    =\int \one_{\{(y,s): \modulus(T_sy)>\eps\}}\sigma(\bdiff
    y)\theta_\alpha(\bdiff s)
    =\eps^{-\alpha} \int_\YY \modulus(y)^\alpha \sigma(\bdiff y)>0
  \end{displaymath}
  for all $\eps>0$ in view of the assumption
  $\sigma(\{\modulus>0\})>0$.
  The same calculation and the already established finiteness of
  $\int \modulus(y)^\alpha\sigma(\bdiff y)$ show that this
  pushforward measure is finite on all sets of the form
  $\{\modulus>\eps\}$, and hence on the ideal $\sY$.
\end{proof}

\begin{example}[Scaling with a regularly varying factor]
  \index{scaling!by a regularly varying factor}
  Let $\YY$ be a perfectly normal Lindel\"of topological space
  equipped with a continuous scaling $T_t$ and the ideal $\sY$
  generated by a continuous finite modulus $\modulus$. Let
  $\zeta\in\RV(\R_+,\mydot,\sR_0,g,c\theta_\alpha)$ be a positive random
  variable independent of a random element $\eta$ in $\YY$ with
  distribution $P_\eta$, and assume that
  $\Prob{\modulus(\eta)>0}>0$. By
  Lemma~\ref{lem:BreimanNew}, if
  $\E [\modulus(\eta)^{\alpha+\delta}]<\infty$ for some $\delta>0$, then
  $T_\zeta\eta\in\RV(\YY,T,\sY,g,\mu)$, where $\mu$ is the pushforward
  of $P_\eta\otimes c\theta_\alpha$ under the map $(y,t)\mapsto
  T_ty$.   
\end{example}

The following example gives the multivariate Breiman's lemma, which is
proved in 
\cite{bas:dav02}; see also \cite[Section~C.3]{bdm}.

\begin{example}[Breiman's lemma for random matrices]
  \index{random matrix}
  Assume that $V$ is a $m\times d$ random matrix, which is independent
  of a regularly varying $d$-dimensional random vector $\xi$. Assume that
  $\xi\in\RV(\R^d,\mydot,\sR_0^d,g,\mu')$. Denote by $\|V\|$ the
  operator norm of the matrix $V$ induced by the Euclidean norm. It
  was shown in \cite{bas:dav02} that if
  \begin{displaymath}
    \E\|V\|^{\alpha+\delta} < \infty
  \end{displaymath}
  for some $\delta>0$, then $V\xi$ is a regularly varying random
  vector. This fact can easily be derived from
  Lemma~\ref{lem:BreimanNew}. If $\|\xi\|>0$, let
  $\eta=\|\xi\|^{-1}(V\xi)$, and otherwise let $\eta=0$, so that
  $\eta$ is a random vector in $\R^m$. Note that $\zeta\eta=V\xi$,
  where $\zeta=\|\xi\|\in\RV(\R_+,\mydot,\sR_0,g,c\theta_\alpha)$.

  Consider the space $\YY=\R^m$ with the ideal $\sY=\sR_0^m$ generated
  by the modulus given by the Euclidean norm.
  Since $\|\eta\|\leq \|V\|$ and $V$ is
  independent of $\xi$, condition \eqref{eq:BreimanAss2} is
  satisfied. Indeed,
  \[
    \E[\|\eta\|^{\alpha+\delta}\mid \zeta]
    \leq \E\|V\|^{\alpha+\delta}<\infty .
  \]
  To check \eqref{eq:eta-zeta}, consider a bounded
  continuous function $h:\R^m\times\R_+\to\R$ such that $h(y,s)=0$
  for all $s\leq a$ for some $a>0$, in particular, $h(0,0)=0$.
  Denote
  \begin{displaymath}
    f(u)=\E h(\|u\|^{-1}Vu,\|u\|),\quad u\in\R^d\setminus\{0\}
  \end{displaymath}
  and set $f(0)=0$.
  Then $f$ is a continuous bounded function
  and $f(u)=0$ for all $u$ with $\|u\|\leq a$. The regular
  variation of $\xi$ implies that 
  \begin{multline*}
    g(t)\E h(\eta,T_{t^{-1}}\zeta)
    = g(t)\E\big[f(t^{-1}\xi)\one_{\xi\neq 0}\big]
    = g(t)\E f(t^{-1}\xi)\\
    \to \int_{\R^d} f(y)\mu'(\bdiff y) \quad \text{as}\; t\to\infty.
  \end{multline*}
  Using the spectral representation
  $\mu'=\sigma'\otimes\theta_\alpha$ with $\sigma'$ being a finite
  measure on the unit Euclidean sphere $\SS^{d-1}$, we obtain 
  \begin{align*}
    \int_{\R^d} f(y)\mu'(\bdiff y)
    &=\int_{\SS^{d-1}}\int_0^\infty
    \E h(Vu,s)\theta_\alpha(\bdiff s)\sigma'(\bdiff u)\\
    &=\int_{\R^d} \int_0^\infty
    h(w,s)\theta_\alpha(\bdiff s)\sigma(\bdiff w),
  \end{align*}
  where
  \begin{displaymath}
    \sigma(A)=\int_{\SS^{d-1}} \Prob{Vu\in A}\sigma'(\bdiff u), \quad
    A\in\sB(\R^m). 
  \end{displaymath}
  Therefore, the pair $(\eta,\zeta)$ satisfies \eqref{eq:eta-zeta} with the
  tail measure $\sigma\otimes\theta_\alpha$.

  By Lemma~\ref{lem:BreimanNew}, if
  \begin{displaymath}
    \sigma(\R^m\setminus\{0\})=\int \Prob{Vu\neq 0}\sigma'(\bdiff u)>0,
  \end{displaymath}
  then $V\xi$ is regularly varying with the normalising function $g$
  and the tail measure
  \begin{displaymath}
    \mu^*(A)=\int \Prob{Vu\in A} \mu'(\bdiff u).
  \end{displaymath}

  For instance, assume that $\xi=(\xi_1,\dots,\xi_d)$ has i.i.d.\
  Pareto($\alpha$) components. If $g(t)=t^\alpha$,
  then $\xi$ is regularly varying with the tail measure
  \begin{displaymath}
    \mu^{(1)}(\bdiff(x_1,\dots,x_d))
    =\sum_{i=1}^d \theta_\alpha(\bdiff x_i)\prod_{j\neq i}\delta_0(\bdiff x_j).
  \end{displaymath}
  Then
  \begin{displaymath}
    \mu^*(A)=\int \sum_{i=1}^d \Prob{sV_i\in A} \theta_\alpha(\bdiff s),
  \end{displaymath}
  where $V_1,\dots,V_d$ are columns of $V$. This can be easily
  generalised to the hidden regular variation setting by scaling with
  $g(t)^k$, which equals $t^{k\alpha}$ in the
  Pareto$(\alpha)$ case. Then $\mu^*$ arises by transforming the measure
  $\mu^{(k)}$ defined in Theorem~\ref{thr:power-space}.  Further
  examples are easy to construct for a linear operator $V$ acting on a
  regularly varying Banach space-valued $\xi$.
\end{example}

Dependency naturally emerges in the pair $(\eta,\zeta)$ when $\zeta$
serves as the temporal index for an independent stochastic process
$(\eta_t)$. This is the setting of
Theorem~\ref{thr:pair}. Combining it with Lemma~\ref{lem:BreimanNew}
yields the following result.

\begin{corollary}
  \label{cor:Breiman-substitution}
  Let $\YY$ be a perfectly normal Lindel\"of topological space with a
  continuous scaling $T_t$ and the ideal $\sY$ generated by a
  continuous finite modulus $\modulus$.  Assume that
  $\zeta\in\RV(\R_+,\mydot,\sR_0,g,c\theta_\alpha)$ and $\zeta$ is
  strictly positive, and let $(\eta_t)_{t\geq 0}$ be a jointly
  measurable stochastic
  process taking values in $\YY$, independent of $\zeta$. Furthermore, suppose that
  \begin{equation}
    \label{eq:hxW-real}
    \eta_t\dto  W \quad  \text{as}\; t\to\infty
  \end{equation}
  for some random element $W$ in $\YY$. Assume that the distribution
  $\sigma$ of $W$ is a tau-additive probability measure satisfying
  $\sigma(\{\modulus>0\})>0$, and that
  \begin{equation}
    \label{eq:sup-eta-uniform}
    \sup_{t>0} \E [\modulus(\eta_t)^{\alpha+\delta}]<\infty
  \end{equation}
  for some $\delta>0$. Then
  $T_\zeta\eta_\zeta\in\RV(\YY,T,\sY,g,\mu)$,
  where $\mu$ is the pushforward of the tau-additive extension of the
  product $\sigma\otimes 
  c\theta_\alpha$ under the map $(y,t)\mapsto T_ty$. 
\end{corollary}
\begin{proof}
  By Theorem~\ref{thr:pair}, applied with $\XX=\R_+$ and
  $\xi=\zeta$, the pair $(\eta_\zeta,\zeta)$ is regularly varying on
  $\YY\times\R_+$ with respect to the scaling acting on the second
  component, with normalising function $g$ and tail measure
  $\sigma\otimes c\theta_\alpha$.

  It remains to check the moment condition in Lemma~\ref{lem:BreimanNew}.
  Since the process $(\eta_t)_{t\geq0}$ is independent of $\zeta$,
  \begin{displaymath}
    \E\big[\modulus(\eta_\zeta)^{\alpha+\delta}\mid \zeta\big]
    =
    \E\big[\modulus(\eta_t)^{\alpha+\delta}\big]_{t=\zeta}
    \leq \sup_{t>0}\E\modulus(\eta_t)^{\alpha+\delta}<\infty .
  \end{displaymath}
  Lemma~\ref{lem:BreimanNew} now yields the claimed regular variation
  of $T_\zeta\eta_\zeta$ and identifies the tail measure as the
  pushforward of $\sigma\otimes c\theta_\alpha$ under
  $(y,t)\mapsto T_ty$.
\end{proof}

Various existing generalisations of Breiman's lemma in the
literature are more restrictive regarding the choice of state space
and typically demand almost sure convergence of $\eta_x$ to $W$. For
instance, a recent generalisation for $\YY$ being a Banach space
is given in \cite[Theorem~3.3]{janssen23}, 
motivated by applications to preferential attachment models. An
application to branching processes can be found in \cite{MR4168330};
this also motivates the second part of Example~\ref{ex:r-walks}.

\begin{example}
  \label{ex:r-walks}
  Assume that $Y\in\RV(\R_+,\mydot,\sR_0,g,c\theta_\alpha)$, and let
  $S_n=\xi_1+\cdots+\xi_n$, $n\geq1$, be a random walk with i.i.d.\
  integrable steps $(\xi_i)_{i\geq 1}$. Assume that $Y$ and the random
  walk are independent. The aim is to establish the regular variation of
  $Y^\gamma S_{\lfloor Y \rfloor}$ for $\gamma\geq 0$. We consider
  two special cases below.

  (a) Assume that $\xi_1\geq 0$ with probability one and that $\xi_1$ is
  not a.s.\ equal to zero, so 
  $a=\E\xi_1>0$, and assume that $\E\xi_1^\beta<\infty$ for some
  $\beta>\max\{1,\alpha\}$. Let $\zeta = Y^{\gamma+1}$ and note that
  $\zeta\in\RV(\R_+,\mydot,\sR_0,g_1,c\theta_{\alpha/(\gamma+1)})$,
  where $g_1(t)=g(t^{1/(\gamma+1)})$. 

  Denote $\eta'_t = t^{-1} S_{\lfloor t \rfloor}$
  and $\eta_t = \eta'_{t^{1/(\gamma+1)}}$. By the strong law of large
  numbers, $\eta_t \to a$ with probability one as $t\to\infty$. By
  Minkowski's inequality and the fact that $\eta_t=0$ for $t\in[0,1)$, 
  \begin{displaymath}
    \sup_{t\in(0,\infty)}  (\E \eta_t^\beta)^{1/\beta}
    = \sup_{t\in(1,\infty)}
    \Big(\E  \big(t^{-1} S_{\lfloor t\rfloor}\big)^\beta\Big)^{1/\beta}
    \leq \sup_{t\in(1,\infty)}
    \frac{\lfloor t  \rfloor}{t}\,
    \big(\E \xi_1^\beta\big)^{1/\beta}< \infty.
  \end{displaymath}
  Since $\beta>\alpha/(\gamma+1)$, the moment condition
  \eqref{eq:sup-eta-uniform} 
  of Corollary~\ref{cor:Breiman-substitution} is verified. Consequently,
  $\zeta \eta_\zeta = Y^\gamma S_{\lfloor Y \rfloor}$ is regularly
  varying with the normalising function $g_1$ and the tail index
  $\alpha/(\gamma+1)$.

  (b) Assume that $\E\xi_1=0$ and $0<\E \xi_1^2<\infty$.
  Consider any $\gamma\geq 0$ such that $\alpha < 2\gamma + 1$. Let
  $ \eta'_t = t^{-1/2} S_{\lfloor t \rfloor}$ and let
  $\eta_t = \eta'_{t^{1/(\gamma+1/2)}}$. Observe that $\eta_t$
  converges in distribution as $t\to\infty$ to a centred normal
  variable $W$ with variance $\E\xi_1^2$. Then, 
  \begin{align*}
    \sup_{t\in(0,\infty)} \E \eta_t^2
    = \sup_{t\in(1,\infty)} \E  \big( t ^{-1/2}
    S_{\lfloor t  \rfloor}\big)^2    \leq \E \xi_1^2<\infty.
  \end{align*}
  Denote $\zeta = Y^{\gamma+1/2}$, so that $\zeta$ is regularly
  varying with normalising function
  $g_2(t)=g(t^{1/(\gamma+1/2)})$ and tail index
  $\alpha/(\gamma+1/2)$.
  Since $\alpha/(\gamma+1/2)< 2$, there exists $\delta>0$ such
  that $\alpha/(\gamma+1/2)+\delta=2$, so that the moment
  condition \eqref{eq:sup-eta-uniform} is satisfied. By
  Corollary~\ref{cor:Breiman-substitution}, the random variable
  $Y^\gamma S_{\lfloor Y \rfloor}$ is regularly varying with 
  normalising function $g_2$ and tail index
  $\alpha/(\gamma+\frac12)$, and tail measure given by the
  pushforward of $c\sigma\otimes \theta_{\alpha/(\gamma+ \frac 12)}$
  under the map $(y,t)\mapsto T_ty$. In particular, the tail measure
  is symmetric since the distribution $\sigma$ of $W$ is symmetric.
\end{example}

\section{Regular variation in metric spaces and Banach spaces}
\label{sec:regul-vari-polish}

\paragraph{Metric spaces}

In metric spaces, it is possible to establish further criteria 
ensuring the regular variation property. Recall that in this setting we
work with the conventional case of Borel measures.

We need the following lemma, which makes it possible to reduce
the task of checking the vague convergence in \eqref{eq:vague-f} to Lipschitz
functions in $\fC(\XX,\sS)$. 

\begin{lemma}
  \label{lemma:Lipschitz}
  Let $(\XX,\dmet)$ be a metric space with an ideal $\sS$. Assume that,
  for each set $A\in\sS$, there exists $r>0$ such that the ideal $\sS$
  contains its $r$-envelope 
  \index{envelope of a set}
  \begin{equation}
    \label{eq:neighbourhood-r}
    A^r=\{x\colons \dmet(x,A)\leq r\}.
  \end{equation}
  Then $\mu_\gamma\vto\mu$ if and only if
  \eqref{eq:vague-f} holds for all 1-Lipschitz functions
  $f\in\fC(\XX,\sS)$.
\end{lemma}
\begin{proof}
  Necessity is immediate.  We now prove sufficiency
  using the characterisation of vague convergence from
  Theorem~\ref{lemma:vVSw}(iii). Fix an $\eps>0$. The condition
  imposed on $\sS$ implies that $\sS$ is topologically consistent.
  Let $F$ be a closed set in $\sS$. By assumption, $F^r\in\sS$
  for some $r>0$. Choose $r$ such that $\mu(F^r)-\mu(F)<\eps$ (it
  exists by outer regularity of $\mu$), and 
  define the Lipschitz function
  \begin{displaymath}
    f_r(x)=\max\big(1-\dmet(x,F)/r,0\big).
  \end{displaymath}
  Although $f_r$ is $1/r$-Lipschitz, the assumed convergence for
  1-Lipschitz functions applies to $r f_r$, and hence also to $f_r$ by
  rescaling.  Then
  \begin{displaymath}
    \limsup_{\gamma} \mu_\gamma(F)
    \leq \limsup_{\gamma} \int f_r(x)\diff \mu_\gamma
    =\int f_r \diff \mu\leq \mu(F^r)<\mu(F)+\eps.
  \end{displaymath}
  Let $G\in\sS$ be open. Then $G^{-r}=\{x\colons
  \dmet(x,G^c)>r\}$  increases 
  to $G$ as $r\downarrow 0$. Choose an $r$ such that
  $\mu(G)-\mu(G^{-r})<\eps$ and define the Lipschitz function
  $f_{-r}(x)=\max(1-\dmet(x,G^{-r})/r,0)$. Then
  \begin{displaymath}
    \liminf_{\gamma} \mu_\gamma(G)
    \geq \liminf_{\gamma} \int f_{-r}(x) \diff \mu_\gamma
    =\int f_{-r} \diff \mu\geq \mu(G^{-r})>\mu(G)-\eps. \qedhere
  \end{displaymath}
\end{proof}

If the ideal $\sS$ contains an envelope
of every set belonging to it (as assumed in Lemma~\ref{lemma:Lipschitz}), then 
the ideal is topologically consistent; this condition is
satisfied for all metric exclusion ideals. By the property of ideals, if
an ideal contains the $r$-envelope $A^r$ for some $r>0$, it also
contains the envelopes $A^s$ for all $s<r$. The following result
derives the regular variation of a random element from the regular
variation property of a net of random elements that approximate it.

\begin{theorem}
  \label{thr:m-space}
  Let $(\XX,\dmet)$ be a metric space equipped with a scaling and a
  scaling-consistent ideal $\sS$ such that, for each set $A$,
  the ideal $\sS$ contains its $r$-envelope $A^r$ for some $r>0$. Let
  $(\eta_\gamma)_{\gamma\in\Gamma}$ be a net of
  random elements in $\XX$. Then $\xi\in\RV(\XX,T,\sS,g,\mu)$ provided 
  the following conditions hold.
  \index{regular variation!on metric spaces}
  \begin{enumerate}[(i)]
  \item $\eta_\gamma\in\RV(\XX,T,\sS,g,\mu_\gamma)$ for some
    $\mu_\gamma\in\Mb$ and all $\gamma\in\Gamma$.
  \item $\mu_\gamma\vto\mu$.
  \item For every closed $A\in\sS$,
    \begin{equation}
      \label{eq:3g}
      \limsup_{t\to\infty} g(t)\Prob{T_{t^{-1}}\xi\in A} <\infty.
    \end{equation}
  \item For all $\delta>0$,
    \begin{equation}
      \label{eq:38}
      \lim_\gamma \limsup_{t\to\infty} g(t)
      \Prob{\dmet(T_{t^{-1}}\xi,T_{t^{-1}}\eta_\gamma)>\delta}=0.
    \end{equation}
  \end{enumerate}
\end{theorem}
\begin{proof}
  By Lemma~\ref{lemma:Lipschitz}, we need to confirm that
  \begin{equation}
    \label{eq:40}
    g(t) \E f(T_{t^{-1}}\xi)\to \int f\diff \mu\quad \text{as}\; t\to\infty
  \end{equation}
  for each non-negative 1-Lipschitz function $f\in\fC(\XX,\sS)$, since
  any $f$ is the difference of its positive and
  negative parts. Let $A=\supp f$. Then $A$ is closed, belongs to $\sS$, and
  $|f|\leq c$ for some $c<\infty$ (note that the condition on $\sS$
  implies its topological consistency). Then, for all $\gamma$,
  \begin{multline*}
    \Big|g(t)\E f(T_{t^{-1}}\xi)-\int f\diff \mu\Big|
    \leq \Big|g(t)\E f(T_{t^{-1}}\xi)-g(t)\E f(T_{t^{-1}}\eta_\gamma)\Big|\\
    +\Big|g(t)\E f(T_{t^{-1}}\eta_\gamma)-\int f\diff \mu_\gamma\Big|
    +\Big|\int f\diff \mu_\gamma -\int f\diff \mu\Big|.
  \end{multline*}
  The second term on the right-hand side converges to zero as
  $t\to\infty$ for all $\gamma$ by (i). The last term converges to
  zero, since $\mu_\gamma\vto\mu$.  Fix an arbitrary $\delta>0$ and bound the
  first term by the sum of
  \begin{align*}
    I_1(t)&=g(t)\E\bigg[\Big|f(T_{t^{-1}}\xi)-
    f(T_{t^{-1}}\eta_\gamma)\Big|
    \one_{\dmet(T_{t^{-1}}\xi,T_{t^{-1}}\eta_\gamma)>\delta}\bigg]\\
    &\leq 2c g(t)\Prob{\dmet(T_{t^{-1}}\xi,T_{t^{-1}}\eta_\gamma)>\delta}
  \end{align*}
  and
  \begin{align*}
    I_2(t)
    &=g(t)\E\bigg[\Big|f(T_{t^{-1}}\xi)-f(T_{t^{-1}}\eta_\gamma)\Big|
    \one_{\dmet(T_{t^{-1}}\xi,T_{t^{-1}}\eta_\gamma)\leq\delta}\bigg]\\
    &=g(t)\E\bigg[\Big|f(T_{t^{-1}}\xi)-f(T_{t^{-1}}\eta_\gamma)\Big|
      \one_{\dmet(T_{t^{-1}}\xi,T_{t^{-1}}\eta_\gamma)\leq\delta}
      \one_{\{T_{t^{-1}}\xi\in A\}\cup\{T_{t^{-1}}\eta_\gamma\in A\}} \bigg]\\
    &\leq \delta g(t)\Prob{T_{t^{-1}}\xi\in A}
    +\delta g(t)\Prob{T_{t^{-1}}\eta_\gamma\in A},
  \end{align*}
  where we use that $f$ is 1-Lipschitz and vanishes outside $A$. 
  By (iv),
  \begin{displaymath}
    \lim_{\gamma}\limsup_{t\to\infty} I_1(t)\leq 2c \limsup_\gamma
    \limsup_{t\to\infty}
    g(t)\Prob{\dmet(T_{t^{-1}}\xi,T_{t^{-1}}\eta_\gamma)>\delta}=0.
  \end{displaymath}
  Furthermore,
  \begin{displaymath}
    \limsup_{t\to\infty} I_2(t)\leq \delta\big(a+\mu_\gamma(A)\big),    
  \end{displaymath}
  where
  \begin{displaymath}
    a=\limsup_{t\to\infty} g(t)\Prob{T_{t^{-1}}\xi\in A}
  \end{displaymath}
  is finite by \eqref{eq:3g}.
  We use (i) together with the closedness
  of $A$ to obtain
  \begin{align*}
    \limsup_{t\to\infty}\Big|g(t)\E f(T_{t^{-1}}\xi)-\int
    f\diff \mu\Big|\leq \limsup_\gamma \delta \big(a+\mu_\gamma(A)\big)
    \leq \delta(a+\mu(A)).
  \end{align*}
  Since $a$ and $\mu(A)$ are finite and depend only on $f$,
  letting $\delta\downarrow 0$ finishes the proof.
\end{proof}

\begin{remark}
  \label{rem:metric-exclusion}
  If $\sS_\cone$ is a metric exclusion ideal obtained by excluding a
  cone $\cone$, and the metric $\dmet$ on $\XX$ is homogeneous, that
  is, $\dmet(T_tx,T_ty)=t\dmet(x,y)$ for all $x,y$, then condition
  (iii) of Theorem~\ref{thr:m-space} can be replaced by
  \begin{displaymath}
    \limsup_{t\to\infty}
    g(t)\Prob{\dmet(\xi,\cone)>tu}<\infty
  \end{displaymath}
  for all $u>0$. This condition holds in particular if $\dmet(\xi,\cone)$ is
  a regularly varying 
  random variable with normalising function $g$. 
\end{remark}

The approximating random elements $\eta_\gamma$ are often constructed
by applying a net of functions $(\psi_\gamma)_{\gamma\in\Gamma}$ to
$\xi$. A net of functions $(\psi_\gamma)$ is said to be \emph{uniformly
bornologically consistent} if, for all $A\in\sS$, there exist
$\gamma_0\in\Gamma$ and $B\in\sS$ such that
$\psi_\gamma^{-1}(A)\subset B$ for all $\gamma\geq\gamma_0$. In
particular, this implies that each $\psi_\gamma$ with $\gamma\geq\gamma_0$
is bornologically consistent. 

\begin{corollary}
  \label{cor:m-space}
  Let $(\XX,\dmet)$ be a metric space equipped with a scaling and a
  scaling-consistent ideal $\sS$ such that, for each set $A$, the
  ideal $\sS$ contains its $r$-envelope $A^r$ for some $r>0$. Let
  $(\psi_\gamma)_{\gamma\in\Gamma}$ be a net of Borel measurable, uniformly
  bornologically consistent functions $\psi_\gamma:\XX\to\XX$ such
  that $\psi_\gamma(x)\to x$ for all $x$ in the union of all sets in
  $\sS$. Then $\xi\in\RV(\XX,T,\sS,g,\mu)$ if condition (iii) of
  Theorem~\ref{thr:m-space} holds and the random elements
  $\eta_\gamma=\psi_\gamma(\xi)$,
  $\gamma\in\Gamma$, satisfy conditions (i) and (iv) therein with
  $\mu_\gamma=\psi_\gamma\mu$ (the pushforward of $\mu$ under $\psi_\gamma$).
\end{corollary}
\begin{proof}
  Without loss of generality, assume that the entire family
  $\psi_\gamma$, $\gamma\in\Gamma$, is uniformly bornologically
  consistent. Let $\mu_\gamma=\psi_\gamma\mu$ be the pushforward
  measure of
  $\mu$ under $\psi_\gamma$. Note that $\mu_\gamma\in\Mb$ since
  $\psi_\gamma$ is bornologically consistent. Then the listed
  conditions correspond to those of Theorem~\ref{thr:m-space} apart
  from (ii), where we need to check that $\mu_\gamma\vto\mu$. Let
  $f\in\fC(\XX,\sS)$ 
  be a continuous bounded function supported on $A\in\sS$. Then
  \begin{align*}
    \Big|\int f\diff (\psi_\gamma\mu) -\int f\diff \mu\Big|
    &\leq \int \big|f(\psi_\gamma(x))-f(x)\big|\mu(\bdiff x)\\
    &= \int_{A\cup \psi_\gamma^{-1}(A)}
    \big|f(\psi_\gamma(x))-f(x)\big|\mu(\bdiff x) \to 0
  \end{align*}
  by the dominated convergence theorem, since $f$ is continuous
  and bounded, $\psi_\gamma(x)\to x$,
  $A\cup \psi_\gamma^{-1}(A) \subset B\in\sS$ for all $\gamma$ and
  $\mu(B)<\infty$.
\end{proof}

The assumption of uniform bornological consistency imposed on the
approximating maps $\psi_\gamma$ in Corollary~\ref{cor:m-space} can be
replaced by the requirement that
$\psi_\gamma\mu\in\Mb$ and $\psi_\gamma\mu\vto\mu$ for the relevant
tail measure $\mu$. Indeed, this is precisely condition (ii) of
Theorem~\ref{thr:m-space}.

\paragraph{Linear normed spaces}
Assume that $\XX$ is a linear normed space.  If we require that the
approximating maps $\psi_\gamma$ are continuous morphisms, it is
possible to establish necessary and sufficient conditions in
the spirit of Corollary~\ref{cor:m-space}.

\begin{corollary}
  \label{cor:banach-zero}
  \index{regular variation!on linear normed spaces}
  Let $\XX$ be a linear normed space with linear scaling and the ideal
  $\sS_0$ obtained by excluding the zero element. Let
  $(\psi_\gamma)_{\gamma\in\Gamma}$ be a net consisting of
  continuous morphisms
  $\psi_\gamma:\XX\setminus\{0\}\to\XX\setminus\{0\}$ such that
  $\psi_\gamma(x)\to x$ and $\|\psi_\gamma(x)\|\leq b\|x\|$ for all
  $x\neq 0$, $\gamma\in\Gamma$, and a constant $b$ 
  independent of $x$ and $\gamma$. Then
  $\xi\in\RV(\XX,\mydot,\sS_0,g,\mu)$ if and only if the random
  elements 
  $\eta_\gamma=\psi_\gamma(\xi)$, $\gamma\in\Gamma$, satisfy
  Condition~(i) of Theorem~\ref{thr:m-space} with
  $\mu_\gamma=\psi_\gamma\mu$ 
  for some $\mu\in\Mb[\sS_0]$ and the following two conditions hold.
  \begin{enumerate}
  \item[(iii\,$'$)] $\|\xi\|\in\RV(\R_+,\mydot,\sR_0,g,c\theta_\alpha)$.
  \item[(iv\,$'$)]  For all $\delta>0$,
    \begin{equation}
      \label{eq:38a}
      \lim_\gamma \limsup_{t\to\infty} g(t)
      \P\big\{\|\xi-\psi_\gamma(\xi)\|>t\delta\big\}=0.
    \end{equation}
  \end{enumerate}
\end{corollary}
\begin{proof}
  Assume that $\xi$ is regularly varying. Then condition (i) of
  Theorem~\ref{thr:m-space} and condition (iii\,$'$) hold, since the
  maps $\psi_\gamma$
  and the norm are continuous, bornologically consistent morphisms, see
  Lemma~\ref{lemma:map-zero}.  The map $x\mapsto\|x-\psi_\gamma(x)\|$
  is a continuous morphism from $\XX\setminus\{0\}$ to $\R_+$,
  which is bornologically consistent from 
  $\sS_0$ to $\sR_0$, since
  \begin{align*}
    \{x\colons \|x-\psi_\gamma(x)\|>\eps\}
    &\subset \{x\colons \|x\|>\eps/2\}\cup
      \{x\colons\|\psi_\gamma(x)\|>\eps/2\}\\
    &\subset \{x\colons \|x\|>\min(1,b^{-1})\eps/2\}.
  \end{align*}
  The regular variation of $\|\xi\|$ now yields (iv\,$'$).
  Sufficiency follows from Corollary~\ref{cor:m-space}. In order to
  confirm that the net $(\psi_\gamma)$ is uniformly bornologically
  consistent, note that $x\in\psi_\gamma^{-1}(\{y\colons \|y\|\geq\eps\})$
  implies
  \begin{displaymath}
    b\|x\|\geq\|\psi_\gamma(x)\|\geq \eps. 
  \end{displaymath}
  Hence, $\|x\|\geq \eps/b$, so
  \begin{displaymath}
    \psi_\gamma^{-1}(\{y\colons \|y\|\geq\eps\})
    \subset \{x\colons \|x\|\geq\eps/b\},
  \end{displaymath}
  uniformly in $\gamma$.
\end{proof}

Consider the family of maps $\psi_L:\XX\to L$ parametrised by
finite-dimensional linear subspaces $L$ of $\XX$ such that
$\|x-\psi_L(x)\|=\inf_{y\in L}\|x-y\|$ for all $x\in\XX$. Such
nearest-point maps exist and are 
continuous if the space $\XX$ is strictly convex, that is, if
$\|sx+(1-s)y\|<1$ for all $s\in(0,1)$ whenever $\|x\|=\|y\|=1$ and
$x\neq y$.  Since
$\|\psi_L(x)\|\leq \|x\|$, Corollary~\ref{cor:banach-zero} is
applicable. If $g$ is a power function, then these conditions
coincide with those imposed in Theorem~4.3 of \cite{MR0521695} to ensure
that a probability law lies in the domain of attraction of an
$\alpha$-stable law in a separable Banach space of Rademacher type
$p\in [1,2]$ (see \cite{MR423451}) with $\alpha<p$.

If the carrier space $\XX$ is separable, it is possible to replace the net of
approximating functions in Theorem~\ref{thr:m-space} and
Corollary~\ref{cor:banach-zero} with a sequence $(\psi_m)_{m\in\NN}$.
In particular, it is easy to construct the approximating maps if $\XX$
is a (necessarily separable) Banach space with a
\emph{Schauder basis}
\index{Schauder basis} \index{Banach space}
$(e_k)_{k\in\NN}$; see \cite[Chapter~4]{MR2766381}. Without loss of
generality, assume that the Schauder basis is normalised, that is,
$\|e_k\|=1$ for all $k\in\NN$. Let
$(e_k^*)_{k\in\NN}$ be the biorthogonal system of linear functionals
on $\XX$, so that
\begin{displaymath}
  x=\sum_{k=1}^\infty \langle e_k^*,x\rangle e_k,\quad x\in\XX.
\end{displaymath}
For $m\in\NN$, let
\begin{equation}
  \label{eq:39}
  \psi_m(x)=\sum_{k=1}^m \langle e_k^*,x\rangle e_k,\quad x\in\XX.
\end{equation}
Then $\psi_m(x)\to x$ as $m\to\infty$, and each $\psi_m$ is a continuous
bornologically consistent morphism, see
Lemma~\ref{lemma:map-zero}. However, uniform bornological
consistency is not automatically guaranteed, since 
$\|\psi_m(x)\|$ is not necessarily bounded by $b\|x\|$ for some fixed
$b>0$.

\begin{corollary}
  \label{cor:banach-schauder}
  Let $\xi$ be a random element in a separable Banach space $\XX$
  equipped with a Schauder basis, linear scaling and the ideal
  $\sS_0$ obtained by excluding the zero element. Assume that $\psi_m$
  is defined by \eqref{eq:39} for each $m\in\NN$.
  Then $\xi\in\RV(\XX,\mydot,\sS_0,g,\mu)$
  if conditions (iii\,$'$) and (iv\,$'$) above, with
  $\Gamma=\NN$, and the following conditions are satisfied.
  \begin{enumerate}
  \item[(i\,$'$)] For all $m$,
    $\psi_m(\xi)\in\RV(\XX,\mydot,\sS_0,g,\psi_m\mu)$.
  \item[(ii\,$'$)] $\psi_m\mu\vto[\sS_0]\mu$ as $m\to\infty$.
  \end{enumerate}
\end{corollary}

The conditions simplify substantially if $\XX$ is a separable
\emph{Hilbert space} \index{Hilbert space}
with an orthonormal basis $(e_k)_{k\in\NN}$. Regular variation in
Hilbert spaces has been studied in \cite{MR4745556}. The coordinate
map from $x=\sum x_ie_i\in\XX$ to the sequence $(x_1,x_2,\ldots)$ is an
isometric isomorphism between $\XX$ and the Hilbert space $\ell_2$ of
square-summable  
sequences. The continuous mapping theorem immediately
implies the following result.

\begin{proposition}
  Let $\XX$ be a separable Hilbert space with an orthonormal basis
  $(e_k)_{k\in\NN}$ and the ideal $\sS_0$ generated by the norm,
  equivalently obtained by excluding the zero element.  Then a random
  element $\xi$ is regularly varying in $\XX$ if and only if the
  sequence $(\langle e_k,\xi\rangle)_{k\in\NN}$ is regularly varying
  in $\ell_2$ with the ideal generated by the norm.
\end{proposition}

A characterisation of regular variation in $\ell_p$ is given in
Proposition~\ref{thr:rv-ellp} (proved later on); specialised for
$p=2$, it yields the following corollary.

\begin{corollary}
  \label{cor:hilbert}
  Let $\xi$ be a random element in a separable
  Hilbert space $\XX$ with 
  linear scaling and the ideal $\sS_0$ generated by the norm. Then
  $\xi\in\RV(\XX,\mydot,\sS_0,g,\mu)$ if and only if the random
  variables $\eta_k=\langle \xi,e_k\rangle$, $k\in\NN$, satisfy the
  following conditions.
  \begin{enumerate}[(i)]
  \item There exist a normalising function $g$ and a family of measures
    $\mu_m\in\Mb[\sR_0^m]$, $m\in\NN$, such that
    $(\eta_1,\dots,\eta_m)\in\RV(\R^m,\mydot, \sR^m_0,g,\mu_m)$ for
    all $m$ and then necessarily $\mu_m=\psi_m\mu$, where $\psi_m$ is
    defined in \eqref{eq:39}. 
  \item $\|\xi\|\in\RV(\R_+,\mydot,\sR_0,g,c\theta_\alpha)$ and then
    $c\theta_\alpha$ is the pushforward of $\mu$ under the
    map $x\mapsto\|x\|$.
  \item For all $\delta>0$,
    \begin{displaymath}
      \lim_{m\to\infty} \limsup_{t\to\infty} g(t)
      \Prob{\sum_{i=m+1}^\infty \eta_i^2>t^2\delta}=0.
    \end{displaymath}
  \end{enumerate}
\end{corollary}

By Proposition~\ref{prop:polar-decomp}(ii), $\xi$ is regularly varying
in a separable Hilbert space if and only if $\|\xi\|$ is regularly
varying and there exists a non-degenerate finite Borel measure $\sigma$ on
the unit sphere, such that
\begin{equation}
  \label{eq:mod-HS}
  \lim_{t\to\infty} \Prob{\xi/\|\xi\|\in A\mid \|\xi\|>t}
  =\sigma(A)
\end{equation}
for all Borel subsets $A$ of the unit sphere $\SS=\{x\colons \|x\|=1\}$
that are $\sigma$-continuity sets.

The authors of \cite{MR0345155} complement \eqref{eq:mod-HS} with an
additional condition requiring that, for all $m\ge 0$ and $a>0$,
\begin{displaymath}
  \lim_{t\to\infty} \frac{\Prob{\|\xi-\psi_m(\xi)\|>at}}
  {\Prob{\|\xi\|>t}}= \frac{c_m}{c_0}a^{-\alpha},
\end{displaymath}
where $\psi_0(\xi)=0$, $c_m>0$ for all $m$,
and $c_m\to 0$ as $m\to\infty$. This
condition (with $m=0$) implies that $\|\xi\|$ is regularly varying and, combined
with \eqref{eq:mod-HS}, characterises all $\xi$ that belong to the
domain of attraction of a non-degenerate strictly
$\alpha$-stable law in the Hilbert space;
see \cite[Theorem~4.11]{MR0345155}. The condition
$c_m>0$ in the cited
result, which also appears as Theorem~3.1 in  \cite{kok:stoev:xion19},
ensures that the law of $\xi$ is genuinely infinite-dimensional (i.e.,
not supported on any finite-dimensional subspace).

\chapter{Special cases and examples}
\label{sec:examples}

\section{Random variables}
\label{sec:random-variables}

Let $\XX$ be a subset of the real line with the Euclidean
topology. Consider a scaling $T$ and a scaling-consistent ideal
$\sS$ on $\XX$. When maxima are considered below, we assume that the maps $T_t$ are
order-preserving. If the scaling is linear, then $\XX$ must be
a subcone of $\R$, that is, either the entire $\R$, or $\R_+=[0,\infty)$,
$\R_-=(-\infty,0]$, or their open versions $(0,\infty)$ and $(-\infty,0)$.

Let $\xi_1,\xi_2,\ldots$ be independent copies of a random variable
$\xi$, and define
\begin{displaymath}
  M_n=\max(\xi_1,\dots,\xi_n),\quad n\geq 1.
\end{displaymath}
Let $(a_n)$ be such that $g(a_n)\sim n$. By
Proposition~\ref{prop:pp}, $\xi\in\RV(\XX,T,\sS,g,\mu)$ implies that,
for each $t\in\R$ such that $(t,\infty)\in\sS$ and
$\mu(\{t\})=0$,
\begin{multline}
  \label{eq:36}
  \Prob{T_{a_n^{-1}} M_n\leq t}\\
  =\Prob{T_{a_n^{-1}}\{\xi_1,\dots,\xi_n\} \cap (t,\infty)=\emptyset} 
  \to e^{-\mu((t,\infty))}
  \quad \text{as}\; n\to\infty. 
\end{multline}
This means that $T_{a_n^{-1}} M_n$ converges in distribution to a 
random variable with cumulative distribution function
$e^{-\mu((t,\infty))}$, $t\in\R$.

In general, if there exist a distribution function $F$ and sequences
$(a_n)$, $(b_n)$ such that for every continuity point $t$ of $F$
\begin{equation}
  \label{eq:MDA}
  \Prob{a_n^{-1} (M_n -b_n)\leq t}\to F(t)
  \quad \text{as}\; n\to\infty,
\end{equation}
we say that $\xi$ belongs to the maximum domain of attraction (MDA) of
the distribution with the distribution function $F$.  The three
classical maximum domains of attraction 
on the real line are the Fr\'echet, Weibull and Gumbel types; see
Chapter~1 of \cite{res87}.
\index{maximum domain of attraction}

As a consequence of \eqref{eq:36}, the MDA of the Fr\'echet distribution
is closely related to the regular variation property.  
Consider $\XX=\Rpp$ with linear scaling and the ideal $\sR_0$
generated by the modulus $\modulus(x)=x$. Since $\sR_0$ is the
  metric exclusion ideal of $0$ and the union of all sets in 
$\sR_0$ does not contain the origin, we may
equivalently let $\XX=[0,\infty)$. Then
$\xi\in\RV(\R_+,\mydot,\sR_0,g,c\theta_\alpha)$ if and only if
\eqref{eq:36} holds and takes the form
\begin{displaymath}
  \Prob{ a_n^{-1} \max(\xi_1,\dots,\xi_n) \leq t }
  \to e^{- ct^{-\alpha}}\quad
  \text{as}\;n\to\infty
\end{displaymath}
for all $t>0$. In other words, the regular variation of $\xi$ implies
that its distribution belongs to the MDA of the (scaled) Fr\'echet
distribution, which appears on the right-hand side.
Conversely, it is known that membership in the MDA of the
Fr\'echet distribution is equivalent to regular variation of the
right tail of $\xi$; see Proposition~11.1 in
\cite{res87}. Therefore, in this case, one can always take $b_n=0$ in
\eqref{eq:MDA}.
\index{Frechet distribution@Fr\'echet distribution}

If \eqref{eq:MDA} holds with $a_n=1$ for all $n\geq 1$, then $\xi$ is
said to belong to the \emph{translative} MDA (tMDA) of
the distribution $F$. It is known that the tMDA of the Gumbel
distribution, as well as the MDA of the Weibull distribution, can be
directly related to the classical notion of regular variation by
transforming $\xi$ appropriately. Here, we argue that there is
no need for such a transformation; in fact, all three MDAs can be
unified using our general approach to regular variation with
an appropriate scaling and a suitable ideal. 

Let $\psi: \Rpp\to\YY$ be a homeomorphism, where $\YY\subset\R$
is endowed with the induced topology. Following Remark~\ref{ex:implied_born},
define a scaling on $\YY$ as
\begin{equation}
  \label{eq:Tpsi}
  T_t^\psi y=\psi(t\psi^{-1}(y)),\quad y\in\YY.
\end{equation}
The pushforward ideal $\psi\sR_0$ consists of all $V\subset\YY$ such
that $\psi^{-1}(V)\in\sR_0$, that is, $\inf \psi^{-1}(V)>0$. The
following examples are closely related to this construction.

\begin{example}[Weibull domain of attraction]
  Fix some real number $a$, and equip $\YY=(-\infty,a)$ with the
  scaling $T^\YY_t y=a+(y-a)/t$ and the ideal $\sY$ with the base
  $(a-n,a)$, $n\in\NN$. Note that
  $T_{[1,\infty)}(a-n,a)=(a-n,a)$, so that sets $(a-n,a)$ are indeed
  semicones. Let $\eta$ be a random variable with values in $\YY$. Suppose that
  $\eta\in\RV(\YY,T^\YY, \sY, g,\mu)$, that is, $\eta$
  is regularly varying
  on $\sY$ under the defined scaling
  with tail measure $\mu$.  By
  letting the normalising function be
  $g(t) = 1/\Prob{\eta \in (a-1/t,a)}$, it is always possible to
  ensure that the tail measure is given by 
  $\mu((a-u,a))=u^\alpha$ for $u>0$, where $\alpha>0$ is its tail
  index. In the notation of
  Proposition~\ref{prop:pp} for $y>0$, this gives
  \begin{multline}
    \label{eq:weibull}
    \Prob{ a_n(\max(\eta_1,\dots,\eta_n) -a) \leq - y }
    =  \Prob{T_{a_n^{-1}}\max(\eta_1,\dots,\eta_n)\leq a-y }\\
    =  \Prob{T_{a_n^{-1}}\{\eta_1,\dots,\eta_n\}\cap (a-y,a) = \emptyset}
    \to e^{- y^{\alpha}} \quad \text{as}\; n\to\infty
  \end{multline}
  for a sequence $a_n$ such that $g(a_n)\sim n$.  Thus, the
  distribution of $\eta$ (which is regularly varying on $\sY$) belongs
  to the \emph{Weibull} MDA. \index{Weibull distribution}
  One can show that the converse implication holds
  as well, that is, \eqref{eq:weibull} implies the regular variation
  of $\eta$; see \cite{res87}.  Alternatively, one can use
  \begin{displaymath}
    \psi(x)=-x^{-1}+a,\quad x>0,
  \end{displaymath}
  as a homeomorphism between $\Rpp$ and $\YY=(-\infty,a)$ to obtain
  the same conclusion with $T^\YY=T^\psi$, see
  \eqref{eq:Tpsi}.
  Observe that the measure $\mu$ with
  $\mu((a-u,a))=u^\alpha$ is the pushforward of
  $\theta_\alpha$ under $\psi$.
\end{example}

\begin{example}[Gumbel domain of attraction]
  Consider $\YY = \R$ with the scaling $T^\YY_t y=y+ \log t$ and the
  ideal $\sY$ with the base $(n,\infty)$, $n\in\ZZ$.
  If a real-valued
  random variable $\zeta$ is regularly varying on on $\YY$ with  with
  this scaling and ideal, then its tail measure is defined by 
  $\nu_\alpha ((u,\infty)) = ce^{-\alpha u}$ for $u\in\R$ and a
  constant $c$.
  If $(a_n)$ satisfies $g(a_n)\sim n$, then, in the notation of
  Proposition~\ref{prop:pp}, for $u\in\R$, this gives
  \begin{multline*}
    \Prob{ \max(\zeta_1,\dots,\zeta_n) -\log a_n \leq u}
    =  \Prob{T_{a_n^{-1}}\max(\zeta_1,\dots,\zeta_n) \leq u }\\
    =  \Prob{T_{a_n^{-1}}\{\zeta_1,\dots,\zeta_n\}\cap ( u,\infty) = \emptyset}
    \to e^{- ce^{- \alpha u }} \quad \text{as}\; n\to\infty.
  \end{multline*}
  Thus, the distribution of $\zeta$, which is regularly varying in this
  space, belongs to the translative MDA of the (scaled) \emph{Gumbel
    distribution}.  \index{Gumbel distribution}
  As in the examples above, one can conclude that the
  converse implication holds as well; see \cite{res87}.
  This also follows because 
  \begin{displaymath}
    \psi(x)=\log x,\quad x>0,
  \end{displaymath}
  is a homeomorphism between $\Rpp$ and $\YY=\R$, which induces the
  scaling on $\YY$ as in \eqref{eq:Tpsi}, and
  $\nu_\alpha = c \psi \theta_\alpha$.  It should be noted that the full
  MDA of a Gumbel distribution involves a general affine
  normalisation and therefore it cannot be handled by this simple
  idea.
\end{example}

\begin{example}
  The function $\psi(x)=e^x$ is a homeomorphism between $(0,\infty)$
  and $\YY=(1,\infty)$. As in Remark~\ref{ex:implied_born}, it
  transforms the linear scaling on $(0,\infty)$ into the scaling
  $T^\psi_t y = y^t$ on $\YY$ and allows one to characterise random
  variables $\eta$ on $(1,\infty)$ for which there exists a sequence
  $(a_n)$ such that
  \[
    T^\psi_{a_n^{-1}}\max(\eta_1,\dots,\eta_n)
    =\max(\eta_1,\dots,\eta_n)^{1/a_n}
  \]
  has a non-trivial weak limit as $n\to\infty$.
\end{example}

\begin{example}
  Consider the map $\psi(x)=x^\beta$ with $\beta>0$ from $\R_+$ to
  $\R_+$. Then $\psi\sR_0=\sR_0$, and $T_t^\psi y= t^\beta y$. We have
  $\xi\in\RV(\R_+,\mydot,\sR_0,g,c\theta_\alpha)$ if and only if
  $\xi^\beta\in\RV(\R_+,T^\psi,\sR_0,g,c\psi\theta_\alpha)$. In
  particular, $\psi\theta_\alpha=\theta_{\alpha/\beta}$, so that
  $\xi^\beta$ also belongs to the MDA of the Fr\'echet distribution;
  however, it possesses a different tail index.

  If $\beta<0$, for example, $\psi(x)=x^{-1}$, then $\psi$ is a
  continuous morphism between $\Rpp$ with linear scaling and
  the same space with the \emph{inverse linear scaling}, acting as
  \index{inverse linear scaling}
  \begin{equation}
    \label{eq:inv-linear}
    T^\psi_t x=t^{-1}x,\quad t>0,
  \end{equation}
  and the ideal $\psi\sR_0$ with the base
  $(0,n)$, $n\in\NN$. Thus, if
  $\xi\in\RV(\Rpp,\mydot,\sR_0,g,c\theta_\alpha)$, and so belongs to
  the MDA of the Fr\'echet distribution, then $\xi^{-1}$ belongs to
  the minimum domain of attraction of the reciprocal Fr\'echet
  distribution, namely
  \begin{displaymath}
    \Prob{a_n\min(\xi_1^{-1},\dots,\xi_n^{-1})\geq t}
    \to e^{-c t^\alpha} \quad \text{as}\; n\to\infty
  \end{displaymath}
  for all $t>0$, where $g(a_n)\sim n$. 
\end{example}

\begin{example}
  Let $\XX$ be the \emph{Sorgenfrey line}, \index{Sorgenfrey line}
  which is the real line with the topology generated by the base
  $[a,b)$ for $a<b$. This space is used in \cite{janson20} in the
  context of the Skorokhod topology, owing to the fact that continuous
  functions $f:\XX\to\R$ are exactly those that are right-continuous
  in the usual sense.  The space $\XX$ is perfectly normal and Hausdorff;
  it is first countable, separable, and hereditary
  Lindel\"of, but not metrisable. The Borel $\sigma$-algebra on $\XX$
  is the same as the Borel $\sigma$-algebra on $\R$. For each fixed
  $t>0$, the map $x\mapsto tx$ is continuous on the Sorgenfrey line,
  but the joint map $(t,x)\mapsto tx$ is not continuous; as a function
  of $t$, it is only right-continuous in the Sorgenfrey topology.
  Nevertheless, it
  is easy to see that the scaling is Borel measurable and, therefore, Baire
  measurable due to the perfect normality property of the space. Since
  the topology on $\XX$ is finer than the Euclidean one, all random
  variables that are regularly varying on $\XX$ are also regularly
  varying on $\R$ with the corresponding ideals.

  Let $\sX$ be the ideal on $\XX$ with the base $(n^{-1},\infty)$,
  $n\in\NN$. If a Borel measure 
  $\nu\in\RV(\R_+,\mydot,\sR_0,g,c\theta_\alpha)$, then 
  $g(t)\nu(t[a,b))\to c\theta_\alpha([a,b))$ for all $0<a<b$. Note that
  the measure $\theta_\alpha$ is tau-additive on the Sorgenfrey line,
  since the Sorgenfrey line is hereditarily Lindel\"of and its Baire and
  Borel $\sigma$-algebras coincide.
  Thus, 
  \eqref{eq:24g} holds on the intervals $[a,b)$, $0<a<b$, which form a
  convergence-determining family for the vague topology associated with
  the ideal $\sX$, which
  implies that
  $\nu\in\RV(\XX,\mydot,\sX,g,c\theta_\alpha)$; see
  Theorem~\ref{thr:basis} and \cite[Theorem~4.3.10]{bogachev18}.
\end{example}

\section{Random vectors}
\label{sec:random-vectors}

\paragraph{Linear scaling}
Let $\XX\subset\R^d$ be a cone in the Euclidean space with linear
scaling $T_tx=tx=(tx_1,\dots,tx_d)$ for
$x=(x_1,\dots,x_d)\in\R^d$. For $k\in\{1,\dots,d\}$, the product ideal
$\sR^d_0(k)$ is the family of all sets $B\subset\R^d$ such that
\begin{displaymath}
  \inf_{x=(x_1,\dots,x_d)\in B} \max_{I\subset \{1,\dots,d\}, \card I=k}
  \min_{i\in I} |x_i|>0,
\end{displaymath}
see Definition~\ref{def:product}, where $\card I$ stands for the
number of elements in $I$.  The metric exclusion ideal
$\sR^d_0$ (obtained by excluding $0$ from $\R^d$ and then restricting
the sets to $\XX$) consists of all subsets of $\XX$ whose closures do not
contain the origin. This coincides with $\sR^d_0(1)$, which is briefly denoted by
$\sR_0^d$, see Example~\ref{example:product}. It is topologically and scaling consistent, and may be generated by various continuous moduli; for instance, one may let
$\modulus(x)=\|x\|_p$ to be any $p$-norm on $\R^d$ with $p\in(0,\infty]$.

\begin{example}[The first quadrant with various moduli]
  \label{ex:moduli-2}
  \label{ex:hidden-RV}
  Let $\XX=\R_+^2$ with linear scaling.  Note that the modulus
  $\modulus_{\max}(x)=\max(x_1,x_2)$ and any other modulus induced by a
  norm on $\R_+^2$ generate the same ideal $\sR_0^2$, whose union set is
  $\R_+^2\setminus\{0\}$. The modulus $\modulus_{\max}$ is a
  special case (for $\beta=0$) of
  \begin{displaymath}
    \modulus^*_\beta(x_1,x_2)=\max\big(x_1^\beta
    x_2^{1-\beta},x_1^{1-\beta}x_2^\beta\big),
    \quad \beta\in[0,1/2].
  \end{displaymath}
  Note that $\{\modulus^*_\beta>0\}=(0,\infty)^2$ for $\beta\neq 0$. The
  function
  \begin{displaymath}
    \modulus_\beta(x_1,x_2)
    =\min\big(x_1^\beta x_2^{1-\beta}, x_1^{1-\beta} x_2^{\beta}\big),
    \quad \beta\in[0,1/2],
  \end{displaymath}
  and its special case $\modulus_{\min}(x_1,x_2)=\min(x_1,x_2)$ for
  $\beta=0$ are also moduli with $\{\modulus_\beta>0\}=(0,\infty)^2$.
  While the union sets for the ideals $\sR_{\modulus_\beta}$ and
  $\sR_{\modulus_\beta^*}$ generated by these moduli are identical, the
  ideals themselves are different; specifically,
  \begin{displaymath}
    \sR^2_0(2)=\sR_{\modulus_{\min}}\subset \sR_{\modulus_\beta}
    \subset \sR_{\modulus_\beta^*}\subset \sR_{\modulus_{\max}}
    =\sR^2_0(1)=\sR^2_0,
  \end{displaymath}
  and the inclusions are strict.  Below we show that the tail
  behaviour of a random vector may vary depending on the
  choice of modulus and the parameter $\beta$: for certain $\beta$ we observe
  regular variation on $\R_+^2$, for others, hidden
  regular variation, and for some $\beta$, the distribution is not
  regularly varying at all.

  Let $\xi=(\xi_1,\xi_2)$ be a random vector in $\XX$, where
  $\xi_1,\xi_2$ are i.i.d.~Pareto(1)-distributed, that is,
  \index{Pareto distribution}
  \begin{displaymath}
    \Prob{\xi_1>t}=t^{-1},\quad  t\geq1.
  \end{displaymath}
  It is well known that $\xi$
  is regularly varying of tail index 1 (with the normalising function
  $g(t)=t$) on $\R_+^2$ with respect to the ideal
  $\sR_{\modulus_{\max}}=\sR_0^2$. The spectral measure is supported
  by the axes, hence $\xi$ is not regularly varying with this
  normalising function on any ideal whose union set is a subset of
  $(0,\infty)^2$, in accordance with Proposition~\ref{prop:sub-ideal}.

  For other moduli and all sufficiently large $t$, we observe the following:
  \begin{align*}
    \Prob{\modulus_{\min}(\xi)>t}&=t^{-2},\\
    \Prob{\modulus_\beta(\xi)>t}&=\frac{1}{1-2\beta}t^{-2}
    +\frac{2\beta}{1-2\beta}t^{-1/\beta},\quad \beta\in(0,1/2),\\
    \Prob{\modulus_{1/2}(\xi)>t}&=t^{-2}(1+2\log t),\\
    \Prob{\modulus^*_\beta(\xi)>t}&=\frac{2-2\beta}{1-2\beta}t^{-1/(1-\beta)}
                                -\frac{1}{1-2\beta}t^{-2},\quad \beta\in(0,1/2).
  \end{align*}
  Note that the main asymptotic term for $\modulus_{1/2}$ has also been
  derived by \cite[Lemma~4.1(4)]{jessen06:_regul}.  In order to
  establish the regular variation of $\xi$ in $(0,\infty)^2$,
  we use the characterisation from
  Theorem~\ref{thr:polar-decomp-4}. Let $\modulus$ be a modulus
  that is continuous and proper on $(0,\infty)^2$. Let $A$ be a
  compact subset of $\{\modulus=1\}\cap (0,\infty)^2$. One can
  represent the set
  $\{x\in(0,\infty)^2\colons
  \modulus(x)\geq t,\modulus(x)^{-1}x\in A\}$ in polar
  coordinates as $\{tsu\colons s\geq 1/\modulus(u),u\in A'\}$, where $A'$ is
  a subset of the unit circle in $(0,\infty)^2$. Since the density of
  $(\xi_1,\xi_2)$ is $x_1^{-2}x_2^{-2}$ on $[1,\infty)^2$, 
  \begin{align*}
    \Prob{\modulus(\xi)^{-1}\xi\in A,\modulus(\xi)>t}
    &=\int_{A'}\diff u \int_{t/\modulus(u)}^\infty (su_1)^{-2}(su_2)^{-2}s \diff s\\
    &=t^{-2} \int_{A'} \modulus(u)^2 u_1^{-2} u_2^{-2} \diff u
  \end{align*}
  for all sufficiently large $t$, since then
  $tu/\modulus(u)\in [1,\infty)^2$, which is the support of $\xi$.
  Since $A'$ is a compact subset of the unit Euclidean circle
  intersected with $(0,\infty)^2$ and the modulus is continuous, we
  have that $0<c_1\leq \modulus(u)\leq c_2<\infty$ for all $u\in A'$,
  so that the last integral is finite and strictly positive.
  Thus, any non-trivial local tail measure on compact spectral subsets
  of $(0,\infty)^2$ must appear on the scale $t^{-2}$.
  This is not the case for $\modulus^*_\beta(\xi)$ with
  $\beta\in[0,1/2)$, whose tail is of order $t^{-1/(1-\beta)}$, nor
  for $\modulus_{1/2}$: in the latter case the normalising scale is
  affected by the logarithmic factor.
  However, $\modulus_\beta$ with
  $\beta\in[0,1/2)$ has a tail of order $2$, implying that $\xi$ is
  regularly varying with index $2$ on 
  $\sR_{\modulus_\beta}$.
  \index{hidden regular variation}
  This is well known for $\modulus_{\min}$
  (that is, for $\beta=0$), in which case $\xi$ is said to exhibit 
  \emph{hidden regular variation}; see \cite{res02} and the book
  \cite{resnick24:_art_findin_hidden_risks} entirely devoted to   
  the hidden regular variation phenomenon.

  Consider the map $\psi(x)=x_1+x_2$ from $(0,\infty)^2$ to
  $\YY=(0,\infty)$ equipped with the ideal $\sR_0$. The map $\psi$ is
  bornologically consistent if the ideal on $(0,\infty)^2$ contains
  $\{(x_1,x_2)\in\XX\colons x_1+x_2>t\}$ for all $t>0$. This is the case if
  the ideal is generated by the modulus
  $\modulus_{\max}(x_1,x_2)=\max(x_1,x_2)$ or any other modulus induced
  by a norm. If we equip $(0,\infty)^2$ with the modulus
  $\modulus_{\min}(x_1,x_2)$, then $\psi$ is not bornologically
  consistent. Indeed, the set
  $\psi^{-1}((y,\infty))=\{(x_1,x_2)\colons x_1+x_2>y\}$,
  $y\in\YY$, neither belongs to $\sR_{\modulus_{\min}}$ nor to
  $\sR_{\modulus_\beta}$ with $\beta\in(0,1/2)$, that is, it does not
  belong to ideals that are compatible with the hidden regular variation property.
\end{example}

\begin{example}[Hidden regular variation in $\R^d$]
  \label{ex:vectors-HRV}
  Let $\xi_1,\dots,\xi_m$ be i.i.d.\ regularly varying random vectors,
  namely, $\xi_1\in\RV(\R^d,\mydot,\sR_0^d,g,\mu)$. Consider
  $(\xi_1,\dots,\xi_m)$ as a random vector in the space $(\R^d)^m$ and
  equip this space with the ideal $(\sR_0^d)^m(k)$ for
  $k\in\{1,\dots,m\}$. Since $T_{t^{-1}}\xi_1\to0$ in probability, the assumptions of
  Theorem~\ref{thr:power-space} are satisfied, and consequently
  \begin{displaymath}
    (\xi_1,\dots,\xi_m)
    \in \RV((\R^d)^m,\mydot,(\sR_0^d)^m(k),g^k,\mu_k),
  \end{displaymath}
  where
  \begin{equation}
    \label{eq:mu-vector-hidden}
    \mu_k\big(\bdiff (x_1,\dots,x_m)\big)
    =\sum_{I\subset\{1,\dots,m\},\card(I)=k}
    \prod_{i\in I}\mu(\bdiff x_i)\prod_{j\notin I}\delta_0(\bdiff x_j).
  \end{equation}
\end{example}

\begin{example} 
  \label{ex:bornology-skew}
  Let $\XX=(0,\infty)^2$ with linear scaling. Define an ideal
  (actually, a bornology) $\sX$ by specifying its open base as
  \begin{displaymath}
    \base_n=\big\{(x_1,x_2)\in(0,\infty)^2\colons 
    x_2> n^{-1}(1+x_1)\big\}, \quad n\geq1.
  \end{displaymath}
  This ideal is generated by a countable family of moduli
  $\modulus_n(x_1,x_2)=x_1+nx_2$, $n\geq1$; it does not satisfy
  Condition~\hyperref[condB]{(B)}. Indeed, if $\VV$ were a bounded semicone from
  \hyperref[condB]{(B)}, then there would exist some
  $n\geq1$ such that $\VV\subset \base_n$. This is impossible,
  since $\VV$ contains points $(a_m,b_m)$, $m\geq1$,
  with $b_m\downarrow 0$ as $m\to\infty$, while
  $b_m< n^{-1}(1+a_m)$ for any fixed $n$ and sufficiently large $m$
  and so $(a_m,b_m)\notin\base_n$. However, if the scaling on $\XX$ is
  instead defined by $T_tx=(x_1,tx_2)$, then
  $\modulus(x_1,x_2)=x_2/(1+x_1)$ is homogeneous, and then $\sX$
  satisfies Condition~\hyperref[condB]{(B)}.
\end{example}

\begin{example}[Linear transforms]
  Let $\psi$ be a linear map from $\XX=\R^d$ to $\YY=\R^m$.  Assume
  that $\XX$ and $\YY$ are equipped with linear scaling and metric
  exclusion ideals $\sR_0^d$ and $\sR_0^m$; consequently, $\psi$ is a
  morphism, which is bornologically consistent by
  Lemma~\ref{lemma:map-zero}. By Theorem~\ref{mapping-theorem}, if
  $\xi\in\RV(\R^d,\mydot,\sR_0^d,g,\mu)$, then
  $\psi(\xi)\in\RV(\R^m,\mydot,\sR_0^m,g,\psi\mu)$ provided the
  pushforward tail measure $\psi\mu$ is non-trivial
  on the ideal $\sR_0^m$, equivalently, not concentrated at the
  origin; see also \cite{bas:dav02} and 
  \cite[Proposition~2.1.12]{kul:sol20}. The case of sums of components
  was addressed in Example~\ref{ex:maps-products}.
\end{example}

\paragraph{Moduli given by the marginals}
Consider moduli $\modulus_i(x)=|x_i|$, $i=1,\dots,d$, defined
by the absolute values of the components of $x\in\R^d$. Each modulus
$\modulus_i$ generates an ideal $\sR_i$, such that $\sR_i\subset \sR_0^d$
and $\sR_0^d$ is the product of the ideals $\sR_1,\dots,\sR_d$.
The following result characterises regular variation on $\sR_0^d$ in
terms of tail measures on the ideals $\sR_i$.

\begin{proposition}
  \label{prop:cumulative-moduli}
  We have $\xi\in\RV(\R^d,\mydot,\sR_0^d,g,\mu)$ if and only if 
  there exist measures $\mu_i\in\Mb[\sR_i]$, $i=1,\dots,d$, not all
  trivial, such that
  \begin{displaymath}
    g(t)\Prob{t^{-1}\xi\in\cdot}\vto[\sR_i]\mu_i,
    \quad i=1,\dots,d,
  \end{displaymath}
  that is, $\xi\in\RV(\R^d,\mydot,\sR_i,g,\mu_i)$ for
  all $i$ such that $\mu_i$ is non-trivial. In this case,
  \begin{equation}
    \label{eq:48}
    \mu\big(A\cap\{\modulus >a\}\big)=\sum_{i=1}^d 
    \mu_i\big(A\cap\{\modulus_i(x)>a\}\cap
    \{\overline\modulus_{i-1}(x)\leq a\}\big), \quad a>0, \; A\in\sR_0^d,
  \end{equation}
  where
  \begin{displaymath}
    \overline\modulus_i(x)=\max_{j=1,\dots,i}\modulus_j(x),\quad i=1,\dots,d,
  \end{displaymath}
  and $\overline\modulus_0(x)=0$.
\end{proposition}
\begin{proof}
  \textsl{Necessity.} The ideal $\sR_i$ is a sub-ideal of
  $\sR_0^d$. However, if
  $A\in\sR_0^d$, then there exists $a>0$ such that
  \[
    A\subset \bigcup_{i=1}^d \{|x_i|>a\},
  \]
  and each set $\{|x_i|>a\}$ belongs to $\sR_i$. Thus, the restriction
  of $\mu$ to at least one $\sR_i$ is non-trivial, and the
  restrictions of $\mu$ to the ideals $\sR_i$ give the measures
  $\mu_i$. The remainder of the proof follows from
  Proposition~\ref{prop:sub-ideal}. The explicit representation of
  $\mu$ is derived as part of the sufficiency argument.

  \smallskip
  \noindent
  \textsl{Sufficiency.}  The proof is inspired by Lemma~3.5 of 
  \cite{dombry18:_tail}. Observe that 
  $\overline\modulus_d(x)=\max(|x_1|,\dots,|x_d|)=\modulus(x)$
  generates the ideal $\sR_0^d$. For each
  Borel $A\in\sR_0^d$ and $a>0$, we partition the
  event according to the first index exceeding the threshold:
  \begin{displaymath}
    \Prob{t^{-1}\xi\in (A\cap\{\modulus>a\})}
    =\sum_{i=1}^d \Prob{t^{-1}\xi\in A,\modulus_i(\xi)>at,
      \overline\modulus_{i-1}(\xi)\leq at}.
  \end{displaymath}
  We first consider sets $A$ from a convergence-determining class
  (e.g., finite unions of balls) chosen so that all sets appearing
  below are continuity sets for the relevant measures $\mu_i$. For
  such $A$, the assumed vague convergence on each ideal $\sR_i$
  implies that
  \begin{align*}
    \lim_{t\to\infty} g(t)\Prob{t^{-1}\xi\in
    (A\cap\{\modulus>a\})}
        =\sum_{i=1}^d  
        \mu_i\big(A\cap\{\modulus_i>a\}\cap
        \{\overline\modulus_{i-1}(x)\leq a\}\big).
  \end{align*}
  This defines the candidate limiting measure $\mu$ through
  \eqref{eq:48} and establishes the convergence on the chosen
  convergence-determining class, so it completes the proof.
\end{proof}

Note that \eqref{eq:48} can be equivalently written for any permutation
of the coordinates.

\begin{example}[Exchangeable random vectors]
  Let $\xi$ be an exchangeable random vector in $\R^d$, that is, its
  distribution is invariant under all permutations of its components. Then $\xi$ is
  regularly varying on the ideal $\sR_0^d$ if and only if $\xi$ is
  regularly varying on the ideal $\sR_1$, that is,
  \begin{displaymath}
    \Prob{t^{-1}\xi\in \cdot \mid |\xi_1|>t}
    \wto \frac{\mu_1(\cdot\cap\{\modulus_1>1\})}
    {\mu_1(\{\modulus_1>1\})}
  \end{displaymath}
  for a measure $\mu_1\in\Mb[\sR_1]$, where $\modulus_1(x)=|x_1|$.

  Then, by exchangeability, $\xi$ is regularly varying on each ideal
  $\sR_i$. If $\pi_i$ denotes a permutation of the coordinates that
  sends the first coordinate to the $i$-th coordinate, then the
  corresponding tail measure is $\mu_i=\pi_{i\#}\mu_1$.  Therefore,
  the tail measure on $\sR_0^d$ is given by
  \[
    \mu(A\cap\{\modulus>a\})
    =
    \sum_{i=1}^d
    \mu_i\big(A\cap\{\modulus_i>a\}\cap
    \{\overline\modulus_{i-1}\leq a\}\big).
  \]
\end{example}

\paragraph{Non-linear scaling on Euclidean space}

\begin{example}
  \label{ex:g=1}
  Consider $\XX=\R_+^2$ with the Euclidean metric and the scaling
  acting only on the first component, whereby $T_t(x_1,x_2)=(tx_1,x_2)$. First,
  consider the ideal $\sX_\modulus$ generated by the modulus
  $\modulus(x_1,x_2)=x_1$, which possesses the base
  $[m^{-1},\infty)\times\R_+$, $m\geq1$. Then $(\xi_1,\xi_2)$ is
  regularly varying on $\sX_\modulus$ if the distribution
  of $(t^{-1}\xi_1,\xi_2)$ conditioned on $\xi_1>ut$ converges weakly as
  $t\to\infty$ for all $u>0$, see
  Theorem~\ref{thr:polar-decomp-4}. For instance, this occurs if
  $\xi_1$ is a regularly varying random variable and $\xi_2$ is
  independent of $\xi_1$; then the tail measure of $(\xi_1,\xi_2)$
  is the product of the tail measure of $\xi_1$ and the distribution
  of $\xi_2$, see Theorem~\ref{thr:pair}.

  Now equip $\R_+^2$ with the metric exclusion ideal $\sR_0^2$
  obtained by excluding the origin (equivalently, generated by the
  Euclidean norm). While the Euclidean norm is not homogeneous under
  the introduced scaling, this ideal is nevertheless consistent under
  scaling along one component. If $\xi=(\xi_1,\xi_2)$ is a random
  vector in $\R_+^2$, then each set
  $(0,\infty)\times (b,\infty)$ with $b>0$ belongs to $\sR_0^2$ and 
  \begin{displaymath}
    g(t)\Prob{T_{t^{-1}}\xi\in [0,\infty)\times (b,\infty)}=
    g(t)\Prob{\xi_2\in (b,\infty)}\to\infty
  \end{displaymath}
  if $g(t)\to\infty$ as $t\to\infty$ and $\Prob{\xi_2\in
    (b,\infty)}>0$. Thus, $\xi$ is not regularly
  varying under the introduced scaling on the ideal
  $\sR_0^2$. However, if $g(t)\equiv1$ is a constant, then
  \begin{displaymath}
    \Prob{T_{t^{-1}}\xi\in (a,\infty)\times (b,\infty)}
    =\Prob{\xi_1>ta,\xi_2>b}
    \to
    \begin{cases}
      0, & a>0,\\
      \Prob{\xi_2>b}, & a<0,
    \end{cases}
  \end{displaymath}
  implying that $T_{t^{-1}}\xi$ converges weakly to the distribution
  of $(0,\xi_2)$, that is, to the measure being
  the product of $\delta_0$ and the distribution of $\xi_2$, see
  Remark~\ref{rem:wide}. 
\end{example}

\begin{example}[Inverse linear scaling]
  \index{scaling!inverse linear}
  Let $\XX=\R^d$ with the inverse linear scaling, that is,
  $T_tx=t^{-1}x$.  The modulus on $\XX$ may then be defined as
  $\modulus(x)=1/\|x\|$. Note that the modulus takes the 
  value infinity at the origin and is continuous with values in the
  extended half-line $[0,\infty]$. The corresponding ideal on
  $\R^d$ is generated by the balls $B_n(0)$, $n\in\NN$, centred at the
  origin. Therefore, this ideal consists of all bounded sets,
  coinciding with the ideal of bounded sets (also known as 
  the Fr\'echet, or Hadamard, ideal on $\R^d$).
\end{example}

\begin{example}[Scaling with different powers]
  Consider a scaling on $\R^d$ given by
  $T_tx=(t^{\alpha_1}x_1,\dots,t^{\alpha_d}x_d)$ with
  $\alpha_1,\dots,\alpha_d\in\R$. To ease the notation, assume that
  $d=3$ and $\alpha_1=1,\alpha_2=-1,\alpha_3=0$. Then
  $\zero=\{0\}\times\{0\}\times\R$  is the set of
  scaling-invariant elements. A modulus can be defined as
  \begin{displaymath}
    \modulus(x)=\max\big(|x_1|,|x_2|^{-1}\big),
    \quad x=(x_1,x_2,x_3)\in\R^3,
  \end{displaymath}
  noticing that it is infinite if $x_2=0$.  A random vector
  $\xi=(\xi_1,\xi_2,\xi_3)$ is regularly varying if the conditional
  distribution of $(t^{-1}\xi_1,t\xi_2,\xi_3)$ given that $|\xi_1|>tu$
  or $|\xi_2|<(tu)^{-1}$ weakly converges as $t\to\infty$ for all $u>0$, see
  Theorem~\ref{thr:polar-decomp-4}.
  For instance, in the union case this holds if the components are
  independent and the two variables $\xi_1$ and $\xi_2^{-1}$ are
  regularly varying on compatible scales; the limiting tail measure is
  then the sum of the two one-component contributions.
\end{example}

\begin{example}[Sparse regular variation]
  For random vectors in $\R_+^d$, the authors of \cite{MR4342579}
  suggest working with limits of conditional distributions of the pair
  $(t^{-1}\modulus(\xi),\pi(T_{t^{-1}}\xi))$ given
  $\modulus(\xi)=\|\xi\|>t$, where $\pi(x)$ is the metric projection
  of $x$ to the boundary of the unit simplex in $\R_+^d$. Then
  convergence of the conditional distribution of this pair defines the
  \emph{sparse regular variation} of $\xi$.
  \index{regular variation!sparse}
  It should be noted that this projection differs
  from the vector normalised by its $\ell_1$-norm and lacks
  bijectivity required for a scaling.  Nevertheless, if we write
  $x\sim y$ whenever $\big(\modulus(x),\pi(T_{\modulus(x)^{-1}}x)\big)=
  \big(\modulus(y),\pi(T_{\modulus(y)^{-1}}y)\big)$,
  one can treat this as a regular variation in an appropriately chosen
  quotient space.
\end{example}

\begin{example}
  \label{ex:scaling-min}
  Define a scaling on $\XX=\R_+^2$ by letting
  \begin{displaymath}
    T_t(x_1,x_2)=\big(x_1+(t-1)\min(x_1,x_2),x_2+(t-1)\min(x_1,x_2)\big).
  \end{displaymath}
  Then the scaling-invariant elements build the set
  \begin{displaymath}
    \zero=\big\{(x_1,x_2)\in\R_+^2\colons \min(x_1,x_2)=0\big\},
  \end{displaymath}
  which is the union 
  of two semiaxes. Consider the ideal $\sR_0^2(2)$, which is generated
  by the continuous modulus $\modulus_{\min}(x_1,x_2)=\min(x_1,x_2)$.
  Let $\eta=(\eta_1,\eta_2)$ be any random vector with values in
  $\SS_{\modulus_{\min}}=\{(x,y)\in\R_+^2\colons \min(x,y)=1\}$. If $\zeta$ is
  a regularly varying positive random variable, then
  \begin{displaymath}
    T_\zeta\eta =(\eta_1+\zeta-1,\eta_2+\zeta-1)
  \end{displaymath}
  is regularly varying in
  the introduced scaling.  The random vector $\xi=(\xi_1,\xi_2)$ built
  of two independent Pareto(1) components is not regularly varying on
  $\sR_0^2(2)$ in this scaling. This is seen by noticing that
  $\Prob{\modulus_{\min}(\xi)>t}=t^{-2}$, while
  \begin{displaymath}
    \Prob{T_{\modulus_{\min}(\xi)^{-1}}\xi\in A,\,\modulus_{\min}(\xi)>t}
    =\Prob{\xi_2\geq t ,\,\xi_1\in[\xi_2,\xi_2+a-1]}
  \end{displaymath}
  is of the order $t^{-3}$ for
  $A=[1,a]\times\{1\}\subset \SS_{\modulus_{\min}}$ with any finite
  $a>1$.
\end{example}

\paragraph{Quotients of Euclidean spaces}
We now consider several examples of equivalence relations on
Euclidean spaces.

\begin{example}
  \label{lines-ex}
  Let $\XX=\R_+^2$ with the Euclidean topology, linear scaling and
  the ideal $\sR_0^2$ obtained by excluding the origin.  For
  $x=(x_1,x_2)$ and $y=(y_1,y_2)$ in $\XX$, let $x$ be equivalent to
  $y$ if $x_1=y_1$. Then Condition~\hyperref[condS]{(S)} holds and
  $[(x_1,x_2)]=\{(x_1,t)\colons t\in\R_+\}$. 
  The quotient space is identified as $\XXT=\R_+$ with the quotient
  map given by $q(x_1,x_2)=x_1$, and the scaling on $\XXT$ is linear.
  Note that $[\sR_0^2]$ consists of sets $B\in\sR_0^2$ such that $B$
  does not intersect some metric neighbourhood of the semiaxis
  $\{(0,x_2)\colons x_2\in\R_+\}$, so that $q\sR_0^2=\sR_0$. If $\modulus$ is
  the Euclidean norm on $\R_+^2$, then the function
  $\ttau\big([(x_1,x_2)]\big)=x_1$ is continuous and
  Lemma~\ref{lemma:q-B} applies.
  \index{quotient space!of Euclidean space}

  Let $\xi=(\xi_1,\xi_2)$ be regularly varying on $\sR_0^2$. By
  Theorem~\ref{rv-XXT}, $q\xi=\xi_1$ is regularly varying on
  $q\sR_0^2=\sR_0$ if the tail measure $\mu$ of $\xi$ is non-trivial on
  $[\sR_0^2]$, that is, $\mu([\eps,\infty)\times\R_+)>0$ for some
  $\eps>0$, equivalently, the tail measure is not entirely concentrated
  on $\{0\}\times\R_+$.  The latter condition is violated, for instance, if
  $\xi=(\xi_1,\xi_2)$ has two independent components, $\xi_1$ has a
  light tail (e.g., is exponential) and $\xi_2$ has a heavy tail (e.g., is
  Pareto). Then $\xi$ is regularly varying on $\sR_0^2$, with tail measure
  supported by the equivalence class of the origin,
  $[\zero]=\{0\}\times\R_+$, and so
  $q\xi=\xi_1$ is not regularly varying in $\XXT$.
  If $\xi_1$ and $\xi_2$ are independent random variables such that
  $\xi_1$ is regularly varying and $\xi_2=e^\eta$ for a Pareto(1) random
  variable $\eta$, so that $\xi_2$ has a super-heavy tail, then
  $\xi=(\xi_1,\xi_2)$ is not regularly varying on $\sR_0^2$ but $q\xi$ is
  regularly varying in the quotient space.

  If $\XX$ is endowed with the ideal $\sR_0^2(2)$ (obtained by
  excluding the semiaxes), then $q\sR_0^2(2)$ does not contain any
  non-empty set, so that the hidden regular variation property does not
  transfer to the quotient space.
\end{example}

\begin{example}
  Let $\XX=\R_+^2$ with the scaling
  $T_tx=(tx_1, x_2)$ for $x=(x_1,x_2)\in \XX$, see also
  Example~\ref{ex:g=1}.
  Define the modulus $\modulus(x)=x_1$ and the corresponding ideal
  $\sX_\modulus$ on $\XX$, which is also the metric exclusion ideal
  obtained by excluding $\{0\}\times\R_+$. Let
  $x=(x_1,x_2)$ be equivalent to $y=(y_1,y_2)$ if $x_1= y_1$. Then
  Condition~\hyperref[condS]{(S)} holds, the 
  quotient space is $\R_+$ with linear scaling, and
  $q\sX_\modulus=\sR_0$. Since $[\sX_\modulus]=\sX_\modulus$, the
  image $q\xi$ of every random vector $\xi$ that is regularly varying
  on $\sX_\modulus$ with the introduced scaling is regularly varying on
  $q\sX_\modulus$.
\end{example}

\begin{example}
  Assume that two points $x,y\in\R^d$ are equivalent if they are equal
  up to a permutation of their components. Since the ideals
  $\sR_0^d(k)$ are invariant under coordinate permutations, their
  saturations remain in the same ideals. Thus the quotient ideal
  $q\sR_0^d(k)$ is naturally identified with the image of
  $\sR_0^d(k)$ under the quotient map. Consequently, regular variation
  of a random vector on $\sR_0^d(k)$ transfers to the quotient space,
  provided the pushforward tail measure is non-trivial on the quotient
  ideal.
\end{example}

\section{Spaces of sequences}
\label{sec:infinite-sequences}

\paragraph{$\R_+^\infty$ with coordinatewise convergence}
Consider the space $\R^\infty$ with linear scaling applied
coordinatewise and the topology of coordinatewise convergence,
which turns $\R^\infty$ into a Polish space metrised by
\index{space!of sequences}
\begin{displaymath}
  \dmet(x,y)=\sum_{i=1}^\infty
  2^{-i}\min\big(|x_i-y_i|,1\big) ,\quad x,y\in\R^\infty;
\end{displaymath}
see, e.g., \cite{lin:res:roy14}. Note that the constructions below are
easy to adapt to the infinite power $\XX^\infty$ of any metric
space $\XX$ with a continuous scaling and an ideal which
satisfies Condition~\hyperref[condB]{(B)}.

We first show that the entire space $\R_+^\infty$ (hence, also $\R^\infty$)
does not admit a finite continuous modulus which is strictly positive
on all sequences $x\neq 0$ (that is, a proper continuous modulus),
meaning that this example is not covered by the setting of star-shaped
spaces adopted in \cite{seg:zhao:mein17} and \cite{kul:sol20}. Assume
that there is an open semicone $\VV$ that satisfies
Condition~\hyperref[condB]{(B)} and such that 
$\cup_{t>0} t\VV$ contains all non-zero sequences. 
Denote by $\delta_{jn}$ the Kronecker delta and by
$e_j=(\delta_{jn})_n=(\ldots,0,1,0,\ldots)$ the element of
$\R_+^\infty$ with the $j$-th component equal to one and all others
equal to zero. For each $j\in\NN$, let
$y_j=\inf \{t\colons te_j\in\VV\}$, and observe that the set on the
right-hand side is always nonempty. Moreover, $y_j>0$ for each
$j\in\NN$, since otherwise $\cap_{t>1} t\VV$ contains all elements
with strictly positive $j$-th component and all other
components equal to zero, and thus is a non-empty set.  Consider a
sequence $x=(x_n)$ of strictly positive numbers such that
$x_n/y_n\to 0$. If $x\in t\VV$ for some $t>0$, then $x_n\geq ty_n$ for
all $n$, which contradicts the choice of $(x_n)$; hence, the
non-trivial sequence $x$ does not belong to $\cup_{t>0} t\VV$.

For $m\in\NN$, denote by
\begin{displaymath}
  \psi_m(x)=(x_1,\dots,x_m,0,\ldots),
\end{displaymath}
the sequence obtained by retaining the first $m$ components of
\begin{displaymath}
  x=(x_1,x_2,\ldots)\in\R^\infty
\end{displaymath}
and replacing all other components with zeros. Recall that
$\pi_m x=(x_1,\dots,x_m)$ denotes the projection of $x$ onto its first
$m$ coordinates. For a measure $\mu$ on $\R^\infty$, denote by
$\psi_m\mu$ and $\pi_m\mu$ its pushforwards by $\psi_m$, respectively,
$\pi_m$.
\index{projection map}

Recall that $\sR_0$ is the ideal on $\R$ obtained by excluding zero.
The product ideal $\sR_0^\infty(k)$ on $\R^\infty$ is
generated by the moduli
\begin{displaymath}
  \modulus_I(x)=\min_{i\in I} |x_i|,
\end{displaymath}
where $I$ runs through the family of subsets of $\NN$ with cardinality
$k$, see Definition~\ref{def:product}. Thus, $A\in\sR_0^\infty(k)$ if
there exists an $m\in\NN$ such that
\begin{displaymath}
  \inf_{x\in A} \max_{I\subset \{1,\dots,m\},\card(I)=k}  \modulus_I(x)>0. 
\end{displaymath}
The ideal $\sR_0^\infty(1)$ is generated by the family of continuous
moduli $\modulus_i(x)=|x_i|$, $i\in\NN$.  Note that $\sR_0^\infty(1)$
is smaller than the ideal $\sR_{\sup}^\infty$ which consists of all
sets $A\subset\R^\infty$ such that $\inf_{x\in A}\sup_{i\in\NN} |x_i|>0$. For
instance, the set $A=\{e_k\colons k\in\NN\}$ of all unit basis vectors does
not belong to $\sR_0^\infty(1)$ but belongs to
$\sR_{\sup}^\infty$. Moreover, $A$ does not belong to any metric
exclusion ideal obtained by excluding a cone. Note also that the
modulus $\modulus_{\max}(x)=\sup_i |x_i|$ may be infinite and is not
continuous, which is seen by evaluating it on the sequence $(e_n)$ of
basis vectors.

Another example of an ideal on $\R^\infty$ is the family
$\sR_{\inf}^\infty$ of all sets $A\subset\R^\infty$ such that
\begin{displaymath}
  \inf_{x=(x_1,x_2,\ldots)\in A} \inf_{i\in\NN} |x_i|>0.
\end{displaymath}
This ideal is generated by the modulus
$\modulus_{\inf}(x)=\inf_{i\in\NN} |x_i|$. Note that the modulus
$\modulus_{\inf}$ is not continuous on
$\R^\infty$: the sequence $(1,\dots,1,1/n,1/n,\ldots)$ (where the
value $1$ occupies the first $n$ places) converges to $(1,1,\ldots)$, but its
infimum converges to zero. Thus, for all $k\in\NN$,
\begin{displaymath}
  \sR_{\inf}^\infty\subset \sR_0^\infty(k)\subset
  \sR_0^\infty(1)\subset \sR_{\sup}^\infty. 
\end{displaymath}
An ideal on $\R_+^\infty$, which is neither contained in $\sR_0^\infty(1)$
nor containing any $\sR_0^\infty(k)$, can be constructed as the family
of subsets $B\subset\R_+^\infty$ such that there exists an
$\eps\in(0,1]$ satisfying
\begin{displaymath}
  \inf_{x=(x_i)\in B}
  \limsup_{n\to\infty} \frac{1}{n}\sum_{i=1}^n \one_{|x_i|>\eps}>0.
\end{displaymath}

As in \cite{lin:res:roy14}, it is possible to endow $\R^\infty$ with a
metric exclusion ideal obtained by excluding a closed cone $\cone$ in
$\R^\infty$, see Definition~\ref{def:exclusion}. While the metric on
$\R^\infty$ is not homogeneous under scaling, it satisfies
\begin{displaymath}
  \min(t,1)\dmet(x,y)\leq \dmet(tx,ty)\leq \max(t,1)\dmet(x,y)
\end{displaymath}
for all $x,y\in\R^\infty$, which follows directly from
\begin{displaymath}
  \min(\min(t,1)x,\min(t,1))\leq 
  \min(tx,1)\leq \min(\max(t,1)x,\max(t,1)). 
\end{displaymath}
Thus, the metric exclusion ideal is scaling consistent, which is also
a corollary of the following result.
In particular, this result implies that $\sR_0^\infty(1)$ coincides with the
metric exclusion ideal $\sR_0^\infty$ obtained by excluding the zero
sequence, a fact noticed in \cite[Page~105]{kul:sol20}.

\begin{lemma}
  \label{lemma:ideal-sequences}
  \index{ideal!on the space of sequences}
  The ideal $\sR_0^\infty(k)$ coincides with the metric exclusion ideal
  obtained by excluding the convex cone $\cone_k$, which consists of sequences
  with at most $(k-1)$ non-zero elements if $k\geq 2$, and $\cone_1$ is
  the set consisting of the zero sequence.
\end{lemma}
\begin{proof}
  Assume that $B\in\sR_0^\infty(k)$. Then there is an $m\in\NN$ such
  that  $x\in B$ implies $\modulus_I(x)\geq \eps>0$  for some set
  $I\subset\{1,\dots,m\}$ of cardinality $k$. Hence, $|x_i|\geq \eps$
  for all $i\in I$, so that $\dmet(x,\cone_k)\geq 2^{-m}\min(\eps,1)$
  for all $x\in B$.

  Assume that $B$ belongs to the metric exclusion ideal, that is,
  $\dmet(x,\cone_k)\geq \eps>0$ for all $x\in B$. Then
  $\dmet(\pi_m x,\pi_m \cone_k)\geq \eps$, where $m$ is chosen
  so that $2^{-m}\geq \eps/2$ and $m\geq k$. Hence,
  \begin{displaymath}
    \inf_{J\subset \{1,\dots,m\},\card(J)\leq k-1}
    \sum_{i\in\{1,\dots,m\}\setminus J} 2^{-i} \min(|x_i|,1)\geq \eps/2.
  \end{displaymath}
  Choosing $J=\emptyset$, we see that there exists an
  $i_1\in\{1,\dots,m\}$ such that
  \begin{displaymath}
    2^{-i_1} \min(|x_{i_1}|,1)\geq \eps/(2m).
  \end{displaymath}
  In case $J=\{i_1\}$, we
  obtain that there exists an $i_2\neq i_1$ such that
  $2^{-i_2} \min(|x_{i_2}|,1)\geq \eps/(2m)$. Continuing this process, we
  obtain $k$ distinct indices
  $I_x=\{i_1,\dots,i_k\}\subset\{1,\dots,m\}$ such that
  \begin{displaymath}
    |x_{i_j}|\geq \min(|x_{i_j}|,1)\geq \eps/(2m), \quad j=1,\dots,k.
  \end{displaymath}
  Thus, $\min_{i\in I_x} |x_i|\geq \eps/(2m)$ for all $x\in B$ and some
  $I_x\subset\{1,\dots,m\}$ of cardinality $k$. 
\end{proof}

The following result shows that checking the regular variation
property in $\R^\infty$ amounts to confirming it for all
finite-dimensional distributions. This fact was mentioned in
\cite[Page~298]{lin:res:roy14} for a metric exclusion ideal
$\sR^\infty_\cone$ obtained by excluding a cone $\cone$, such that
$\pi_m\cone$ is a closed strict subset of $\R^m$ and
$\psi_m(\cone)\subset \cone$ for all $m\in\NN$. Below we provide an
independent proof of this result, assuming only that the projections of
$\cone$ are non-trivial.

\begin{lemma}
  \label{lemma:fidi-sequence}
  Let $\sR^\infty_\cone$ be a metric exclusion ideal on $\R^\infty$.  Assume
  that $\cone_m=\pi_m\cone$ is a strict subset of $\R^m$ for all
  sufficiently large $m$. Then
  \begin{displaymath}
    \xi\in\RV(\R^\infty,\mydot,\sR^\infty_\cone,g,\mu)
  \end{displaymath}
  if and only if
  $\eta_m=\pi_m(\xi)\in\RV(\R^m,\mydot,\sR^m_{\cone_m},g,\mu_m)$ for all
  sufficiently large $m$, where $\sR^m_{\cone_m}$ is the metric
  exclusion ideal on $\R^m$ with the Euclidean metric obtained by
  excluding $\cone_m$, $\mu_m\in\Mb[\sR^m_{\cone_m}]$ is a Borel measure
  on $\R^m$ for each $m\in\NN$, and then $\mu_m=\pi_m\mu$ for
  $\mu\in\Mb[\sR^\infty_\cone]$.
\end{lemma}
\begin{proof}
  \textsl{Necessity}. 
  Assume that $\xi$ is regularly varying in $\R^\infty$. We apply the
  continuous mapping theorem to the continuous morphism
  $x\mapsto \pi_m x$. If $A\in\sR^m_{\cone_m}$, then for each
  $y\in\pi_m^{-1}(A)$, we have 
  \begin{displaymath}
    \dmet(y,\cone)\geq \dmet\big(\pi_m(y),\cone_m\big)\geq
    2^{-m} \min(m^{-1/2}\eps,1)>0,
  \end{displaymath}
  where $\eps$ is a positive lower bound for the Euclidean distance
  between $x\in A$ and $\cone_m$. The factor $m^{-1/2}$ results from
  the fact that if the Euclidean distance between two vectors in
  $\R^m$ exceeds $\eps$, then there is at least one component
  in which these vectors differ by at least $\eps/\sqrt{m}$. Thus,
  $\pi_m$ is a bornologically 
  consistent morphism. To show that $\mu_m=\pi_m\mu$ is non-trivial on
  $\sR^m_{\cone_m}$, let $m$ be sufficiently large, so that $\cone_m$
  is a strict subset of $\R^m$. For each $\eps>0$ and $m$ such that
  $2^{-m}\leq \eps$, we have
  \begin{align*}
    \mu_m\Big(\big\{x\in\R^m\colons \dmet\big((x,0,0,\ldots),
    &\cone_m\times\{0\}^\infty\big)\geq
    \eps\big\}\Big)\\
    &= \mu\Big(\big\{y\colons \dmet\big(\psi_m(y),
      \cone_m\times\{0\}^\infty\big)\geq \eps \big\}\Big)\\
    &\geq \mu\Big(\big\{y\colons \dmet(y,\cone)\geq 2\eps \big\}\Big)>0
  \end{align*}
  for some $\eps>0$. Thus, $\mu_m$ is non-trivial. Therefore,
  $\eta_m\in\RV(\R^m,\mydot,\sR^m_{\cone_m},g,\mu_m)$ by the continuous
  mapping theorem. 


  \smallskip
  \noindent
  \textsl{Sufficiency} is derived from Theorem~\ref{thr:m-space} with
  $\eta_m=\psi_m \xi$. Condition~(i) holds by the assumption and the
  above argument. For (ii), we need to show that
  $\psi_m\mu\vto[\sR^\infty_\cone]\mu$ as $m\to\infty$. Consider a
  1-Lipschitz function $f\in\fC(\R^\infty,\sR^\infty_\cone)$, which is
  supported by a subset of $A\in\sR^\infty_\cone$. Without loss of generality, it
  is possible to assume that $\dmet(y,\cone)\geq \eps>0$ for all
  $y\in A$.  Since $\dmet(\pi_m(y),\cone_m)\leq \dmet(y,\cone)$,
  we have
  \begin{displaymath}
    \dmet(\psi_m(y),\cone)\leq
    \dmet(y,\cone)+\dmet(y,\psi_m(y))\leq 
    \dmet(\pi_m(y),\cone_m)+2^{-m}
    \leq \dmet(y,\cone)+2^{-m}.
  \end{displaymath}
  Therefore, since $A\subset\{x\colons \dmet(x,\cone)\ge\eps\}$, 
  \begin{displaymath}
    \psi_m^{-1}(A)\subset \{y\in\R^\infty\colons \dmet(\psi_m(y),\cone)\geq\eps\}
    \subset \{y\in\R^\infty\colons \dmet(y,\cone)\geq\eps/2\}\in\sR^\infty_\cone
  \end{displaymath}
  for all $m$ such that $2^{-m}\leq\eps/2$. For each $\delta>0$, we have
  $\dmet(x,\psi_m(x))\leq \delta$ for all $x\in\R^\infty$ and all
  sufficiently large $m$. Then
  \begin{align*}
    \Big|\int f(x)\mu(\bdiff x)
    &-\int f(\psi_m(x))\mu(\bdiff x)\Big|
      \leq \delta\mu(A\cup\psi_m^{-1}(A))\\
    &\leq \delta \big(\mu(A)+\mu(\psi_m^{-1}(A))\big)\\
    &\leq \delta \big(\mu((\cone^\eps)^c)+\mu((\cone^{\eps/2})^c)\big).
  \end{align*}
  Letting $\delta\downarrow 0$ confirms (ii). 
  
  Let $A\in\sR^\infty_\cone$. Then
  \begin{displaymath}
    \Prob{t^{-1}\xi\in A}\leq \Prob{t^{-1}\eta_m\in \pi_m A},
  \end{displaymath}
  so that (iii) follows from the regular variation of $\eta_m$ and the
  fact that $\pi_m A$ belongs to a non-trivial metric exclusion ideal
  $\sR^m_{\cone_m}$.  Finally, for all $\delta>0$,
  \begin{displaymath}
    \Prob{\dmet(t^{-1}\xi,t^{-1}\eta_m)>\delta}=0
  \end{displaymath}
  if $\delta>2^{-m}$, so that \eqref{eq:38} holds.
\end{proof}

If we exclude $\cone=\{0\}$, and so $\sR^\infty_\cone=\sR_0^\infty(1)$,
then Lemma~\ref{lemma:fidi-sequence} becomes Corollary~5.1.3 from
\cite{kul:sol20}. 

\begin{corollary}[see \protect{\cite{kul:sol20}}]
  We have $\xi\in
  \RV(\R^\infty,\mydot,\sR_0^\infty(1),g,\mu)$ for a random sequence
  $\xi$ if and only if all
  finite-dimensional distributions of $\xi$ are regularly varying,
  namely there exists a normalising function $g$ such that
  $(\xi_1,\dots,\xi_m)\in \RV(\R^m,\mydot,\sR_0^m,g,\mu_m)$ for all
  sufficiently large $m$.  
\end{corollary}

In contrast to the cited work, the above formulation imposes the
regular variation property on finite-dimensional distributions of
all sufficiently high orders. Indeed, it is possible for a finite
part of a regularly varying sequence  not to be regularly varying, e.g.,
if it is a constant.

\begin{example}
  \label{ex:increasing-sup}
  Let $N$ be any random variable with values in $\NN$, and let
  $\eta\in\RV(\R_+,\mydot,\sR_0,g,\theta_\alpha)$ be independent of
  $N$. Then $\xi=\eta e_N$ (that is, the sequence with $\eta$ at the
  $N$-th place and zeros elsewhere) is regularly varying in $\R^\infty$
  with the ideal $\sR_0^\infty(1)$, and its tail measure is given by  
  $\sum_{k=1}^\infty \theta_\alpha(\bdiff x_k)\Prob{N=k}$. The projections
  $\pi_m\xi$ are regularly varying on $\sR_0^m$ for all $m$ such that
  $\Prob{N\leq m}>0$ (to avoid a trivial tail measure).
  However, the regular variation property may fail if $N$
  and $\eta$ are dependent, for instance, if $N$ is the integer part
  of $\eta$. 
\end{example}

\begin{example}
  If $\xi\in\RV(\R^\infty,\mydot,\sR_0^\infty(k),g,\mu)$ for some
  $k\in\NN$ and the tail measure $\mu$ is non-trivial on the smaller
  ideal $\sR_{\inf}^\infty$, then $\xi$ is also regularly varying on
  the ideal $\sR_{\inf}^\infty$, see
  Proposition~\ref{prop:sub-ideal}(i). This happens if the components
  of $\xi$ are strongly dependent, so that the tail measure of $\xi$
  gives positive mass to sequences $\{x\colons \inf |x_i|>\eps\}$ for some
  $\eps>0$. For instance, this is the case for a random sequence
  $\xi=(\eta,\zeta,\zeta,\ldots)$ composed of a regularly varying
  random variable $\eta$ followed by components  all equal
  to another independent random variable $\zeta$.
\end{example}

\begin{example}[Cumsum operation]
  \index{cumsum operation}
  Consider a map $\psi:\R_+^\infty\to\R_+^\infty$ given by
  cumulative sums \index{cumulative sums}
  $(\psi(x))_i=x_1+\cdots+x_i$. This is a bornologically consistent
  continuous morphism from $\R_+^\infty$ with the ideal $\sR_0^\infty(1)$
  to $\R_+^\infty$ with the ideal $\sR_0^\infty(1)$ (and thus for all
  smaller ideals on the target space such as $\sR_0^\infty(k)$ for any
  $k\in\NN$ or $\sR_{\inf}^\infty$). If
  $\xi\in\RV(\R_+^\infty,\mydot,\sR_0^\infty(1),g,\mu)$ and $\psi\mu$
  is non-trivial on $\sR_0^\infty(k)$, then
  $\psi(\xi)\in\RV(\R_+^\infty,\mydot,\sR_0^\infty(k),g,\psi\mu)$.  In
  this way, we recover the observation from
  \cite[Section~4.5.1]{lin:res:roy14}. Note that $\psi$ is not
  bornologically consistent if the original space is equipped with the
  ideal $\sR_0^\infty(k)$ for any $k\geq2$. Indeed,
  $\psi^{-1}(\{x\colons \inf x_i\geq 1\})$ contains the set
  $[1,\infty)\times\R_+^\infty$, which is not in $\sR_0^\infty(2)$.
\end{example}

\paragraph{Stationary sequences and tail processes}
Consider double-sided, real-valued sequences $x=(x_i)_{i\in\ZZ}$ with
the topology of pointwise convergence. For such a sequence $x\in\R^\ZZ$,
define $\modulus(x)=|x_0|$. This function is a modulus, which
vanishes on sequences with $x_0=0$, which include many non-zero
sequences $x$ and thus is not proper. This 
modulus defines the ideal $\sX_{x_0}$, which is topologically and
scaling consistent, and the union set $\UU$ is the set of sequences
with $x_0\neq 0$. Further ideals on the space $\R^{\ZZ}$ are
constructed analogously to their counterparts on $\R^\infty$, e.g.,
$\sR_0^\ZZ(1)$ is the metric exclusion ideal obtained by excluding the
zero sequence or, equivalently, the ideal generated by the countable
family of moduli $\modulus_i(x)=|x_i|$, $i\in\ZZ$, see
Lemma~\ref{lemma:ideal-sequences}. 

Consider the shift operator $(\shift_k x)=(x_{i-k})_{i\in\ZZ}$ for
$k\in\ZZ$.  A probability measure $\nu$ on $\R^\ZZ$ is said to be
\emph{stationary} if it is invariant under $\shift_k$ for all
$k\in\ZZ$.
\index{shift operator}
By Theorem~\ref{thr:polar-decomp-4}, a random sequence $\xi$ is
regularly varying on $\sX_{x_0}$ with tail measure $\mu$ if and only
if \index{stationary random sequence}
\begin{displaymath}
  \Prob{T_{t^{-1}}\xi\in \cdot \mid |\xi_0|>t}
  \vto \frac{\mu(\cdot\cap\{\modulus>1\})}{\mu(\{\modulus>1\})}
  \quad \text{as}\; t\to\infty.
\end{displaymath}
If a stationary random sequence is regularly varying in this setting,
then the limiting measure on the right-hand side determines the distribution of
the \emph{tail process} as defined in \cite{bas:seg09}.
\index{tail process}
The following is a part of Theorem~2.1 from \cite{bas:seg09},
formulated in terms of ideals and with a proof not relying on the
Prokhorov theorem and is therefore easily extensible to sequences with values
in fairly general topological spaces.

\begin{proposition}
  \label{prop:stationary-sequence}
  Let $\xi$ be a stationary random sequence. Then $\xi$ is regularly
  varying on the ideal $\sX_{x_0}$ if and only if it is regularly
  varying on the ideal $\sR_0^\ZZ(1)$.
\end{proposition}
\begin{proof}
  \textsl{Sufficiency} follows from the continuous mapping theorem,
  since $\sX_{x_0}$ is a subideal of $\sR_0^\ZZ(1)$ and the tail
  measure is necessarily non-trivial on $\sX_{x_0}$ due to stationarity. 

  \smallskip
  \noindent
  \textsl{Necessity.} We apply
  Proposition~\ref{prop:extension-invariant} to the family
  $\mathsf{G}$ of transformations $\shift_k$, $k\in\ZZ$, and notice
  that the ideal $\sR_0^\ZZ(1)$ is the smallest ideal that contains
  all images of $\sX_{x_0}$ under all shifts. This also follows from
  Proposition~\ref{prop:cumulative-moduli}, which derives the tail
  measure for finite-dimensional distributions from the regular
  variation property on the ideals generated by $\modulus(x)=|x_i|$. 
\end{proof}

\paragraph{Spaces of $p$-summable sequences}
For $p\in[1,\infty]$, consider the space $\ell_p$ of $p$-summable
(bounded if $p=\infty$) sequences with the $\ell_p$-norm.
\index{p-summable sequences@$p$-summable sequences}
Note that $\ell_\infty$ is a non-separable metric space.

The ideals $\sR_0^\infty(k)$, $k\in\NN$, on $\R^\infty$ induce ideals
on $\ell_p$ by taking the intersection of their members with $\ell_p$; so
we keep the same notation $\sR_0^\infty(k)$ for such ideals on
$\ell_p$. Let $\sS_0$ be the ideal on $\ell_p$ generated by the norm
if $p\in[1,\infty]$, which is equivalently defined as
the metric exclusion ideal obtained
by excluding the zero sequence.
While $\sR_0^\infty(1)\subset\sS_0$,
the reverse inclusion fails. Indeed, the set $A=\{e_1,e_2,\ldots\}$,
which consists of all basis vectors, belongs to $\sS_0$ but not to
$\sR_0^\infty(1)$. Further metric exclusion ideals on $\ell_p$ can be
obtained by excluding any other cone; for instance, the family $\cone_k$ of
sequences with at most $(k-1)$ non-zero entries for $k\geq 2$. Unlike
Lemma~\ref{lemma:ideal-sequences}, these ideals do not coincide with
the ideal $\sR_0^\infty(k)$ in $\ell_p$. 

\begin{proposition}
  Let $1\leq p_1<p_2\leq\infty$.  If a random sequence $\xi$ is regularly
  varying in $\ell_{p_1}$ with the ideal $\sS_0$, then $\xi$ is regularly
  varying in $\ell_{p_2}$ with the same ideal $\sS_0$.
\end{proposition}
\begin{proof}
  The result follows from applying the continuous mapping theorem to the
  natural embedding of $\ell_{p_1}$ into $\ell_{p_2}$; see, e.g.,
  \cite[Proposition~1.5.2]{chen:mas:ros20}. The bornological
  consistency follows from the fact that $\|x\|_{p_1}\geq\|x\|_{p_2}$
  and thus 
  \begin{displaymath}
    \{x\colons \|x\|_{p_2}\geq \eps\}
    \subset \{x\colons \|x\|_{p_1}\geq \eps\}. 
  \end{displaymath}
  Furthermore, the pushforward of a non-trivial measure under the
  embedding remains non-trivial. 
\end{proof}

\begin{theorem}
  \label{thr:rv-ellp}
  Let $\xi$ be a random element in $\ell_p$ with $p\in[1,\infty]$.
  Then $\xi\in\RV(\ell_p,\mydot,\sS_0,g,\mu)$ if and only if the
  following three conditions hold.
  \index{regular variation!in ellp@in $\ell_p$}
  \begin{enumerate}[(i)]
  \item $\xi\in\RV(\R^\infty,\mydot,\sR_0^\infty(1),g,\mu)$.
  \item $\|\xi\|_p\in\RV(\R_+,\mydot,\sR_0,g,c\theta_\alpha)$, where
    $c\theta_\alpha$ is the pushforward of $\mu$ under the map
    $x\mapsto\|x\|_p$. 
  \item For all $\delta>0$, 
    \begin{equation}
      \label{eq:32}
      \lim_{m\to\infty} \limsup_{t\to\infty}
      \frac{\Prob{\|\xi-\psi_m(\xi)\|_p> t\delta}}
      {\Prob{\|\xi\|_p>t}}=0.
    \end{equation}
  \end{enumerate}
\end{theorem}
\begin{proof}
  \textsl{Necessity.} The embedding map from $\ell_p$ into $\R^\infty$ is
  continuous and $(\sS_0,\sR_0^\infty(1))$-bornologically consistent,
  since $\{x\colons \max_{i=1,\dots,m} |x_i|>\eps\}$ is a subset of
  $\{x\colons \|x\|_p>\eps\}$. To apply the continuous mapping theorem, we need to
  show that the pushforward of $\mu$ under this embedding map is
  non-trivial. Fix an $\eps>0$ such that $\mu(A)>0$ for 
  $A=\{x\in\ell_p\colons \|x\|_p>\eps\}$. Then
  \begin{displaymath}
    A_m=\{x\in\ell_p\colons \|\psi_m(x)\|_p> \eps\}
  \end{displaymath}
  is an increasing sequence of sets whose union is $A$, so that
  $\mu(A_m)>0$ for all sufficiently large $m$. It remains to note that
  $A_m$ belongs to $\sR_0^\infty(1)$, since
  \begin{displaymath}
    \max_{i=1,\dots,m}|x_i|>\eps m^{-1/p}
  \end{displaymath}
  for all $x\in A_m$. 

  \smallskip
  \noindent
  \textsl{Sufficiency}
  follows from Corollary~\ref{cor:m-space} and 
  Theorem~\ref{thr:m-space} with $\Gamma=\NN$ and
  $\psi_\gamma=\psi_m$. Condition~(i) of Theorem~\ref{thr:m-space}
  follows from the assumed regular variation on
  $\sR_0^\infty(1)$ by taking projections, see
  Lemma~\ref{lemma:fidi-sequence}. The maps 
  $\psi_m$ are uniformly bornologically consistent, since
  $\|\psi_m(x)\|_p\leq \|x\|_p$. Condition~(iii) from
  Theorem~\ref{thr:m-space} follows from  
  the regular variation of $\|\xi\|_p$. Finally, (iv) therein is equivalent to
  \eqref{eq:32} in view of the fact that $g(t)$ is of the same order as
  $\Prob{\|\xi\|_p>t}$. 
\end{proof}

If a sequence $\xi$ is stationary, it is possible to replace 
condition (i) with the requirements listed in
Proposition~\ref{prop:stationary-sequence}.
Condition~(iii) in Theorem~\ref{thr:rv-ellp} is essential, as the
following example shows.

\begin{example}
  Let $V$ be a random variable with values in $\{2,3,\ldots\}$ and
  with a Pareto tail.
  Consider the sequence $\xi$ with $\xi_1=V$, $\xi_V=V$ and $\xi_i=0$
  for $i\neq 1,V$.  This sequence $\xi$ is regularly varying in
  $\R^\infty$ but is not regularly varying in $\ell_p$ for any $p\in[1,\infty]$.
\end{example}

Another example, showing that (iii) is essential in $\ell_2$, is given
in \cite[Proposition~3.2]{MR4745556}. It follows from Proposition~3.3
therein that, to ensure regular variation in $\ell_2$
(equivalently, in any separable Hilbert space), condition (i) can be
combined with a relative compactness assumption on
$g(t)\Prob{t^{-1}\xi\in \cdot}$.

\begin{example}
  Let $\eta=(\eta_n)_{n\in\NN}$ be a sequence of i.i.d.\ regularly
  varying positive random variables with tail index $\alpha=1$, that is,
  $\eta_1\in\RV(\R_+,\mydot,\sR_0,g,\theta_1)$, where
  \begin{equation}
    \label{eq:31}
    g(t)=\frac{1}{\Prob{\eta_1>t}}, \quad t>0.
  \end{equation}
  It is known from \cite[Section~4.5.1]{lin:res:roy14} that the
  sequence $\eta$ is regularly varying in $\R_+^\infty$ with the ideal
  $\sR_0^\infty(1)$, and its tail measure $\mu$ is supported by
  sequences with only one non-zero element.  However, the sequence
  $\eta$ does not belong to the space $\ell_p$ for any
  $p\in[1,\infty]$. Consider the sequence $\zeta=(\zeta_n)$ given by
  \begin{displaymath}
    \zeta_n=a_n\eta_n,\quad n\geq1,
  \end{displaymath}
  with a normalising sequence $a=(a_n)$ of positive numbers such that
  $\sum a_n^\delta<\infty$ for some $\delta\in (0,1)$. Fix a
  $p\in[1,\infty)$. By Corollary~4.2.1 and Section~15.3 from 
  \cite{kul:sol20} or Lemma A.3 from \cite{MR1789987},
  $\|\zeta\|_p^p=\sum (a_n\eta_n)^p$ 
  converges almost surely and is regularly varying with tail index
  $1/p$.
  Indeed, since \(E\eta_1^\delta<\infty\) and \(\sum a_n^\delta<\infty\),
  we have \(\sum (a_n\eta_n)^\delta<\infty\) almost surely, which implies
  \(\sum (a_n\eta_n)^p<\infty\) for every \(p\ge1\).
  Hence, $\zeta\in\ell_p$ and the norm $\|\zeta\|_p$ is
  regularly varying with index 1. With the normalising function given in
  \eqref{eq:31}, the tail measure of $\|\zeta\|_p$ is
  $\|a\|_p^p\theta_1$.

  By Lemma~\ref{lemma:fidi-sequence}, the sequence $\zeta$ is
  regularly varying in $\R_+^\infty$ with the ideal $\sR_0^\infty(1)$, and
  the tail measure is
  \begin{displaymath}
    \mu(\bdiff x)=\sum_{n=1}^\infty \theta_1(a_n^{-1} dx_n)\prod_{i\neq
      n}\delta_0(\bdiff x_i), 
  \end{displaymath}
  where $\delta_0$ is the Dirac measure at zero. By
  \cite[Corollary~4.2.1]{kul:sol20}, the random variable $\big(\sum_{i=m+1}^\infty
  (a_i\eta_i)^p\big)^{1/p}$
  is regularly varying with the tail measure
  $(\sum_{i=m+1}^\infty a_i)\theta_1$. Thus, \eqref{eq:32} holds with
  \begin{displaymath}
    \limsup_{t\to\infty}
    \frac{\Prob{\|\zeta-\psi_m(\zeta)\|_p> t\delta}}
    {\Prob{\|\zeta\|_p>t}}
    =\delta^{-1}
    \frac{\sum_{i=m+1}^\infty a_i}{\sum_{i=1}^\infty a_i},
  \end{displaymath}
  and Theorem~\ref{thr:rv-ellp} yields that $\zeta$ is regularly
  varying in $\ell_p$.
\end{example}

\paragraph{Spaces $c_0$ and $c_{00}$}
Let $c_0$ be the space of
real-valued sequences $(x_i)_{i\in\NN}$ that
converge to zero as $i\to\infty$, equipped with the uniform
distance.  Furthermore, let $c_{00}$ be the space of sequences
$\{(x_1,\dots,x_n,0,\ldots)\colons x_i\in\R, n\geq 1\}$ with only
finitely many non-zero
entries viewed as a subspace of 
$\ell_\infty$ with the $\ell_\infty$-norm. Unlike
\cite[Section~4.1]{janssen23}, it is not necessary to consider the 
completion of the space $c_{00}$, since our technique does not
rely on any completeness assumption.
\index{space c0@space $c_0$}

Let $\sX_0$ be the ideal on $c_0$ or $c_{00}$ obtained by excluding
the zero sequence, equivalently, the ideal generated by the
$\ell_\infty$-norm. If the conditions of Theorem~\ref{thr:rv-ellp}
hold, then a random element $\xi$ in $c_0$ or $c_{00}$ is regularly
varying, provided that the tail measure is supported by the
corresponding space and is non-trivial on it.

\section{Continuous functions}
\label{sec:continuous-functions}

\paragraph{Finite-dimensional distributions}
In this and subsequent sections, we consider spaces of real-valued
functions on $[0,1]$ or on $\R$ with various topologies.  Unless
stated otherwise, we apply linear scaling, i.e.\ the conventional
scaling of function values. We then write $t^{-1}x$ instead of
$T_{t^{-1}}x$ for a function $x(u)$, $u\in\II$, where $\II$ denotes
the definition domain. It is straightforward to extend all results to
functions with values in $\R^d$ and, furthermore, to functions with
values in any topological space with continuous scaling, using the
tools developed previously for general spaces. It is also
possible to consider functions parametrised by $\R^d$ or general
Polish spaces.

Let $\II$ be the parameter domain of a random function $\xi:\II\to\R$,
usually $\II=[0,1]$ or $\II=\R$.

\begin{definition}
  \label{def:fidi}
  A real-valued random function $\xi$ on $\II$ is said to be
  \emph{finite-dimensional regularly varying} if there exist a
  measurable function 
  $g:\Rpp\to\Rpp$ and a family of measures
  $\mu_{u_1,\dots,u_m}\in\Mb[\sR_0^m]$ parametrised by
  $u_1,\dots,u_m\in \II$ and $m\in\NN$, which are non-trivial for some
  $m\in\NN$ and $u_1,\dots,u_m\in \II$, such that
  \index{regular variation!finite-dimensional}
  \index{finite-dimensional regular variation}
  \begin{equation}
    \label{eq:27v}
    \big(\xi(u_1),\dots,\xi(u_m)\big)
    \in\RV(\R^m,\mydot,\sR_0^m,g,\mu_{u_1,\dots,u_m})
  \end{equation}
  for all $u_1,\dots,u_m\in \II$ and $m\in\NN$. 
\end{definition}

While the measures $\mu_{u_1,\dots,u_m}$ are defined on subsets of
$\R^m$, we tacitly use the same notation for the corresponding cylindrical
measures, whenever the set under consideration is cylindrical with
respect to the coordinates $u_1,\dots,u_m$, that is,
\begin{displaymath}
  \mu_{u_1,\dots,u_m}(A)
  =\mu_{u_1,\dots,u_m}\big(\{(x(u_1),\dots,x(u_m))\colons x\in A\}\big).
\end{displaymath}
Recall that the product ideal $\sR_0^m$ consists of all sets $A$ in
$\R^m$ such that the $\inf\{\modulus_m(y)\colons y\in A\}>0$, where
\begin{equation}
  \label{eq:50}
  \modulus_m(y_1,\dots,y_m)=\max_{i=1,\dots,m} |y_i|=\|y\|_\infty,
  \quad y=(y_1,\dots,y_m)\in\R^m.
\end{equation}
Thus, $A\times\R\in \sR_0^{m+1}$ whenever $A\in\sR_0^m$. Because of this and
by construction, the tail measures $\mu_{u_1,\dots,u_m}$ form a
consistent family. Extensions of infinite measures (including Radon
measures on topological spaces) have been constructed in
\cite{MR499063} and \cite{MR0404570}. Such a construction yields a
$\sigma$-finite measure on the cylindrical $\sigma$-algebra on
$\R^\II$ whose projections are the given finite-dimensional
distributions. If $\xi$ takes values in $\R^d$, then \eqref{eq:27v} is
amended by replacing $\sR_0^m$ with $(\sR_0^d)^m$, which is the
product of $m$ ideals $\sR_0^d$.

The definition of finite-dimensional regular variation treats
finite-dimensional distributions of order $m$ in relation to the ideal 
$\sR_0^m$, which is generated by the modulus given in 
\eqref{eq:50}. Theorem~\ref{thr:polar-decomp-4} implies that
\eqref{eq:27v} holds if and only if
\begin{multline*}
  \Prob{t^{-1}\big(\xi(u_1),\dots,\xi(u_m)\big)\in\cdot \,\mid\,
    \modulus_m \big(\xi(u_1),\dots,\xi(u_m)\big)>t}\\
  \wto\frac{\mu_{u_1,\dots,u_m}(\cdot\cap\{\modulus_m>1\})}
  {\mu_{u_1,\dots,u_m}(\{\modulus_m>1\})} \quad \text{as}\; t\to\infty.
\end{multline*}
It is possible to impose weak convergence above,
conditionally on large values of other moduli, possibly depending on
the values of $\xi$ outside of $u_1,\dots,u_m$ or even on its entire
path.

A natural choice for the reference modulus is given by the value
of $\xi$ at a point $u_0$ fixed in the parameter domain $\II$. Then
$\xi$ is said to be finite-dimensional regularly varying in relation
to the modulus $\modulus_{u_0}(\xi)=|\xi(u_0)|$ if
\begin{equation}
  \label{eq:27x}
  \Prob{t^{-1}\big(\xi(u_1),\dots,\xi(u_m)\big)\in\cdot \,\mid\,
    |\xi(u_0)|>t}
  \wto\frac{\mu_{u_1,\dots,u_m}(\cdot\cap\{x\colons |x(u_0)|>1\})}
  {\mu_{u_1,\dots,u_m}(\{x\colons |x(u_0)|>1\})}
\end{equation}
and $\mu_{u_1,\dots,u_m}(\{x\colons |x(u_0)|>1\})\in(0,\infty)$. 
This choice is particularly relevant if $\II=\R$ or $\II=\ZZ$ and
$\xi$ is \emph{stationary}. In this case, it is possible to set $u_0=0$ and
refer to \cite[Theorem~2.1]{bas:seg09}, which establishes the
equivalence of finite-dimensional regular variation as specified
in Definition~\ref{def:fidi} and the property \eqref{eq:27x} with
$u_0=0$, when combined with the regular variation of $\xi(0)$ and the
stationarity assumption, see also
Propositions~\ref{prop:cumulative-moduli} and
\ref{prop:C-stationary}. The limiting probability measure on the
right-hand side of \eqref{eq:27x} determines the distribution of the
\emph{tail process} of $\xi$.
\index{tail process}

\paragraph{Uniform metric}
Let $\Cont([0,1])$ be the (Polish) space of continuous functions
$x:[0,1]\to\R$ with the uniform metric. The study of regular
variation on this space was initiated in \cite{haan-lin2001}; see
\cite{kim:kok24} for recent work concerning applications of this
concept.
All papers on this subject, following \cite{haan-lin2001}, endow the
carrier space with the ideal generated by the modulus
$\modulus(x)=\|x\|_\infty$ (implicitly present in all
constructions), which is the uniform norm of the function $x$. This is
also the metric exclusion ideal $\sC_0$ obtained by excluding the zero
function. In this case, Condition~\hyperref[condS]{(S)} holds and, by
Proposition~\ref{prop:polar-decomp}(ii) (which corresponds to
\cite[Corollary~2.13]{haan-lin2001}), a random continuous function
$\xi$ is regularly varying if and only if
$\|\xi\|$ is regularly varying and the conditional distribution
of $\xi/\|\xi\|_\infty$, given that $\|\xi\|_\infty>t$, converges weakly
to a probability measure on
$\SS_\modulus=\{x\in \Cont([0,1])\colons \|x\|_\infty=1\}$.

The following result is another criterion for regular variation on
$\sC_0$. For a continuous function $x$ on $[0,1]$, define
\begin{displaymath}
  \varpi_\eps(x)=\sup_{u_1,u_2\in[0,1], |u_1-u_2|\leq \eps}
  |x(u_1)-x(u_2)|,\quad \eps>0.  
\end{displaymath}
Let $\Gamma$ be the net consisting of all finite strictly increasing
sequences 
\begin{displaymath}
  \gamma=(u_1,\dots,u_m)\subset[0,1], \quad m\in\NN,
\end{displaymath}
ordered by the
inclusion of the corresponding sets $\{u_1,\dots,u_m\}$.
Denote $\#\gamma=m$ and
\begin{displaymath}
  |\gamma|=\max_{i=1,\dots,m+1}(u_i-u_{i-1}),
\end{displaymath}
where $u_0=0$ and $u_{m+1}=1$.  Furthermore, denote by $\pi_\gamma$
the projection of $x\in\Cont([0,1])$ onto
\begin{displaymath}
  \pi_\gamma x=(x(u_1),\dots,x(u_m))\in\R^{\#\gamma}.
\end{displaymath}
Recall the modulus $\modulus_m$ defined in \eqref{eq:50}, which
generates the ideal $\sR_0^m$ on $\R^m$.  For $m\in\NN$, denote
$B_m=\{\modulus_m>1\}$.
Then $\modulus_\gamma(x)=\modulus_{\#\gamma}(\pi_\gamma x)$ is a
modulus on $\Cont([0,1])$, which generates the ideal $\sC_\gamma$ on
$\Cont([0,1])$.  Note that $\modulus_\gamma(x)>1$ if and only if
$\pi_\gamma x\in B_{\#\gamma}$.

While an analogue of the following result for the space $\DD([0,1])$
of c\`adl\`ag functions is well known from, e.g.,
\cite[Theorem~10]{hul:lin05}, it 
does not seem that the corresponding variant for continuous functions has been
explicitly formulated in the literature as presented below.

\begin{theorem}
  \label{thr:C}
  We have $\xi\in\RV(\Cont([0,1]),\mydot,\sC_0,g,\mu)$ if and only if
  there exist a $\gamma_0\in\Gamma$ and a measurable function $g$ such
  that the following two conditions are satisfied.
  \index{regular variation!of random continuous functions}
  \begin{enumerate}[(i)]
  \item For each $\gamma\geq\gamma_0$, we have
    \begin{equation}
      \label{eq:pigammaRV}
      \pi_\gamma\xi\in\RV(\R^{\#\gamma},\mydot,\sR_0^{\#\gamma},g,\mu_\gamma), 
    \end{equation}
    where $(\mu_\gamma)_{\gamma\geq\gamma_0}$ is a family of
    (non-trivial) tail measures, and 
    \begin{equation}
      \label{eq:c-lim}
      \vartheta=\sup_{\gamma\geq\gamma_0} \mu_\gamma(B_{\#\gamma})<\infty.
    \end{equation}
  \item For all $\delta>0$,
    \begin{equation}
      \label{eq:27d}
      \lim_{\eps\downarrow 0} \limsup_{t\to\infty}
      g(t)\Prob{\varpi_\eps(\xi)>t\delta}=0. 
    \end{equation}
  \end{enumerate}
  If (i) and (ii) hold, then $\mu_\gamma=\pi_\gamma\mu$, where
  $\mu\in\Mb[\sC_0]$ is the tail measure of $\xi$, and
  $\|\xi\|_\infty\in\RV(\R_+,\mydot,\sR_0,g,\vartheta\theta_\alpha)$,
  where $\vartheta$ is defined in \eqref{eq:c-lim}.
\end{theorem}
\begin{proof}
  \textsl{Necessity.} Assume that
  $\xi\in\RV(\Cont([0,1]),\mydot,\sC_0,g,\mu)$.  Statement (i)
  follows from the continuous mapping theorem. Indeed, the map
  $x\mapsto\pi_\gamma x$ is continuous and bornologically consistent,
  since
  \begin{displaymath}
    \big\{x\in\Cont([0,1])\colons \modulus_{\#\gamma}(\pi_\gamma x)>\eps\big\}
    \subset \big\{x\in\Cont([0,1])\colons \|x\|_\infty>\eps\big\}\in\sC_0.
  \end{displaymath}
  
  The projection $\mu_\gamma=\pi_\gamma\mu$ becomes the tail measure
  of $\xi_\gamma$. The projected tail measures are non-trivial for all
  sufficiently fine $\gamma$. Indeed,
  \[
    \{x:\|x\|_\infty>1\}
    =
    \bigcup_{q\in\QQ\cap[0,1]}\{x:|x(q)|>1\},
  \]
  and the left-hand side has positive $\mu$-measure. Hence there
  exists $q\in\QQ\cap[0,1]$ such that $\mu(\{x:|x(q)|>1\})>0$.  Taking
  $\gamma_0=\{q\}$, we obtain $(\pi_\gamma\mu)(B_{\#\gamma})>0$ for
  all $\gamma\geq\gamma_0$.  Furthermore, $\mu_\gamma$ is homogeneous,
  and thus $B_{\#\gamma}$ is a $\pi_\gamma\mu$-continuity set. Hence,
  \begin{multline*}
    (\pi_\gamma\mu)(B_{\#\gamma})
    =\lim_{t\to\infty} g(t)\Prob{\|\xi_\gamma\|_\infty>t}\\
    \leq \lim_{t\to\infty} g(t)\Prob{\|\xi\|_\infty>t}
    = \mu\big(\{x\colons \|x\|_\infty>1\}\big). 
  \end{multline*}
  Thus, \eqref{eq:c-lim} holds with
  $\vartheta\leq \mu(\{x\colons \|x\|_\infty>1\})$. 
  
  Fix any $\delta>0$.  Since the distributions of $t^{-1}\xi$
  conditional on $\|\xi\|_\infty>\delta t/2$ weakly converge by
  Theorem~\ref{thr:polar-decomp-4} and thus form a weakly relatively
  compact family, the Prokhorov theorem, together with the criterion
  for relative compactness in $\Cont([0,1])$ (see, e.g.,
  \cite[Theorem~7.3]{Bi99}) imply that
  \begin{displaymath}
    \lim_{\eps\downarrow 0} \limsup_{t\to\infty}
    \Prob{\varpi_\eps(\xi)>t\delta\,\mid\,\|\xi\|_\infty>\delta t/2}=0.
  \end{displaymath}
  Since $\|x\|_\infty\geq \varpi_\eps(x)/2$, we have
  \begin{align*}
    \Prob{\varpi_\eps(\xi)>t\delta\,\mid\,\|\xi\|_\infty>\delta t/2}
    &=\frac{\Prob{\varpi_\eps(\xi)>t\delta,\;\|\xi\|_\infty>\delta t/2}}
      {\Prob{\|\xi\|_\infty>\delta t/2}}\\
    &=\frac{\Prob{\varpi_\eps(\xi)>t\delta}}
      {\Prob{\|\xi\|_\infty>\delta t/2}}.
  \end{align*}
  Finally, \eqref{eq:27d} holds, since
  $g(t)\Prob{\|\xi\|_\infty>\delta t/2}$ converges to a finite
  positive limit as $t\to\infty$ by the regular variation of
  $\xi$.

  \smallskip
  \noindent
  \textsl{Sufficiency.}
  Without loss of generality, assume that the imposed conditions hold
  for all $\gamma\in\Gamma$. Note that
  \begin{equation}
    \label{eq:45}
    \big|\|x\|_\infty-\modulus_\gamma(\pi_\gamma x)\big|\leq \varpi_{|\gamma|}(x).
  \end{equation}
  We prove first that
  \begin{displaymath}
    \|\xi\|_\infty\in\RV(\R_+,\mydot,\sR_0,g,\vartheta\theta_\alpha).
  \end{displaymath}
  For each $a>0$,
  \begin{displaymath}
    g(t)\Prob{\|\xi\|_\infty>at}
    \geq g(t)\Prob{\modulus_\gamma(\xi)>at}.
  \end{displaymath}
  Hence,
  \begin{displaymath}
    \liminf_{t\to\infty}\;
    g(t)\Prob{\|\xi\|_\infty>at}
    \geq a^{-\alpha}\mu_\gamma(B_{\#\gamma}).
  \end{displaymath}
  Furthermore, for each $\delta\in(0,a)$, \eqref{eq:45} yields that, 
  \begin{align*}
    g(t)\P\big\{\|\xi\|_\infty&>at\big\}
    \leq g(t)\Prob{\modulus_\gamma(\xi)>at-\varpi_{|\gamma|}(\xi)}\\
    &\leq g(t)\Prob{\modulus_\gamma(\xi)>at-\varpi_{|\gamma|}(\xi),
      \varpi_{|\gamma|}(\xi)\leq \delta t}
      +g(t)\Prob{\varpi_{|\gamma|}(\xi)>\delta t}\\
    &\leq g(t)\Prob{\modulus_\gamma(\xi)>(a-\delta)t}
      +g(t)\Prob{\varpi_{|\gamma|}(\xi)>\delta t}.
  \end{align*}
  By the regular variation of $\modulus_\gamma(\xi)$,
  \begin{displaymath}
    \limsup_{t\to\infty} g(t)\Prob{\|\xi\|_\infty>at}
    \leq (a-\delta)^{-\alpha}\mu_\gamma(B_{\#\gamma})
    +\limsup_{t\to\infty}g(t)\Prob{\varpi_{|\gamma|}(\xi)>\delta t}.
  \end{displaymath}
  Thus,
  \begin{multline*}
    a^{-\alpha}\mu_\gamma(B_{\#\gamma})
    \leq \liminf_{t\to\infty} g(t)\Prob{\|\xi\|_\infty>at}
    \leq \limsup_{t\to\infty} g(t)\Prob{\|\xi\|_\infty>at}\\
    \leq (a-\delta)^{-\alpha}\mu_\gamma(B_{\#\gamma})
    +\limsup_{t\to\infty}g(t)\Prob{\varpi_{|\gamma|}(\xi)>\delta t}.
  \end{multline*}
  Both sides converge to $a^{-\alpha}\vartheta$, with $\vartheta$
  defined in \eqref{eq:c-lim}, by letting $|\gamma|\to0$ and then letting
  $\delta\to0$. Indeed, the quantities $\mu_\gamma(B_{\#\gamma})$ are
  increasing in $\gamma$, because the projected tail measures are
  consistent and $B_{\#\gamma}$ is enlarged under refinement of
  $\gamma$. Hence, by choosing refinements with $|\gamma|\to0$, their
  values increase to $\vartheta$. Thus,
  \begin{equation}
    \label{eq:43rv-xi}
    \lim_{t\to\infty} g(t)\Prob{\|\xi\|_\infty>at}=a^{-\alpha}\vartheta.
  \end{equation}
  
  Denote
  \begin{displaymath}
    U_a=\big\{x\in\Cont([0,1])\colons \|x\|_\infty>a\big\},\quad a>0.
  \end{displaymath}
  The sets $U_{1/n}$, $n\in\NN$, form a base for $\sC_0$.  Fix any
  $a>0$.  If $b\geq a$, then
  \begin{displaymath}
    \limsup_{t\to\infty}g(t)
    \Prob{\|\xi\|_\infty>bt,\|\xi\|_\infty> at}
    =b^{-\alpha}\vartheta
  \end{displaymath}
  by \eqref{eq:43rv-xi}.  Furthermore, for each $\delta>0$,
  \begin{multline*}
    \lim_{\eps\downarrow0}
    \limsup_{t\to\infty}
    g(t)\Prob{\varpi_\eps(t^{-1}\xi)>\delta,\|\xi\|_\infty>at}\\
    \leq \lim_{\eps\downarrow0}\limsup_{t\to\infty}
    g(t) \Prob{\varpi_\eps(\xi)>t\delta} =0.
  \end{multline*}
  By the Prokhorov theorem (see, e.g., \cite[Theorem~7.3]{Bi99}), for
  some $t_0>0$, the finite measures
  \begin{displaymath}
    g(t)\Prob{t^{-1}\xi\in\cdot, \|\xi\|_\infty> a t},\quad
    t\geq t_0,
  \end{displaymath}
  are relatively compact in the space of finite measures on $U_a$ and so
  the measures $g(t)\Prob{t^{-1}\xi\in\cdot}$ on $U_a$
  are relatively compact in the weak topology.
  Assume that they converge to a finite measure $\bar\mu_a$ along a
  subsequence $(b_n)_{n\in\NN}$, i.e.
  \begin{equation}
    \label{eq:27bg}
    g(b_n)\Prob{b_n^{-1}\xi \in \cdot}
    \wto \bar\mu_a(\cdot)\quad \text{as}\; n\to\infty. 
  \end{equation}
  For each $\gamma\in\Gamma$, the continuous mapping theorem implies
  that
  \begin{displaymath}
    g(b_n)\Prob{b_n^{-1}\xi_\gamma \in \cdot}
    \wto \pi_\gamma\bar\mu_a(\cdot)\quad \text{as}\; n\to\infty.
  \end{displaymath}
  Thus, $\mu_\gamma=\pi_\gamma\bar\mu_a$.  If $\bar\mu_a'$ is another
  limit in \eqref{eq:27bg}, then these limits have the same
  finite-dimensional distributions and thus coincide. Taking into
  account the
  regular variation of $\|\xi\|_\infty$, we have
  \begin{equation}
    \label{eq:27c}
    \Prob{t^{-1}\xi \in \cdot \mid \|\xi\|_\infty>at}
    \wto \tilde\mu_a(\cdot)\quad \text{as}\; t\to\infty,
  \end{equation}
  where $\tilde\mu_a$ is a probability measure on $U_a$.  Note that
  $\tilde\mu_{as}(sD)=\tilde\mu_a(D)$ for all $s>0$ and Borel sets
  $D\subset U_a$. Furthermore, $\tilde{\mu}_a(U_b)=(b/a)^{-\alpha}$
  for $b\geq a$, and thus $U_b$ is a $\tilde{\mu}_a$-continuity set
  for all $b\geq a$. Then
  \begin{align*}
    g(t)\P\{t^{-1}\xi\in D, &\|\xi\|_\infty>at\}
    =g(t)\Prob{t^{-1}\xi\in (D\cap U_a)}\\
    &=\Prob{t^{-1}\xi\in (D\cap U_a)\mid \|\xi\|_\infty>at}
      g(t)\Prob{\|\xi\|_\infty>at}\\
    &\to \vartheta a^{-\alpha} \tilde\mu_a(D\cap U_a)
    \quad \text{as}\; t\to\infty
  \end{align*}
  for each Borel set $D$ in $\Cont([0,1])$, which is a
  $\tilde\mu_a$-continuity set.  Denote
  \begin{displaymath}
    \mu'_a(D)=\vartheta a^{-\alpha} \tilde\mu_a(D\cap U_a), \quad a>0.
  \end{displaymath}
  If $0<a<b$ and $D\subset U_b$, then
  \begin{displaymath}
    g(t)\Prob{t^{-1}\xi\in D, \|\xi\|_\infty>at}
    =g(t)\Prob{t^{-1}\xi\in D, \|\xi\|_\infty>bt}.
  \end{displaymath}
  Thus, $\mu'_a(D)=\mu'_b(D)$ for all Borel sets
  $D\subset U_b$, that is, $\mu'_a(D)$ and $\mu'_b(D)$ agree
  on Borel subsets of $U_b$. By Lemma~\ref{lemma:enclosed}, there exists a
  measure $\mu$ on $\sC_0=\cup_{a>0}\{D\colons D\subset U_a\}$ that agrees
  with $\mu'_a$ on all subsets of $U_a$ for all $a>0$. In
  particular, $\mu(U_a)=\vartheta a^{-\alpha}$ for all $a>0$. Therefore,
  \begin{multline*}
    g(t)\Prob{t^{-1}\xi\in \cdot, \|\xi\|_\infty>t}\\
    =g(t)\Prob{t^{-1}\xi\in (\cdot\cap U_1), \|\xi\|_\infty>t}
    \wto
    \mu(\cdot\cap U_1) \quad \text{as}\; t\to\infty.    
  \end{multline*}
  Furthermore,
  \begin{displaymath}
    \mu(U_a)=\lim_{t\to\infty} g(t)\Prob{\|\xi\|_\infty>at}=\vartheta a^{-\alpha},
  \end{displaymath}
  and $\mu(U_1)=\vartheta$, so that we finally obtain
  \begin{displaymath}
    \Prob{t^{-1}\xi\in \cdot\,\mid\, \|\xi\|_\infty>t}
    \wto \frac{\mu(\cdot\cap\{\modulus>1\})}
    {\mu(\{\modulus>1\})} \quad \text{as}\; t\to\infty,
  \end{displaymath}
  where $\modulus(x)=\|x\|_\infty$. 
  By Theorem~\ref{thr:polar-decomp-4}, the random element $\xi$
  is regularly varying in $\Cont([0,1])$ with the ideal $\sC_0$.
\end{proof} 

The condition $\gamma\geq \gamma_0$ in Theorem~\ref{thr:C} is
explained by the fact that the values of $\xi$ at some points in
$[0,1]$ are not necessarily regularly varying, for instance, they may
be constant.

If $\xi$ is regularly varying on $\sC_0$, then the regular variation
property is transferred to smaller ideals, provided that the tail measure is
non-trivial on them, see Proposition~\ref{prop:sub-ideal}. Below we
give several examples of such subideals of $\sC_0$.  While the ideal
$\sC_0$ arises from the requirement that the supremum norm of a
function is at least a given positive number, further ideals can be
defined by considering lower bounds for the absolute value of the
function on subsets of 
$[0,1]$. Let $\sJ$ be any family of non-empty subsets of
$[0,1]$. Define the family of moduli parametrised by $J\in\sJ$ and
given by
\begin{displaymath}
  \modulus_J(x)=\inf_{u\in J} |x(u)|,\quad x\in \Cont([0,1]).
\end{displaymath}
This family generates a scaling-consistent ideal denoted by
$\sC_{\sJ}$.

For instance, if $\sJ$ consists of a single set $[0,1]$,
then this ideal (denoted by $\sC_{\inf}$) is generated by the modulus
\begin{displaymath}
  \modulus_{\inf}(x)=\inf\big\{|x(u)|\colons u\in[0,1]\big\}.
\end{displaymath}
Note that this modulus vanishes on some functions (which may be chosen
to be invariant under scaling) and
thus is not proper. The regular variation of $\xi$ on the ideal
$\sC_{\inf}$ necessarily entails that $\xi$ with positive probability
has no zeros in $[0,1]$, since otherwise $\modulus_{\inf}(\xi)=0$
almost surely.

If $\sJ$ consists of all singletons, then $\sC_{\sJ}$ is generated by
the moduli $\modulus_u(x)=|x(u)|$ for all $u\in[0,1]$. This ideal is
strictly smaller than the ideal $\sC_0$ generated by the norm. A
considerably more complicated choice is to let $\sJ$ be the family of
all segments of length at least $\delta>0$.  A different type of ideal on
$\Cont([0,1])$ is generated by the modulus
$\modulus(x)=\int_0^1 |x(u)|\diff u$. This is also a subideal of $\sC_0$.

\begin{example}[Hidden regular variation]
  \label{ex:function-hidden}
  \index{hidden regular variation!of continuous functions}
  We illustrate the hidden regular variation phenomenon on the space
  of continuous functions. Consider the ideals $\sC_0$ and
  $\sC_{\inf}$ on $\Cont([0,1])$.  The union of all sets in
  $\sC_{\inf}$ is the subcone of $\Cont([0,1])$, which consists of all
  continuous functions $x:[0,1]\to\R\setminus\{0\}$.  Consider a random
  continuous function
  \begin{displaymath}
    \xi(u)=V_1 u+V_2(1-u), \quad u\in[0,1],
  \end{displaymath}
  where
  $V_1,V_2$ are two independent Pareto(1) random variables. Note that
  $\modulus_{\inf}(\xi)=\min(V_1,V_2)$ is regularly varying with index
  $2$, and $\|\xi\|_\infty=\max(V_1,V_2)$ is regularly varying with
  index $1$. The random function $\xi$ is regularly varying on the
  ideal $\sC_0$ with normalising function $g(t)=t$ (which can be
  easily seen by applying the continuous mapping theorem) and its tail
  measure $\mu$ is supported by functions $au$ and $a(1-u)$,
  $u\in[0,1]$, where $a>0$.
  In fact, $\mu$ is the sum of the
  pushforwards of $\theta_1$ under two maps, one associating
  $z\in(0,\infty)$ with the function $zu$, $u\in[0,1]$,
  and the other one with $z(1-u)$,
  $u\in[0,1]$. This tail measure is supported by functions which are
  not strictly positive on $[0,1]$ and thus the tail measure vanishes on
  the ideal $\sC_{\inf}$.

  Choosing the normalising function $g(t)=t^2$ ensures that $\xi$ is
  regularly varying on the ideal $\sC_{\inf}$ with the tail measure
  supported by functions $au+b(1-u)$; this is the pushforward of
  $\theta_1^{\otimes 2}$ under the map from $(z_1,z_2)\in(0,\infty)^2$ to
  the function $z_1u+z_2(1-u)$. It is easy to extend this example to
  random broken lines or splines constructed from a regularly varying
  random vector.
\end{example}

The following sufficient criterion infers regular
variation in $\sC_{\inf}$ from the hidden regular variation property
of finite-dimensional distributions, which is defined with respect to
the ideals $\sR_0^m(m)$ generated by the moduli
\begin{displaymath}
  \modulus_{\min}(x_\gamma)=\min(|x(u_1)|,\dots,|x(u_m)|)
\end{displaymath}
for $\gamma=(u_1,\dots,u_m)$.

\begin{proposition}
  \label{prop:C-varpi-hidden}
  Let $\xi$ be a random element in $\Cont([0,1])$. Then
  \begin{displaymath}
    \xi\in\RV(\Cont([0,1]),\mydot,\sC_{\inf},g,\mu)
  \end{displaymath}
  if \eqref{eq:27d}
  holds and there exists a $\gamma_0\in\Gamma$ such that
  \begin{displaymath}
    \pi_\gamma\xi
    \in\RV(\R^{\#\gamma},\mydot,\sR^{\#\gamma}_0(\#\gamma),g,\mu_\gamma)
  \end{displaymath}
  for all $\gamma\geq \gamma_0$, where
  $(\mu_\gamma)_{\gamma\geq\gamma_0}$ is a family of non-trivial tail
  measures satisfying 
  \begin{displaymath}
    \vartheta_{\inf}=\inf_{\gamma\geq\gamma_0}
    \mu_\gamma(A_{\#\gamma})>0,
    \quad
    A_m=\{y\in\R^m\colons \modulus_{\min}(y)>1\}.
  \end{displaymath}
\end{proposition}
\begin{proof}
  The proof follows the scheme of the proof of sufficiency in 
  Theorem~\ref{thr:C}. With
  \begin{displaymath}
    0\leq \modulus_{\min}(\pi_\gamma x)-\modulus_{\inf}(x)
    \leq \varpi_{|\gamma|}(x)
  \end{displaymath}
  in place of \eqref{eq:45}, we derive that
  \begin{multline*}
    (a+\delta)^{-\alpha}\mu_\gamma(A_{\#\gamma})
    -\limsup_{t\to\infty} g(t)\Prob{\varpi_{|\gamma|}(\xi)>\delta t}
    \leq \liminf_{t\to\infty} g(t)\Prob{\modulus_{\inf}(\xi)>at}\\
    \leq \limsup_{t\to\infty} g(t)\Prob{\modulus_{\inf}(\xi)>at}
    \leq a^{-\alpha}\mu_\gamma(A_{\#\gamma}),
  \end{multline*}
  where $A_m=\{y\in\R^m\colons \modulus_{\min}(y)>1\}$.  By refining
  the finite point sets $\gamma$ and letting
  $|\gamma|\to0$ and then $\delta\to0$, we see that both sides
  converge to $a^{-\alpha}\vartheta_{\inf}$. Thus,
  $\modulus_{\inf}(\xi)\in\RV(\R_+,\mydot,\sR_0,g,\vartheta_{\inf}\theta_\alpha)$. 
  The proof is then concluded exactly as in the sufficiency part of
  Theorem~\ref{thr:C}, replacing the sets $U_a=\{\|x\|_\infty>a\}$ by
  \begin{displaymath}
    U_a=\big\{x\in\Cont([0,1])\colons \modulus_{\inf}(x)>a\big\},\quad a>0. \qedhere
  \end{displaymath}
\end{proof}

\begin{example}[Polynomials on the unit interval]
  \label{example:polynom}
  It is easy to construct a regularly varying random function by
  applying a homogeneous map to a regularly varying random vector. For
  instance, one may map a point $y=(y_0,\dots,y_d)\in\R^{d+1}$ to the
  polynomial
  \begin{equation}
    \label{eq:37}
    \psi(y)(u)=y_0+y_1u+\cdots+y_{d} u^{d}, \quad u\in[0,1].
  \end{equation}
  Then $\psi$ is a continuous morphism between $\R^{d+1}$ with 
  linear scaling and $\Cont([0,1])$, and the inverse image of the zero
  function is exactly the zero vector in $\R^{d+1}$. By
  Lemma~\ref{lemma:map-zero}, $\psi$ is bornologically consistent as a
  map from the ideal $\sR_0^{d+1}$ to $\sC_0$. The continuous mapping
  theorem implies that, if $\eta$ is a regularly varying random vector
  in $\R^{d+1}$ on the ideal $\sR_0^{d+1}$, then
  \begin{displaymath}
    \xi(u)=\eta_0+\eta_1u+\cdots+\eta_du^d, \quad u\in[0,1],
  \end{displaymath}
  is regularly varying on $\sC_0$. In particular, we deduce that
  $\|\xi\|_\infty$ is a regularly varying random variable. If $\eta$
  has independent components, then the tail measure of $\eta$ is
  supported by the axes of $\R^{d+1}$, so that the tail measure of
  $\xi$ is supported by scalar multiples of the monomials $1,u,\dots,u^d$.


  Assume that $d=1$, i.e.\ we deal with the random affine function
  $\xi(u)=\eta_0+\eta_1u$, $u\in[0,1]$, where $\eta_0$ and $\eta_1$
  are independent Pareto(1) random variables. Then $\xi$ is regularly
  varying on $\sC_0$ with the normalising function $g(t)=t$, and its
  tail measure is the sum of the pushforwards of $\theta_1$ under the
  two maps
  associating $y>0$ with the constant function $x(u)=y$ and with the
  linear function $x(u)=yu$, respectively.


  If $\Cont([0,1])$ is equipped with the ideal $\sC_{\inf}$, then the
  map $(y_0,y_1)\to x(u)=y_0+y_1u$ from $\R_+^2$ to $\Cont([0,1])$ is
  bornologically consistent if $\R_+^2$ is equipped with the ideal
  generated by the modulus $\modulus(y_0,y_1)=y_0$.
  For the random vector $(\eta_0,\eta_1)$ considered above,
  \begin{displaymath}
    \Prob{t^{-1}(\eta_0,\eta_1)\in \cdot\mid \eta_0>t}
    \wto \frac{\mu\big(\cdot\cap [1,\infty)\times\R_+\big)}
    {\mu\big([1,\infty)\times\R_+\big)},
  \end{displaymath}
  where the tail measure $\mu$ is $\theta_1\otimes\delta_0$.  By the
  continuous mapping theorem, $\xi$ is regularly varying on
  $\sC_{\inf}$ with the tail
  measure supported by the pushforward of $\theta_1$ under the map
  that sends $y\in(0,\infty)$ to the constant function $x(u)=y$,
  $u\in[0,1]$.
\end{example}

\begin{example}
  \label{ex:series}
  Let $\eta=(\eta_n)_{n\geq0}$ be a regularly varying sequence in
  $\ell_1$ with the ideal generated by the norm. Since $\eta$ is in
  $\ell_1$, the following function 
  \begin{displaymath}
    \xi(u)=\sum_{n=0}^\infty \eta_n u^n,\quad u\in[0,1],
  \end{displaymath}
  is well defined and is continuous. The regular variation of $\eta$
  implies that $\xi$ 
  is regularly varying in $\Cont([0,1])$ with the ideal
  $\sC_0$. This follows from the continuity of the map $\eta\mapsto\xi$ and
  Lemma~\ref{lemma:map-zero}, which ensures the bornological
  consistency of this map. 
\end{example}

\begin{example}
  Let $\eta=(\eta_1,\dots,\eta_k)$ be a random vector in
  $\R^k$. Furthermore, let $Z=\{z_1,\dots,z_k\}$ be a finite subset of
  $[0,1]$. Define
  \begin{displaymath}
    \xi(u)=\max_{z_i\in Z} (u-z_i)^2\eta_i. 
  \end{displaymath}
  The map $\psi$ from $(y_1,\dots,y_k)\in\R^k$ to the function
  $\psi(y)(u)=\max_i (u-z_i)^2 y_i$ is a continuous morphism, which is
  also bornologically consistent if $\Cont([0,1])$ is equipped with
  the ideal $\sC_0$, since
  \begin{displaymath}
    \Big\{(y_1,\dots,y_k)\colons \sup_{u\in[0,1]}|\psi(y)(u)|>\eps\Big\}
    \subset \Big\{(y_1,\dots,y_k)\colons a\|y\|_\infty>\eps\Big\}\in\sR_0^k,
  \end{displaymath}
  where
  \begin{displaymath}
    a=\max_{z_i\in Z} \sup_{u\in[0,1]} (u-z_i)^2
    =\max_{z_i\in Z}\max\{z_i^2,(1-z_i)^2\}.
  \end{displaymath}
  If $\eta$ is regularly varying
  on $\sR_0^k$, then the continuous mapping theorem implies that
  $\xi=\psi(\eta)$ is regularly varying in $\Cont([0,1])$ with the
  ideal $\sC_0$. Its tail measure is the pushforward of the tail
  measure $\mu$ by the map $\psi$. In
  particular, if $\eta$ consists of i.i.d.\ Pareto(1) random variables,
  then $g(t)=t$ and the tail measure of $\xi$ is the sum of 
  the pushforwards of $\theta_1$ under the maps
  $y\mapsto x(u)=(u-z_i)^2 y$, $i=1,\dots,k$.

  Assume now that $k\geq2$.
  The map $\psi$ is bornologically consistent from $\sR_0^k(2)$ to
  $\sC_{\inf}$
  Indeed, if
  \[
    \inf_{u\in[0,1]}|\psi(y)(u)|>\eps,
  \]
  then, at each point $z_l\in Z$,
  \[
    \psi(y)(z_l)=\max_{i\neq l}(z_l-z_i)^2y_i>\eps.
  \]
  Thus at least two components of $y$ exceed $\eps/b$, where
  \[
    b=\max_{i,j}(z_i-z_j)^2,
  \]
  and hence
  \[
    \big\{y\colons \inf_{u\in[0,1]}|\psi(y)(u)|>\eps\big\}
    \subset
    \bigcup_{i<j}
    \{y\colons b\min(y_i,y_j)>\eps\}\in\sR_0^k(2).
  \]
  If all components of $\eta$ are i.i.d.\ Pareto(1) random variables,
  then $\eta$ is regularly varying on $\sR_0^k(2)$ with the
  normalising function $g(t)=t^2$ and the tail measure 
  \begin{displaymath}
    \mu(\bdiff y)=\sum_{1\leq i<jleq k}
    \theta_1(\bdiff y_i) \theta_1(\bdiff y_j)\prod_{l\neq i,j} \delta_0(\bdiff y_l).
  \end{displaymath}
  Thus, $\xi$ is regularly varying on $\sC_{\inf}$ with the tail
  measure equal to the pushforward of this tail measure under the map
  $\psi$. 
\end{example}

The following three examples are effectively deduced from
Theorem~\ref{mapping-theorem}, which concerns the transfer of the regular
variation property under continuous bornologically consistent
mappings.

\begin{example}[Real-valued maps on \protect{$\Cont([0,1])$}]
  Consider a map
  \begin{displaymath}
    \psi(x) = \int_0^1 x(u) \diff u
  \end{displaymath}
  from $\Cont([0,1])$
  with the ideal $\sC_0$ to $\YY=\R$ with linear scaling and the
  ideal $\sR_0$ generated by $\modulus(y)=|y|$.  Then
  \begin{align*}
    \psi^{-1}\big(\{y\colons |y|\geq \eps\}\big)
    &=\Big\{x\in\Cont([0,1])\colons \Big|\int x(u)\diff u\Big|\geq \eps\Big\}\\
    &\subset \big\{x\in\Cont([0,1])\colons \|x\|_\infty\geq \eps\big\} 
  \end{align*}
  for each $\eps>0$, so that $\psi$ is a continuous bornologically
  consistent morphism.  If $\xi$
  is regularly varying in $\Cont([0,1])$, namely,
  $\xi\in\RV(\Cont([0,1]),\mydot,\sC_0,g,\mu)$, then $\zeta=\int_0^1
  \xi(u)\diff u$ is a 
  regularly varying random variable with the same normalising function
  $g$, unless the tail measure of $\xi$
  is supported only by continuous functions that integrate to
  zero. The tail measure of $\zeta$ is the pushforward of $\mu$ under
  the integral map $x\mapsto \int_0^1 x(u)\diff u$.

  Consider also a map $\tilde\psi(x)=\sup x -\inf x$ with values in
  $\R_+$.  Since
  \begin{align*}
    \tilde\psi^{-1}\big(\{y\colons y\geq \eps\}\big)
    &=\big\{x\in\Cont([0,1])\colons \sup x-\inf x\geq \eps\big\}\\
    &\subset \big\{x\in\Cont([0,1])\colons \|x\|_\infty\geq \eps/2\big\}, 
  \end{align*}
  for all $\eps>0$, the map $\tilde\psi$ is a bornologically
  consistent morphism. Thus, if $\xi$ is regularly varying in
  $\Cont([0,1])$ with the tail measure that is not exclusively
  supported by constant 
  functions, then $\sup\xi-\inf\xi$ is a regularly varying random
  variable. Its tail measure is the pushforward of the tail measure of
  $\xi$ under the map $x\mapsto \sup x-\inf x$.
\end{example}

\begin{example}[Derivative map]
  \label{example:diff}
  \index{derivative map}
  Let $\XX$ be the space of continuously differentiable functions on
  $[0,1]$, and let $\YY=\Cont([0,1])$ be the space of continuous
  functions, both equipped with linear scaling. Equip $\XX$ with the
  norm $\|x\|=\|x\|_\infty+\|x'\|_\infty$,
  which is the sum of the uniform norms of the function and its
  derivative, and $\YY$ with the uniform norm, with the ideals
  generated by these norms chosen as moduli. Then the map $\psi(x)=x'$
  associating a function $x\in\XX$ with its derivative is a continuous
  bornologically consistent morphism. If $\xi$ is
  regularly varying on $\XX$ and its tail measure is not supported
  solely by constant functions, then $\psi(\xi)=\xi'$ is regularly
  varying on $\YY$ and its tail measure is the pushforward of the tail
  measure of $\xi$ under the map $x\mapsto x'$.
  This example can be generalised for higher-order
  differential operators.
\end{example}

\begin{example}[Quotient spaces of continuous
  functions]\label{cont-func}
  \index{quotient space!of continuous functions}
  Let $\XX = \Cont[0,1]$ be the family of all continuous functions
  $x:[0,1]\to\R$ with the linear scaling and the ideal $\sC_0$
  generated by the norm. 
  For $x, y\in \Cont[0,1]$, let $x\sim y$ if $x-y$ is a constant
  function.  Then the equivalence class of $x$ is $[x]=\{x+c\colons c\in\R\}$
  and Condition~\hyperref[condS]{(S)} holds, so that the scaling $\tilde{T}_t$ on $\XXT$
  is defined by \eqref{scaling-on-XXT}.  Since
  \begin{displaymath}
    \ttau([x])=\inf\big\{\|x+c\|_\infty\colons c\in\R\big\}
    =\frac{1}{2}\big(\sup x - \inf x\big)
  \end{displaymath}
  is continuous, Lemma~\ref{lemma:q-B} applies, that is, $q\sC_0$
  satisfies Condition~\hyperref[condB]{(B)}.  If $\xi$ is regularly
  varying on $\sC_0$ and its tail measure $\mu$ is not supported
  exclusively by constant functions, then its equivalence class
  $\tilde\xi=q\xi$ is regularly varying on $q\sC_0$ and its tail measure is
  the pushforward of $\mu$ under the quotient map.  This condition is
  violated if, for instance, $\xi$ is identically equal to a random
  constant or, more generally, if $\sup\xi-\inf\xi$ has a lighter tail
  than $\|\xi\|_\infty$. In this setting, it is not possible to use
  the ideal $\sC_{\inf}$ on the original space of continuous
  functions, since $q\sC_{\inf}$ does not contain any non-empty set.

  A more complicated equivalence relation arises by letting
  $x\sim y$ if $x-y\in\ P_k$, where $P_k$ is the family of polynomials
  of degree at most $k$ for a fixed $k\in\NN$. Then
  \begin{displaymath}
    \ttau([x])=\inf\big\{\|x+p\|_\infty\colons p\in P_k\big\}.
  \end{displaymath}
  The infimum is attained for a unique $p=p_x$, which is the best
  approximating polynomial of order at most $k$; see, e.g.,
  \cite{isa:kel94}. The regular variation property of $\xi$ is transferred
  to $\tilde\xi$ if the tail measure of $\xi$ is not exclusively supported
  by $P_k$. 
\end{example}

\paragraph{Ideals generated by improper moduli on $\Cont([0,1])$}
Consider the ideal $\sC(u_0)$ on $\Cont([0,1])$ generated by the
modulus $\modulus_{u_0}(x)=|x(u_0)|$, where $u_0$ is a fixed point in
$[0,1]$. This modulus vanishes on non-trivial functions and thus is not
proper.
Recall that $\Gamma$ is the family of all finite point sets
$\{u_1,\dots,u_m\}$ with $u_1<\cdots<u_m$ and $m\in\NN$. If
$\{u_0\}\leq \gamma$ (that is, $\gamma$ contains $u_0$), then let
$\sR^{\#\gamma}(u_0)$ denote
the ideal on $\R^{\#\gamma}$ generated by the
modulus given by the absolute value of the component of the vector of
length $\#\gamma$ corresponding to $u_0$.
\index{modulus!improper}

\begin{proposition}
  Let $\xi$ be a random element in $\Cont([0,1])$ and let
  $u_0\in[0,1]$. Then $\xi\in\RV(\Cont([0,1]),\mydot,\sC(u_0),g,\mu)$
  if and only if the following conditions are satisfied.
  \begin{enumerate}[(i)]
  \item For each $\gamma\in\Gamma$ with $\{u_0\}\leq \gamma$, we have
    $\pi_\gamma\xi\in\RV\big(\R^{\#\gamma},\mydot,\sR^{\#\gamma}(u_0),g,\mu_\gamma\big)$.
  \item For all $\delta>0$,
    \begin{displaymath}
      \lim_{\eps\downarrow 0} \limsup_{t\to\infty}
      \Prob{\varpi_\eps(\xi)>t\delta\mid |\xi(u_0)|>t} =0. 
    \end{displaymath}
  \end{enumerate}
\end{proposition}
\begin{proof}
  \textsl{Necessity.}
  By Theorem~\ref{thr:polar-decomp-4},
  \begin{displaymath}
    \Prob{t^{-1}\xi\in\cdot\mid |\xi(u_0)|>t}
    \wto \frac{\mu(\cdot\cap \{x\colons |x(u_0)|>1\})}
    {\mu(\{x\colons |x(u_0)|>1\})}
    \quad \text{ as}\; t\to \infty. 
  \end{displaymath}
  In particular, $\xi(u_0)\in\RV(\R,\mydot,\sR_0,g,\pi_{u_0}\mu)$, and
  the conditional distributions of $t^{-1}\xi$ given $|\xi(u_0)|>t$, for all
  sufficiently large $t$, form a relatively compact family, so they are
  tight, which yields (ii).  If $\{u_0\}\leq \gamma$, then the map
  $x\mapsto\pi_\gamma x$ is continuous and
  $(\sC(u_0),\sR^{\#\gamma}(u_0))$-bornologically consistent. This
  immediately implies (i) with $\mu_\gamma=\pi_\gamma\mu$.

  \smallskip
  \noindent
  \textsl{Sufficiency.}  By (i), the finite-dimensional distributions
  of $t^{-1}\xi$ given $\modulus_{u_0}(\xi)=|\xi(u_0)|>at$ for any
  given $a>0$ converge. Together with the one-dimensional convergence at $u_0$,
  condition (ii) guarantees that the corresponding conditional
  distributions of $t^{-1}\xi$ are tight. Thus, they are relatively
  weakly compact on the set $\{x\in\Cont([0,1])\colons |x(u_0)|>a\}$. By
  Lemma~\ref{lemma:rsc}, they are relatively compact in the vague
  topology corresponding to the ideal $\sC(u_0)$. The proof is
  concluded similarly to the proof of Theorem~\ref{thr:C}.  
\end{proof}

By a similar argument, one may consider other
improper moduli, for example,
$\modulus(x)=\min\big(|x(u_0)|,|x(v_0)|\big)$,
which is
generated by the values of the function at two points or any finite fixed set
of points.

\paragraph{Topology of uniform convergence on compact sets}
Continuous real-valued functions on $\R$ (or any locally compact space
such as $\R^d$) are often equipped with the topology of uniform
convergence on compact sets (also called the locally uniform
topology),
\index{locally uniform topology}
\index{uniform convergence on compact sets}
which makes $\Cont_{\text{loc}}(\R)$ a separable metrisable
space; see \cite{warner58}.  Consider the ideal $\sC_{0,\text{loc}}$
on $\Cont_{\text{loc}}(\R)$ generated by a family of continuous moduli
\begin{displaymath}
  \modulus_{[-n,n]}(x)=\sup_{u\in [-n,n]}\big|x(u)\big|,
  \quad n\in\NN,\; x\in\Cont_{\text{loc}}(\R).
\end{displaymath}
In the case of a general parameter space, instead of $[-n,n]$, one takes
the uniform norms of a function restricted to a family of compact sets
$(K_n)_{n\in\NN}$ that increases to the whole space, for example, nested
balls of radii increasing to infinity. Note that $\sC_{0,\text{loc}}$
is strictly smaller than the ideal $\sC_0$ generated by the uniform
metric on the whole of $\R$. For instance, the set
\begin{displaymath}
  A=\big\{x(u)=(1-|u-k|)_+\colons k\in\ZZ\big\} 
\end{displaymath}
belongs to $\sC_0$, but does not belong to $\sC_{0,\text{loc}}$.
While $\sC_{0,\text{loc}}\subset \sC_0$, it is not possible to refer
directly to Proposition~\ref{prop:sub-ideal} in order to deduce the
regular variation property on $\sC_{0,\text{loc}}$ from such a
property on $\sC_0$, since the topologies of the uniform and locally
uniform convergence differ.

By \cite[Proposition~16.6]{kalle}, a sequence of random functions
converges weakly in $\Cont_{\text{loc}}(\R)$ if and only if their
restrictions to all compact sets weakly
converge. Therefore, showing the regular variation on
$\Cont_{\text{loc}}(\R)$ with the ideal $\sC_{0,\text{loc}}$ amounts
to checking that the restrictions to any compact interval are
regularly varying, for example, by applying Theorem~\ref{thr:C}.
Let $\Gamma$ denote the directed set of all finite subsets of $\R$,
ordered by inclusion.

\begin{proposition}
  \label{prop:C-varpi-local}
  Let $\xi$ be a random element in $\Cont_{\text{loc}}(\R)$. Then
  \begin{displaymath}
    \xi\in\RV(\Cont_{\text{loc}}(\R),\mydot,\sC_{0,\text{loc}},g,\mu)
  \end{displaymath}
  if and only if the following conditions are satisfied.
  \begin{enumerate}[(i)]
  \item 
    There exist $n_0\in\NN$ and $\gamma_0\in\Gamma$ with
    $\gamma_0\subset[-n_0,n_0]$ such that, for each $n\geq n_0$ and each finite
    set $\gamma$ satisfying $\gamma_0\leq\gamma\subset[-n,n]$, the
    finite-dimensional random vector  
    $\xi_\gamma$ is regularly varying with a common normalising
    function $g\in\RV_\alpha$ and tail measure $\mu_\gamma$, and
    \begin{displaymath}
      \vartheta_n=\sup_{\gamma_0\leq
        \gamma\subset[-n,n]} \mu_\gamma(B_{\#\gamma})\in(0,\infty),
      \quad n\geq n_0.
    \end{displaymath}
  \item For all $\delta>0$ and $n\in\NN$,
    \begin{displaymath}
      \lim_{\eps\downarrow 0} \limsup_{t\to\infty}
      g(t)\Prob{\varpi_{\eps,n}(\xi)>t\delta} =0,
    \end{displaymath}
    where
    \begin{displaymath}
      \varpi_{\eps,n}(x)=\sup_{u_1,u_2\in[-n,n], |u_1-u_2|\leq \eps}
      |x(u_1)-x(u_2)|, \quad x\in\Cont(\R). 
    \end{displaymath}
  \end{enumerate}
\end{proposition}
\begin{proof}
  \textsl{Necessity} follows from the necessity part of
  Theorem~\ref{thr:C}, noticing that the projection of the tail
  measure $\mu$ is non-trivial on $\Cont([-n,n])$ with all sufficiently
  large $n$.

  \smallskip
  \noindent
  \textsl{Sufficiency.}  Denote by $\xi|_{[-n,n]}$ the restriction of
  $\xi$ onto $[-n,n]$. By Theorem~\ref{thr:C}, for each $n\geq n_0$
  there exists a measure $\mu^{(n)}$ on $\Cont([-n,n])$ such that
  \begin{displaymath}
    g(t)\Prob{t^{-1}\xi|_{[-n,n]}\in D, \modulus_{[-n,n]}(\xi)>at}
    \to \mu^{(n)}\big(D\cap\{\modulus_{[-n,n]}>a\}\big)
    \quad \text{as}\; t\to\infty
  \end{displaymath}
  for each $\mu^{(n)}$-continuity set $D\subset\Cont([-n,n])$ and $a>0$.
  Furthermore,
  \begin{displaymath}
    \mu^{(n)}(\{\modulus_{[-n,n]}>1\})
    =\lim_{t\to\infty} g(t)\Prob{\modulus_{[-n,n]}(\xi)>t}=\vartheta_n. 
  \end{displaymath}
  If $n'>n$, the restriction map
  $r_{n,n'}:\Cont([-n',n'])\to\Cont([-n,n])$ is continuous and
  bornologically consistent. By the continuous mapping theorem,
  $r_{n,n'}\mu^{(n')}$ is the tail measure of $\xi|_{[-n,n]}$. By
  uniqueness of the tail measure obtained from Theorem~\ref{thr:C}, we
  have
  \[
    r_{n,n'}\mu^{(n')}=\mu^{(n)} .
  \]
  That is, if $D$ is a
  Borel subset of $\{x\in\Cont([-n,n])\colons \modulus_{[-n,n]}>a\}$ for some
  $a>0$, and $D'$ is the set of all functions in $\Cont([-n',n'])$,
  whose restrictions to $[-n,n]$ belong to $D$, then
  $\mu^{(n)}(D)=\mu^{(n')}(D')$.
  Thus, the compatible measures $\mu^{(n)}$, $n\ge n_0$, can
  be merged to produce a measure $\mu\in\Mb[\sC_{0,\text{loc}}]$
  that extends them, see Lemma~\ref{lemma:enclosed}. 

  The ideal $\sC_{0,\text{loc}}$ is generated by a countable
  collection of moduli $\modulus_{[-n,n]}$, $n\ge n_0$. Note that
  \begin{displaymath}
    \max_{i=k,\dots,n}\modulus_{[-i,i]}(x)=\modulus_{[-n,n]}(x), \quad k\leq n.
  \end{displaymath}
  By Proposition~\ref{prop:polar-many-tau}, it remains to verify the
  corresponding convergence on a convergence-determining class. Fix
  $n\geq n_0$ and $a>0$. Let $m\geq n$, and let
  $D=r_m^{-1}D_m$ be a cylinder set, where
  $r_m:\Cont_{\mathrm{loc}}(\R)\to\Cont([-m,m])$ is the restriction map
  and $D_m\subset\Cont([-m,m])$ is chosen so that
  \[
    D_m\cap\{\modulus_{[-n,n]}>a\}
  \]
  is a $\mu^{(m)}$-continuity set. Then Theorem~\ref{thr:C}, applied
  on $[-m,m]$, gives
  \[
    g(t)\Prob{t^{-1}\xi\in D,\modulus_{[-n,n]}(\xi)>at}
    \to
    \mu\big(D\cap\{\modulus_{[-n,n]}>a\}\big).
  \]
  Cylinder sets determined by compact restrictions form a
  convergence-determining class for $\Cont_{\mathrm{loc}}(\R)$. Hence
  the convergence extends to all $\mu$-continuity sets in the ideal,
  and Proposition~\ref{prop:polar-many-tau} yields regular variation
  in $\Cont_{\mathrm{loc}}(\R)$.
\end{proof}

\begin{example}[Random polynomial]
  \index{random polynomial}
  A random polynomial of degree at most $d$ (see Example~\ref{example:polynom})
  is regularly varying in $\Cont_{\text{loc}}(\R)$ if its coefficients
  constitute a regularly varying random vector in $\R^{d+1}$. Indeed,
  the map $\psi$ from $\R^{d+1}$ to polynomials given in \eqref{eq:37}
  is continuous if the space of continuous functions on $\R$ is
  equipped with the topology of uniform convergence on compact
  sets. It is bornologically consistent, since
  \begin{displaymath}
    \modulus_{[-n,n]}(\psi(y))=\sup_{u\in[-n,n]}|\psi(y)(u)|
    \leq \sum_{i=0}^d |y_i|n^i\leq n^{d}\|y\|_1 
  \end{displaymath}
  for each $n\in\NN$. 
\end{example}

\begin{example}
  Let $\eta$ be a Pareto(1) random variable in $\R_+$, and let
  \begin{displaymath}
    \xi(u)=\eta+(u-\eta)_+ - (u-2\eta)_+=
    \begin{cases}
      \eta, & u\leq \eta,\\
      u, & \eta<u\leq 2\eta,\\
      2\eta, & 2\eta<u,
    \end{cases}
  \end{displaymath}
  which is a bounded continuous random function.  Let
  $\gamma=(u_1,\dots,u_m)$ for $u_1<\cdots<u_m$, and let
  $A\in\sR_0^m$. If $\xi(u_m)>t$, then we must have $2\eta>t$, and
  thus $\eta>t/2$, implying that $u_m\leq \eta$ and $\xi(u)=\eta$ for
  all $u\leq u_m$ for all sufficiently large $t$. Therefore,
  \begin{align*}
    \lim_{t\to\infty}
    t & \Prob{t^{-1}\big(\xi(u_1),\dots,\xi(u_m) \big)\in A,
        \max\big(\xi(u_1),\dots,\xi(u_m)\big)>t}\\
    & = \lim_{t\to\infty}
    t \Prob{t^{-1}\big(\xi(u_1),\dots,\xi(u_m) \big)\in A,
        \xi(u_m)>t}\\
    &=\lim_{t\to\infty}
    t\Prob{t^{-1}(\eta,\dots,\eta)\in A, \eta>t}
    =\theta_1 \big (\{s>0\colons (s,\dots,s)\in A\}\big).
  \end{align*}
  For the final equality, $A$ has to be a continuity set for the
  pushforward of $\theta_1$ under the map $s\mapsto(s,\dots,s)$; this
  is the case, for instance, for finite unions of rectangles whose
  boundaries have zero measure under this pushforward.
  Therefore, the tail measure $\mu_\gamma$ on $\R^m$ is
  supported by 
  vectors with all identical components. For instance,
  $\mu_\gamma((a,\infty)^m)=1/a$, $a>0$.
  Condition~(ii) follows, since $\varpi_{\eps,n}(\xi)=\eps$. The tail
  measure $\mu$ is supported by constant functions and is the
  pushforward of $\theta_1$ under the map associating $y>0$ with the
  constant function $x(u)=y$.  Note that
  \begin{multline*}
    t\Prob{\modulus_{[-n,n]}(\xi)>at}
    =t\Prob{\xi(n)>at}\\
    =t\Prob{\eta>at,n\leq \eta}+t\Prob{n>at,\eta<n\leq 2\eta}
    +t\Prob{2\eta>at,2\eta<n}.
  \end{multline*}
  The last two summands vanish for all sufficiently large $t$. Thus,
  \begin{displaymath}
    \lim_{t\to\infty} t\Prob{\modulus_{[-n,n]}(\xi)>at}
    =\theta_1((a,\infty)),
  \end{displaymath}
  and so the pushforward of the tail measure $\mu$ under the map
  $x\mapsto \modulus_{[-n,n]}(x)$ is $\vartheta_n\theta_1$ with
  $\vartheta_n=1$ for all $n$.
  
  It should be noted that the random function $\xi$ is not regularly
  varying on $\Cont(\R)$ with the ideal $\sC_0$. If this were the
  case, then we would have
  \begin{displaymath}
    t\Prob{\|\xi\|_\infty>t}=t\Prob{\eta> t/2}\to 2\quad
    \text{as}\; t\to\infty.
  \end{displaymath}
  However,
  \begin{displaymath}
    \mu_\gamma(B_{\#\gamma})=\lim_{t\to\infty}
    t\Prob{\max\big(\xi(u_1),\dots,\xi(u_m)\big)>t}
    =\lim_{t\to\infty}t\Prob{\eta>t}=1,
  \end{displaymath}
  so that the factor $\vartheta=2$ in the tail measure
  $\vartheta\theta_1$ of $\|\xi\|_\infty$ is not equal to the
  supremum of $\mu_\gamma(B_{\#\gamma})$.
\end{example}

The following example involves a non-linear scaling on $\Cont(\R_+)$.

\begin{example}
  Let $\Cont_{\text{loc}}(\R_+)$ be the space of continuous functions
  $x:\R_+\to\R$ with the topology of uniform convergence on compact
  subsets of $\R_+$ and the continuous scaling $(T_tx)(u)=t x(tu)$,
  $u\in\R_+$. Note that this scaling is not continuous in $t$ if the
  space of continuous functions is endowed with the topology of
  uniform convergence.  Let $\sC(0)$ be the ideal on
  $\Cont_{\text{loc}}(\R_+)$ generated by the modulus
  $\modulus_0(x)=|x(0)|$.
  Define a function $\xi(u)=\eta(1-\eta u)_+$, $u\geq0$, where $\eta$
  is a Pareto(1) random variable. Since $\eta$ is regularly varying
  and the map $y\mapsto T_y x$ is continuous and bornologically
  consistent, the random function $\xi=T_\eta x$ obtained as the
  scaling of the function
  $x(u)=(1-u)_+$ is regularly varying. It should be
  noted that $\xi$ is not regularly varying under linear scaling on
  $\Cont_{\text{loc}}(\R_+)$, since in that case the tail measure
  would be supported by 
  discontinuous functions.
\end{example}

\paragraph{Stationary random continuous functions}
In the important setting of stationary random continuous functions, it
is possible to amend Proposition~\ref{prop:C-varpi-local} as follows.
\index{stationary continuous functions}

\begin{proposition}
  \label{prop:C-stationary}
  Let $\xi$ be a stationary random continuous function on $\R$. Then
  $\xi\in\RV(\Cont_{\text{loc}}(\R),\mydot,\sC_{0,\text{loc}},g,\mu)$
  if and only if
  $\xi\in\RV(\Cont_{\text{loc}}(\R),\mydot,\sC(0),g,\mu_0)$, where
  $\mu_0$ is the restriction of $\mu$ to the ideal $\sC(0)$ generated
  by $\modulus_0(x)=|x(0)|$. This is 
  also equivalent to condition (ii) of
  Proposition~\ref{prop:C-varpi-local} combined with the following
  condition.
  \begin{enumerate}[(i)]
  \item[(i')] There is a normalising function $g$ such that, for all
    $\gamma\in\Gamma$ with $\{0\}\leq\gamma$, the random vector 
    $\xi_\gamma$ is regularly varying on $\R^{\#\gamma}$ with the
    ideal generated by $\modulus_0(x)=|x(0)|$, the normalising
    function $g$, and the tail measure $\mu_\gamma$.
  \end{enumerate}
\end{proposition}
\begin{proof}
  The equivalence of the finite-dimensional regular variation
  properties with respect to the ideals $\sR_0^{\#\gamma}$ and
  $\sR^{\#\gamma}(0)$ corresponding to
  $\sC_\gamma$ and $\sC(0)$
  follows from Proposition~\ref{prop:cumulative-moduli}, taking into
  account the stationarity and
  Proposition~\ref{prop:extension-invariant}. This also confirms the
  asserted equivalence.
\end{proof}

\paragraph{Topology of pointwise convergence}
Let $\Cp(\R)$ be the family of continuous functions $x:\R\to\R$
\index{pointwise convergence}
\index{continuous functions!topology of pointwise convergence}
with the topology of pointwise convergence, that is, $x_n$ converges
to $x$ if $x_n(u)\to x(u)$ for all $u\in\R$. This space is extensively
studied in topology; see, e.g., \cite{tkach11}. Since the parameter 
space $\R$ is uncountable, the space $\Cp(\R)$ does not have a
countable topological base (it is not even first countable); it is not
metrisable and thus is not Polish. This space is completely regular
(see \cite[Problem~068]{tkach11}) and Souslin (see
\cite[Corollary~1.8]{ban:wan20}).  It is known that a completely
regular Souslin space is perfectly normal; see
\cite[Theorem~6.7.7]{bogachev07}, so that the Baire and Borel
$\sigma$-algebras on $\Cp(\R)$ coincide. Each Borel measure on a
completely regular Souslin space is Radon; see
\cite[Theorem~4.1.11(iii)]{bogachev18}.

Consider the ideal $\sC$ on $\Cp(\R)$ generated by the moduli
$\modulus_v(x)=|x(v)|$ for all rational $v\in\R$, that is, $A\in\sC$
means that there exists a finite set $v_1,\dots,v_n$ of rational
numbers such that
\begin{equation}
  \label{eq:ideal-Cp}
  \inf_{x\in A} \max_{i=1,\dots,n} \big|x(v_i)\big|>0. 
\end{equation}
This ideal is topologically and scaling consistent. The general
definition of regular variation applies to this ideal and makes it
possible to study regular variation on $\Cp(\R)$. 

\begin{example}
  For $a\geq 0$ and $b\in(0,1/2]$, define the function
  \begin{displaymath}
    x_{a,b}(u)=
    \begin{cases}
      a+au/b, & 0\leq u\leq b,\\
      3a-au/b, & b<u\leq 2b,\\
      a, & 2b<u\leq 1,
    \end{cases}
    \qquad u\in[0,1],
  \end{displaymath}
  and let $x_{a,0}(u)=a$ for all $u\in[0,1]$ and $a\geq0$. The map
  $(a,b)\mapsto x_{a,b}$ is continuous as a map from
  $\R_+\times[0,1/2]$ to $\Cp([0,1])$. Note that continuity fails (as
  $b\to 0$) if the target space is equipped with the topology of
  uniform convergence. Equip $\R_+\times[0,1/2]$ with the scaling
  applied to the first component and the ideal generated by the
  modulus $\modulus(a,b)=a$. Then the map $(a,b)\mapsto x_{a,b}$ is a
  morphism, which is bornologically consistent, since
  \begin{displaymath}
    \big\{(a,b)\colons x_{a,b}(u)> \eps\big\}
    \subset \big\{(a,b)\colons a>\eps/2\big\}
  \end{displaymath}
  for all $u\in[0,1]$ and $\eps>0$.  Let
  $\eta\in\RV(\R_+,\mydot,\sR_0,g,\theta_\alpha)$ with $\eta\ge2$
  almost surely, and let
  \begin{displaymath}
    \xi(u)=x_{\eta,1/\eta}(u)=
    \begin{cases}
      \eta+\eta^2 u, & 0\leq u\leq 1/\eta,\\
      3\eta-\eta^2 u, & 1/\eta<u\leq 2/\eta,\\
      \eta, & 2/\eta<u\leq 1,
    \end{cases}
    \qquad u\in[0,1].
  \end{displaymath}
  By Theorem~\ref{thr:pair}, $(\eta,1/\eta)$ is regularly varying
  under the scaling applied to the first component, so that the
  continuous mapping theorem yields that $\xi$ is regularly varying in
  $\Cp([0,1])$ with the normalising function $g$ and the tail measure
  $\mu$ supported by constant functions, namely, the pushforward of
  $\theta_\alpha$ under the map $a\mapsto x_{a,0}$.  Regular
  variation in the space $\Cont([0,1])$ of continuous functions with
  the uniform metric fails, since
  \begin{displaymath}
    g(t)\Prob{\sup\xi >t}=g(t)\Prob{2\eta>t}\to 2^\alpha
    \quad \text{as}\; t\to\infty.
  \end{displaymath}
  Indeed, if $\xi$ were regularly varying in $\Cont([0,1])$, then its
  tail measure would also be $\mu$, and then the limit of
  $g(t)\Prob{\sup\xi >t}$ would be $\mu(\{x\colons \sup
  x>1\})=\theta_\alpha((1,\infty))=1$. 
\end{example}

\section{Other functional spaces}
\label{sec:other-funct-spac}

\paragraph{C\`adl\`ag functions}
Consider now the space $\Dfun(\R^d,\R^l)$ of (possibly vector-valued)
c\`adl\`ag functions on Euclidean space equipped with linear
scaling. It is well known that this space, equipped with the Skorokhod
topology, is Polish; see \cite{Bi99} and \cite{kalle}. For $d>1$, one
uses the corresponding multiparameter Skorokhod topology. The general
definition of regular variation applies to the space
$\Dfun(\R^d,\R^l)$ once an ideal on this space is fixed.
\index{cadlag functions@c\`adl\`ag functions}

Regularly varying random functions in the space $\Dfun([0,1],\R)$
(the Skorokhod space) have been studied in \cite{hul:lin05,MR2213968},
which work with the metric exclusion ideal $\sD_0$ obtained by
excluding the zero function. This ideal is generated by the uniform
norm $\|x\|_\infty$, which is a continuous modulus on
$\Dfun([0,1],\R)$. The results from \cite{MR2213968} follow from our
general results for regular variation on ideals satisfying
Condition~\hyperref[condB]{(B)}; for instance, Theorem~4 therein
follows from 
Proposition~\ref{prop:polar-decomp} and
Theorem~\ref{thr:polar-decomp-4}, while Theorem~8 follows
from the continuous
mapping theorem adapted to the case of discontinuities.  Our
Theorem~\ref{thr:C} is a counterpart to Theorem~10 from
\cite{hul:lin05}, with the only difference being the use of the
relative compactness condition in the Skorokhod space
$\Dfun([0,1],\R)$ with the relevant continuity modulus.

Regular variation on the space of c\`adl\`ag functions on $\R$ or
$\R^d$ has been studied in detail in \cite{bladt:hash22} and
\cite{MR4376586}.  The authors of \cite{bladt:hash22} and
\cite{MR4376586} work with the ideal $\sD_{0,\text{loc}}$ on the space
$\Dfun(\R^d,\R^l)$, which is generated by the moduli
$\modulus_K(x)=\sup_{u\in K}\|x(u)\|$ for a family of compact sets $K$
(for example, $K=[-n,n]$ for functions on $\R$) which cover the entire
parameter space. This is also the metric exclusion ideal on $\Dfun(\R^d,\R^l)$
obtained by using an extension of the Skorokhod metric; see Appendix~A
from \cite{bladt:hash22}, in particular, Theorem~A.1(v) therein.

The general definition of regular variation applies to the space
$\Dfun(\R^d,\R^l)$ with the ideal $\sD_{0,\text{loc}}$, namely,
$\xi$ is regularly varying if
\begin{displaymath}
  g(t)\Prob{t^{-1}\xi\in\cdot}\vto[\sD_{0,\text{loc}}]
  \nu(\cdot) \quad\text{as}\; t\to\infty.
\end{displaymath}
It should be noted that the modulus $\modulus(x)=\|x(u)\|$ for any
given $u\in\R^d$ is not continuous in the Skorokhod topology, that is,
$x_n(u)$ does not necessarily converge to $x(u)$ if $x_n\to x$ in
$\Dfun(\R^d,\R^l)$, unless $u$ is a continuity point of the
limiting function. Therefore, it is not possible to use the continuous
mapping theorem to infer that the
finite-dimensional distributions of a regularly varying c\`adl\`ag
function are regularly varying themselves. 

To alleviate the difficulty caused by the possible discontinuity of the
above-mentioned modulus, the author of \cite{MR4376586} forces the 
normalising function $g$ in \eqref{eq:xi-RV} to be
$g(t)=1/\Prob{\|\xi(u_0)\|>t}$ for some $u_0$, which could be chosen as
the origin in view of the stationarity assumption adopted there. If
$\xi$ is a stationary process with values in $\Dfun(\R^d,\R^l)$, then
any given $u$ is almost surely a continuity point of $\xi$.  In the
non-stationary setting, choosing the reference point $u_0$ may
influence the regular variation property of the random function as the next
example shows. 

\begin{example}
  Consider the stochastic process $\xi$ given by
  \begin{displaymath}
    \xi(u)=
    \begin{cases}
      \eta_1, & u<0,\\
      \zeta+\eta_2u, & u\geq 0,
    \end{cases}
  \end{displaymath}
  where $\eta_1$ and $\eta_2$ are two independent random variables
  with regularly varying distributions, for instance, Pareto(1), and
  $\zeta$ is another regularly varying random variable with a lighter
  tail independent of $\eta_1$ and $\eta_2$. The finite-dimensional
  distributions of $\xi$ that include at least one point different
  from $0$ are regularly varying with the normalising function
  corresponding to the tails of $\eta_1$ and $\eta_2$. However,
  the conditional distribution of $t^{-1}\xi(u)=t^{-1}\eta_1$
  for $u<0$
  given that $\xi(0)>t$, converges to the atomic measure at zero. If
  we choose $u_0=-1$ as the reference point, then, conditionally on
  $\xi(u_0)>t$, the distribution of $t^{-1}\xi(u)$, $u<0$, has a non-trivial
  limit.
\end{example}

The authors of \cite{bladt:hash22} work in the not necessarily
stationary setting and set the normalising function to be
$g(t)=1/\Prob{\|\xi(u_0)\|>t}$, additionally assuming that
$\|\xi(u_0)\|$ is a regularly varying random variable. Then 
results akin to Proposition~\ref{prop:polar-decomp},
Theorem~\ref{thr:polar-decomp-4} and Theorem~\ref{thr:C} characterise
the regular variation of $\xi$ together with the regular
variation of $\|\xi(u_0)\|$ for some $u_0\in \R^d$.  In contrast to
\cite[Definition~4.10]{bladt:hash22}, Definition~\ref{def:RV} allows a
more flexible choice of the normalising function.

A criterion for regular variation can be formulated using the
relative compactness criterion in the Skorokhod space as in
Theorem~\ref{thr:C} or using the regular variation
criterion in metric spaces formulated in Theorem~\ref{thr:m-space}. We
refer to Theorem~4.15 of \cite{bladt:hash22} for results along these
lines. It is also possible to obtain regular variation in 
topologies other than the standard $J_1$-topology on the family of
c\`adl\`ag functions, proposed in \cite{MR84897} and discussed in
\cite{MR1876437}. 

\begin{example}[Regular variation in $M_1$-topology]
  Consider the stochastic process $\xi$ on $[0,1]$ given by
  \begin{displaymath}
    \xi(u)=
    \begin{cases}
      0, & u< \frac 12,\\
      \eta, & u\in \left[ \frac 12,\frac 12 + \frac 1{2\eta}\right),\\
      2 \eta, & u\in \left[ \frac 12 + \frac 1{2\eta},1 \right],
    \end{cases}
  \end{displaymath}
  where $\eta>1$ is a regularly varying random variable, say
  Pareto(1). All finite-dimensional distributions of $\xi$ are
  regularly varying, and the same holds for the modulus
  $\modulus(\xi)=\sup_{u\in [0,1]} |\xi(u)| = 2\eta$. However, $\xi$ is
  not regularly varying in $J_1$-topology; this follows from the
  failure of the oscillation condition in \cite[Theorem~10,
  Equation~(7)]{hul:lin05}.
  However, it is not too difficult to show that $\xi$ is
  regularly varying in the $M_1$ topology on $\Dfun([0,1],\R)$; see
  \cite{MR1876437}.
\end{example}

\paragraph{Upper semicontinuous functions}
Let $\XX=\USC([0,1])$ be the family of non-negative upper
semicontinuous functions $x:[0,1]\to[0,\infty]$ (which are not
identically equal to infinity) equipped with the hypo-topology; see
\cite{mo1} and \cite{nor87}.
\index{upper semicontinuous functions}
\index{hypograph}
Recall that the hypograph of $x$ is
defined as $\mathrm{hypo} x=\{(u,v)\colons v\leq x(u)\}$ and that the
hypo-topology is the Fell topology on the family of hypographs, see
Section~\ref{sec:regul-vari-rand-1}. With this topology,
$\USC([0,1])$ becomes a 
locally compact second-countable Hausdorff space, which is
metrisable. It should be noted that convergence in the
hypo-topology does not necessarily imply pointwise convergence;
for instance, the sequence $x_n(u)=\one_{u=v_n}$ with $v_n\to v$ as
$n\to\infty$ converges to $\one_{u=v}$ in $\USC([0,1])$, but not
pointwise at $u=v$.

It is possible to generate an ideal on $\USC([0,1])$ using the uniform
norm as the modulus. Furthermore, the family $\USC([0,1])$ can be
equipped with an exclusion ideal obtained by excluding the zero
function. Another ideal can be obtained by excluding the family of
spike functions $a\one_{u=b}$ for $a\geq 0$ and $b\in[0,1]$.

\begin{example}
  The family $\Cont([0,1])$ of continuous functions
  $x:[0,1]\to[0,\infty)$ is a cone in $\USC([0,1])$.  The topology on
  $\Cont([0,1])$ generated by the uniform metric is stronger than the
  topology induced from $\USC([0,1])$.  Equip $\USC([0,1])$ and
  $\Cont([0,1])$ with the ideal generated by the uniform norm. If a
  random element is regularly varying in $\Cont([0,1])$, then it is
  also regularly varying in $\USC([0,1])$ with the pushforward of the
  tail measure under the natural embedding of $\Cont([0,1])$ into
  $\USC([0,1])$. However, the inverse implication
  fails.  Consider a random function given by
  \begin{displaymath}
    \xi(u)=\eta \big(1-\eta|u-\zeta|\big)_+, 
    \quad u\in[0,1],
  \end{displaymath}
  where $\zeta$ is a continuous random variable on $[0,1]$ and $\eta$
  is an independent regularly varying random variable with tail index
  $\alpha$. Then $\xi$ is regularly varying in $\USC([0,1])$ with the
  tail measure $\mu$ supported by functions $x(u)=c\one_{u=a}$ for
  $c>0$ and $a\in[0,1]$, and the tail index is $\alpha$. This can be
  shown by applying Theorem~\ref{thr:polar-decomp-4} or using the
  continuous mapping from $(a,b,c)$ to
  $x_{a,b,c}(u)=a(1-b^{-1}|u-c|)_+$, $u\in[0,1]$, with $a=\eta$,
  $b=1/\eta$ and $c=\zeta$. Consequently, a variant of
  Theorem~\ref{thr:pair} applies with the to the auxiliary pair
  $\eta_x=(1/x,\zeta)\in\R^2$.  Note that the random function $\xi$ is
  not regularly varying on $\Cont([0,1])$ with the uniform topology,
  since the normalised functions develop increasingly narrow spikes
  and are not relatively compact in that topology; their hypo-limit is
  supported by spike functions, which do not belong to $\Cont([0,1])$.
\end{example}

If $\xi$ is a random upper semicontinuous function, then its hypograph
is a random closed set. Using this representation, it is possible to
use results for random closed sets from
Section~\ref{sec:regul-vari-rand-1} in the space
$[0,1]\times[0,\infty]$ equipped with linear scaling acting on the
second component. 


\paragraph{Integrable functions}
For $p\in[1,\infty)$, let $\XX=\Lp(\R^d)$ be the space of
$p$-integrable functions (with respect to the Lebesgue measure), so
that $\XX$ is a separable Banach space. Equip $\Lp(\R^d)$
with linear scaling and the ideal $\sS_0$ generated by the norm,
which is also the metric exclusion ideal obtained by excluding 
zero, that is, the equivalence class of functions that vanish
almost everywhere. In order to check 
regular variation, it is possible to use results from
Section~\ref{sec:regul-vari-polish}.
\index{space Lp@space $\Lp(\R^d)$}

\begin{example}
  Let $\eta=(\eta_i)_{i\in\NN}$ be a random sequence in $\ell_1$,
  which is also regularly varying in $\ell_1$ with the ideal generated
  by the norm. Consider a sequence of deterministic functions
  $f_i\in \Lp[1](\R)$, $i\in\NN$, such that
  $\sup_i\|f_i\|_1<\infty$. Then $\xi=\sum_i \eta_i f_i$ converges
  almost surely in $\Lp[1](\R)$. If the pushforward of the tail measure
  of $\eta$ under the map
  \[
    (y_i)_{i\in\NN}\mapsto \sum_i y_i f_i
  \]
  is non-trivial, then $\xi$ is regularly varying in $\Lp[1](\R)$ by
  the continuous mapping theorem.
\end{example}

The embedding map from $\Cont([0,1])$ with the ideal generated by the
uniform norm to $\Lp([0,1])$ is continuous and bornologically
consistent, so that a regularly varying function in $\Cont([0,1])$ is
also regularly varying in $\Lp([0,1])$. This fact was noted in
\cite[Section~3.2]{MR4745556} for $p=2$. The continuity and
bornological consistency properties fail if the functions are
considered on the real line. Furthermore, the inverse statement fails:
there is a sample continuous process that is regularly varying in
$\Lp[2]([0,1])$, but not in $\Cont([0,1])$; see
\cite[Proposition~3.5]{MR4745556}.

\paragraph{Sobolev space}
The following result concerns regular variation in the Sobolev space
$W^{1,p}(\R)$; see, e.g., \cite{MR2527916}. It is a generalisation of
Example~\ref{example:diff}, and, in turn, can be further generalised
to other Sobolev spaces and weak derivatives of higher orders.
\index{Sobolev space}

\begin{proposition}
  Let $\xi$ be regularly varying in the Sobolev space $W^{1,p}(\R)$
  with linear scaling and the ideal generated by the norm. Assume that
  the image of the tail measure of $\xi$ under the weak derivative
  operator is non-trivial.  Then the weak derivative of $\xi$ is
  regularly varying in 
  $\Lp(\R)$ with linear scaling and the ideal generated by the norm.
\end{proposition}
\begin{proof}
  Recall that the Sobolev norm of $\xi$ is the $p$-sum combination
  of the $p$-norms of
  $\xi$ and its weak derivative; hence, the map from $\xi$ to its
  weak derivative is a continuous bornologically consistent morphism.
\end{proof}

\section{Random  measures and point processes}
\label{sec:space-count-meas}

Assume that $\XX$ is a Polish space equipped with a continuous
scaling $T$ and a scaling-consistent ideal $\sS$ on $\XX$ with a
countable open base. Recall that $\Mb$ is the family of all Borel
measures that are finite on each set in $\sX$ and are supported
on the union of all sets in $\sS$. It was shown in
\cite[Proposition~3.1]{bas:plan19} that the family $\Mb$ of Borel
measures with the vague topology is Polish.
A metric on $\Mb$ can be constructed as
\begin{equation}
  \label{eq:31b}
  \dmet_{\mathrm{Pr}}(\me_1,\me_2)=\int_1^\infty
  e^{-t}\min\big\{1,\dmet_{\mathrm{Pr}}^{\mathrm{s}}
  \big(\me_1|_{G_t},\me_2|_{G_t}\big)\big\} \diff t,
\end{equation}
where $\{G_t,t\geq1\}$ is the system of sets constructed from a
countable open base $(\base_n)_{n\in\NN}$ of $\sS$ as in
Lemma~\ref{lemma:cont-base}. Note that it is possible to set
$G_t= T_t \VV$, $t \geq 1$, if Condition~\hyperref[condS]{(S)} holds. Furthermore,
$\dmet_{\mathrm{Pr}}^{\mathrm{s}} (\me_1|_{G_t},\me_2|_{G_t})$
is the symmetrised version of the \emph{Prokhorov
  distance}
\index{Prokhorov distance}
between finite measures obtained as restrictions of
$\me_1,\me_2\in\Mb$ onto $G_t$; see \cite{schuhmacher08},
\cite{kalle17} and \cite{dtw22}. Namely, 
\begin{displaymath}
  \dmet_{\mathrm{Pr}}^{\mathrm{s}} (\nu,\nu_1)
  =\inf\big\{r>0\colons \nu(A)\leq \nu_1(A^r)+r,
  \nu_1(A)\leq \nu(A^r)+r, \; A\in\sB(\XX)\big\},
\end{displaymath}
where $\nu$ and $\nu_1$ are finite measures and
$A^r=\{x\colons \dmet(x,A)\leq r\}$ stands for the $r$-neighbourhood
($r$-envelope) of $A$ 
in the metric $\dmet$ on $\XX$.
The symmetrisation is essential, since the total masses of the
restrictions of $\me_1$ and $\me_2$ may differ.

\begin{remark}
  All results from this section hold, assuming that $\XX$ is a
  completely regular metrisable space with an ideal having a countable
  open base. Then the family of tau-additive measures is
  metrisable. If $\XX$ is separable, then the family $\Mb$ is also
  separable; see \cite[Theorem~5.1.3]{bogachev18} for the case of
  finite measures, which is easy to extend to infinite measures with
  the help of Theorem~\ref{lemma:vVSw}.
\end{remark}

The vague topology on $\Mb$ (equivalently, the above defined metric on
$\Mb$) generates the Borel $\sigma$-algebra, which is also the
smallest $\sigma$-algebra that makes the maps $\me\mapsto\me(A)$ measurable
for all Borel $A\in\sS$. A random element in $\Mb$ with this
$\sigma$-algebra is said to be a \emph{random measure}.
\index{random measure}

\index{scaling!of measures}
A possible scaling of measures is the linear scaling
applied to their values, that
is, $(T_t\nu)(A)=t\nu(A)$ for all Borel $A$. Since this scaling amounts
to the linear scaling of functions, this time specialised to functions
$\sB(\XX)\to[0,\infty]$ with arguments are Borel sets, regular
variation can be handled as it was before for spaces of
functions. Linear scaling is not useful for normalised measures
(e.g., probability measures) or for integer-value measures.

We use a scaling of measures that is lifted from a scaling
$T$ on $\XX$ as in \eqref{eq:23}, so that
\begin{displaymath}
  T_t\nu(A)=\nu(T_{t^{-1}}A), \quad \nu\in\Mb,
\end{displaymath}
for all Borel $A$ in $\sS$ (the lifted scaling is denoted by the
same letter). In this way, the scaling is defined on measures with
restrictions on their values, such as counting measures or probability
measures. Note that the zero measure and all measures supported by the
set of scaling-invariant elements in $\XX$ are scaling invariant.

\paragraph{Ideals on counting measures}
A particularly important subfamily of measures in $\Mb$ is the
family $\Mp$ of \emph{counting measures}, that is, measures whose
values are non-negative integers.
\index{counting measure} \index{support of a counting measure}
We denote by $\supp\me$ the
\emph{support} of $\me\in\Mp$, that is, the set of all $x\in\XX$ with
$\me(\{x\})\geq 1$. Note that $(\supp\me)\cap A$ is finite for all
$A\in\sX$, so that $\supp\me$ is at most countable (if $\sX$ has
countable base) and
\begin{displaymath}
  \me=\sum_{x_i\in \supp\me} c_i\delta_{x_i}, 
\end{displaymath}
where $c_i=\me(\{x_i\})$ are called the \emph{multiplicities} of
the corresponding support points.
\index{multiplicity} 
A random element $\etap$ in $\Mp$ equipped with its Borel $\sigma$-algebra
$\sB(\Mp)$ is called a \emph{point process}.
\index{point process}

\begin{lemma}
  Let $\sX$ be a topologically and scaling consistent ideal with a
  countable open base on a Polish space $\XX$ equipped with a 
  continuous scaling. Then the lifted scaling on $\Mp$ is 
  continuous.
\end{lemma}
\begin{proof}
  It is well known that the vague convergence of counting measures is
  equivalent to the point-to-point convergence (pointwise convergence
  of the atoms) of suitable
  enumerations of the support points restricted to a countable open
  base of $\sX$; see, e.g., \cite[Proposition~2.8]{bas:plan19}.
  Namely, $\me_n\vto \me$ as $n\to\infty$ if and only if, for all sets
  $\base_l$, $l\in\NN$, from a countable convergence-determining base
  of $\sS$ consisting of $\me$-continuity sets,
  \begin{displaymath}
    \me_n|_{\base_l}= \sum_{i=1}^k  \delta_{x_i^{(n)}},\quad
    \me|_{\base_l}=\sum_{i=1}^k \delta_{x_i},
  \end{displaymath}
  where atoms are listed with their multiplicites, 
  for all sufficiently large $n$ and $x_i^{(n)}\to x_i$ as
  $n\to\infty$ for all $i=1,\dots,k$.
  
  We need to show that $\me_n\vto \me$ and $t_n \to t >0$ imply
  $T_{t_n}\me_n \vto T_{t}\me$. Without loss of generality, assume that
  $t_n \to 1$ and $\cl \base_l\subset \base_{l+1}$ for all $l$. By
  separating closed sets $\cl\base_l$ and $\base_{l+1}^c$ with 
  continuous functions and taking their suitable level sets as a new
  base, it is possible to construct two new bases $(\base'_l)$ and
  $(\base''_l)$ such that $\cl \base'_l\subset \base''_l$,
  $\cl \base''_l\subset\base'_{l+1}$,
  $\me(\base'_l)=\me(\base''_l)$ and
  $\me(\partial\base'_l)=\me(\partial\base''_l)=0$ for all $l\in\NN$.
  Let
  \begin{displaymath}
    \me|_{\base'_l}=\sum_{i=1}^k \delta_{x_i}.
  \end{displaymath}
  Since $\me_n\vto\me$, we have
  \begin{displaymath}
    \me_n|_{\base''_l}=\me_n|_{\base'_l}= \sum_{i=1}^k  \delta_{x_i^{(n)}}
  \end{displaymath}
  with $x_i^{(n)}\to x_i$ as $n\to\infty$ for all $i=1,\dots,k$.
  Since the scaling is continuous,
  $T_{t_n} x_i^{(n)}\in\base''_l$ for all sufficiently large $n$ and
  each $i=1,\dots,k$. Furthermore, there exists an $n_1$ such that, for
  each $y$ in the support of $\me_n$ lying outside of $\base'_{l+1}$,
  we have $T_{t_n}y\notin \base''_l$ for all $n\geq n_1$. Thus, we have
  the pointwise convergence of all points in
  $T_{t_n}\me_n|_{\base'_l}$ to the points in
  $T_1\me|_{\base'_l}=\me|_{\base'_l}$. Since this holds for all
  elements of a countable convergence-determining base, we conclude
  $T_{t_n}\me_n\to T_1\me=\me$. 
\end{proof}

\begin{lemma}
  Assume that $\sX$ is an ideal that satisfies Condition~\hyperref[condB0]{(B$_0$)}. Then 
  the only scaling-invariant element of $\Mp$ is the zero 
  measure.
\end{lemma}
\begin{proof}
  Let $T_t\me=\me$ for some $t>1$ and a non-trivial $\me$. Then for
  each $x\in\supp\me$ we have $T_tx\in\supp\me$.  Since $\me\in\Mp$ is
  supported by the union of all sets in the ideal,
  Condition~\hyperref[condB0]{(B$_0$)} implies that no point of
  $\supp\me$ is scaling invariant; hence $T_sx\neq x$ for all $s>1$.
  Thus, the set $\{T_{t^n}x\colons n\geq 0\}$ is countably infinite
  and is contained in the support of $\me$.  Without loss of
  generality, assume that $\modulus(x)>1$ for a modulus $\modulus$
  compatible with $\sX$, that is, $A=\{\modulus>1\}\in\sX$.  Then
  $\{T_{t^n}x\colons n\geq 0\}\subset A$, consequently,
  $\me(A)=\infty$ for some $A\in\sX$, contradicting the fact that
  $\me$ should be finite on $A\in\sX$.
\end{proof}

For $k\geq 1$, denote by $\Mpk$ the subset of $\Mp$ consisting of
all counting measures with total mass at most $k$, and let
$\Mpk[0]$ consist only of the zero measure (recall that all measures in
$\Mp$ are assumed to be supported by the union of all sets in
$\sX$). Furthermore, let $\Mpx$ be the family of all subsets
$\Me\subset \Mp$ such that there exists a Borel $A\in\sX$ with
$\me(A)\geq k$ for all $\me\in \Me$.  It is easy to see that $\Mpx$ is a
topologically and scaling consistent ideal with a countable open
base. The union of all sets in $\Mpx$ is the complement of
$\Mpk[k-1]$.

If $\sX$ satisfies Condition~\hyperref[condS]{(S)} and is thus
generated by a modulus 
$\modulus$, 
then the ideal $\Mpx$ is generated by the modulus
\begin{displaymath}
  \modulus_{\mathrm{p}}(\me)
  =\sup\big\{t\geq0\colons \me(\{\modulus> t\})\geq k\big\},
\end{displaymath}
where $\sup\emptyset=0$.  This modulus is continuous; this follows
from the point-to-point characterisation of vague convergence of
counting measures, since $\modulus_{\mathrm{p}}(\me)$ is the $k$-th
largest value of $\modulus$ among the atoms of $\me$, counted with
multiplicities.

\begin{lemma}
  \label{lemma:mpk}
  Assume that $\Mp$ is metrised by \eqref{eq:31b} and let $\sX$ be a
  topologically and scaling consistent ideal with a countable open
  base on a Polish space $\XX$. Then, for each $k\in\NN$, the ideal
  $\Mpx$ coincides with the metric exclusion ideal on $\Mp$ obtained by
  excluding the set $\Mpk[k-1]$.
\end{lemma}
\begin{proof}
  Note that $\Mpk[k-1]$ is closed in the vague topology on
  $\Mp$. Assume that $\Me\in\Mpx$, so that there exists a set $\base_l$
  from the countable base of $\sX$ such that $\me(\base_l)\geq k$ for
  all $\me\in \Me$. The sets $G_t$ from \eqref{eq:31b} are constructed
  to satisfy $G_n=\base_n$ for all sets $\base_n$, $n\in\NN$ from the
  base of $\sX$, see
  Lemma~\ref{lemma:cont-base}. If $\me'\in\Mpk[k-1]$ and $\me\in \Me$,
  then
  \begin{displaymath}
    \dmet_{\mathrm{Pr}}(\me,\me')
    \geq e^{-l}\min\big\{1,\dmet_{\mathrm{Pr}}^{\mathrm{s}}
    \big(\me|_{\base_l},\me'|_{\base_l}\big)\}.
  \end{displaymath}
  By the definition of the Prokhorov metric,
  \begin{align*}
    \dmet_{\mathrm{Pr}}^{\mathrm{s}}
    \big(\me|_{\base_l},\me'|_{\base_l}\big)
    &\geq \inf\big\{r>0\colons \me(A\cap \base_l)\leq \me'(A^r\cap \base_l)+r, \;
      A\in\sB(\XX)\big\}\\
    &\geq \inf\big\{r>0\colons \me(A\cap \base_l)\leq \me'(\base_l)+r,\;
      A\in\sB(\XX)\big\}\\
    &\geq \inf\big\{r>0\colons \me(A\cap \base_k)\leq k-1+r,\; A\in\sB(\XX)\big\}\\
    &\geq \inf\big\{r>0\colons \me(\base_l)\leq k-1+r\big\}\geq 1.
  \end{align*}
  Therefore, $\dmet_{\mathrm{Pr}}(\me,\me')\geq e^{-l}$, 
  meaning that the distance from $\me$ to $\Mpk[k-1]$ is at
  least $e^{-l}$ for all $\me\in \Me$. This means that $\Me$ belongs
  to the metric exclusion ideal.

  Now we prove the reverse inclusion. Assume that the distance between
  all elements of $\Me$ to the set $\Mpk[k-1]$ is at least $\eps>0$,
  that is, $\dmet_{\mathrm{Pr}}(\me,\me')\geq\eps$  for all
  $\me\in\Me$ and $\me'\in \Mpk[k-1]$. Note that $\me'(G_t)\leq k-1$
  for all $t$ and the sets $G_t$ used in \eqref{eq:31b}. Then
  \begin{displaymath}
    \dmet_{\mathrm{Pr}}(\me,\me')=\int_1^\infty e^{-t}
    \min\big\{1,\dmet_{\mathrm{Pr}}^{\mathrm{s}}
    \big(\me|_{G_t},\me'|_{G_t}\big)\big\} \diff t\geq \eps.
  \end{displaymath}
  By the monotonicity of $(G_t)$, the distance $\dmet_{\mathrm{Pr}}^{\mathrm{s}}
  \big(\me|_{G_t},\me'|_{G_t}\big)$ is non-decreasing in $t$.
  If $\me(G_{t_0})\leq k-1$ for some $G_{t_0}$ (where $t_0$ may depend
  on $\me$), then letting $\me'=\me|_{G_{t_0}}$
  yields that
  \begin{displaymath}
    \dmet_{\mathrm{Pr}}^{\mathrm{s}}
    \big(\me|_{G_t},\me'|_{G_t}\big)=0,\quad t\leq t_0, 
  \end{displaymath}
  so that
  \begin{displaymath}
    \dmet_{\mathrm{Pr}}(\me,\me')
    =\int_{t_0}^\infty e^{-t}
    \min\big\{1,\dmet_{\mathrm{Pr}}^{\mathrm{s}}
    \big(\me|_{G_t},\me'|_{G_t}\big)\big\} \diff t
    \leq e^{-t_0}.
  \end{displaymath}
  Thus, $e^{-t_0}\geq\eps$, so that $t_0\leq -\log\eps$. Fix any
  $t>-\log\eps$. Then we necessarily have $\me(G_t)\geq k$ for all
  $\me\in\Me$.
\end{proof}

For $\me\in\Mp$ and $k\in\NN$, let $\me^{(k)}$ denote the $k$-th
\emph{factorial product measure} of $\me$, that is, the counting
measure on $\XX^k$ that is obtained as
\index{factorial product measure}
\begin{equation}
  \label{eq:fact-pm}
  \me^{(k)}
  =\sideset{}{^{\neq}}\sum_{i_1,\dots,i_k\in\{1,\dots,N\}}
  \delta_{(x_{i_1},\dots,x_{i_k})},
\end{equation}
where the atoms of $\me=\sum_{i=1}^N\delta_{x_i}$ with
$N\in\NN\cup\{\infty\}$ are listed with
their multiplicities, and the sum is taken over distinct indices
$i_1,\dots,i_k$; see
\cite[Section~4.2]{LastPenrose17}. Note that
\begin{displaymath}
  \me^{(k)}(A^k)=\me(A)(\me(A)-1)\cdots (\me(A)-k+1)
\end{displaymath}
is the factorial moment of $\me(A)$. Hence, $\me^{(k)}(A^k)$ is
divisible by $k!$. Furthermore, $\me^{(k)}(A^k)\geq 1$ implies
$\me^{(k)}(A^k)\geq k!$. 

Recall that  $\sX^k(k)$ denotes the product ideal on $\XX^k$
generated by Cartesian powers $A^k$ for $A\in\sX$.

\begin{lemma}
  \label{lemma:products}
  For each $k\in\NN$, the map
  \begin{displaymath}
    \me\mapsto \varphi_k(\me)=\me^{(k)}
  \end{displaymath}
  is continuous on $\Mp$ and is a bicontinuous injection on
  $\Mp\setminus\Mpk[k-1]$. This map and its inverse on its image are
  bornologically consistent maps between the ideals $\Mpx$ and
  $\Mpi[\sX^k(k)]$.
\end{lemma}
\begin{proof}
  Continuity follows from
  \cite[Proposition~2.8]{bas:plan19}.
  Assume that
  $\Me\in \Mpi[\sX^k(k)]$ and that $\Me$ is contained in the image of
  $\varphi_k$. Then there exists a set $A\in\sX$ such that
  \begin{displaymath}
    \Me\subset \big\{\nu\colons \nu(A^k)\geq1\big\}.
  \end{displaymath}
  For $\nu=\me^{(k)}$, the condition $\nu(A^k)\geq1$ is equivalent to
  $\me(A)\geq k$, and then $\nu(A^k)\geq k!$, implying that
  \begin{displaymath}
    \varphi_k^{-1}(\Me)=\big\{\me\colons \me(A)\geq k\big\}\in \Mpx.
  \end{displaymath}
  In the opposite direction, if $\Me\subset\big\{\me\colons \me(A)\geq
  k\big\}$ for some 
  $A\in\sX$, then $\me^{(k)}(A^k)\geq k!\geq1$ for all $\me\in \Me$. 
\end{proof}

\paragraph{Convergence and regular variation}
Note that $\Mb[\Mpx]$ is the family of measures $\mu$ on
$\Mp$ that are finite on $\Mpx$, namely, on all sets of the form
$\{\me\in\Mp\colons \me(A)\geq k\}$ 
for all Borel $A\in\sX$, and are
supported on the union of all sets in
$\Mpx$, that is, $\mu(\Mpk[k-1])=0$.
The following result characterises vague convergence on $\Mpx$. It was
formulated as Theorem~2.5 in \cite{dtw22} without proof; we provide
below an independent proof.

\begin{lemma}
  \label{lemma:conv-Sk}
  Let $\sX$ be a topologically consistent ideal with a countable open
  base on a Polish space $\XX$ and let $k\in\NN$. Then a sequence of measures
  $\mu_n\in\Mb[\Mpx]$, $n \geq 1$, vaguely converges on the ideal
  $\Mpx$ to the measure $\mu\in \Mb[\Mpx]$ as $n\to\infty$ if and only
  if there exists a base $(\base_i)_{i\in \NN}$ of $\sX$ such that
  \begin{equation}
    \label{eq:pp-ideal}
    \int e^{-\int f\diff \me}\one_{\me(\base_i)\geq k}\,\mu_n(\bdiff \me) 
    \to \int e^{-\int f\diff \me}\one_{\me(\base_i)\geq k}\,\mu(\bdiff
    \me)
    \quad \text{as}\; n\to\infty
  \end{equation}
  for all $i\in\NN$ and bounded Lipschitz functions $f:\XX\to\R_+$
  supported by $\base_i$.
\end{lemma}
\begin{proof}
  \textsl{Necessity.}
  Fix any countable base $(\base_i)_{i\in\NN}$ of $\sX$. Note that
  \begin{equation}
    \label{eq:43a}
    \De_i=\big\{\me\in\Mp\colons \me(\base_i)\geq k\big\}, \quad i\in\NN,
  \end{equation}
  is an open base of the ideal $\Mpx$.
  Let $r_i:\XX\to[0,1]$ be a continuous function that exactly
  separates $\base_{i+1}^c$ and $\cl\base_i$, that is,
  $\base_{i+1}=\{r_i>0\}$ and $\cl\base_i=\{r_i=1\}$ (note that we can
  always ensure that $\cl \base_i\subset \base_{i+1}$). Such function
  exists since the carrier space is metrisable.  Consider the sets
  \begin{displaymath}
    \Me_{i,u}=\big\{\me\in\Mp\colons \me(\{r_i>u\})\geq k\big\},\quad u\in[0,1).
  \end{displaymath}
  Then $\Me_{i,u'}\subset \Me_{i,u}$ if $u\leq u'$.  The sets
  $\Me_{i,u}$ belong to the ideal $\Mpx$. They 
  are open, since the map $\me\mapsto\me(\{r_i>u\})$ is lower
  semicontinuous. Furthermore, the closure of $\Me_{i,u}$ in $\Mp$ is 
  \begin{displaymath}
    \cl \Me_{i,u}=\big\{\me\in\Mp\colons \me(\{r_i\geq u\})\geq k\big\}.
  \end{displaymath}
  Since $\mu$ is finite on $\Mpx$, we have
  \begin{displaymath}
    \Me_{i,u}\subset \Me_{i,0}=\big\{\me\in\Mp\colons \me(\base_{i+1})\geq k\big\}
  \end{displaymath}
  and $\mu(\Me_{i,0})<\infty$, we obtain that
  $\mu(\partial \Me_{i,u})=0$ for all but an at most countable family
  of $u$, and thus for at least one $u_i\in(0,1)$. Replacing $\De_i$
  with the set $\Me_{i,u_i}$, we can adjust the base of $\sX$ by
  replacing $\base_i$ with $\{r_i>u_i\}$ and redefine
  $(\De_i)_{i\in\NN}$ by \eqref{eq:43a}. Then the base
  $(\De_i)_{i\in\NN}$ of $\Mpx$ consists of $\mu$-continuity
  sets. This proves the necessity, since the functional $\me\mapsto
  \int f\diff\me$ is continuous in the vague topology on $\sX$ and so
  the integrand in
  \eqref{eq:pp-ideal} is continuous $\mu$-almost everywhere.

  \noindent
  \textsl{Sufficiency.}  We take the base $(\base_i)_{i\in\NN}$ of
  $\sX$ and the corresponding base $(\De_i)_{i\in\NN}$ of $\Mpx$
  constructed in the proof of the necessity. 
  By Theorem~\ref{lemma:vVSw}(iv), it suffices to show that
  \begin{equation}
    \label{eq:mun-to-mu}
    (\mu_n|_{\De_i})\wto (\mu|_{\De_i}) \quad \text{as}\; n\to\infty
  \end{equation}
  for all $i\in\NN$. The restrictions of $\mu_n$ and $\mu$ on $\De_i$
  are finite measures on the family $\Mp$ of counting measures finite
  on the ideal $\sX$. A criterion for the weak convergence of measures
  on $\Mp$ in terms of convergence of their Laplace functionals is
  formulated in Proposition~4.1 from \cite{bas:plan19}. While it
  involves checking \eqref{eq:pp-ideal} for the family of all bounded
  Lipschitz functions not necessarily restricted to those supported by
  sets in the involved ideal, its proof in \cite{bas:plan19}
  accommodates the case of functions supported by sets from the ideal. It
  remains to recognise the sides of \eqref{eq:pp-ideal} as the
  Laplace functionals of $\mu_n|_{\De_i}$ and $\mu|_{\De_i}$ to
  confirm \eqref{eq:mun-to-mu}.
\end{proof}

In case of $k=1$ we get a useful equivalent condition for the weak
convergence, which also appears as part of Theorem~A.1 in
\cite{dombry18:_tail}. 

\begin{corollary}
  \label{cor:conv-Sk-1}
  Let $\sX$ be a topologically consistent ideal with a countable open
  base on a Polish space $\XX$. Then a sequence of measures
  $\mu_n\in\Mb[\Mpi]$, $n \geq 1$, vaguely converges on the ideal
  $\Mpi$ to the measure $\mu\in \Mb[\Mpi]$ as $n\to\infty$ if and only
  if there exists a base $(\base_i)_{i\in \NN}$ of $\sX$ such that
  \begin{equation}
    \label{eq:mun-total}
    \mu_n(\{\me\colons \me(\base_i)\geq1\})\to
    \mu(\{\me\colons \me(\base_i)\geq1\})\quad
    \text{as}\; n\to\infty
  \end{equation}
  and 
  \begin{equation}
    \label{eq:pp-ideal-1}
    \int \Big[1-e^{-\int f\diff \me}\Big]\mu_n(\bdiff \me) 
    \to \int \Big[1-e^{-\int f\diff \me}\Big]\mu(\bdiff \me) \quad
    \text{as}\; n\to\infty
  \end{equation}
  for all $i\in\NN$ and bounded Lipschitz functions $f:\XX\to\R_+$
  supported by $\base_i$.
\end{corollary}
\begin{proof}
  If $f$ is supported by $\base_i$, then
  $\int f\diff \me=0$ whenever $\me(\base_i)=0$. Therefore,
  \begin{displaymath}
    1-e^{-\int f\diff \me}
    =\Big[1-e^{-\int f\diff \me}\Big]\one_{\me(\base_i)\geq1}
  \end{displaymath}
  and thus
  \begin{displaymath}
    e^{-\int f\diff \me}\one_{\me(\base_i)\geq1}
    =\one_{\me(\base_i)\geq1}-\Big(1-e^{-\int f\diff \me}\Big).
  \end{displaymath}
  Therefore, the convergence in \eqref{eq:pp-ideal} of
  Lemma~\ref{lemma:conv-Sk} with $k=1$ is equivalent to the
  convergence stated in \eqref{eq:pp-ideal-1} together with the
  convergence of the masses of the set $\{\me\colons \me(\base_i)\geq1\}$. 
\end{proof}
 
Lemma~\ref{lemma:conv-Sk} is usually applied on the assumption that
the set $\{\me\colons \me(\base_i)\geq k+1\}$ is negligible under $\mu_n$, and
thus the integration may be reduced to measures $\me$ with
$k$ atoms. If $k=1$ then the major advantage in working with
\eqref{eq:pp-ideal-1} is the fact that the integrand vanishes if
$\me(\base_i)=0$. However, this argument does not transfer to
$k\ge2$. Indeed, $\int f\diff\me$ is the sum of $f(x_i)$ for $x_i$
from the support of $\me$, which does not necessarily vanish if
$\me(\base_i)\leq k-1$ -- instead this only results in the elimination
of the terms $f(x_i)$ corresponding to $x_i\notin\base_i$.  The following
result provides an alternative condition to Lemma~\ref{lemma:conv-Sk},
which avoids this problem and so provides a generalisation of
Corollary~\ref{cor:conv-Sk-1} for $k\ge2$.  

\begin{lemma}
  \label{lemma:conv-Sk-product}
  Let $\sX$ be a topologically consistent ideal with a countable open
  base on a Polish space $\XX$ and let $k\in\NN$. Then a sequence of measures
  $\mu_n\in\Mb[\Mpx]$, $n\in\NN$, vaguely converges on the ideal
  $\Mpx$ to the measure $\mu\in \Mb[\Mpx]$ as $n\to\infty$ if and only
  if there exists a base $(\base_i)_{i\in \NN}$ of $\sX$ such that
  \begin{equation}
    \label{eq:mun-total-k}
    \mu_n(\{\me\colons \me(\base_i)\geq k\})\to
    \mu(\{\me\colons \me(\base_i)\geq k\})\quad
    \text{as}\; n\to\infty
  \end{equation}
  and
  \begin{equation}
    \label{eq:pp-ideal-product}
    \int \Big[1-e^{-\int f\diff \me^{(k)}}\Big]\mu_n(\bdiff \me) 
    \to \int \Big[1-e^{-\int f\diff \me^{(k)}}\Big]\mu(\bdiff \me) \quad
    \text{as}\; n\to\infty
  \end{equation}
  for all $i\in\NN$ and bounded Lipschitz functions $f:\XX^k\to\R_+$
  supported by $\base_i^k$.
\end{lemma}
\begin{proof}
  By Lemma~\ref{lemma:products}, applied on
  $\Mp\setminus\Mpk[k-1]$, $\mu_n\vto[\Mpx]\mu$ if and only if
  the pushforward of $\mu_n$ under the map $\me\mapsto\me^{(k)}$
  vaguely converges on the ideal $\Mpi[\sX^k(k)]$ to the pushforward
  of $\mu$ under the same map. By Corollary~\ref{cor:conv-Sk-1}, this
  is the case if and only if \eqref{eq:pp-ideal-product} and
  \eqref{eq:mun-total-k} hold. For the latter, note that $\me(A)\geq
  k$ if and only if $\me^{(k)}(A^k)\geq 1$ for any $A\in\sX$.
  Thus, the convergence of
  the total masses of $\mu_n$ to the total mass of $\mu$ implies that
  the masses of the above mentioned pushforwards of the sets
  $\{\me\colons \me^{(k)}(\base_i^k)\geq1\}$ converge.
\end{proof}

The general definition of regular variation specialises to define
regularly varying measures on $\sM(\Mpx)$, that is,
regularly varying point processes on the ideal $\Mpx$.
\index{point process!regularly varying}
Namely, a point process $\etap$ is $\RV(\Mp,T,\Mpx,g,\mu)$ if
\begin{displaymath}
  g(t)\Prob{T_{t^{-1}}\etap\in\cdot }\vto[\Mpx] \mu(\cdot). 
\end{displaymath}
By Lemma~\ref{lemma:conv-Sk}, this is equivalent to showing that
\begin{displaymath}
  g(t)\E\Big[e^{-\int f \diff T_{t^{-1}}\etap}
  \one_{\etap(T_t\base_i)\geq k}\Big]
  \to \int e^{-\int f\diff\me}
  \one_{\me(\base_i)\geq k}\mu(\bdiff \me)
  \quad \text{as}\; t\to\infty
\end{displaymath}
for all $i\in\NN$ and bounded Lipschitz
functions $f$ with support in $\base_i$. Alternatively, by
Lemma~\ref{lemma:conv-Sk-product}, it suffices to show that
\begin{displaymath}
  g(t)\Prob{\etap(T_t\base_i)\geq k}\to \mu(\{\me\colons \me(\base_i)\geq k\})
\end{displaymath}
and
\begin{equation}
  \label{eq:PP-conv-k}
  g(t)\E \Big[1-e^{-\int f\diff (T_{t^{-1}}\etap)^{(k)}}\Big]
  \to \int \Big[1-e^{-\int f\diff \me^{(k)}}\Big]\mu(\bdiff \me) \quad
  \text{as}\; t\to\infty
\end{equation}
for all $i\in\NN$ and bounded Lipschitz functions $f:\XX^k\to\R_+$ supported by
$\base_i^k$.

\index{Janossy measure}
The \emph{Janossy measure} of order $k\in\NN$ is defined for $B\in\sX$ as
\begin{equation}
  \label{eq:janossy-def}
  J_{k,B}(A)=\frac{1}{k!}\E \big[\one_{\etap(B)=k}
  \etap^{(k)}(A\cap B^k)\big], \quad A\in\sB(\XX^k),
\end{equation}
see \cite[Definition~4.6]{LastPenrose17}. In particular,
\begin{equation}
  \label{eq:Janossy:A=B}
  J_{k,B}(B^k)=\E \big[\one_{\etap(B)=k}\big]=\Prob{\etap(B)=k}.
\end{equation}

\begin{theorem}
  \label{thr:process-k}
  Let $\sX$ be a topologically and scaling consistent ideal with a
  countable open base $(\base_i)_{i\in\NN}$ on a Polish space $\XX$
  equipped with continuous scaling and let $k\in\NN$.
  Furthermore, let $\etap\in\Mp$
  be a point process such that $\Prob{\etap(T_t\base_i)=k}>0$ for all
  $t\geq1$ and all $i\in\NN$. Impose the following conditions.
  \begin{enumerate}[(i)]
  \item There is a non-trivial $\mu\in\Mb[\sX^k(k)]$ such that, for
    all $B\in\sX$ and all continuous bounded functions $h:\XX^k\to\R$
    supported by $B^k$,
    \begin{multline}
      \label{eq:Janossy-RV-k}
      g(t)\int_{\XX^k} h(T_{t^{-1}}(x_1,\dots,x_k))
      J_{k,T_tB}(\bdiff(x_1,\dots,x_k))\\
      \to \int_{\XX^k} h(x_1,\dots,x_k)\,\mu(\bdiff(x_1,\dots,x_k))
      \quad \text{as}\; t\to\infty.
    \end{multline}
  \item For all $i\in\NN$,
    \begin{equation}
      \label{eq:more-than-k}
      g(t)\Prob{\etap(T_t\base_i)\geq k+1}\to 0 \quad \text{as}\; t\to\infty. 
    \end{equation}
  \end{enumerate}
  Then $\etap\in\RV(\Mp,T, \Mpx[k],g,\mu^*)$, that is,
  $\etap$ is regularly varying on $\Mpx[k]$ with the
  normalising function $g$ and the tail measure
  \begin{equation}
    \label{eq:mu-star-k}
    \mu^*_{(k)}(\Be)=\int_{\XX^k}\one_{\sum_{i=1}^k \delta_{x_i}\in \Be}
    \mu(\bdiff (x_1,\dots,x_k)) 
  \end{equation}
  for Borel $\Be\in\Mp[\sX^k(k)]$. Equivalently, the tail measure
  $\mu^*_{(k)}$ is the pushforward of $\mu$ under the map
  \begin{equation}
    \label{eq:me-map-k}
    x=(x_1,\dots,x_k)\mapsto \me_x=\sum_{i=1}^k \delta_{x_i}.
  \end{equation} 
\end{theorem}
\begin{proof}
  We apply Lemma~\ref{lemma:conv-Sk-product}. Without loss of
  generality, it is possible to assume that all $\base_i^k$ are
  $\mu$-continuity sets.   By
  \eqref{eq:Janossy:A=B} and \eqref{eq:more-than-k},
  \begin{align*}
    \lim_{t\to\infty}g(t)\E \one_{\etap(T_t\base_i)\geq k}
    &=\lim_{t\to\infty}
    g(t)\E \one_{\etap(T_t\base_i)=k}\\
    &=\lim_{t\to\infty} g(t)J_{k,T_t\base_i}(T_t\base_i^k)\\
    &=\mu(\base_i^k)
      =\mu^*_{(k)}(\{\me\colons \me(\base_i)=k\})
      =\mu^*_{(k)}(\{\me\colons \me(\base_i)\geq k\}).
  \end{align*}
  This follows from \eqref{eq:Janossy-RV-k} by a standard argument,
  approximating the indicator of $\base_i^k$ with continuous
  functions.

  Let $f:\XX^k\to\R_+$ be a bounded
  Lipschitz function supported by $\base_i^k$.
  By \eqref{eq:more-than-k}, the limit of the left-hand side of
  \eqref{eq:PP-conv-k} coincides with the limit of 
  \begin{multline*}
    g(t)\E \Big[\Big(1-e^{-\int f\diff (T_{t^{-1}}\etap)^{(k)}}\Big)
    \one_{\etap(T_t\base_i)=k}\Big]\\
    = g(t) \int_{T_t\base_i^k}
    \Big(1-e^{-\int f \diff (T_{t^{-1}}\me_x)^{(k)}}\Big)
    J_{k,T_t\base_i}(\bdiff (x_1,\dots,x_k)).
  \end{multline*}
  By (i) applied to the function
  \begin{displaymath}
    h(x_1,\dots,x_k)=1-e^{-\int f\diff\me_x^{(k)}},
  \end{displaymath}
  which is bounded and supported by $\base_i^k$, 
  we obtain
  \begin{multline*}
    g(t) \int_{T_t\base_i^k}
    \Big(1-e^{-\int f \diff (T_{t^{-1}}\me_x)^{(k)}}\Big)
    J_{k,T_t\base_i}(\bdiff (x_1,\dots,x_k))\\
    \to \int_{\XX^k} \Big(1-e^{-\int f \diff \me_x^{(k)}}\Big)
      \mu(\bdiff (x_1,\dots,x_k)) \quad \text{as}\; t\to\infty.
  \end{multline*}
  Thus,
  \begin{multline*}
    g(t)\E \Big[\Big(1-e^{-\int f\diff (T_{t^{-1}}\etap)^{(k)}}\Big)
    \one_{\etap(T_t\base_i)\geq k}\Big]\\
    \to \int_{\XX^k} \Big(1-e^{-\int f \diff \me_x^{(k)}}\Big)
    \mu(\bdiff (x_1,\dots,x_k))
    \quad \text{as}\; t\to\infty.
  \end{multline*}
  The result follows from Lemma~\ref{lemma:conv-Sk-product}.
\end{proof}

\begin{example}[Binomial point process]
  \label{ex:rv-binomial-pp}
  Let $\etap=\delta_{\xi_1}+\cdots+\delta_{\xi_m}$, where
  $\xi_1,\dots,\xi_m$ are i.i.d.\ in a Polish space $\XX$ equipped
  with a continuous scaling and an ideal $\sX$. In this case, $\etap$
  is said to be a \emph{binomial point process}.
  \index{binomial point process}
  Assume that the common
  distribution $\nu$ of the $\xi_i$'s satisfies
  $\nu\in\RV(\XX,T,\sX,g,\mu_1)$. Fix $k\in\{1,\dots,m\}$.
  Then the Janossy
  measure of $\etap$ of order $k$ of a Borel set $A$ in $\XX^k$ and
  for a Borel set $B$ in $\XX$ is  given by 
  \begin{align*}
    J_{k,B}(A)
    &=\frac{1}{k!} \E\big[\etap^{(k)}(A\cap B^k)
      \one_{\etap(B)=k}\big]\\
    &=\frac{1}{k!} \E\bigg[\sideset{}{^{\neq}}\sum_{i_1,\dots,i_m
      \in \{1,\dots,m\}} \one_{(\xi_{i_1},\dots,\xi_{i_k})\in (A\cap
      B^k)}
      \prod_{l\notin\{i_1,\dots,i_k\}}\one_{\xi_l\notin B}\bigg]\\
    &=\frac{1}{k!} \frac{m!}{(m-k)!}
      \Prob{(\xi_1,\dots,\xi_k)\in (A\cap B^k)}
      \Prob{\xi_{i_j}\notin B, j=k+1,\dots,m}\\
    &=\binom{m}{k} \nu^{\otimes k}(A\cap B^k)(1-\nu(B))^{m-k}.
  \end{align*}
  Then \eqref{eq:Janossy-RV-k} holds with the normalising function
  $g(t)^k$ and 
  \begin{displaymath}
    \mu=\binom{m}{k} \mu_1^{\otimes k}.
  \end{displaymath}
  Furthermore, if $k\leq m-1$,
  \begin{align}
    g(t)^k\Prob{\etap(T_t\base_i)\geq k+1}
    &=g(t)^k\sum_{j=k+1}^m \binom{m}{j} \nu(T_t\base_i)^j
      (1-\nu(T_t\base_i))^{m-j}\notag \\
    \label{eq:rv-binomial-k+1}
    &\leq (g(t) \nu(T_t\base_i))^k C_m
    \sum_{j=k+1}^m \nu(T_t\base_i)^{j-k}
    \to 0 
  \end{align}
  as $t\to\infty$, since the first factor is bounded and
  $\nu(T_t\base_i)\to0$. Here $C_m$ is the maximum of the involvex
  binomial coefficients. 
  By Theorem~\ref{thr:process-k}, $\etap$
  is regularly varying on the ideal $\Mpik$ with the normalising
  function $g^k$  and the tail
  measure $\mu^*_{(k)}$ being the pushforward of $\mu$ under the map
  \eqref{eq:me-map-k}. 
\end{example}

\paragraph{Regular variation of Poisson processes}
Theorem~\ref{thr:process-k} yields the following result in the case where
$\etap$ is a Poisson process.
\index{Poisson process}
Recall that $\etap$ is a Poisson process with intensity measure
$\lambda$ if the numbers $\etap(A_1),\dots,\etap(A_k)$ are independent
for disjoint Borel sets $A_1,\dots,A_k$ from the ideal $\sX$ and
$\etap(A)$ is Poisson distributed with mean $\lambda(A)$ for each
Borel $A\in\sX$; see \cite[Section~3.1]{LastPenrose17}. 

\begin{theorem}
  \label{thr:poisson-process}
  Let $\sS$ be a topologically and scaling consistent ideal with a
  countable open base on a Polish space $\XX$ equipped with a
  continuous scaling. Furthermore, let $k\in\NN$ and let $\etap\in\Mp$
  be a Poisson point process with intensity measure
  $\lambda\in\RV(\XX,T,\sS,g,\mu)$. Then $\etap$ is regularly varying
  on $\Mp$ with the lifted scaling, the ideal $\Mpx$, the normalising
  function $g^k$ and the tail measure
  \index{Poisson process!regular variation of}
  \begin{displaymath}
    \mu^*_{(k)}(\Be)=\frac{1}{k!} \int_{\XX^k}
    \one_{\sum_{i=1}^k\delta_{x_i}\in \Be}\mu^{\otimes
      k}(\bdiff (x_1,\dots,x_k))
  \end{displaymath}
  for all Borel $\Be$ in $\Mpx$.
\end{theorem}
\begin{proof}
  Condition (i) of Theorem~\ref{thr:process-k} holds, since
  \begin{displaymath}
    J_{k,B}(A)=\frac{e^{-\lambda(B)}}{k!}
    \lambda^{\otimes k}(A\cap B^k), \quad A\in\sB(\XX^k),
  \end{displaymath}
  see \cite[Example~4.8]{LastPenrose17}. Extend the scaling to act
  componentwise on $\XX^k$ and note that $\lambda\in\RV(\XX,T,\sS,g,\mu)$
  implies $\lambda^{\otimes k}\in\RV(\XX^k,T,\sS^k(k),g^k,\mu^{\otimes
    k})$. Then,
  \begin{displaymath}
    g(t)^k J_{k,T_tB}(T_t A)= g(t)^k \frac{1}{k!}
    e^{-\lambda(T_tB)}\lambda^{\otimes k}(T_t(A\cap
    B^k))\to \frac{1}{k!} \mu^{\otimes k}(A\cap B^k)
    \quad \text{as}\; t\to\infty
  \end{displaymath}
  for all $\mu$-continuity sets $B$ and $\mu^{\otimes k}$-continuity
  sets $A$.  Condition~(ii) follows from
  \begin{displaymath}
    g(t)^k\Prob{\etap(T_tB)\geq k+1}=g(t)^k O(\lambda(T_tB)^{k+1})\to
    0 
    \quad \text{as}\; t\to\infty
  \end{displaymath}
  due to the properties of the Poisson distribution.
\end{proof}

\begin{example}
  \label{ex:inv-lin}
  Let $\etap$ be the Poisson process on $\R^d$ with intensity measure
  $\lambda$. Equip $\R^d$ with linear scaling and the ideal
  $\sR^d_0(m)$ for some $m\in\{1,\dots,d\}$. If
  $\lambda\in\RV(\R^d,\mydot,\sR^d_0(m),g,\mu)$, then $\etap$ is
  regularly varying on $\Mpik[\sR_0^d(m)]$ with with the normalising
  function $g^k$ and the tail measure
  supported on $k$-point counting measures on $\R^d$. Note that the case of a
  stationary Poisson process is excluded, since the Lebesgue measure
  is not finite on some sets in the ideal $\sR^d_0(m)$.

  A similar statement holds if $\R^d$ is equipped with the inverse
  linear scaling given in \eqref{eq:inv-linear}.
  \index{inverse linear scaling} \index{scaling!inverse linear}
  For instance, let
  $\sR$ be the Fr\'echet bornology on $\R^d$, that is, $\sR$ has the
  base $\{x\colons \|x\|<n\}$ for $n\in\NN$. If $\lambda$ is the Lebesgue
  measure, then $\lambda$ is regularly varying on $\sR$ under the
  inverse scaling with the normalising function $g(t)=t^d$, the tail
  measure being $\lambda$ itself due to its homogeneity property. Then,
  for $k=1$, the stationary Poisson process $\etap$ is regularly varying under
  the inverse scaling with the tail measure being the pushforward of
  $\lambda$ under the map $x\mapsto\delta_x$.
\end{example}

\begin{example}[Poisson process of hyperplanes]
  \label{ex:hyperplanes}
  Let $\XX=\R_+\times\SS^{d-1}$, which is the polar representation of
  the Euclidean space $\R^d$ using the Euclidean sphere
  $\SS^{d-1}$. Assume that the scaling on $\XX$ is linear, applied to
  the first component, and equip $\XX$ with the ideal $\sX$ generated
  by sets $(n,\infty)\times\SS^{d-1}$ for $n\in\NN$. Let
  $\lambda=\nu\otimes\kappa$ be the product of a measure
  $\nu\in\Mb[\sR_0]$ and a finite measure $\kappa$ on
  $\SS^{d-1}$. Each point $(r,u)\in\XX$ defines a hyperplane
  $\{x\in\R^d\colons \langle x,u\rangle=r\}$.  \index{hyperplane}
  \index{Poisson hyperplane tessellation} In this way, a Poisson point
  process $\etap$ on $\XX$ with intensity $\lambda$ is interpreted as
  a Poisson process on the family of hyperplanes, the so-called
  \emph{Poisson hyperplane tessellation}; see, e.g.,
  \cite{sch:weil08}. If
  $\nu\in\RV(\R_+,\mydot,\sR_0,g,\theta_\alpha)$, then, for each
  $k\in\NN$, $\etap$ is regularly varying on the ideal $\Mpx[k]$ with
  respect to the scaling applied to the radial component and with the
  normalising function $g^k$.  Its tail measure is the pushforward of
  $\big(\theta_\alpha\otimes\kappa\big)^{\otimes k}/k!$ under the map
  \begin{equation}
    \label{eq:hyp-transform}
    \big((t_1,u_1),\dots,(t_k,u_k)\big)\mapsto
    \delta_{(t_1,u_1)}+\cdots+\delta_{(t_k,u_k)}.
  \end{equation}
  If $\nu$ is the Lebesgue measure on $(0,\infty)$, then $\etap$
  determines the stationary Poisson hyperplane tessellation. In this
  case, we choose the inverse linear scaling on $(0,\infty)$, and then
  $\etap$ is regularly varying on $\Mpik[\sX]$ with $\sX$ generated by
  $(0,n)\times\SS^{d-1}$, $n\in\NN$, see Example~\ref{ex:inv-lin}. The
  normalising function is $t^k$ and the tail measure is the image of 
  $\big(\nu\otimes\kappa\big)^{\otimes k}/k!$ under the map
  given at \eqref{eq:hyp-transform}.
\end{example}

\begin{example}[Translation equivalence of counting measures]
  Let $\Mp[\sR_0^d]$ be the space of counting measures on $\R^d$ with
  linear scaling and the ideal $\sR_0^d$.
  Assume that two counting measures $\me$ and $\me'$ are
  equivalent if there is a $z\in\R^d$ such that $\me(B-z)=\me'(B)$ for
  all Borel sets $B$. In other words, $\me\sim \me'$ if the two
  counting measures are 
  identical up to a translation.
  Condition~\hyperref[condS]{(S)} holds in this case. The
  scaling-invariant measures are the zero measure and all those supported
  at the origin. Note that $[\zero]$, that is, the equivalence class
  of $\delta_0$, is the family $\Mpk[1]$ of
  counting measures with a support consisting of at most a single
  point.
  \index{counting measure!translation equivalence}

  The quotient space of counting measures up to translation
  equivalence is useful when considering the convergence of clusters of
  extremes for time series \cite{basrak18} and marked point processes
  \cite{basrak22}. In these works, measures were translated to move a
  specified point (called an anchor) to the origin. However, our
  setting makes it possible to avoid specifying the anchor (which
  serves the role of the selection map).
  \index{anchor}


  The saturation $[\Be]$ of $\Be\subset\Mp[\sR_0^d]$
  (that is, the union of the equivalence
  classes $[\me]$ for all $\me\in\Be$) is in $\Mpik[\sR_0^d]$ if there exists
  an $A\in\sR_0^d$ such 
  that $\me(A+z)\geq k$ for all $\me\in \Be$ and all $z\in\R^d$. This
  condition means that all $\me\in \Me$ have at least $k$ points
  outside any ball of a fixed radius $\eps>0$. By centring this ball
  at a point in the support of $\me$, we see that
  $[\Me]\in\Mpik[\sR_0^d]$ is only possible if
  $\Me\in\Mpiko[\sR_0^d]$, that is,
  \begin{displaymath}
    \big[\Mpik[\sR_0^d]\big] \subset \Mpiko[\sR_0^d], \quad k\in\NN. 
  \end{displaymath}

  Let $\etap$ be the Poisson process with intensity measure
  $\lambda\in\RV(\R^d,\mydot,\sR_0^d,g,\mu)$.  By
  Theorem~\ref{thr:poisson-process}, the tail measure of $\etap$ on
  $\Mpik[\sR_0^d]$ is supported by counting measures with exactly
  $k$ points and it is obtained as the pushforward of 
  $\mu^{\otimes k}/k!$ by the map
  $(x_1,\dots,x_k)\mapsto \{x_1,\dots,x_k\}$. Then the further
  pushforward of the tail measure to the quotient space is trivial on
  $\big[\Mpik[\sR_0^d]\big]$, but is non-trivial on
  $\big[\Mpikm[\sR_0^d]\big]$ if $k\geq2$, since the support of the tail
  measure $\mu$ always contains points with distances exceeding any
  positive number. The latter is due to the fact that $\mu$ is
  homogeneous and non-trivial. Thus, $q\mu^{\otimes k}$ is non-trivial on
  $q \Mpikm[\sR_0^d]$ and so $q\etap$ is regularly varying in the
  quotient space on the ideal $q \Mpikm[\sR_0^d]$. Its tail
  measure is the image of $\mu^{\otimes k}/k!$ under the map
  $(x_1,\dots,x_k)\mapsto q\{x_1,\dots,x_k\}$.
\end{example}

The following example deals with doubly stochastic Poisson processes. 

\begin{example}[Cox process]
  Let $\etap$ be the \emph{Cox process} with the driving random measure
  \index{Cox process}
  $\Lambda\in\Mb$, that is, conditional on $\Lambda$, $\etap$ is the
  Poisson point process with intensity measure $\Lambda$. Assume that
  $\Lambda(A)$ is square-integrable for each $A\in\sX$. Then
  \begin{displaymath}
    \Prob{\etap(A)\geq
      2}=\E\Big[1-e^{-\Lambda(A)}-\Lambda(A)e^{-\Lambda(A)}\Big]
    \leq \E \min\big(1,\Lambda(A)^2\big).
  \end{displaymath}
  The first Janossy measure is given by
  \begin{displaymath}
    J_{1,B}(A)=\E\big[\etap(A\cap B)\one_{\etap(B)=1}\big]
    =\E \big[\Lambda(A\cap B)e^{-\Lambda(B)}\big].
  \end{displaymath}
  Thus, $\etap$ is regularly varying on $\Mpx[1]$ with the tail
  measure given by \eqref{eq:mu-star-k} with $k=1$ if, for all $B\in\sX$,
  \begin{displaymath}
    g(t)\E\Big[e^{-\Lambda(T_tB)} \int
    f(T_{t^{-1}}x)\Lambda(\bdiff x)\Big]
    \to \int f\diff \mu \quad \text{as}\; t\to\infty
  \end{displaymath}
  for all $f\in\Cont(\XX,\sX)$ with support contained in $B$, and
  \begin{displaymath}
    g(t) \E \min\big(1,\Lambda(T_tA)^2\big)\to 0\quad \text{as}\;
    t\to\infty. 
  \end{displaymath}
\end{example}

\paragraph{Independently marked processes with scaling applied to the
  marks. }
Consider counting measures on the product $\YY\times\XX$ of two Polish
spaces, where the first component designates the position and the
second component is the mark of a point. Let
\begin{displaymath}
  \zetap=\sum_i \delta_{y_i}
\end{displaymath}
be a finite point process on $\YY$ (that is, $\zetap(\YY)$ is
almost surely finite). We call $\zetap$ the \emph{ground process}
\index{ground process} and let
\begin{displaymath}
  \etap=\sum_i \delta_{(y_i,\xi_i)}
\end{displaymath}
be an \emph{independently marked point process},
\index{independently marked point process}
where $\xi_1,\xi_2,\ldots$ are independent copies of a random element
$\xi$ in $\XX$, called the \emph{typical mark}. 
\index{typical mark}

If, for some $k\in\NN$, the random variable $\zetap(\YY)^k$ is
integrable, then the intensity measure of the factorial process
$\zetap^{(k)}$ is finite and is denoted by $\kappa^{(k)}$. This
measure is said to be the 
$k$-th \emph{factorial moment measure} of $\zetap$. Note that
$\kappa^{(k)}$ vanishes if $\zetap(\YY)\leq k-1$ almost surely.
\index{factorial moment measure}

\begin{theorem}
  \label{thr:marked-rv}
  Assume that the Polish space $\XX$ is equipped with a continuous scaling
  and a topologically and scaling consistent ideal $\sX$ with a
  countable open base, and let $\sZ$ be the ideal on $\YY\times\XX$
  generated by $\YY\times B$, $B\in\sX$, where $\YY$ is another Polish
  space.  Furthermore, let $\etap$ be a point process with independent
  marks such that the ground process $\zetap$ satisfies
  $\E\zetap(\YY)^k<\infty$ and $\Prob{\zetap(\YY)\geq k}>0$ for some
  $k\in\NN$, and thus admits a non-trivial finite factorial
  moment measure $\kappa^{(k)}$.
  If the typical mark satisfies
  $\xi\in\RV(\XX,T,\sX,g,\mu)$, then $\etap$ is regularly varying on
  $\Mp[\sZ]$ with the lifted scaling $T_t(y,x)=(y,T_tx)$, the ideal
  $\Mpik[\sZ]$, the normalising function $g^k$, and the tail measure
  \begin{displaymath}
    \mu^*_{(k)}(\Be)=\frac{1}{k!} \int_{\YY^k} \int_{\XX^k}
    \one_{\sum_{i=1}^k\delta_{(y_i,x_i)}\in \Be}\,
    \mu^{\otimes k}(\bdiff (x_1,\dots,x_k))\,\kappa^{(k)}(\bdiff (y_1,\dots,y_k))
  \end{displaymath}
  for Borel $\Be$ from $\Mpik[\sZ]$.
\end{theorem}
\begin{proof}
  We apply Theorem~\ref{thr:process-k}. Let $C$ be a Borel set in
  $\YY^k$, let $B\in\sX$, and denote by $\nu$ the distribution of the
  typical mark.  The Janossy measure of the point process $\etap$ is
  given by
  \begin{align*}
    J_{k,\YY\times B}&(C\times B^k)
    =\frac{1}{k!}\E\big[\etap^{(k)}(C\times B^k)
      \one_{\etap(\YY\times B)=k}\big]\\
    &=\sum_{n=0}^\infty
      \frac{1}{k!}\E\big[\etap^{(k)}(C\times B^k)
      \one_{\etap(\YY\times B)=k}\one_{\zetap(\YY)=k+n}\big].
  \end{align*}
  By independence of the marks,
  \begin{displaymath}
    \E\big[\etap^{(k)}(C\times B^k)
    \one_{\etap(\YY\times B)=k}\one_{\zetap(\YY)=k+n}\big]
    =\E\big[\zetap^{(k)}(C)
    \one_{\zetap(\YY)=k+n}\big]\nu(B)^k(1-\nu(B))^n.
  \end{displaymath}
  Thus,
  \begin{align*}
    J_{k,\YY\times (T_tB)}
    &(C\times (T_tB)^k)
    =\frac{1}{k!} \sum_{n=0}^\infty
      \E\big[\zetap^{(k)}(C)
      \one_{\zetap(\YY)=k+n}\big]\nu(T_tB)^k(1-\nu(T_tB))^n.
  \end{align*}
  The sum is dominated by
  \begin{displaymath}
    \sum_{n=0}^\infty
      \E\big[\zetap^{(k)}(C)
    \one_{\zetap(\YY)=k+n}\big]=\E\big[\zetap^{(k)}(C)\big]<\infty
  \end{displaymath}
  for all $t>0$. Thus, the dominated convergence theorem applies, and
  letting $t\to\infty$ yields that
  \begin{displaymath}
    g(t)^k J_{k,\YY\times (T_tB)}
    \vto \frac{1}{k!} \sum_{n=0}^\infty
      \E\big[\zetap^{(k)}(C)
      \one_{\zetap(\YY)=k+n}\big]\mu^{\otimes k}(B^k)
      =\frac{1}{k!}\kappa^{(k)}(C)\mu^{\otimes k}(B^k).
  \end{displaymath}

  Finally, we need to check that the probability of having at least
  $k+1$ points with large marks is negligible.
  Denote by $N=\zetap(\YY)$ and $p_t=\nu(T_tB)$. Then
  \begin{align*}
    \Prob{\etap\big(\YY\times T_tB)\geq k+1}
    &\leq \sum_{m=k+1}^\infty \Prob{N=m} \binom{m}{k+1} p_t^{k+1}\\
    &\leq p_t^{k+1}I(n)+\Prob{N>n},
  \end{align*}
  where $n=n(t)$ grows with $t$ so that $n(t)\sim ag(t)$ for a
  constant $a>0$, and 
  \begin{displaymath}
    I(n)=\sum_{m=k+1}^n \Prob{N=m} \binom{m}{k+1}.
  \end{displaymath}
  Since $\E N^k<\infty$, we have
  \begin{equation}
    \label{eq:Nk-sum}
    \sum_{m=1}^\infty m^{k-1}\Prob{N\geq m}<\infty.
  \end{equation}
  Therefore,  $m^k\Prob{N>m}\to 0$ as
  $m\to\infty$, and thus $g(t)^k\Prob{N>n(t)}\to 0$ as $t\to\infty$.
  Furthermore,
  \begin{displaymath}
    I(n)\leq \frac{n}{(k+1)!}\sum_{m=1}^n m^k\Prob{N=m}
    \leq \frac{n}{(k+1)!}\E N^k.
  \end{displaymath}
  Thus,
  \begin{displaymath}
    g(t)^k p_t^{k+1}I(n)
    \leq \frac{\E N^k}{(k+1)!} g(t)^k p_t^{k+1} n(t)
    =\frac{\E N^k}{(k+1)!} (g(t)p_t)^{k+1} \frac{n(t)}{g(t)}
    \sim \frac{\E N^k}{(k+1)!} a.
  \end{displaymath}
  Choosing $a=\lim n(t)/g(t) > 0$ sufficiently small and noticing
  that $g(t)p_t$ converges to a constant by the regular variation of
  the marks confirm that
  \begin{displaymath}
    g(t)^k \Prob{\etap\big(\YY\times T_tB)\geq k+1}\to0 \quad
    \text{as}\; t\to\infty. \qedhere
  \end{displaymath}
\end{proof}

Theorem~\ref{thr:marked-rv} replicates Theorem~3.2 of \cite{dtw22},
however, with a substantially shorter proof. It is possible to
generalise it to scalings which act jointly on the ground process and
the marks --- this is straightforward to do by checking the conditions of
Theorem~\ref{thr:process-k}. 
Lemma~\ref{lemma:conv-det-product} may be used to confirm the regular
variation of the Janossy measure in the product space.

\begin{example}[Ground process on a single-point space]
  Let $\XX$ be a Polish space with a continuous scaling and a
  topologically and scaling consistent ideal $\sX$ with a countable
  open base.  Consider the point process
  \begin{displaymath}
    \etap = \sum_{i=1}^Z \delta_{\xi_i},
  \end{displaymath}
  where $(\xi_i)_{i\geq1}$ are i.i.d.\ in $\XX$ and $Z$ is a
  non-negative integer-valued 
  random variable that is integrable of order $k$ and independent of
  all $\xi_i$. This point process can be considered as an
  independently marked point process on the product of a single-point
  space $\YY=\{0\}$ and $\XX$. The ground process $\zetap$ has a point
  at $0$ of multiplicity $Z$, so that its factorial moment measure of
  order $k$ is supported by a single point with mass
  $\E \big[Z(Z-1)\cdots(Z-k+1)\big]$.  If
  $\xi_1\in\RV(\XX,T,\sX,g,\mu)$, then $\etap$ is regularly varying on
  $\Mpx$ with the normalising function $g^k$ and the tail measure
  \begin{displaymath}
    \mu^*_{(k)}(B)=\E \binom{Z}{k}\int_{\XX^k} \one_{\sum_{i=1}^k
      \delta_{x_i}\in B}
    \mu^{\otimes k}(\bdiff (x_1,\dots,x_k))
  \end{displaymath}
  for Borel $B\subset\Mp[\sX]$. If  $k=1$, then
  $\etap$ is
  regularly varying on $\Mpx[1]$ with the normalising function
  $g$. Its tail measure $\mu^*_{(1)}$ equals $\E Z$ times
  the pushforward of $\mu$
  under the map $x\mapsto\delta_x$.
\end{example}

\begin{example}[Shot-noise process]
  Let
  \begin{displaymath}
    \etap=\sum_{i=1}^N \delta_{(y_i,\xi_i)}
  \end{displaymath}
  be an independently marked point process, where the ground process
  $\{y_i, i=1,\dots,N\}\subset\R$ is a finite point process.  Assume
  that the typical mark is positive and satisfies
  $\xi\in\RV(\R_+,\mydot,\sR_0,g,\theta_\alpha)$. Assume also that the
  ground process has almost surely bounded total mass and non-trivial
  intensity measure $\kappa$. Then, by Theorem~\ref{thr:marked-rv}
  with $k=1$, $\etap$ is regularly varying on the space $\Mp[\sZ]$ of
  point processes on $\R\times\R_+$ with the lifted scaling acting on
  the mark component, the ideal $\Mpi[\sZ]$, and the normalising
  function $g$, where $\sZ$ is generated by the sets
  $\R\times [s,\infty)$, $s>0$.
  
  Let $f:\R\to\R$ be a continuous function with bounded support which
  is not identically zero. The sum
  \begin{displaymath} 
   \zeta(u)=\sum_{i=1}^N \xi_i f(u-y_i),\quad u\in\R,
  \end{displaymath}
  is an example of a \emph{shot-noise process}.
  \index{shot-noise process} By construction, the map that associates
  $\etap$ with 
  $\zeta$ is continuous as a map from $\Mp$ to $\Cp(\R)$, the space of
  continuous functions with the topology of pointwise
  convergence. This map is also a morphism (which respects the
  scaling) if linear scaling acts only on the marks. By the continuous
  mapping theorem, provided the pushforward tail measure is
  non-trivial, $\zeta$ is regularly varying in $\Cp(\R)$ with the
  ideal $\sC$ given in \eqref{eq:ideal-Cp}.  The tail measure of
  $\zeta$ is the pushforward of $\kappa\otimes \theta_\alpha$ under
  the map that sends $(y,s)$ to the 
  function $u\mapsto s f(u-y)$, $u\in\R$.
\end{example}

\paragraph{Random probability measures}
\label{sec:rpm}
Let $\XX$ be a completely regular Sous\-lin space (in particular, a
Polish space) with a continuous scaling.  Consider the family $\Mone$
of probability measures $\nu$ on $\XX$ with the topology of weak
convergence, which makes $\Mone$ a Souslin space as well; see
\cite[Theorem~5.1.8]{bogachev18}. A \emph{random probability measure} is a
random element $\xip$ in $\Mone$ with its Borel $\sigma$-algebra.
\index{random probability measure}
We use the scaling defined at
\eqref{eq:23} as $(T_t\nu)(B)=\nu(T_{t^{-1}}B)$; this scaling is
continuous on $\Mone$.

Let $\sX$ be a topologically and scaling consistent ideal on $\XX$
with a countable base. For $r\in(0,1)$, define the ideal $\sN_r$
on $\Mone$ as the family of all sets $\Npp\subset\Mone$ such that
\begin{displaymath}
  \inf\big\{\nu(A)\colons \nu\in \Npp\big\}>r
\end{displaymath}
for some Borel set $A\in\sX$. This ideal is topologically and scaling
consistent and admits a countable open base.
Below we give several examples of using the continuous mapping theorem
to derive the regular variation property of random probability
measures.

\begin{example}[Uniform distribution on random interval]
  \label{ex:prob-m}
  Consider the family of probability measures on $\XX=\R^1$ with 
  linear scaling and the ideal $\sR_0$. 
  Let $\psi$ be the map that associates with each $(x_1,x_2)\in\R^2$
  the uniform distribution $\nu=\psi(x_1,x_2)$ on
  \begin{displaymath}
    \big[\min(x_1,x_2),\max(x_1,x_2)\big]=\conv(\{x_1,x_2\}),
  \end{displaymath}
  where $\conv(\cdot)$ denotes the convex hull. If $x_1=x_2$,
  then $\nu$ is the Dirac measure at $x_1$. This map
  is a morphism between $\R^2$ with the ideal $\sR_0^2$ and $\Mone$
  with the ideal $\sN_r$ for any $r\in(0,1)$, since each $A\in\sR_0$
  is contained in the complement of $[-a,a]$ for some $a>0$ and thus 
  \begin{displaymath}
    \big\{(x_1,x_2)\colons \psi(x_1,x_2)([-a,a]^c)> r \big\}
    \subset \big\{(x_1,x_2)\in\R^2\colons \max(|x_1|,|x_2|)>a\}
    \in\sR_0^2.
  \end{displaymath}
  Indeed, if both $x_1$ and $x_2$ belong to $[-a,a]$, then
  $\psi(x_1,x_2)([-a,a]^c)=0$, and thus a positive value in the latter
  expression means that at least one endpoint must lie outside of
  $[-a,a]$.
  
  If $(\xi_1,\xi_2)\in\RV(\R^2,\mydot,\sR_0^2,g,\mu)$ is a regularly
  varying random vector, then the random probability measure
  $\xip=\psi(\xi_1,\xi_2)$ is regularly varying on $\sN_r$ and its tail
  measure is the pushforward of $\mu$ under the map
  $(x_1,x_2)\mapsto \psi(x_1,x_2)$. For instance, if $\xi_1,\xi_2$ are
  i.i.d.\ and regularly varying on $\sR_0^1$, then $\mu$ is supported
  on the axes and so the tail measure of $\xip=\psi(\xi_1,\xi_2)$ is
  supported on uniform distributions on intervals having one endpoint at
  the origin. 
\end{example}

\begin{example}[Normal distribution with random mean and variance]
  Let $\XX=\R\times\R_+$ with linear scaling. Consider the map $\psi(x,b)$
  from $(x,b)\in\R\times\R_+$ to $\Mone$ obtained by letting $\psi(x,b)$ be the
  normal distribution with mean $x$ and variance $b^2$. This map is a
  morphism. For every $a>0$ and $b>0$, we have 
  \begin{align*}
    \big\{(x,b)\in \R\times\R_+\colons &\; \psi(x,b)([-a,a]^c)> r \big\}\\
    &= \big\{(x,b)\in \R\times\R_+\colons \Phi(-(a-x)/b)+\Phi(-(a+x)/b)>
    r \big\},
  \end{align*}
  where $\Phi$ is the c.d.f.\ of the standard normal distribution.
  The case $b=0$ is understood by continuity, corresponding to the
  Dirac measure at $x$.  The function $\Phi(-(a-x)/b)+\Phi(-(a+x)/b)$
  decreases for $x\leq 0$ and increases for $x\geq0$; its minimum
  is attained at $x=0$ and equals $2\Phi(-a/b)$.
  Therefore, the set  
  \begin{displaymath}
    \big\{(x,b)\in \R\times\R_+\colons \Phi(-(a-x)/b)+\Phi(-(a+x)/b)>
    r \big\}
  \end{displaymath}
  is a subset of the complement of $[-a/2,a/2]\times[0,b_r]$ with
  $b_r=-a/(2\Phi^{-1}(r/2))$, and thus belongs to $\sR_0^2$. Note that
  $b_r>0$, since $\Phi^{-1}(r/2)<0$. Thus, the map $\psi$ is
  $(\sR_0^2,\sN_r)$ bornologically consistent.  By the continuous
  mapping theorem, if
  $(\xi,\eta)\in\RV(\R\times\R_+,\mydot,\sR_0^2,g,\mu)$, then
  $\xip=\psi(\xi,\eta)\in\RV(\Mone,\mydot,\sN_r,g,\psi\mu)$.
\end{example}

\begin{example}[Empirical probability measures]
  Let $\XX=\R^d$ with the ideal $\sX=\sR_0^d$. 
  Let $\nu^*_n$ be the empirical probability measure constructed from
  a sample $\xi_1,\dots,\xi_n$ of i.i.d.\ realisations of
  $\xi\in\RV(\R^d,\mydot,\sR_0^d,g,\mu)$. Let $r\in(0,1)$, and denote by $m$
  the smallest integer strictly greater than $nr$, that is,
  $m=\lfloor nr\rfloor+1$. The map
  \begin{displaymath}
    \psi:(x_1,\dots,x_n)\mapsto \frac{1}{n}\sum_{i=1}^n \delta_{x_i}
  \end{displaymath}
  is a morphism that is bornologically consistent between the product
  ideal $\sX^n(m)$ (with $\sX=\sR_0^d$) and the ideal $\sN_r$. Indeed,
  for $A$ being the complement of $[-a,a]^d$, we have $\nu_n^*(A)>r$
  exactly if at least $m$ points out of $x_1,\dots,x_n$ lie in $A$.
  Since $\xi_1,\dots,\xi_n$ are independent,
  the random vector $(\xi_1,\dots,\xi_n)$ is regularly varying on
  $\sX^n(m)$ with the normalising function $g^m$ and the tail measure
  \begin{displaymath}
    \mu'\big(\bdiff (x_1,\dots,x_n)\big)=\sum_{I\subset\{1,\dots,n\},\card(I)=m}
    \prod_{i\in I}\mu(\bdiff x_i)\prod_{j\notin I}\delta_0(\bdiff x_j).
  \end{displaymath}
  The pushforward $\psi\mu'$ of this measure is non-trivial on $\sN_r$,
  since, for some $A\in\sR_0^d$ with $\mu(A)>0$,
  \begin{multline*}
    \mu'\big(\{(x_1,\dots,x_n)\colons \psi(x_1,\dots,x_n)(A)> r\}\big)\\
    =\mu'\big(\{(x_1,\dots,x_n)\colons \card(\{i:x_i\in A\})\geq
    m\}\big)
    \geq \mu(A)^m>0.
  \end{multline*}
  Thus, $\xip=\nu^*_n$ is regularly varying on $\sN_r$ with the
  normalising function $g^m$ and tail measure $\tilde{\mu}=\psi\mu'$.
  This tail measure can be
  written as the pushforward of $\binom{n}{m}\mu^{\otimes m}$ by the
  map
  \begin{displaymath}
    (x_1,\dots,x_m)\mapsto
    n^{-1}\Big(\delta_{x_1}+\cdots+\delta_{x_m}
    +(n-m)\delta_0\Big).
  \end{displaymath}
\end{example}

\section{Random closed sets}
\label{sec:regul-vari-rand-1}

\paragraph{Topology and scaling}
Let $\sF=\sF(\XX)$ be the family of closed subsets of a Polish space
$\XX$ metrised by the metric $\dmet$. Since not much is known about
topologies on closed sets if $\XX$ is not Polish (see \cite{beer93}),
the Polish assumption on $\XX$ is imposed throughout this section.
\index{space!of closed sets}

Each closed set $F$ in $\XX$ is uniquely identified by its
\emph{distance function}
\index{distance function}
\begin{displaymath}
  \dmet(x,F)=\inf\big\{\dmet(x,y)\colons y\in F\big\}, \quad x\in\XX.
\end{displaymath}
The distance function of the empty set is set to be infinite. The space
$\sF$ is endowed with the \emph{Wijsman topology}, namely, $F_n\to F$
if $\dmet(x,F_n)\to \dmet(x,F)$ as $n\to\infty$ for all $x\in\XX$.
\index{Wijsman topology}
It is known that the Wijsman topology turns $\sF$ into a Polish space;
see \cite[Theorem~2.5.4]{beer93}.  The Wijsman topology on the family
$\sF'=\sF\setminus\{\emptyset\}$ of non-empty closed sets in a ball
compact space can be
metrised by the Hausdorff--Busemann metric
\index{Hausdorff--Busemann metric}
\begin{displaymath}
  \dmet_\sF(F_1,F_2)=\sup_{x\in\XX} e^{-\dmet(x_0,x)}
  \big|\dmet(x,F_1)-\dmet(x,F_2)\big|,
\end{displaymath}
where $x_0$ is any fixed point in $\XX$.

A Polish space is said to be \emph{ball compact} (also called
\emph{Heine--Borel}) if each closed bounded ball is compact.
\index{space!ball compact} \index{Fell topology}
In this case, the Wijsman topology coincides with the Fell
topology on $\sF$; see \cite[Theorem~5.1.10]{beer93}. A sequence
$\{F_n,n\geq1\}$ converges to $F$ in the \emph{Fell topology} if, for
each open $G$ such that $F\cap G\neq\emptyset$ and each compact $K$
such that $K\cap F=\emptyset$, we have $F_n\cap G\neq\emptyset$ and
$F_n\cap K=\emptyset$ for all sufficiently large $n$.  While the Fell
topology may be defined on a general topological space $\XX$, it
fails to be Hausdorff if $\XX$ is not locally compact.
Note that each locally compact second-countable Hausdorff
space is metrisable by a complete metric and thus is Polish.

The scaling operation $T_t$ on $\XX$ is naturally extended to act on
$\sF$ by setting
\begin{displaymath}
  T_tF=\{T_tx\colon x\in F\}
\end{displaymath}
and $T_t\emptyset=\emptyset$; it will also be denoted by $T_t$. Note that the
Hausdorff--Busemann metric $\dmet_\sF$ is not homogeneous under the
scaling, even if $\dmet$ is homogeneous.

\begin{lemma}
  Assume that $\XX$ is a Polish space equipped with a 
  continuous scaling and that at least one of the following conditions holds.
  \begin{enumerate}[(a)]
  \item The metric $\dmet$ on $\XX$ is homogeneous.
  \item $\XX$ is ball compact.
  \end{enumerate}
  Then the scaling on $\sF$ is continuous in the Wijsman topology.
\end{lemma}
\begin{proof}
  Since $\sF$ is Polish, it suffices to show sequential
  continuity. Assume that $t_n\to t$ and $F_n\to F$ in the Wijsman
  topology as $n\to\infty$.

  (a) Assume that $x_n\to x$ in $\XX$ and $F_n\to F$ in the Wijsman
  topology. By the triangle inequality,
  \begin{displaymath}
    |\dmet(x_n,F_n)-\dmet(x,F)|
    \leq \dmet(x_n,x)
    +|\dmet(x,F_n)-\dmet(x,F)|,
  \end{displaymath}
  so that $\dmet(x_n,F_n)\to \dmet(x,F)$ as $n\to\infty$ if
  $F\neq\emptyset$. Applying this argument to $x_n=T_{t_n^{-1}}x$
  yields that
  \begin{displaymath}
    \dmet(x,T_{t_n}F_n)=t_n \dmet(T_{t_n^{-1}}x,F_n)\to
    t\dmet(T_{t^{-1}}x,F)=\dmet(x,T_tF)\quad \text{as}\; n\to\infty. 
  \end{displaymath}
  If $F=\emptyset$, then $\dmet(x,F_n)\to\infty$, and the estimate
  $|\dmet(x_n,F_n)-\dmet(x,F_n)|\leq \dmet(x_n,x)$ yields
  $\dmet(x_n,F_n)\to\infty$.

  (b) In this case, the Wijsman topology coincides with the Fell
  topology. Assume that $(T_tF\cap G)\neq\emptyset$ for an open set
  $G$, so that there exists $x\in F$ such that $T_tx\in G$. Since
  $F_n\to F$, there exists a sequence $x_n\in F_n$ such that
  $x_n\to x$. Then $T_{t_n}x_n\to T_tx$, hence,
  $T_{t_n}x_n\in G$ and 
  $T_{t_n}F_n\cap G\neq\emptyset$ for all sufficiently large $n$.
  
  Assume that $T_tF\cap K=\emptyset$ for a compact set $K$ and
  $t_n\to t$ as $n\to\infty$. If $(T_{t_{n_k}}F_{n_k}\cap
  K)\neq\emptyset$ for a 
  subsequence $(n_k)_{k\in\NN}$, take $x_{n_k}\in F_{n_k}$ such that
  $T_{t_{n_k}}x_{n_k}\in K$. Passing to a further subsequence, assume
  that $T_{t_{n_k}}x_{n_k}\to y\in K$.  Applying $T_{t_{n_k}^{-1}}$ to
  this convergence and using the continuity of the scaling, we see
  that $x_{n_k}\to T_{t^{-1}}y$. Since 
  $F_n\to F$ we have $T_{t^{-1}}y\in F$, so that $y$ belongs to $T_t F\cap K$
  and the latter intersection is non-empty, which is a contradiction.
\end{proof}

\begin{remark}
  \label{rem:sets-scaling}
  Recall that $\zero$ stands for the family of all scaling-invariant
  elements of $\XX$.  The family of scaling-invariant elements in
  $\sF$ is substantially larger than the family of subsets of $\zero$;
  for instance, all closed cones in $\XX$ (and $\XX$ itself)
  are scaling invariant under $T_t$
  for all $t>0$. The empty set is also scaling invariant. In
  particular, the space $\sF$ is not star-shaped with respect to the
  scaling operation $T_t$.  In the space of
  closed sets, one can easily identify members that are
  scaling invariant under $T_t$ only for some $t\neq1$. For this, fix
  $x\notin\zero$ and $t\neq1$, and let $F=\cl\{T_{t^k}x\colons k\in\ZZ\}$,
  where $\ZZ$ denotes the set of all integers. Then $F$ is
  scaling invariant in the sense that $T_tF=F$ for this particular
  $t\neq1$ and all its powers.
\end{remark}

For a given set $A\subset\XX$, denote
\begin{align*}
  \sF_A&=\{F\in\sF\colons F\cap A\neq\emptyset\},\\
  \sF^A&=\{F\in\sF\colons F\cap A=\emptyset\},\\
  \sF_{\supset A}&=\{F\in\sF\colons F\supset A\}.
\end{align*}
In the following, we consider random closed sets that almost
surely take values in a subfamily of $\sF$ that is a cone in the
chosen scaling. 
For instance, such a subfamily can be $\sF_{\supset \cone}$ for a
given cone $\cone$ in $\XX$, e.g., the family of all closed sets
containing the origin in $\XX=\R^d$ with $\cone=\{0\}$. Other examples
are the family $\sF^\cone$ of all sets that do not intersect $\cone$
or the family of all singletons.  Passing to a subfamily is often
motivated by the fact that a continuous modulus can be explicitly
defined on subfamilies of $\sF$ rather than on the entire family $\sF$.

An important subfamily of $\sF$ is the family $\sK=\sK(\XX)$ of
\emph{compact} sets in $\XX$. This family is metrised by the
\emph{Hausdorff metric}
\index{Hausdorff metric}
\begin{displaymath}
  \dmet_{\mathrm{H}}(K_1,K_2)=\inf\big\{r>0\colons
  K_1\subset K_2^r,K_2\subset
  K_1^r\big\},\quad K_1,K_2\in\sK,
\end{displaymath}
where $K^r$ stands for the $r$-envelope (or $r$-neighbourhood) of $K$
in the metric 
$\dmet$, see \eqref{eq:neighbourhood-r}. The distance from a non-empty
$K$ to the empty set is defined as infinite. It is well known that
\begin{displaymath}
  \dmet_{\mathrm{H}}(K_1,K_2)=\sup_{x\in\XX} \big|\dmet(x,K_1)-\dmet(x,K_2)\big|
\end{displaymath}
for non-empty $K_1$ and $K_2$. 
If $\XX$ is Polish, then $\sK$ with the Hausdorff metric is also
Polish; see, e.g., \cite[Theorem~D.9]{mo1}.

In $\XX=\R^d$ we often consider the family $\sK_c^d$ of
non-empty \emph{compact convex} sets.  The trace of the Fell topology on
$\sK_c^d$ coincides with the Hausdorff metric topology; see
\cite[Theorem~12.3.4]{sch:weil08}.  Unlike the family $\sF$ of all
closed sets in $\R^d$, the space $\sK_c^d$ is a star-shaped metric
space with the only scaling-invariant element being the singleton $\{0\}$.

\paragraph{Ideals on closed sets}
A generic ideal on $\sF$ is denoted by $\fF$ and on $\sK$ by $\fK$.
Below we consider a special (but still very general construction)
that derives an ideal on $\sF$ 
from an ideal $\sS$ on the carrier space $\XX$. Namely, the ideal
$\sS$ gives rise to the ideal $\fF^\sS$ on $\sF$, which is the
smallest ideal containing the families of sets
\index{ideal!on closed sets}
\begin{displaymath}
  \sF^{A^c}=\{F\in\sF\colons F\subset A\}
\end{displaymath}
for all $A\in\sS$.  If $\sS=\sS_\cone$ is the metric exclusion ideal
obtained by excluding a cone $\cone$, then $\fF^\sS$ is the smallest
ideal that contains the families
\begin{displaymath}
  \sF^{\cone^r}=\{F\colons F\cap\cone^r=\emptyset\}, \quad r>0,
\end{displaymath}
where $\cone^r=\{x\colons \dmet(x,\cone)\leq r\}$ is the $r$-envelope of
$\cone$. In this case,  we 
write $\fF^\cone$ instead of $\fF^{\sS_\cone}$. It is easy to see that the
ideal $\fF^\cone$ is scaling consistent if the underlying metric on
$\XX$ is homogeneous.

Recall that $\sF_{\supset \cone}$ denotes the family of all closed
sets that contain $\cone$. The ideal $\fF_{\supset \cone}$ is the
smallest ideal that 
contains the families $\{F\in\sF_{\supset \cone}\colons
F\not\subset\cone^r\}$ for $r>0$.

\begin{example}
  Let $\XX=\R^d$ with the Euclidean metric and the ideal
  $\sR_0^d(k)$, see Example~\ref{example:product}. If $k=1$, then
  $\fF^{\sR_0^d(1)}=\fF^0$, that is, 
  $\fF^\cone$ with $\cone=\{0\}$. The ideal $\fF^0$ consists of
  families of closed sets such that the union of all sets in the
  family, after taking closure, is
  disjoint from the origin. The ideal $\fF^{\sR_0^d(d)}$ consists of
  all families of closed sets such that their union does not intersect
  a metric neighbourhood of the union of all coordinate hyperplanes. 

  If $\cone=\{0\}$, then $\sF_{\supset \cone}$ is the family of all
  closed sets that contain the origin, and the ideal
  $\fF_{\supset \cone}$ (denoted by $\fF_0$) consists of families of
  closed sets that all contain the origin and all are not contained in
  $B_r(0)$ for some $r>0$. 
\end{example}

In the following, we identify two continuous moduli on subfamilies of
$\sF$ and relate them to ideals.

\begin{proposition}
  \label{ex:F-zero}
  Let $\modulus$ be an inf-compact continuous modulus on a ball
  compact Polish space $\XX$. Then the function
  $\ttau(F)=\inf\{\modulus(x)\colons x\in F\}$ (see also \eqref{eq:12}) is a
  continuous modulus on the family $\sF'$ of non-empty closed sets. The
  ideal generated by $\ttau$ on $\sF'$ coincides with the ideal $\fF^\cone$,
  where $\cone=\{\modulus=0\}$.
\end{proposition}
\begin{proof}  
  The function $\ttau$ is evidently homogeneous and finite.  Assume
  that $F_n\to F$ in the Fell topology and $\ttau(F)<a$ for some
  $a>0$. Since $F$ intersects the open set $\{\modulus<a\}$, Fell
  convergence implies that 
  $F_n$ intersects this set for all sufficiently large $n$, so that
  $\ttau(F_n)<a$. If $\ttau(F)>a>0$, then $F$ does not intersect the 
  compact set $\{\modulus\leq a\}$, so that $F_n$ does not intersect
  this set and $\ttau(F_n)>a$ for all sufficiently large $n$.

  If $F$ intersects $\{\modulus=0\}$, then $\ttau(F)=0$. Assume now
  that $F$ does not intersect $\{\modulus=0\}$ and $\ttau(F)=0$, so that
  there exists a sequence $x_n\in F$ with $\modulus(x_n)\to0$.
  Since $x_n\in\{\modulus\leq\eps\}$ for all sufficiently large $n$
  and $\modulus$ is inf-compact, there is a subsequence $(x_{n_k})_{k\in\NN}$
  that converges to $x\in F$. By the continuity of the modulus,
  $\modulus(x)=0$, which contradicts the fact that $F$
  does not intersect the set $\{\modulus=0\}$. This argument also
  confirms that $\ttau(F)>0$, implying that $F$ is not scaling
  invariant.

  Since $\ttau(F)>a$ is equivalent to $F\subset \{\modulus>a\}$, the
  ideal generated by the families $\{F\in\sF'\colons \ttau(F)>a\}$ for all
  $a>0$ coincides with the ideal $\fF^{\sS_\modulus}$. By
  Lemma~\ref{lemma:inf-compact-excluded}, the latter ideal coincides
  with the metric exclusion ideal generated by excluding
  $\cone=\{\modulus=0\}$.
\end{proof}

A set $F$ is said to be \emph{star-shaped} if $T_tF\subset F$ for all
$t\in(0,1]$. Denote by $\sF_s(\cone)$ the family of all star-shaped
sets $F$ in $\sF_{\supset \cone}$ with $\cone=\{\modulus=0\}$, and
let
\index{star-shaped set}
\begin{equation}
  \label{eq:suptau}
  \suptau(F)=\sup\big\{\modulus(x)\colons x\in F\big\},\quad F\in\sF.
\end{equation}

\begin{proposition}
  \label{ex:tau-bounded}
  Let $\modulus$ be an inf-compact continuous modulus on a ball
  compact Polish space $\XX$ that generates the metric exclusion
  ideal obtained by excluding a cone $\cone$. Then the function
  $\suptau$ is a continuous modulus on $\sF_s(\cone)$, which is
  positive on all sets $F\neq \cone$ and generates the ideal
  $\fF_{\supset \cone}$.
\end{proposition}
\begin{proof}
  The function $\suptau$ is homogeneous and strictly positive on sets
  in $\sF_s(\cone)$ that do not coincide with $\cone$. We show that
  it is also continuous. Assume that $F_n\to F$ in the Fell
  topology. Let $\suptau(F)>t$ for some $t>0$.  Then
  $F\cap\{\modulus>t\}\neq\emptyset$, so that
  $F_n\cap\{\modulus>t\}\neq\emptyset$ and $\suptau(F_n)>t$ for all
  sufficiently large $n$.  Assume that $\suptau(F)<t$. Then
  $F\cap\{\modulus=t\}=\emptyset$. Since $\{\modulus=t\}$ is a compact
  set (as a closed subset of $\{\modulus\leq t\}$),
  $F_n \cap\{\modulus=t\}=\emptyset$ for all sufficiently large
  $n$.  By the star-shapedness assumption, this implies $\suptau(F_n)<
  t$.  If 
  $\suptau(F)=\infty$, then for each $c>0$ there exists $x\in F$ with
  $\modulus(x)>c$. Assume that $\limsup \suptau(F_{n_k})<c<\infty$ for a
  subsequence $(n_k)$. Since $\{\modulus>c\}$ is open, Fell
  convergence $F_n\to F$  implies that $F_{n_k}\cap
  \{\modulus>c\}\neq\emptyset$ 
  eventually, meaning that $\suptau(F_{n_{k_l}})>c$, which is a
  contradiction. Thus, $\suptau(F_n)\to\infty$. 
  
  Finally,
  \begin{displaymath}
    \big\{F\in\sF_s(\cone)\colons \suptau(F)>\delta\big\}\subset 
    \big\{F\in\sF_s(\cone)\colons F\not\subset \{\modulus\leq r\}\big\}
    \subset \big\{F\in\sF_s(\cone)\colons \suptau(F)>\eps\big\}
  \end{displaymath}
  for $0<\eps<r<\delta$. The proof is finished by referring to
  Lemma~\ref{lemma:inf-compact-excluded} which establishes that the
  ideal generated by $\modulus$ coincides with the metric exclusion
  ideal obtained by excluding $\cone$.
\end{proof}

Note that $\suptau$ is not continuous if we drop the star-shapedness
assumption.  Indeed, fix a compact set $F$ in $\R^d$ and consider a
sequence $(x_n)$ with $\|x_n\|\to\infty$ so that $x_n$ eventually
escapes any compact set. Then, while $F_n=F\cup\{x_n\}\to F$ in the
Fell topology, $\suptau(F_n)\geq \modulus(x_n)$ and thus does not converge to
$\suptau(F)$.

The following result implies that the moduli $\ttau$ and $\suptau$ are
continuous on the family of non-empty compact sets without the need to
impose additional assumptions.

\begin{proposition}
  Assume that $\modulus$ is a continuous modulus on a Polish space $\XX$. Then
  the moduli $\ttau$ and $\suptau$ are continuous
  on $\sK$ with the Hausdorff metric.
\end{proposition}
\begin{proof}
  Assume that $K_n\to K$ with $K\neq\emptyset$. If $\ttau(K)>t$, then $K$ does
  not intersect the $\{\modulus\leq t\}$, hence also does not
  intersect a neighbourhood of this set in view of compactness of $K$. By
  compactness, there is an $r>0$, such that $K^r$ does not intersect
  $\{\modulus\leq t\}$. Then $K_n\cap\{\modulus\leq t\}=\emptyset$,
  since $K_n\subset K^r$ for all sufficiently large $n$. Hence,
  $\ttau(K_n)>t$ for all sufficiently large $n$. If $\ttau(K)< t$,
  then there is a point $x\in K$ with $\modulus(x)<t$. Then
  $\modulus(y)<t$ for all $y\in B_r(x)$ for some $r>0$. Since $x\in K$
  and $K\subset K_n^r$ for all sufficiently large $n$, we have
  $B_r(x)\cap K_n\neq\emptyset$, meaning that $\ttau(K_n)<t$.

  If $\suptau(K)<t$, then
  $K\cap \{\modulus\geq t\}=\emptyset$. By the compactness of $K$, we have
  that $K^r\cap \{\modulus\geq t\}=\emptyset$ for some $r>0$, so that
  $K_n\cap \{\modulus\geq t\}=\emptyset$ and $\suptau(K_n)<t$ for all
  sufficiently large $n$. If $\suptau(K)>t$, then
  $K\cap\{\modulus>t\}\neq\emptyset$, so that the open set
  $\{\modulus>t\}$ contains a ball of radius $r>0$ centred at
  $x\in K$. This ball intersects $K_n$ for all sufficiently large $n$,
  so that $\suptau(K_n)>t$ for all sufficiently large $n$.
\end{proof}

Given that $\sF$ is Polish, it is possible to equip it with metric
exclusion ideals. For instance, it is possible to exclude a family
of closed sets, such as a point, the family of all singletons, or the
family of lower-dimensional sets for closed sets in $\R^d$.

\paragraph{Ideals on compact sets in $\R^d$}
\label{sec:ideals-compact-sets}
Consider the family $\sK^d$ of compact sets in $\R^d$ equipped with
linear scaling and the modulus $\modulus(x)=\|x\|$ given by the
Euclidean norm. The ideal $\fK^0$ is generated by the modulus $\ttau$,
which equals the distance from the origin to the nearest point of $K$
and is infinite if $K=\emptyset$. Hence, $\fK^0$ is the trace of
$\fF^0$ on $\sK^d$.  Furthermore, let $\fK_0$ be the Hausdorff metric
exclusion ideal on non-empty sets in $\sK^d$ obtained by excluding
$\cone=\{0\}$. This ideal is generated by the modulus
\begin{equation}
  \label{eq:14}
  \suptau(K)
  =\sup\big\{\|x\|\colons x\in K\big\}=\|K\|, \quad K\in\sK^d, \; K\neq\emptyset,
\end{equation}
with the convention $\suptau(\emptyset)=0$. Note that $\{\emptyset\}$ belongs to
$\fK^0$, but does not belong to $\fK_0$. The value of $\suptau(K)$ is
the distance from the origin to the farthest point of $K$, while
$\ttau(K)$ is the distance from the origin to the nearest point of
$K$.
\index{ideal!on compact sets}

On suitable subfamilies, for instance on compact convex sets containing
the origin, define the modulus 
\begin{equation}
  \label{eq:inscribed}
  \modulus_{\inf}(K)=\ttau(K^c)=\sup\big\{r>0\colons B_r(0)\subset K\big\},
\end{equation}
which equals the maximal radius of the ball centred at the origin and
inscribed in $K$. The corresponding ideal is denoted by
$\fK_{\inf}$. Let $\fK_{(\inf)}$ be the ideal generated by the modulus
that is the radius of the largest inscribed ball in $K$, that is, this
modulus is the supremum of $r>0$ such that $B_r(x)\subset K$ for some
$x$.

For $k=0,\dots,d-1$, let $\fK_{(k)}$ be the metric exclusion ideal on
the family $\sK^d$ of compact sets obtained by excluding the
family of all compact sets with affine dimension at most $k$. The ideal
$\fK_{(k)}$ is generated by the modulus 
\begin{displaymath}
  \modulus_{(k)}(K)=
  \inf_{H\in A(d,k)} \inf_{L\in\sK^d,L\subset H}
  \dmet_{\mathrm{H}}(K,L),\quad K\in\sK^d,
\end{displaymath}
where $A(d,k)$ is the affine Grassmannian,
\index{Grassmannian!affine}
that is, the family of all
affine subspaces of dimension $k$. Note that
\begin{displaymath}
  \inf_{L\in\sK^d,L\subset H}\dmet_{\mathrm{H}}(K,L)
  =\inf\{\eps>0\colons K\subset H+B_\eps(0)\}.
\end{displaymath}
The ideals $\fK_{(k)}$ are invariant under translations, namely, with
each family of sets, they contain also one obtained by applying any
translation to all of its members.  For $k=0$, the ideal
$\fK_{(0)}$ consists of all families $\sA$ of compact sets such that
\begin{displaymath}
  \inf_{K\in\sA} \inf_{x\in\R^d} \dmet_{\mathrm{H}}(\{x\},K)
  =\inf_{K\in\sA}\inf_{x\in\R^d} \|K-x\|>0,
\end{displaymath}
where $\|K-x\|=\suptau(K-x)=\sup_{y\in K}\|y-x\|$. 
Equivalently, $\sA$ is in $\fK_{(0)}$ if and only if
\begin{displaymath}
  \inf_{K\in\sA} \diam(K)>0,
\end{displaymath}
where $\diam(K)$ is the supremum of the distance between two points in
$K$. The modulus $\modulus_{(d-1)}(K)$ equals one half of the smallest width of
$K$, and thus the ideal $\fK_{(d-1)}$ consists of families $\sA$ of
compact sets such that, for some $\eps>0$, none of the sets $K\in\sA$
has width in any direction smaller than a positive constant.

An alternative family of ideals $\fK_{[k]}$
is obtained by excluding the family of
all compact subsets of linear subspaces of dimension
$k\in\{0,\dots,d-1\}$. The ideal $\fK_{[k]}$ is generated by the
modulus
\begin{equation}
  \label{eq:modulus-[k]}
  \modulus_{[k]}(K)=
  \inf_{H\in G(d,k)} \inf_{L\in\sK^d,L\subset H}
  \dmet_{\mathrm{H}}(K,L),
  \quad K\in\sK^d,
\end{equation}
where $G(d,k)$ is the linear Grassmannian,
\index{Grassmannian!linear}
that is, the family of all linear
subspaces of dimension $k$. In other words, 
a family $\sA$ of
compact sets belongs to $\fK_{[k]}$ if there is an $\eps>0$ such that
none of the sets $K\in\sA$ is contained in the $\eps$-neighbourhood of
any linear subspace $H$ of dimension $k$. If $k=0$, we recover the ideal $\fK_0$.

For each $k=0,\dots,d-2$, we have $\fK_{[k+1]}\subset \fK_{[k]}$ and 
$\fK_{(k+1)}\subset \fK_{(k)}$. Furthermore, for each $k=0,\dots,d-1$,
\begin{displaymath}
  \fK_{\inf}\subset \fK_{(\inf)}\subset
  \fK_{(k)}\subset \fK_{[k]}\subset \fK_0. 
\end{displaymath}
Some relations between various ideals are shown in Figure~\ref{fig:2}.

\begin{figure}
  \centering
  \begin{tikzpicture}[
  node distance=10mm and 18mm,
  every node/.style={draw, rounded corners, align=center},
  >=stealth
]

\node (K0) {$\fK_0$};

\node (Kbr0)  [below=of K0] {$\fK_{[0]}=\fK_0$};
\node (Kbr1)  [below=of Kbr0] {$\fK_{[1]}$};
\node (Kbr2)  [below=of Kbr1] {$\vdots$};
\node (Kbrd1) [below=of Kbr2] {$\fK_{[d-1]}$};

\node (Kpar0)  [left=of Kbr0] {$\fK_{(0)}$};
\node (Kpar1)  [below=of Kpar0] {$\fK_{(1)}$};
\node (Kpar2)  [below=of Kpar1] {$\vdots$};
\node (Kpard1) [below=of Kpar2] {$\fK_{(d-1)}$};

\node (Kinf) [below=of Kpard1] {$\fK_{\inf}$};

\node (Kzero) [right=35mm of K0] {$\fK^0$};

\draw[->] (Kbrd1) -- (Kbr2);
\draw[->] (Kbr2)  -- (Kbr1);
\draw[->] (Kbr1)  -- (Kbr0);

\draw[->] (Kpard1) -- (Kpar2);
\draw[->] (Kpar2)  -- (Kpar1);
\draw[->] (Kpar1)  -- (Kpar0);

\draw[->] (Kpar0) -- (Kbr0);
\draw[->] (Kpar1) -- (Kbr1);
\draw[->] (Kpar2) -- (Kbr2);
\draw[->] (Kpard1) -- (Kbrd1);

\draw[->] (Kbr0) -- (K0);

\draw[->] (Kinf) -- (Kpard1);

\draw[dashed] (Kzero) -- (K0);

\end{tikzpicture}
  \caption{Relations between ideals on the family of compact sets}
  \label{fig:2}
\end{figure}

Further translation-invariant ideals on $\sK_c^d$ are generated by the
moduli $\modulus(K)=V_i(K)^{1/i}$, where $V_i$ with $i=1,\dots,d$ is
one of the intrinsic volumes; see \cite[Chapter~4]{schn2}.
\index{intrinsic volumes}
All these functions
are continuous in the Hausdorff metric. Particularly important
examples are $V_d$ which is the Lebesgue measure (or $d$-dimensional
volume) and $V_1$ which is
proportional to the mean width.

\paragraph{Vague convergence of measures on closed sets and regular
  variation}
A random element $X$ in $\sF$ with the Borel $\sigma$-algebra
generated by the Wijsman topology
is said to be a \emph{random closed set}.
\index{random closed set}
It is well known that $X$ is
a random closed set if and only if $\{X\cap G\neq\emptyset\}$ is
measurable for all open sets $G$; see \cite[Theorem~1.3.3]{mo1}.
If $\XX$ is locally
compact second-countable Hausdorff, $G$ above may be replaced by a
compact set $K$ and then the distribution of a random closed set is
uniquely determined by the probabilities
\begin{displaymath}
  \Prob{X\in\sF_K}=\Prob{X\cap K\neq\emptyset}
\end{displaymath}
for all compact sets $K\in\sK$. Moreover, if $\mu$ is a (possibly
infinite) measure on $\sF$, which is finite on $\sF_K$ for all
$K\in\sK$ (or is finite on $\sF^K$ for all non-empty $K\in\sK$) then
$\mu$ is uniquely determined by its values on $\sF_K$ (or $\sF^K$) for
all non-empty compact sets $K$; see \cite[Theorem~1.1.33]{mo1}.

\begin{remark}
  \label{rem:FK}
  The vague convergence of measures on the family of non-empty closed
  sets of a locally compact second-countable Hausdorff space is usually
  defined by working with the values of measures on the families
  $\sF_K=\{F\in \sF\colons F\cap K\neq\emptyset\}$ for all non-empty compact
  sets $K$ in $\XX$; see \cite[Page~130]{mo1}. These families build a 
  topologically and scaling consistent ideal on $\sF$, which is also
  the Hadamard bornology on the family $\sF'$ of non-empty closed
  sets. This choice is motivated by the fact that these families are
  compact in the Fell topology on $\sF'$ and form a convergence
  determining class. Note that this ideal contains all non-empty
  scaling-invariant closed sets, and thus does not satisfy Condition~\hyperref[condB0]{(B$_0$)}.
\end{remark}

The regular variation concept for random closed sets appears in
\cite[Definition~4.2.4]{mo1}, implicitly referring to the bornology
from Remark~\ref{rem:FK}, and in
\cite{mikosch:pawlas:samorodnitsky_2011} for random compact sets,
again implicitly using the modulus $\suptau$ defined in 
\eqref{eq:suptau}.  It should be noted that, if a random closed set
with positive probability takes scaling-invariant values that 
belong to a member of an ideal, then it is not regularly varying,
since the tail measure becomes infinite.

In all our examples, the regular variation property of random closed
sets is derived by applying the continuous mapping theorem.
Alternatively, it is possible to argue from the definition of 
vague convergence, as shown below.

If $\XX$ is a ball compact Polish space, the families $\sF_K$ form a
family that determines the weak convergence of probability measures
(or finite measures) on $\sF$; see \cite[Theorem~1.7.7]{mo1}. If $\mu$
is a finite Borel measure on $\sF$, then $\sF_K$ is a $\mu$-continuity
set if
\begin{equation}
  \label{eq:mu-K}
  \mu(\sF_K)=\mu\big(\sF_{\Int K}\big);
\end{equation}
see \cite[Lemma~1.7.4]{mo1}. This is
justified by the fact that $\sF_K\setminus\sF_{\Int K}$ is the
boundary of $\sF_K$ in the Fell topology.

The following result provides a criterion for the regular variation of a
random closed set on the ideal $\fF^0$ in $\R^d$ equipped with 
linear scaling. Recall that $\sF^K=\{F\colons F\cap K=\emptyset\}$. 

\begin{proposition}
  \label{prop:FK}
  A random closed set $X$ in $\R^d$ is regularly varying on the ideal
  $\fF^0$ if and only if there exists a measure $\mu\in\Mb[\fF^0]$
  such that 
  $\ttau(X)\in\RV(\R_+,\mydot,\sR_0,g,c\theta_\alpha)$ and
  \index{random closed set!regular variation of}
  \begin{equation}
    \label{eq:sets-FK}
    \Prob{X\cap T_tK=\emptyset \mid X\cap B_t(0)=\emptyset}
    \to \frac{\mu\big(\sF^K\cap\{\ttau>1\}\big)}{\mu\big(\{\ttau>1\}\big)}
    \quad \text{as}\; t\to\infty
  \end{equation}
  for all compact sets $K$ such that
  $\mu\big((\sF_K\setminus\sF_{\Int K})\cap\{\ttau>1\}\big)=0$, and
  then $c=\mu(\{F\in\sF\colons F\cap B_1(0)=\emptyset\})$.
\end{proposition}
\begin{proof}
  \textsl{Necessity.}
  The set
  \begin{displaymath}
    \sF^{B_r(0)}=\{F\in\sF\colons F\cap B_r(0)=\emptyset\}
    =\{\ttau>r\}
  \end{displaymath}
  belongs to $\fF^0$. By homogeneity and finiteness of $\mu$ on
  $\{\ttau>r\}$, its boundary $\{\ttau=r\}$ has $\mu$-measure zero,
  so it is a $\mu$-continuity set. Then
  \begin{displaymath}
    g(t)\Prob{\ttau(X)>rt}
    =g(t)\Prob{T_{t^{-1}}X\in \sF^{B_r(0)}}
    \to \mu\big(\sF^{B_r(0)}\big)\quad \text{as}\; t\to\infty,
  \end{displaymath}
  so that $\ttau(X)$ is regularly varying and the expression for $c$
  follows. Furthermore, \eqref{eq:sets-FK} immediately follows from
  the regular variation of $X$.
  
  \smallskip
  \noindent
  \textsl{Sufficiency.}
  It suffices to prove that
  \begin{equation}
    \label{eq:sets-FK-1}
    g(t)\Prob{T_{t^{-1}}X\in \cdot,\; X\cap B_t(0)=\emptyset}
    \to \mu\big(\cdot\cap\{\ttau>1\}\big)
    \quad \text{as}\; t\to\infty.
  \end{equation}
  Since the left-hand side is bounded by
  \begin{displaymath}
    g(t)\Prob{X\cap B_t(0)=\emptyset}
    =g(t)\Prob{T_{t^{-1}}X\cap B_1(0)=\emptyset}
    \to \mu\big(\{\ttau>1\}\big) \quad \text{as}\; t\to\infty,
  \end{displaymath}
  it is possible to apply Helly's theorem (see
  \cite[Theorem~1.7.3]{mo1}) to the uniformly bounded sequence of
  measures
  \begin{equation}
    \label{eq:g-seq}
    g(t)\Prob{T_{t^{-1}}X\in \cdot,T_{t^{-1}}X\cap
      B_1(0)=\emptyset},\quad 
    t\geq 1. 
  \end{equation}
  The uniform bound holds, since
  \begin{displaymath}
    g(t)\Prob{T_{t^{-1}}X\in \cdot,T_{t^{-1}}X\cap
      B_1(0)=\emptyset}
    \leq g(t)\Prob{\ttau(X)>t} \to c\quad \text{as}\;
    t\to\infty. 
  \end{displaymath}
  The sequence of measures given in \eqref{eq:g-seq} has a weakly
  convergent subsequence, and the limit is uniquely identified by its
  values on the sets $\sF^K\cap\{\ttau>1\}$, for compact $K$
  satisfying the above continuity condition. Thus, the convergence in
  \eqref{eq:sets-FK-1} holds.
\end{proof}

\begin{example}[Random half-space]
  \index{random half-space}
  Let $\xi\in\RV(\R^d,\mydot,\sR_0^d,g,\mu)$, and let
  \begin{displaymath}
    X=\big\{u\in\R^d\colons \langle u,\xi\rangle\geq \|\xi\|^2\big\}
  \end{displaymath}
  be the half-space with $\xi$ on its boundary and the inner normal
  $\xi$. The regular variation property of $X$ on the ideal $\fF^0$
  follows from the continuous mapping theorem. Alternatively, we can
  apply Proposition~\ref{prop:FK}. Let $K$ be a compact set. Then
  \begin{multline*}
    \Prob{X\cap T_tK=\emptyset\mid X\cap B_t(0)=\emptyset}\\
    =\Prob{X\cap T_t(K\cup B_1(0))=\emptyset\mid X\cap B_t(0)=\emptyset},
  \end{multline*}
  so we assume without loss of generality that $K$ contains the origin.
  Next, $\ttau(X)=\|\xi\|$, and
  Theorem~\ref{thr:polar-decomp-4} implies that, as $t\to\infty$,
  \begin{multline*}
    \Prob{X\cap T_tK=\emptyset\mid X\cap B_t(0)=\emptyset}
    =\Prob{\|\xi\|> t h_K(\xi/\|\xi\|)\mid \|\xi\|>t}\\
    \to \frac{1}{\mu\big(\{u\colons \|u\|>1\}\big)}
    \int_{\R^d} \one_{\|u\|>h_K(u/\|u\|)}\mu(\bdiff u)
    =\frac{1}{\sigma(\SS^{d-1})}
    \int_{\SS^{d-1}} h_K(v)^{-\alpha} \sigma(\bdiff v),
  \end{multline*}
  where $\mu=\sigma\otimes\theta_\alpha$ with the spectral measure
  $\sigma$ supported on the unit Euclidean sphere $\SS^{d-1}$ and
  $h_K$ is the support function of $K$, see \eqref{eq:21}. Note that
  $h_K$ is non-negative. The
  tail measure of $X$ is the pushforward of $\mu$ under
  the map $x\mapsto \big\{u\in\R^d\colons \langle u,x\rangle\geq \|x\|^2\big\}$.
\end{example}

The following result relates regular variation on the ideals
$\fF_\cone$ and $\fF^\cone$.

\begin{proposition}
  Let $\cone=\{0\}$ in $\XX=\R^d$.  Let $X$ be a random closed convex
  set with values in the family $\sF_0$ of closed sets that contain the
  origin, and assume that $X$ is regularly varying on the ideal
  $\fF_0$. Furthermore, assume
  that $X$ is almost surely regular closed (coincides with the closure
  of its interior). Then the closed complement $Y=\cl(X^c)$ is
  regularly varying on the ideal $\fF^0$.
\end{proposition}
\begin{proof}
  It is shown in \cite[Lemma~7.4]{kab:mar:mol} that the map
  $F\mapsto \cl(F^c)$ is bicontinuous between the family of non-empty
  convex regular closed sets containing the origin and their
  complements in $\R^d$. This map is a morphism. Furthermore, it is
  bornologically consistent, since
  \begin{displaymath}
    \{F\colons \ttau(\cl(F^c))> \eps\}
    \subset \{F\colons \suptau(F)\geq \eps\}. 
  \end{displaymath}
  Indeed if $\ttau(\cl(F^c))> \eps$, then $B_\eps(0)\cap
  F^c=\emptyset$, and thus $B_\eps(0)\subset F$ implying
  $\suptau(F)\geq \eps$. 
  The result follows from the continuous mapping theorem. 
\end{proof}

\paragraph{Set-valued maps}
Let $\YY$ be a topological space with scaling and with a topologically and
scaling consistent ideal $\sY$. If
$\varphi:\YY\mapsto\sF(\XX)$ is a set-valued morphism, which is
continuous and bornologically consistent, then the image of any
regularly varying random element in $\YY$ is a regularly varying
random closed set in $\XX$, provided that the pushforward of the tail
measure is non-trivial on the ideal imposed on $\sF$. In this way, it is
possible to obtain numerous random closed sets. Below we consider only
the case $\XX=\R^d$ with linear scaling. It can easily be generalised
for point processes on star-shaped metric spaces.

An important family of random closed sets arises as the supports of point
processes, which can be realised by mapping a counting
measure to its closed support. Consider the ideal $\sR_0^d$ on $\R^d$ and
let $\etap$ be a point process with realisations in $\Mp[\sR_0^d]$,
meaning that at most a finite number of points of $\etap$ lie outside
any non-trivial ball centred at the origin, so that $\supp\etap$ is
almost surely bounded.
\index{point process!support set}
If the support $\supp\etap$ is not closed, then this may be
only because the origin is an accumulation point of
$\supp\etap$, and then
$\cl(\supp\etap)=\{0\}\cup\supp\etap$. Furthermore,
$X=\cl(\supp\etap)$ is a random compact set, where the measurability
condition holds since $\{X\cap G\neq\emptyset\}=\{\etap(G)\geq1\}$ for
each open set $G$. If $\etap=0$, then $X=\emptyset$. 

The major difficulty in passing from regular variation of a point
process to the regular variation of its support as a random closed set
lies in the fact that a neighbourhood of the origin is excluded to
construct ideals on counting measures, while it is naturally present
in random closed sets. In other words, under downscaling many points of
the point process converge to the origin and are not part of the
tail measure, whereas their limit, the origin, may be an essential
point of the random closed set.

In the sequel, we use the random variable $\dmet(0,\etap)$, which is
the infimum of the distances between the origin and the points in
$\etap$. Its measurability follows from the fact that
$\{\dmet(0,\etap)>t\}=\{\etap(B_t(0)\setminus\{0\})=0\}$.
Note that, for a counting
measure $\me\in\Mp[\sR_0^d]$, we have  
$\dmet(0,\me) \in [0,\infty)$ if $\me$ is non-zero and we set
$\dmet(0,\me)=\infty$ if $\me=0$.

Next, we consider counting measures together with the distance between
the origin and their supports. For this, we work in
the product space $\Mp[\sR_0^d] \times [0,\infty]$ with the
product topology and the ideal $\Mpi[\sR_0^d]\times[0,\infty]$
generated by sets $\Me\times[0,\infty]$, where $\Me\in \Mpi[\sR_0^d]$.
In other words, a set in $\Mp[\sR_0^d] \times [0,\infty]$ is considered
bounded if it is bounded in the first component. 
By construction, if $\Me\in \Mpi[\sR_0^d]$ then
each counting measure $\me\in\Me$ is non-zero. The scaling on the
product space is given by $T_s(\me,r)=(T_s\me,sr)$ 
with the convention $T_s\infty=\infty$.

Denote 
\begin{displaymath}
  \Me'=\big\{ (\me,r)\in\Mp[\sR_0^d]\times[0,\infty)\colons 
  r= \dmet(0,\me) \mbox{ or } r = 0 \big\}\cup\{(0,\infty)\}.
\end{displaymath}
Consider the map $\psi: \Me' \to \sK^d$ given by
\begin{equation}
  \label{eq:psi-me}
  \psi(\me,r) = \cl \Big(\supp \big(\me + \delta_0\one_{r = 0}\big)\Big)
  =
  \begin{cases}
    \cl(\supp\me), & r>0,\\
    \cl(\supp\me) \cup\{0\}, & r=0,
  \end{cases}
\end{equation}
if $\me\ne0$. 
If $\me=0$ and $r=0$, we let $\psi(0,0)=\{0\}$, and
$\psi(0,\infty)=\emptyset$.
Note that $\psi(0,r)$ with $r\in(0,\infty)0$ is not defined since
$(0,r)$ does not belong to $\Me'$.

Now we apply the map $\psi$ to a point process $\etap$. 
Since $\etap$ is almost surely finite on every set from $\sR_0^d$,
$\dmet(0,\etap)\in(0,\infty)$ implies the existence of a point
$x \in \etap$ such that $\dmet(0,x) = \dmet(0,\etap)$. If
$\dmet(0,\etap)=0$, then $0$ is either a point of $\etap$ or its the
accumulation point. 
Therefore,
\begin{displaymath}
  \big(\etap, \dmet(0,\etap)\big) \in \Me',
\end{displaymath}
and $\psi(\etap,\dmet(0,\etap))$ equals the closed support
$\cl(\supp\etap)$. 
The set $\cl(\supp\etap)$ is either empty or almost surely
bounded, since $\etap$ contains at most a finite number of points
outside any neighbourhood of the origin. If $\etap=0$, then the closed
support is empty and so coincides with $\psi(0,\infty)$. 

Note that the space $\sK^d$ of
compact sets in $\R^d$ is equipped with linear scaling $\mydot$ and
the space of counting measures is equipped with linear scaling applied
to their atoms.

\begin{theorem}
  \label{th:supp-etap}
  Let $\etap$ be a point process in $\Mp[\sR_0^d]$ such that
  \begin{equation}
    \label{eq:RV-etap-c}
    \big(\etap, \dmet(0,\etap)\big) \in\RV\big(\Mp[\sR_0^d] \times [0,\infty),
    \mydot,\Mpi[\sR_0^d]\times [0,\infty),g,\mu^*\big).
  \end{equation}
  Then $\cl(\supp\etap)\in \RV(\sK^d,\mydot,\fK_0,g,\tilde\mu)$, where
  the tail measure $\tilde\mu$ is the pushforward of $\mu^*$ under the
  map $\psi$ from \eqref{eq:psi-me}.
\end{theorem}
\begin{proof}
  We first prove that $\Me'$ is a closed subset of
  $\Mp[\sR_0^d] \times [0,\infty]$.  Indeed, assume that
  $(\me_n,r_n)\to(\me,r)$. If $\me\neq 0$, then $\me(B_t(0)^c)\geq1$
  for some $t>0$, that is, $\me$ has at least one atom outside the
  closed ball $B_t(0)$.  Then the atoms of $\me_n$ and $\me$ in 
  the open set $B_t(0)^c$ can be enumerated to ensure their pointwise
  convergence. Hence, $\limsup \dmet(0,\me_n)\leq
  \dmet(0,\me)$. Furthermore, if $\me(B_t(0))=0$, then also
  $\me_n(B_{t-\eps}(0))=0$ for any $\eps>0$ and all sufficiently large
  $n$, so that $\liminf \dmet(0,\me_n) \geq \dmet(0,\me)$. Thus,
  $\dmet(0,\me_n)\to \dmet(0,\me)$.  If $\me=0$, then, for every
  $t>0$, the support of $\me_n$ is contained in the ball $B_t(0)$ for
  all sufficiently large $n$, so that $r_n\to0$. In both cases, the
  limit satisfies $r=\dmet(0,\me)$ or $r=0$.  Furthermore,
  $(0,\infty)$ is an isolated point in $\Me'$. Therefore, $\Me'$ is
  closed and Theorem~\ref{lemma:vVSw} implies that the tail measure
  $\mu^*$ of $\big(\etap,\dmet(0,\etap)\big)$ is supported on the set
  $\Me'$. 

  Since $\cl(\supp\etap) = \psi\big(\etap,\dmet(0,\etap)\big)$, it
  remains to show that $\psi$ satisfies the assumptions of the
  continuous mapping theorem.
  The map $\psi$ is clearly a morphism. It is also bornologically
  consistent, since
  \begin{displaymath}
    \big\{\me\colons \|\cl(\supp\me)\|\geq\eps \big\}
    =\big\{\me\colons \me(B_\eps(0)^c)\geq 1 \big\}\in
    \Mpi[\sR_0^d].
  \end{displaymath}
  We show that $\psi$ is continuous as the map from $\Me'$ endowed
  with the product of the vague topology and the Euclidean topology on
  $[0,\infty)$ to $\sK^d$ with the Hausdorff metric.  Because the
  spaces involved are sequential, it suffices to show that
  $(\me_n,r_n) \to (\me,r)$ implies $\psi(\me_n,r_n) \to \psi(\me,r)$
  for any sequence $(\me_n,r_n)_{n\in\NN}\subset\Me'$ and $(\me,r) \in
  \Me'$. Denote $K_n=\psi(\me_n,r_n)$ and $K=\psi(\me,r)$. 

  Since $(0,\infty)$ is an isolated point, there are no sequences
  converging to it. 
  Assume that $(\me,r)=(0,0)$, so that $K=\psi(0,0)=\{0\}$.
  Since $\me_n\to\me$, for each $\eps>0$
  we have $\supp\me_n\subset B_\eps(0)$ for all sufficiently large
  $n$. Hence $K_n\subset B_\eps(0)$ and so
  $\dmet_{\mathrm{H}}(K_n,K)\leq\eps$. 

  Next, assume that $\me\neq 0$ and $r=0$. Assume first that
  $\dmet(0,\me)>t$, that is, 
  $\me(B_t(0))=0$ for some $t>0$. Without loss of generality, assume that
  $\me(\partial B_t(0))=0$. Then $r_n=0$ for all
  sufficiently large $n$, since $\me_n(B_t(0))=0$ and so
  $\dmet(0,\me_n)\geq t$. For any $\eps>0$, there exist $k$ and $n_1$ such
  that for all $n\geq n_1$,
  \begin{displaymath}
    \me_n|_{B_t^c}= \sum_{i=1}^k  \delta_{x_i^{(n)}},\quad
    \me|_{B_t^c}=\sum_{i=1}^k \delta_{x_i}
  \end{displaymath}
  and $\|{x_i^{(n)}}-x_i\| < \eps$ for all $i=1,\ldots, k$ and
  sufficiently large $n$; see
  \cite[Proposition~2.8]{bas:plan19}. Since $r_n=0$ and $r=0$, we have
  $K\cap B_t(0)=K_n\cap B_t(0)=\{0\}$, and thus
  \begin{equation}
    \label{eq:outside-t}
    \dmet_{\mathrm{H}}(K_n,K)=
    \dmet_{\mathrm{H}}(K_n\cap B_t(0)^c,K\cap B_t(0)^c)
    \leq \eps. 
  \end{equation}

  If $\dmet(0,\me)=0$, then we apply the above argument to the
  counting measures restricted to the complement of $B_\eps(0)$ for
  any small $\eps>0$. Then
  \begin{displaymath}
    \dmet_{\mathrm{H}}(K_n,K)\leq 
    \dmet_{\mathrm{H}}(K_n\cap B_\eps(0)^c,K\cap B_\eps(0)^c)
    +\eps\leq 2\eps. 
  \end{displaymath}

  Finally, consider the case $\me\ne0$ and $r>0$. Then take any
  $t\in(0,r)$ such that $\me(\partial B_t(0))=0$ and note that $r_n>t$
  for all sufficiently large $n$. By restricting the counting measures
  $\me_n$ and $\me$ onto the complement of $B_t(0)$ and using
  point-to-point convergence, we obtain that \eqref{eq:outside-t}
  holds in this case.
\end{proof}

We now discuss a variant of Theorem~\ref{th:supp-etap} for point
processes that are regularly varying on the ideal $\Mpik[\sR_0^d]$
with $k\geq2$. In this case, tail measures of point processes are
supported by counting measures with at least
$k$ atoms.  However, by examining
the supports viewed as compact sets, it is not directly
possible to infer how many points lie outside any ball, so that there
is apparently no direct analogue of the ideal $\Mpik[\sR_0^d]$ in the
family of compact sets. This is also due to the fact that the cardinality
map $K\mapsto\card(K)$ is not continuous on the space of compact sets.

Below we explain how to handle this for the cases when $k\leq d$. For
this, we use ideals $\fK_{[j]}$ defined in 
Section~\ref{sec:ideals-compact-sets}. The restriction on $k$ being at
most $d$ results
from the fact that lower-dimensional closed sets in $\R^d$ are easier to
discern and so it becomes possible to phrase the event that a point
process has at least $k$ points outside a small ball in terms of
ideals on random compact sets. Recall that $\fK_{[0]}=\fK_0$. 

\begin{theorem}
  \label{th:supp-etap-k}
  Let $\etap$ be a point process in $\Mp[\sR_0^d]$ such that, for some
  $k\in\{2,\dots,d\}$, 
  \begin{equation}
    \label{eq:RV-etap-c-k}
    \big(\etap, \dmet(0,\etap)\big) \in\RV\big(\Mp[\sR_0^d] \times [0,\infty),
    \mydot,\Mpik[\sR_0^d]\times [0,\infty),g,\mu^*\big).
  \end{equation}
  Then $\cl(\supp\etap)\in
  \RV(\sK^d,\mydot,\fK_{[k-1]},g,\tilde\mu)$, 
  and the tail measure $\tilde\mu$ is the pushforward of $\mu^*$ under the
  map $\psi$ from \eqref{eq:psi-me}. 
\end{theorem}
\begin{proof}
  Repeating the argument from Lemma~\ref{lemma:products}, we see that
  the map $(\me,\dmet(0,\me))\mapsto (\me^{(k)},\dmet(0,\me))$ is
  continuous and bornologically consistent between
  $\Mpik[\sR_0^d]\times [0,\infty)$ and
  $\Mpi[(\sR_0^d)^k(k)]\times [0,\infty)$. This yields the regular
  variation of the pair $(\etap^{(k)},\dmet(0,\etap))$.

  Next, we adjust the map $\psi$ by letting
  \begin{displaymath}
    \psi(\me^{(k)},r)
    =\cl\Big(
    \bigcup_{(x_1,\dots,x_k)\in\supp \me^{(k)}}\{x_1,\dots,x_k\}
    \Big)
  \end{displaymath}
  if $r>0$
  and the union of this set and the origin if $r=0$. Note that the
  values of $\psi$ are 
  compact subsets of $\R^d$ and
  $\psi(\me^{(k)},\dmet(0,\me))=\cl(\supp\me)$. Arguing as in
  Theorem~\ref{th:supp-etap}, we see that $\psi$ is continuous. We
  need to show its bornological consistency. Recall that $\fK_{[k-1]}$ is
  generated by the modulus $\modulus_{[k-1]}$
  given in \eqref{eq:modulus-[k]}. For each
  $\eps>0$,
  \begin{align*}
    \{\me\colons \modulus_{[k-1]}(\cl(\supp\me))\geq \eps\}
    &=\bigcap_{H\in G(d,k-1)} \bigcap_{L\in\sK^d,L\subset H}
    \{\me\colons \dmet_{\mathrm{H}}(\cl(\supp\me),L)\geq \eps\}\\
    &= \bigcap_{H\in G(d,k-1)} 
    \{\me\colons \cl(\supp\me)\not\subset H+B_\eps(0)\}\\
    &\subset  
    \{\me\colons \me(B_{\eps}(0)^c)\geq k\}.
  \end{align*}
  Indeed, if at most $k-1$ atoms lie outside $B_{\eps}(0)$, 
  then these atoms, together with the origin, are contained in some 
  $(k-1)$-dimensional linear subspace $H\in G(d,k-1)$, and so
  $\cl(\supp\me)$ is a subset of $H+B_\eps(0)$. Thus, the map $\psi$
  is bornologically consistent.
  The continuous mapping theorem
  yields the result.  
\end{proof}

\begin{remark}
  \label{rem:rv-etap-check}
  By an analogue of Lemma~\ref{lemma:conv-Sk} together with
  Lemma~\ref{lemma:conv-det-product}, the regular variation property
  \eqref{eq:RV-etap-c} can be confirmed by checking that
  \begin{multline}
    \label{eq:h-inf}
    g(t)\E \Big[e^{-\int f\diff T_{t^{-1}}\etap}
    \one_{\etap(T_t\base_i)\geq 1}h(\dmet(0,T_{t^{-1}}\etap))\Big]\\
    \to \int e^{-\int f\diff \me}\one_{\me(\base_i)\geq
      1}h(r)\,\mu^*(\bdiff (\me,r)) \quad
    \text{as}\; t\to\infty
  \end{multline}
  for all $\base_i$ from a countable open base of the ideal $\sR_0^d$, bounded
  Lipschitz functions $f:\R^d\to\R_+$ with support contained in $\base_i$, 
  and all bounded Lipschitz functions $h:[0,\infty)\to\R$. Note that
  $\dmet(0,T_{t^{-1}}\etap)=t^{-1}\dmet(0,\etap)$.

  Furthermore, \eqref{eq:RV-etap-c-k} can be confirmed (using the
  result of Lemma~\ref{lemma:conv-Sk-product}) by checking that
  \begin{displaymath}
    g(t)\Prob{\etap(T_t\base_i)\geq k+1}\to 0,
  \end{displaymath}
  and that 
  \begin{multline}
    \label{eq:PP-conv-k-with-d}
    g(t)\E \Big[\Big(1-e^{-\int f\diff (T_{t^{-1}}\etap)^{(k)}}\Big)
    h(\dmet(0,T_{t^{-1}}\etap))\one_{\etap(T_t\base_i)=k}\Big]\\
    \to \int \Big(1-e^{-\int f\diff \me^{(k)}}\Big)h(r)\one_{\me(\base_i)=k}
    \mu^*(\bdiff (\me,r)) \quad
    \text{as}\; t\to\infty
  \end{multline}
  for all $i\in\NN$, bounded Lipschitz functions $f:(\R^d)^k\to\R_+$
  with cupport contained in $\base_i^k$, and all bounded Lipschitz functions
  $h:[0,\infty)\to\R$.   
\end{remark}

\begin{example}[Counterexample to regular variation of the support]
  The regular variation of a point process $\etap$ on the ideal
  $\Mpi[\sR_0]$ (without pairing it with $\dmet(0,\etap)$)
  does not necessarily imply the regular variation of
  $X=\cl(\supp\etap)$.  Let $\xi$ be a random variable which is not
  regularly varying on $\R$ with the ideal $\sR_0$, but is such that
  $|\xi|$ is regularly varying on the same space, namely,
  $|\xi|\in\RV(\R_+,\mydot,\sR_0,g,\mu')$. Such a distribution of $\xi$
  may be constructed using oscillating tails, for example, letting
  $\xi=RS$, where $R$ is a Pareto($\alpha$) random variable and
  \begin{align*}
    \Prob{S=1\mid R=r}&=\frac{1}{2}+\frac{1}{4}\sin (\log r),\\
    \Prob{S=-1\mid R=r}&=\frac{1}{2}-\frac{1}{4}\sin (\log r).
  \end{align*}
  Then the ratio $\Prob{\xi>t}/\Prob{|\xi|>t}$ oscillates and has no
  limit. The tail measure $\mu'$ of $|\xi|$ is $\theta_\alpha$.

  Define the point process on the positive half-line by letting
  \begin{displaymath}
    \etap = \delta_{|\xi|} +\delta_{|\xi|^{-1}}\one_{\xi>0}.
  \end{displaymath}
  By Lemma~\ref{lemma:conv-Sk} with
  $k=1$ and the sets $\base_i=(i^{-1},\infty)$ forming an open base for
  the ideal $\sR_0$ on $\R_+$, $\etap$ is regularly varying on
  $\Mpi[\sR_0]$ with the tail 
  measure $\mu$ being the pushforward of $\mu'$ under the map
  $x\mapsto\delta_x$ and the normalising function
  $g(t)=t^\alpha$. Indeed,
  \begin{displaymath}
    t^\alpha \Prob{\etap((t/i,\infty))\geq 1}
    =t^\alpha\Prob{|\xi|> t/i}\to i^\alpha=\theta_\alpha(\base_i)
  \end{displaymath}
  and
  \begin{multline*}
    t^\alpha \E \Big[e^{-f(t^{-1}|\xi|)-f(t^{-1}/|\xi|)}
    \one_{\etap((t/i,\infty))\geq1}\one_{\xi>0} \Big]
    +t^\alpha \E \Big[e^{-f(t^{-1}|\xi|)}
    \one_{\etap((t/i,\infty))\geq1}\one_{\xi\leq 0} \Big]\\
    =t^\alpha \E \Big[e^{-f(t^{-1}|\xi|)}
    \one_{|\xi|> t/i}\Big]
    \to \int_0^\infty  e^{-f(x)}\one_{x>i^{-1}}
    \theta_\alpha(\bdiff x)\quad \text{as}\; t\to\infty,
  \end{multline*}
  where, to combine two summands, we have used the fact that $f$ is
  supported by $\base_i$ and so $f(t^{-1}/|\xi|)=0$ if $|\xi|>t/i$ and
  $t$ is sufficiently large.

  The random compact set
  $X=\supp\etap$ takes value $\{|\xi|\}$ if $\xi\leq 0$ and
  $\{|\xi|,|\xi|^{-1}\}$ if $\xi>0$. Let $\sA_s=\{\{x\}\colons x\geq s\}$
  be the family of
  singletons $\{x\}$ with $x\geq s$ for $s>1$. This is a closed subset
  of $\sK^1$ that is bounded away from the origin in the Hausdorff
  metric and thus belongs to the ideal $\fK_0$. By homogeneity,
  $\sA_s$ would be a continuity set of the tail measure for all but
  at most countably many values of $s$. However,
  \begin{displaymath}
    t^{\alpha}\Prob{X \in T_t\sA_s}
    =t^{\alpha}\Prob{|\xi|>ts,\xi\leq 0}
    =t^{\alpha}\Prob{\xi<-ts}
  \end{displaymath}
  does not converge as $t\to\infty$. Choosing an $s$ among the
  continuity values of the putative tail measure shows that $X$ is not
  regularly varying on $\fK_0$. Checking the condition of
  Theorem~\ref{th:supp-etap}, note that the left-hand side of
  \eqref{eq:h-inf} becomes
  \begin{multline*}
    t^\alpha \E \Big[e^{-f(t^{-1}|\xi|)-f(t^{-1}/|\xi|)}
    \one_{\etap((t/i,\infty))\geq1}\one_{\xi>0}h(t^{-1}|\xi|^{-1}) \Big]\\
    +t^\alpha \E \Big[e^{-f(t^{-1}|\xi|)}
    \one_{\etap((t/i,\infty))\geq1}\one_{\xi\leq 0}h(t^{-1}|\xi|) \Big].
  \end{multline*}
  For suitable choices of $f$ and non-constant $h$, the values of $h$ at
  $t^{-1}|\xi|^{-1}$ and at $t^{-1}|\xi|$ prevent the two summands
  from being combined. The oscillations of the positive and negative
  tails of $\xi$ then remain visible, and the limit in
  \eqref{eq:h-inf} does not exist. Thus, the additional assumption in
  Theorem~\ref{th:supp-etap} indeed fails.
\end{example}

\begin{example}[Set of i.i.d.\ points]
  \label{ex:etap-n-points}
  Let $\etap=\delta_{\xi_1}+\cdots+\delta_{\xi_m}$ be a binomial point
  process in $\R^d$, where $m$ is fixed
  and $\xi_1,\dots,\xi_m$ are independent copies of
  $\xi$ with distribution $\nu\in\RV(\R^d,\mydot,\sR_0^d,g,\mu)$. It
  is shown in Example~\ref{ex:rv-binomial-pp} that,  for
  $k\in\{1,\dots,m\}$, 
  \begin{displaymath}
    \etap\in\RV\big(\Mp[\sR_0^d],\mydot,\Mpik[\sR_0^d],g^k,\mu^*_{(k)}\big), 
  \end{displaymath}
  where $\mu^*_{(k)}$ is the image of $\binom{m}{k}\mu^{\otimes k}$ under the map
  $x\mapsto\me_x$, see \eqref{eq:me-map-k} for the definition of $\me_x$. 

  The support of $\etap$ is the 
  random compact set $X=\{\xi_1,\dots,\xi_m\}$.
  In order to show that $X$ is regularly varying,
  we confirm the joint regular variation of $\etap$ and
  $\dmet(0,\etap)$ on the ideal $\Mpik[\sR_0^d]\times\R_+$. First, it
  is easy to see that, for a constant $c_{m,k}$,
  \begin{displaymath}
    g(t)^k\Prob{\etap(T_t\base_i)\geq k+1}
    \leq g(t)^k c_{m,k} \nu(T_t\base_i)^{k+1}\to 0\quad
    \text{as}\; t\to\infty. 
  \end{displaymath}
  Let $f$ be a
  continuous bounded function with support contained in $\base_i^k$, where
  $\base_i=B_{1/i}(0)^c$ and let $h:[0,\infty)\to\R$ be a bounded continuous
  function. Taking into account \eqref{eq:rv-binomial-k+1}, by
  \eqref{eq:PP-conv-k-with-d}, we need to find the limit as
  $t\to\infty$ of
  \begin{align*}
    g(t)^k&\E
     \Big[\Big(1-e^{-\int f\diff (T_{t^{-1}}\etap)^{(k)}}\Big)
      h(\dmet(0,T_{t^{-1}}\etap))\one_{\etap(T_t\base_i)=k}\Big]\\
    &=\binom{m}{k}g(t)^k\int_{T_t\base_i^k}\int_{T_t B_{1/i}(0)^{m-k}}
      \Big(1-e^{-\int f\diff (T_{t^{-1}}\me_{(x_1,\dots,x_k)}^{(k)})}\Big)\\
    &\times  
      h(t^{-1}\min\{\|x_1\|,\dots,\|x_m\|\})
    \nu^{\otimes k}(\bdiff(x_1,\dots,x_k))
      \nu^{\otimes (m-k)}(\bdiff(x_{k+1},\dots,x_m)).
  \end{align*}
  If $m=k$, then the limit is
  \begin{align*}
    \int_{\base_i^k}\Big(1-e^{-\int f\diff\me_x^{(k)}}\Big)
    h(\dmet(0,\me_x)) \binom{m}{k} \mu^{\otimes k}(\bdiff(x_1,\dots,x_k)),
  \end{align*}
  and then the tail measure $\mu^*$ in \eqref{eq:RV-etap-c-k} is the
  pushforward of $\binom{m}{k}\mu^{\otimes k}$ under the map
  $x\mapsto (\me_x,\dmet(0,\me_x))$.
  If $k\leq m-1$, then the limit is 
  \begin{align*}
    \int_{\base_i^k}\Big(1-e^{-\int f\diff\me_x^{(k)}}\Big)
    h(0) \binom{m}{k} \mu^{\otimes k}(\bdiff(x_1,\dots,x_k)), 
  \end{align*}
  and then $\mu^*$ is the
  pushforward of $\binom{m}{k}\mu^{\otimes k}$ under the map
  $x\mapsto (\me_x,0)$.
  
  We now apply Theorem~\ref{th:supp-etap-k} (or
  Theorem~\ref{th:supp-etap} if $k=1$), noticing the requirement
  that $k\leq d$.  Thus, for $k\leq \min\{m-1,d\}$ and $m\geq 2$, the
  random compact set $X=\{\xi_1,\dots,\xi_m\}$ is regularly varying on
  the ideal $\fK_{[k-1]}$ with the normalising function $g^k$. Its
  tail measure $\tilde\mu$ is the pushforward of
  $\binom{m}{k}\mu^{\otimes k}$ (with $\mu^*$ being the intermediate
  object) under the map $x\mapsto \{0,x_1,\dots,x_k\}$, while if
  $k=m\leq d$, it is the pushforward of the same measure under the map
  $x\mapsto \{x_1,\dots,x_k\}$ for $x=(x_1,\dots,x_k)\in(\R^d)^k$.
\end{example}

\begin{example}[Support of a Poisson process]
  \label{ex:etap-supp}
  Let $\etap$ be the Poisson process in $\R^d$ with intensity measure
  \begin{displaymath}
    \lambda\in\RV\big(\R^d,\mydot,\sR_0^d,g,\mu\big).
  \end{displaymath}
  We explore the regular variation of $X=\cl(\supp\etap)$ on the ideal
  $\fK_{[k-1]}$ for $k\in\{1,\dots,d\}$ using
  Theorem~\ref{th:supp-etap} for $k=1$ and
  Theorem~\ref{th:supp-etap-k} for $k\geq2$.
  
  If the total mass of $\lambda$ is infinite, then $\dmet(0,\etap)=0$
  almost surely, and thus Theorem~\ref{thr:poisson-process} applies. Indeed,
  since $\lambda$ is finite on the complement of any ball, it is
  infinite in every neighbourhood of the origin, so that there are
  infinitely many points of $\etap$ in any ball around the origin. The
  tail measure $\mu^*$ of $\big(\etap,\dmet(0,\etap)\big)$ on the
  ideal $\Mpik[\sR_0^d]\times [0,\infty)$ is the
  pushforward of $\frac{1}{k!}\mu^{\otimes k}$ under the map
  $x\mapsto (\me_x,0)$ for $x=(x_1,\dots,x_k)\in(\R^d)^k$.
  Then $X=\cl(\supp\etap)$ is
  regularly varying with tail measure $\tilde{\mu}$ being the pushforward of
  $\frac{1}{k!}\mu^{\otimes k}$ by the map $x\mapsto
  \{0,x_1,\dots,x_k\}$.  

  Assume now that $\lambda(\R^d)<\infty$. The key observation below is
  that, if the point process consists of $k$ points, then the tail
  measure of the random compact set $X$ is supported on finite sets of
  cardinality $k$, while 
  if the point process has at least $(k+1)$ points, then the tail
  measure of $X$ is concentrated on sets of $(k+1)$ points, one of
  them being the origin.

  First, note that  
  \begin{displaymath}
    g(t)^k\Prob{\etap(T_t\base_i)\geq k+1}
    \leq g(t)^k 2\frac{\lambda(T_t\base_i)^{k+1}}{(k+1)!}
    \to 0 \quad \text{as}\; t\to\infty.
  \end{displaymath}
  where the inequality holds for sufficiently large $t$ such that
  $\lambda(T_t\base_i)\leq (k+1)/2$.
  
  Fix an $i\in\NN$ and let $\base_i$ be an element of an open base of
  $\sR_0^d$. Let $f:(\R^d)^k\to\R_+$ be a bounded Lipschitz function
  supported by $\base_i^k$, and let $h:\R_+\to\R$ be a bounded
  Lipschitz function. Then the left-hand side of \eqref{eq:PP-conv-k-with-d}
  becomes
  \begin{align*}
    I(t)
    &=g(t)^k\E \Big[\Big(1-e^{-\int f\diff (T_{t^{-1}}\etap)^{(k)}}\Big)
      h(\dmet(0,T_{t^{-1}}\etap))\one_{\etap(T_t\base_i)=k}\Big]\\
    &=\frac{1}{k!}e^{-\lambda(T_t\base_i)} g(t)^k \int_{(T_t\base_i)^k}
      \Big(1-e^{-\int f\diff (T_{t^{-1}}\me_x)^{(k)}}\Big)\\
    &\qquad\qquad\times 
      \E \big[h(\min\{\dmet(0,T_{t^{-1}}\me_x),
      \dmet(0,T_{t^{-1}}\etap_t)\})\big]
      \lambda^{\otimes k}(\bdiff(x_1,\dots,x_k)),
  \end{align*}
  where $\etap_t$ is $\etap$ restricted to the complement of
  $T_t\base_i$. Note that
  \begin{displaymath}
    p_t=\Prob{\etap_t=0}=\Prob{\etap((T_t\base_i)^c)=0}
    =e^{-\lambda((T_t\base_i)^c)}\to e^{-\lambda(\R^d)}=p \quad
    \text{as}\; t\to\infty. 
  \end{displaymath}
  We the expression for $I(t)$ in two parts $I_1(t)+I_2(t)$
  depending on whether
  $\etap_t$ is non-trivial. If $\etap_t=0$, then $p_t
  e^{-\lambda(T_t\base_i)}=p$ and we obtain
  \begin{align*}
    I_1(t)&=\frac{p}{k!} g(t)^k \int
    \Big(1-e^{-\int f\diff (T_{t^{-1}}\me_x)^{(k)}}\Big)
    h(\dmet(0,T_{t^{-1}}\me_x))
    \lambda^{\otimes k}(\bdiff(x_1,\dots,x_k))\\
    &\to \frac{p}{k!}  \int
    \Big(1-e^{-\int f\diff (\me_x)^{(k)}}\Big)
    h(\dmet(0,\me_x))
    \mu^{\otimes k}(\bdiff(x_1,\dots,x_k))
  \end{align*}
  by the regular variation of $\lambda^{\otimes k}$. If $\etap_t$ is
  not zero, then
  \begin{displaymath}
    \dmet(0,T_{t^{-1}}\etap_t)\leq \dmet(0,T_{t^{-1}}\me_x)
  \end{displaymath}
  for all $x\notin (\base_i)^k$. Thus,
  \begin{align*}
    I_2(t)
    &=\frac{1}{k!} e^{-\lambda(T_t\base_i)} g(t)^k \int
      \Big(1-e^{-\int f\diff (T_{t^{-1}}\me_x)^{(k)}}\Big)
      \E \big[h(\dmet(0,T_{t^{-1}}\etap_t))\mid \etap_t\neq0\big]\\
    &\qquad\qquad\qquad\qquad\qquad\qquad \times 
      \Prob{\etap_t\neq 0}
      \lambda^{\otimes k}(\bdiff(x_1,\dots,x_k)).
  \end{align*}
  By the Lipschitz property of $h$, the difference between $I_2(t)$
  and the same expression with $h(0)$ in place of
  $h(\dmet(0,T_{t^{-1}}\etap_t))$ is bounded by a constant times
  \[
    g(t)^k \int_{(T_t\base_i)^k}
    \Big(1-e^{-\int f\diff (T_{t^{-1}}\me_x)^{(k)}}\Big)
    \E\big[\dmet(0,T_{t^{-1}}\etap_t)\one_{\{\etap_t\neq0\}}\big]
    \lambda^{\otimes k}(\bdiff x).
  \]
  Since $\lambda(\R^d)<\infty$, the point process $\etap$ has finitely
  many atoms almost surely. Hence
  $\dmet(0,T_{t^{-1}}\etap_t)\one_{\{\etap_t\neq0\}}\to0$ almost
  surely, and it is bounded by $1/i$ for $\base_i=B_{1/i}(0)^c$.
  Dominated convergence therefore shows that this error tends to zero.
  As $t\to\infty$, we have
  $e^{-\lambda(T_t\base_i)}\Prob{\etap_t\neq0} \to 1-p$, and thus
  \begin{displaymath}
    I_2(t)\to \frac{1-p}{k!}\int
      \Big(1-e^{-\int f\diff (\me_x)^{(k)}}\Big)
      h(0)\mu^{\otimes k}(\bdiff(x_1,\dots,x_k))
  \end{displaymath}  
  Thus, \eqref{eq:RV-etap-c-k} holds with 
  $\mu^*$ equal to the sum of two measures: one obtained as the
  pushforward of $\frac{p}{k!}\mu^{\otimes k}$ under the map
  $x\mapsto(\me_x,\dmet(0,\me_x))$ and the other of the measure
  $\frac{1-p}{k!}\mu^{\otimes k}$ 
  under the map $x\mapsto(\me_x,0)$. 
  Consequently, $X=\cl(\supp\etap)$ is regularly varying on $\fK_{[k-1]}$
  with the normalising function $g^k$ and the tail measure $\tilde\mu$
  obtained as the sum of the pushforward of $\frac{p}{k!}\mu^{\otimes
    k}$ under the map 
  $x\mapsto \{x_1,\dots,x_k\}$ and the pushforward of the measure
  $\frac{1-p}{k!}\mu^{\otimes k}$ 
  under the map $x\mapsto \{0,x_1,\dots,x_k\}$.
\end{example}

\begin{example}[Ranges of random functions]
  Regularly varying random closed sets naturally appear as ranges of
  regularly varying random functions. Let $\xi_t$, $t\in[0,1]$, be
  random continuous functions which is regularly varying on the space
  $\Cont([0,1])$ with the uniform metric and the ideal $\sC_0$
  generated by the uniform norm taken as the modulus. Then the range
  of $x\in\Cont([0,1])$ is given by $\psi(x)=[\inf x,\sup x]$. The
  map $\psi$ is continuous from $\Cont([0,1])$ to the
  family $\sK^1$ of compact sets on the line with the Hausdorff
  metric. It is also $(\sC_0,\fK_0)$-bornologically
  consistent. Therefore, if $\xi$ is a regularly varying stochastic
  process, then its range $X=\psi(\xi)$ is a regularly varying random
  compact set. This can be generalised to continuous random functions
  with values in $\R^d$. 
\end{example}

\paragraph{Convex hull maps}
Now consider set-valued maps with convex compact values.
Note that the convex hull mapping is continuous in the
Hausdorff metric.
\index{convex hull}
If $X$ is a random compact set in $\R^d$ that is regularly varying on
$\fK_{[k]}$ for some $k\in\{0,\dots,d-1\}$, then its \emph{convex
  hull} $\conv(X)$ is also regularly varying on $\fK_{[k]}$ by the
continuous mapping theorem. This applies to random compact sets
in Examples~\ref{ex:etap-n-points} and~\ref{ex:etap-supp}
obtained as supports of point processes. However, a particularly
interesting ideal for random convex sets is $\fK_{\inf}$ and we now
focus on this latter ideal. 

\begin{lemma}
  \label{lemma:conv-hull-map}
  The map $(\me,r)\mapsto \conv(\psi(\me,r))$, where $\psi$ is defined
  in \eqref{eq:psi-me}, is continuous from $\Me'$ to $\sK_c^d$ and also
  bornologically consistent from the ideal
  $\sM_{\mathrm{p}}^{(d+1)}\big(\sR_0^d\big)\times [0,\infty)$ 
  to the ideal $\fK_{\inf}$.  
\end{lemma}
\begin{proof}
  The continuity follows from the continuity of the convex hull
  map. To show its bornological consistency, let $\eps>0$. Then
  \begin{displaymath}
    \big\{(\me,r)\colons \conv(\psi(\me,r))\supset B_\eps(0)\big\}
    \subset \{\me\colons \me(B_{\eps/2}(0^c))\geq d+1\}\times [0,\infty).
  \end{displaymath}
  This is based on the fact that if the convex hull of a finite set of
  points contains $B_\eps(0)$, then there are at least $d+1$ points in
  this set lying outside of the ball of radius $\eps/2$.  Indeed, if
  only at most $d$ points lie outside of $B_{\eps/2}(0)$, then these
  points lie in a half-space whose boundary passes through the
  origin.  In this case, the convex hull of the whole set of points is
  contained in the convex hull of the union of a half-space
  and $B_{\eps/2}(0)$, and thus cannot contain $B_\eps(0)$.
\end{proof}

\begin{example}[Convex hull of i.i.d.\ points]
  \label{ex:interior-origin}
  Let $\xi_1,\dots,\xi_m$ with $m\geq2$ be i.i.d.\ random vectors in
  $\R^d$ such that $\xi_1\in\RV\big(\R^d,\mydot,\sR_0^d,g,\mu\big)$. By
  the continuous mapping theorem applied to the convex hull of
  $X=\supp\etap$ for the point process
  $\etap=\delta_{\xi_1}+\cdots+\delta_{\xi_m}$ from 
  Example~\ref{ex:etap-n-points} and letting $k=1$, we see that
  $\conv(X)=\conv(\{\xi_1,\dots,\xi_m\})$ is regularly varying on
  $\fK_0$ with the normalising function $g$ and the tail measure
  $\tilde\mu$ being the image of $m\mu$ under the map
  $x\mapsto\conv(\{0,x\})$. This means that the tail measure is supported on 
  line segments with one endpoint at the origin. Since the segments
  $\conv(\{0,x\})$ have empty interiors, the tail measure
  vanishes on the ideal $\fK_{\inf}$, and thus $\conv(X)$ is not regularly
  varying on $\fK_{\inf}$ if the normalisation is done by the function
  $g$.

  If $m\geq d+1$ and $k\in\{2,\dots,d\}$, then $\conv(X)$
  is regularly varying on $\fK_{[k-1]}$ with the tail measure being
  the pushforward of $\binom{m}{k}\mu^{\otimes k}$ under the map
  \begin{displaymath}
    x=(x_1,\dots,x_k)\mapsto\conv(\{0,x_1,\dots,x_k\}).
  \end{displaymath}
  In this case, the tail
  measure also vanishes on the ideal $\fK_{\inf}$. 
  
  Next we consider the case $k\geq d+1$, assuming that $m\geq d+1$. Then
  $(\etap,\dmet(0,\etap))$ satisfies satisfies the analogue of
  \eqref{eq:RV-etap-c-k} with $k=d+1$ and the
  tail measure $\mu^*$ being the pushforward of
  $\binom{m}{k}\mu^{\otimes k}$ under the map
  $x=(x_1,\dots,x_k)\mapsto (\me_x,\dmet(0,\me_x))$ if $k=m$ and
  under the map $x=(x_1,\dots,x_k)\mapsto (\me_x,0)$ if $k\leq m-1$. 


  By Lemma~\ref{lemma:conv-hull-map}, $\conv(X)$ is regularly varying
  on $\sK^d_c$ with the ideal $\fK_{\inf}$ and normalising function
  $g^{d+1}$. Its tail measure
  $\tilde{\mu}$ is the pushforward of $\binom{m}{d+1}\mu^{\otimes
    (d+1)}$ by the map $x=(x_1,\dots,x_{d+1})\mapsto
  \conv(\{0,x_1,\dots,x_{d+1}\})$ if $m\geq d+2$ and by the map to
  $\conv(\{x_1,\dots,x_{d+1}\})$ if $m=d+1$. This holds if $\tilde{\mu}$
  gives positive mass to convex compact sets which contain the origin
  in their interior. For instance, this is the case if $\mu$ is not
  supported by any closed half-space whose boundary contains the origin.
\end{example}

\begin{example}[Convex hull of Poisson process]
  \label{ex:chull-K-null}
  Let $\etap\in\Mp[\sR_0^d]$ be a Poisson point process on $\R^d$ with
  intensity measure $\lambda\in\RV(\R^d,\mydot,\sR_0^d,g,\mu)$. Since
  $\lambda(B_\eps(0)^c)$ is finite for any $\eps>0$, the point process
  $\etap$ has only a finite number of points outside any ball centred
  at the origin, so that $X=\cl(\supp\etap)$ is almost surely
  bounded, hence, $X$ is a random compact set.  Recall that either
  $\supp\etap$ is closed or $\cl(\supp\etap)=\{0\}\cup\supp\etap$, see
  Example~\ref{ex:etap-supp}.

  Then the analogue of \eqref{eq:RV-etap-c-k} holds for $k=d+1$ with the
  tail measure $\mu^*$ being the pushforward of
  $\frac{1}{(d+1)!}\mu^{\otimes (d+1)}$ under the map
  $x=(x_1,\dots,x_{d+1})\mapsto (\me_x,0)$ if
  $\lambda(\R^d)=\infty$. If $\lambda$ is a finite measure with
  $p=\exp\{-\lambda(\R^d)\}\in(0,1)$, then 
  $\mu^*$ is the sum of the pushforward of $p\frac{1}{(d+1)!}\mu^{\otimes(d+1)}$
  under the map $x\mapsto (\me_x,\dmet(0,\me_x))$,
  where here and below $x=(x_1,\dots,x_{d+1})$,
  and the pushforward of $(1-p)\frac{1}{(d+1)!}\mu^{\otimes (d+1)}$
  under the map $x=(x_1,\dots,x_{d+1})\mapsto (\me_x,0)$. 

  By the continuous mapping theorem and
  Lemma~\ref{lemma:conv-hull-map} applied to the closed convex hull
  map, we see that $\conv(X)\in\RV(\sK_c^d,\mydot,\fK_{\inf},g^{d+1},\tilde\mu)$,
  where $\tilde\mu$ is the pushforward of
  $\frac{1}{(d+1)!}\mu^{\otimes (d+1)}$ under the map
  $x\mapsto \conv(\{0,x_1,\dots,x_{d+1}\})$ if $\lambda$ is
  infinite. If $\lambda$ is finite, then $\tilde{\mu}$ is the sum of
  the pushforward of $p\frac{1}{(d+1)!}\mu^{\otimes (d+1)}$
  under the map $x\mapsto \conv(\{x_1,\dots,x_{d+1}\})$ and the
  pushforward of $(1-p)\frac{1}{(d+1)!}\mu^{\otimes (d+1)}$
  under the map $x\mapsto \conv(\{0,x_1,\dots,x_{d+1}\})$.
\end{example}

\begin{example}[Random ball]
  \label{ex:r-ball}
  Let $(\eta,\xi)$ be a random vector in $\YY=\R_+\times\R^d$ with
  linear scaling applied to both components. Define a random ball
  \begin{displaymath}
    X=\varphi\big((\eta,\xi)\big)=B_{\eta}(\xi).
  \end{displaymath}
  Then $\varphi$ is a continuous morphism between $\R_+\times\R^d$ and
  $\sK_c^d$. Equip the space $\R_+\times\R^d$ with the ideal
  $\sR_0^{d+1}$ defined as the product $\sR_0\otimes\sR_0^d$ of the ideal
  $\sR_0$ on $\R_+$ and $\sR_0^d$ on $\R^d$. Recall that the ideal
  $\fK_0$ is generated by the modulus $\suptau(K)=\|K\|$. For all
  $\eps>0$,
  \begin{displaymath}
    \big\{(r,x)\colons \suptau(\varphi(r,x))>\eps\big\}
    =\big\{(r,x)\colons r+\|x\|>\eps\big\}.
  \end{displaymath}
  Thus, $\varphi$ is a bornologically consistent map between the ideal
  $\sR_0^{d+1}$ and $\fK_0$.
  If $(\eta,\xi)$ is regularly varying on $\sR_0^{d+1}$, then
  $X=B_\eta(\xi)$ is regularly varying by the continuous mapping
  theorem. For instance, if $\eta$ and $\xi$ are independent with the
  same normalising functions, then the tail measure of $X$ is supported
  either on singletons or on balls centred at the origin.

  We can endow $\R_+\times\R^d$ with the product
  ideal $\sR_0\times\sR_0^d$, so that bounded sets are products of two
  bounded sets. This ideal is smaller than $\sR_0^{d+1}$, and thus the
  map $\varphi$ is not bornologically consistent on the ideal
  $\fK_0$. To handle this, assume $d\geq2$ and endow $\sK_c^d$ with the ideal
  $\fK_{[1]}\cap\fK^0$, which is generated by the minimum of
  $\modulus{[1]}(K)$ and $\ttau(K)$. 
  Then $\modulus{[1]}(B_r(x))>\eps$ implies that
  $B_\eta(\xi)$ is not a subset of $H+B_\eps$ for any line $H$ passing
  through the origin, so that $r>\eps$. Since also 
  $\ttau(B_r(x))=\|x\|-r>\eps$, we obtain that 
  $\|x\|>\eps$. This means that the map defining $X$ is bornologically
  consistent. Assume that $\eta$ and $\xi$ are independent,
  $\eta\in\RV(\R_+,\mydot,\sR_0,g_1,\mu_1)$
  and $\xi\in\RV(\R^d,\mydot,\sR_0^d,g_2,\mu_2)$. Then the pair
  $(\eta,\xi)$ is regularly varying on $\sR_0\times\sR_0^d$ with the
  normalising function $g_1g_2$ and the tail measure
  $\mu_1\otimes\mu_2$. In this case, $X$ is regularly varying with the
  tail measure being the pushforward of $\mu_1\otimes\mu_2$ under the
  map $(r,x)\mapsto B_r(x)$.   
\end{example}

\paragraph{Inverse functions}
Another rather general scheme relies on defining random sets from
inverse functions. Let $q:\XX\to\YY$ be a continuous map from a
Polish space $\XX$ to a Hausdorff space $\YY$.
The continuity of $q$ and the fact that singletons are closed
in $\YY$ ensure that
\index{inverse map}
\begin{equation}
  \label{eq:9}
  \Psi(x) =q^{-1}(q(x))=\big\{z\in\XX\colons q(z)=q(x) \big\} 
\end{equation}
defines a set-valued map from $\XX$ to the space $\sF=\sF(\XX)$ of closed
subsets of $\XX$. Examples of such set-valued inverses arise in the
context of quotient spaces in Section~\ref{sec:quotient-spaces-under},
where $q$ is the quotient map
\index{quotient map}
and $\Psi(x)=q^{-1}(q(x))=[x]$ is the equivalence class of $x$.

Assume that $q(z)=q(x)$ for any $x,z\in\XX$ implies
$q(T_tz)=q(T_tx)$ for all $t>0$, which is
Condition~\hyperref[condS]{(S)} in the setting 
of quotient spaces. Then the set-valued map $\Psi$ defined in
\eqref{eq:9} satisfies
\begin{equation}
  \label{eq:11}
  \Psi(T_tx)=T_t\Psi(x), \quad t>0,x\in\XX,
\end{equation}
meaning that $\Psi$ is a morphism from $\XX$ to $\sF$.

The following result does not impose any topological assumptions on $\XX$.
Consequently, it applies also in cases where the Fell topology on
$\XX$ is not Hausdorff. 

\begin{lemma}
  \label{lemma:Psi-continuous}
  Assume that $\Psi$ is defined by \eqref{eq:9} with a
  continuous open map $q$. Then $\Psi$ is continuous as a map from
  $\XX$ to $\sF$ endowed with the Fell topology. 
\end{lemma}
\begin{proof} 
  Let $x_\gamma\to x$ for a net $(x_\gamma)_{\gamma\in\Gamma}$.  If
  $\Psi(x)\cap G\neq\emptyset$ for an open set $G$, then
  $q(x)\in q(G)$. Since $q(x_\gamma)\to q(x)$ and $q(G)$
  is open, there exists $\gamma_0$ such that
  $q(x_\gamma)\in q(G)$ for all $\gamma\geq\gamma_0$. Hence,
  $\Psi(x_\gamma)\cap G\neq\emptyset$ for all $\gamma\geq\gamma_0$. If
  $\Psi(x)\cap K=\emptyset$ for a compact set $K$, then
  $q(x)\notin q(K)$. Since $q(K)$ is compact (being continuous
  image of a compact set) and closed in the Hausdorff space $\YY$,
  we have that $q(x_\gamma)\notin q(K)$ for all
  $\gamma\geq\gamma_0$. 
\end{proof}

\begin{lemma}
  \label{lemma:cones}
  Assume that \eqref{eq:11} holds.  If
  \;$\zero \cap\Psi(x)\neq\emptyset$ for some $x\in\XX$, then
  $T_t\Psi(x)=\Psi(x)$ for some $t\neq1$.
\end{lemma}
\begin{proof}
  Assume that $y\in \Psi(x)\cap\zero$, which implies that $T_ty=y$ for some
  $t\neq1$ and $q(x)=q(y)$. Then, by the equivariance property
  \eqref{eq:11}, 
  \begin{align*}
    T_t\Psi(x)&=\big\{T_tz\colons q(z)=q(y) \big\}\\
    &=\big\{z\colons q(T_{t^{-1}} z)=q(T_{t^{-1}}y) \big\}
    =\big\{z\colons q(z)=q(y) \big\}=\Psi(x). \qedhere 
  \end{align*}
\end{proof}

Note that $\zero\cap\Psi(x)\neq\emptyset$ is equivalent to 
$q(x)\in q(\zero)$.
Therefore, every set $\Psi(x)$ either intersects $\zero$, in which case
it is invariant under at least one non-trivial scaling, or is a closed
set disjoint from $\zero$.

\begin{theorem}
  \label{thr:Psi}
  Assume that $\XX$ is a ball compact Polish space equipped with a
  continuous scaling action such that the set $\zero$ of
  scaling-invariant elements in $\XX$ is compact. Consider the metric exclusion ideal
  $\sX_\zero$ obtained by excluding $\zero$. Assume that either the metric
  $\dmet$ is homogeneous, or that $\sX_\zero$ is generated by an
  inf-compact continuous modulus. Furthermore, assume that $\Psi$ is
  defined by \eqref{eq:9} with a continuous open map $q$ and that
  \eqref{eq:11} holds. If $\xi\in\RV(\XX,T,\sS_\zero,g,\mu)$ and the
  tail measure $\mu$ is non-trivial on the set 
  $\{x\in\XX\colons \Psi(x)\cap\zero=\emptyset\}$, then $\Psi(\xi)$ is a
  random closed set that is regularly varying on the ideal
  $\fF^\zero$. 
\end{theorem}
\begin{proof}
  Note that $\Psi(x)$ is a closed set for every $x\in\XX$.  By
  Lemma~\ref{lemma:Psi-continuous}, $\Psi$ is continuous with respect
  to the Fell topology.  Thus, $\Psi(\xi)$ is a random closed set and
  $\Psi$ is a 
  continuous morphism. Recall that $\fF^\zero$ is the ideal 
  consisting of all families of sets which do not intersect
  $\zero^r=\{x\in\XX\colons \dmet(x,\zero)\leq r\}$ for some $r>0$. Since
  $\zero$ is compact, the sets in $\fF^\zero$ are precisely the closed
  sets disjoint from $\zero$ and their union is the complement of
  $\zero$. Furthermore, since $x\in\Psi(x)$, 
  \begin{displaymath}
    \big\{x\colons \Psi(x)\cap \zero^r=\emptyset\big\}
    \subset \big\{x\colons x\notin \zero^r\big\}\in\sX_\zero, \quad r>0. 
  \end{displaymath}
  Thus $\Psi$ is bornologically consistent, and the result follows from
  the continuous mapping theorem (Theorem~\ref{mapping-theorem}). Note
  that under the imposed conditions $\sX_\zero$ is scaling consistent
  either by the homogeneity of the metric or by 
  Proposition~\ref{ex:F-zero}.
\end{proof}

\begin{example}
  Let $q=(q_1,\dots,q_m):\R^d\to\R^m$ be a continuous
  homogeneous function which is also an open map. By
  Lemma~\ref{lemma:map-zero}, the map $q$ is bornologically consistent if the
  spaces are equipped with the ideals $\sR_0^d$ and $\sR_0^m$,
  respectively. If $\xi\in\RV(\R^d,\mydot,\sR_0^d,g,\mu)$, then
  \begin{displaymath}
    \Psi(\xi)=\bigcap_{i=1}^m \big\{x\in\R^d\colons 
    q_i(x)=q_i(\xi)\big\}.
  \end{displaymath}
  If $\mu$ is non-trivial on $\{x\in\R^d\colons 0\notin \Psi(x)\}$, then
  Theorem~\ref{thr:Psi} implies that the random closed set
  $X=\Psi(\xi)$ is regularly varying on $\fF^\zero$.
\end{example}

We now adapt the setting of quotient spaces from
Section~\ref{sec:quotient-spaces-under} to study the regular
variation of equivalence classes as random closed sets in
$\XX$. Recall that $[x]$ denotes the equivalence class of $x\in\XX$
under the quotient map $q$. 

\begin{proposition}
  \label{prop:equiv-sets}
  Assume that $\XX$ is a ball compact Polish space
  equipped with a continuous scaling and an ideal $\sX$
  generated by an inf-compact, continuous modulus $\modulus$.  Let
  $\xi\in\RV(\XX,T,\sS,g,\mu)$. If Condition~\hyperref[condS]{(S)} holds, $q$ is an open
  map, and $\mu$ is non-trivial on $[\sS]$, then $[\xi]$ is a regularly
  varying random closed set on the ideal $\sF^\zero(\XX)$ generated by the modulus
  $\ttau$.
\end{proposition}
\begin{proof}
  The map $x\mapsto[x]=\Psi(x)$, which associates each $x\in\XX$ with its
  equivalence class, is covered by
  Theorem~\ref{thr:Psi} with $\Psi(x)=q^{-1}(q(x))$, see
  \eqref{eq:9}, where $q$ is the quotient map. 
  If Condition~\hyperref[condS]{(S)} holds, then $\Psi$ satisfies \eqref{eq:11}.  By
  Lemma~\ref{lemma:cones}, $\Psi(x)=[x]$ is either a cone or has an
  empty intersection with $\zero$. In the latter case, $[x]$ is a
  closed set disjoint from $\zero$, hence $[x]$ belongs to the
  family $\sF^\zero(\XX)$. The statement now follows from 
  Theorem~\ref{thr:Psi}.
\end{proof}

\paragraph{Examples: regularly varying random convex compact sets}
Below we present several examples involving probability distributions
on the family $\sK_c^d$ of convex compact sets in $\R^d$ with 
linear scaling.

\begin{example}[Support functions]
  \label{mps-example} 
  Define a map $\psi$ from $\sK_c^d$ to the space $\Cont(\SS^{d-1})$
  of continuous functions on the unit sphere $\SS^{d-1}$ in $\R^d$ by
  associating each compact convex set $K$ with its \emph{support function},
  defined as
  \index{random convex compact set}
  \index{support function}
  \begin{equation}
    \label{eq:21}
    h_K(u)= \sup\{\langle u,a\rangle\colons a\in K\},\quad u\in \SS^{d-1};
  \end{equation}
  see \cite[Section~1.7.1]{schn2}.  The space $\Cont(\SS^{d-1})$ is
  equipped with the uniform metric, linear scaling and the modulus
  given by the uniform norm.  Since the Hausdorff distance between
  compact convex sets equals the uniform distance between their
  support functions,
  \begin{displaymath}
    \suptau(K)=\|K\|=\sup\big\{h_K(u)\colons
    u\in\SS^{d-1}\big\}=\|h_K\|_\infty.
  \end{displaymath}
  It follows that $\psi$ is an isometry. This map $\psi$ is
  bornologically consistent by 
  Lemma~\ref{lemma:map-zero} and therefore satisfies the conditions of
  Theorem~\ref{mapping-theorem}.  A random compact convex set $X$ is
  regularly varying on $\fK_0$ if and only if its
  support function $h_X$ is regularly varying on the family of
  continuous functions on the unit sphere with the ideal $\sC_0$
  generated by the norm. In this way we recover an observation made
  in \cite{mikosch:pawlas:samorodnitsky_2011}.

  Furthermore, if $0\in X$ almost surely, then $X$ is regularly varying on
  $\fK_{\inf}$ if and only if $h_X$ is regularly varying on
  $\Cont(\SS^{d-1})$ with the ideal $\sC_{\inf}$ generated by
  \begin{displaymath}
    \modulus_{\inf}(h_K)=\inf_{u\in\SS^{d-1}} h_K(u).
  \end{displaymath}
  
  The mean width of $K$ is defined as
  \begin{equation}
    \label{eq:mean-width}
    \bar{b}(K)=\frac{1}{d\kappa_d}\int_{\SS^{d-1}} h_K(u)\diff u,
  \end{equation}
  where $(d\kappa_d)$ is the surface area of $\SS^{d-1}$. The regular
  variation of $X$ then corresponds to the regular
  variation of the support function on the ideal generated
  by the integral functional. 
\end{example}

\begin{example}[Regularly varying selection]
  \label{ex:Steiner-point}
  The \emph{Steiner point} of a convex compact set $K$ is defined by
  \begin{equation}
    \label{eq:20}
    \steiner(K) = \frac{1}{\kappa_d} \int_{\SS^{d-1}} h_K(u)\,
    u \diff u, \quad K\in \sK_c^d,
  \end{equation} 
  where $\kappa_d$ is the volume of the unit Euclidean ball in $\R^d$,
  and the integration is with respect to the $(d-1)$-dimensional
  Hausdorff measure on the unit sphere $\SS^{d-1}$; see
  \cite[Eq.~(1.31)]{schn2}. By \cite[Eq.~(1.34)]{schn2}, $\steiner(K)$
  always belongs to $K$. Since the Hausdorff metric equals the uniform
  distance between support functions, $\steiner: \sK_c^d\to\R^d$ is a
  continuous morphism. Furthermore, 
  \begin{displaymath}
    \big\{K\in\sK_c^d\colons \|\steiner(K)\|\geq\eps\big\}\subset
    \big\{K\in\sK_c^d\colons \|K\|\geq\eps\big\}
  \end{displaymath}
  shows that the map $K\mapsto\steiner(K)$ is bornologically
  consistent between $\fK_0$ and the ideal $\sR_0^d$.  If $X$ is
  regularly varying on $\fK_0$, then $\steiner(X)$ is regularly
  varying in $\R^d$ provided that the tail measure of $X$ assigns
  positive mass to the family of sets whose Steiner point is not the
  origin. For instance, consider the random set $X$ from
  Example~\ref{ex:interior-origin} obtained as the convex hull of $m$
  points $\xi_1,\dots,\xi_m$, $m\geq2$, which are independent copies
  of $\xi\in\RV(\R^d,\mydot,\sR_0^d,g,\mu)$. Then, with the
  normalising function $m^{-1}g$, the tail measure of $X$ is the
  pushforward of $\mu$ under the map $x\mapsto[0,x]$ and so
  $\steiner(X)\in\RV(\R^d,\mydot, \sR_0^d,m^{-1}g,\mu_1)$, where
  $\mu_1$ is the pushforward of $\mu$ under the map $x\mapsto x/2$.

  The Steiner point $\steiner(X)$ is an example of a \emph{selection} of $X$,
  that is, a random vector that belongs to $X$ almost surely.
  \index{selection}
  The above argument shows that a regularly varying convex compact
  random set $X$ admits a regularly varying selection if the tail
  measure of $X$ is not entirely supported by convex compact sets whose 
  Steiner point is at the origin.  It is natural to conjecture that a
  much more general 
  variant of this statement holds, namely, every regularly varying
  random closed set admits a selection that is itself regularly varying on a
  suitably chosen non-trivial ideal.
\end{example}

The following example concerns passing to a subcone in the family of
compact convex sets. 

\begin{example}[Random ellipsoid]
  \label{ex:sets-hrv}
  For a random vector $\xi=(\xi_1,\dots,\xi_d)$ in $[0,\infty)^d$, let
  \begin{displaymath}
    X=\psi(\xi)=\Big\{u\in\R^d\colons \sum_{i=1}^d \xi_i^{-2}u_i^2 \leq 1\Big\}
  \end{displaymath}
  be the ellipsoid with semiaxes given by the components of $\xi$
  with the convention that if $\xi_i=0$ then the ellipsoid is
  degenerate in the $i$-th direction. Since 
  \begin{displaymath}
    \big\{x\in[0,\infty)^d\colons \|\psi(x)\|\geq\eps\big\}
    =\big\{x\in [0,\infty)^d\colons \max x_i\geq\eps\big\},
  \end{displaymath}
  the map $\psi$ is bornologically consistent between $\sR_0^d$
  (restricted to $[0,\infty)^d$) and $\fK_0$.  If
  $\xi\in\RV(\R_+^d,\mydot,\sR_0^d,g,\mu)$, then
  $X\in\RV(\sK_c^d,\mydot,\fK_0,g,\psi\mu)$ by the continuous mapping
  theorem. If the components of $\xi$ are i.i.d., then the tail
  measure $\mu$ is supported by the axes. Consequently, its pushforward
  $\psi\mu$ is supported by the family of segments $[-ue_i,ue_i]$
  passing through the origin, where $u>0$ and $e_1,\dots,e_d$ are the 
  basis vectors in $\R^d$.

  Note that $X$ belongs to the family $\sK_{c,0}^d$ of convex
  compact sets containing the origin.  Since $\psi\mu$ is supported by
  segments, its support does not contain any set containing the
  origin in its interior, hence, $X$ is not regularly varying on the ideal
  $\fK_{\inf}$ with the normalising function $g$.  However,
  \begin{displaymath}
    \modulus_{\inf}(\psi(\xi))=\min(\xi_1,\dots,\xi_d)=\modulus_{\min}(\xi),
  \end{displaymath}
  and $X=\psi(\xi)$ becomes regularly varying if $\xi$ is regularly
  varying on the ideal $\sR_0^d(d)$.

  It is possible to connect various orders of the hidden regular
  variation of $X$ and $\xi$ by considering ideals $\fK_{[k]}$
  obtained by excluding the family of linear subspaces of dimension at most
  $k$. Then, for $k=1,\dots,d-1$, the map $\psi$ is bornologically
  consistent from the ideal $\sR_0^d(k+1)$ to the ideal
  $\fK_{[k]}$, since the condition that at
  least  $d-k$ semiaxes exceed a threshold prevents the image
  ellipsoid the image ellipsoid from being approximable by a set of
  linear dimension at most $k$.
\end{example}

\begin{example}[Volume map]
  \label{ex:volumes}
  Let $\psi(K)=V_d(K)^{1/d}$ be the $d$-th root of the Lebesgue
  measure of $K\in \sK_c^d$. Then $\psi$ is a continuous
  bornologically consistent morphism between $\sK_c^d$ with the ideal
  $\fK_0$ and $\R_+$ with the ideal $\sR_0$. Indeed, it is impossible
  for $\|K_n\|\to0$ to hold for a sequence of compact sets
  $(K_n)_{n\in\NN}$, satisfying $V_d(K_n)\geq \eps>0$ for all $n$. If
  $X$ is a regularly varying random convex compact set on the ideal
  $\fK_0$ and its tail measure is not entirely supported by sets of
  volume zero, then $V_d(X)^{1/d}$ is regularly varying. Consequently,
  $V_d(X)$ is regularly varying after applying the power map
  $x\mapsto x^d$, with the corresponding induced scaling and
  normalising function.
\end{example}

\begin{example}[Intrinsic volumes]
  \label{ex:intr-volumes}
  Consider the map from $\sK_c^d$ to $\R_+^d$ given by
  \begin{displaymath}
    \psi(K)=\big(V_1(K),\dots,V_d(K)\big),
  \end{displaymath}
  where $V_1,\dots,V_d$ are the 
  intrinsic volumes; see \cite[Section~4.1]{schn2}. In particular,
  $V_d$ is the volume.
  \index{intrinsic volumes}
  Then $\psi$ is a bornologically consistent morphism
  between $\sK_c^d$ with the ideal $\fK_0$ and $\R_+^d$ with the scaling
  \begin{equation}
    \label{eq:3s}
    T_tx=(tx_1,\dots,t^dx_d)
  \end{equation}
  and the ideal $\sR_0^d$. Indeed,
  \begin{displaymath}
    \big\{K\in\sK_c^d\colons \max_{i=1,\dots,d} V_i(K)>\eps\big\}
    \subset \big\{K\colons \|K\|>c(\eps)\big\}\in\fK_0
  \end{displaymath}
  for some $c(\eps)>0$ chosen so that
  $\max_i V_i\big(B_{c(\eps)}(0)\big)\leq \eps$. Note that $\psi(K)=0$ if and
  only if $K$ is a singleton. If $X$ is a regularly
  varying random compact convex set and its tail measure is not
  supported exclusively by singletons, then $\psi(X)$ is regularly
  varying in $\R_+^d$ with the scaling given in \eqref{eq:3s}. For
  instance, consider the random set $X$ from
  Example~\ref{ex:interior-origin} obtained as the convex hull of $m$
  points $\eta_1,\dots,\eta_m$, which are independent copies of
  $\eta\in\RV(\R^d,\mydot,\sR_0^d,g,\mu)$. It is shown in
  Example~\ref{ex:interior-origin} that
  $X\in\RV(\sK_c^d,\mydot,\fK_0,m^{-1}g,\tilde{\mu})$, where
  $\tilde{\mu}$ is the pushforward of $\mu$ under the map
  $x\mapsto\conv(\{0,x\})$.
  Since the value of $V_1$ on a unit segment
  in $\R^d$ is easily seen to be one, 
  \begin{displaymath}
    \psi\big(\conv(\{0,x\})\big)=\big(\|x\|,0,\dots,0\big).
  \end{displaymath}
  Then $\psi(X)$ is regularly varying on $\sR_0^d$ with the
  tail measure given by the pushforward of $\mu$ under the map
  $x\mapsto (\|x\|,0,\dots,0)$.
\end{example}

\begin{example}[Uniform distribution on a random set]
  Let $X$ be a random compact convex set in $\R^d$. Consider the map $\psi$ that
  associates $X$ with the uniform probability distribution on $X$; in
  case of a lower-dimensional $X$ we consider the distribution
  proportional to the Lebesgue measure in the affine hull of $X$. This
  map $\psi$ is clearly homogeneous and is also bornologically
  consistent as the map from the ideal $\fK_0$ to the ideal $\sN_r$
  introduced in Section~\ref{sec:rpm}. Indeed,
  \begin{displaymath}
    \{K\in\sK^d\colons \psi(K)(B_a(0)^c)>r\}\subset
    \{K\in\sK^d\colons \|K\|>a\}\in\fK_0
  \end{displaymath}
  for all $a>0$.  Thus, the map is bornologically consistent. If, in
  addition, this map is continuous at the relevant points of the tail
  measure, then the continuous mapping theorem yields regular
  variation of the associated random probability measure.
  \index{random probability measure}
\end{example}

\paragraph{Examples: equivalence classes of compact convex sets} The
following examples deal with quotient spaces of the space $\sK_c^d$ of
convex compact sets in $\R^d$ metrised by the Hausdorff metric and
equipped with linear scaling. We follow the setting of
Section~\ref{sec:quotient-spaces-under}.

\begin{example}[Equivalence up to translations]
  \label{ex:circumscribed}
  For $K,L\in\sK_c^d$, let $K\sim L$ if $K+a=L$ for some $a\in \R^d$,
  meaning that some translation of $K$ coincides with $L$. Then
  $[K]=\{K+a\colons a\in\R^d\}$, in particular, $[\zero]$ is the family of
  all singletons. Condition~\hyperref[condS]{(S)} holds, so that the scaling on the
  quotient space $\XXT=\tilde{\sK}_c^d$ of $\sK_c^d$ is well defined.

  Using the correspondence between compact convex sets and their
  support functions, we see that the translation equivalence of sets
  corresponds to the equivalence of support functions up to an
  additive term of the form $\langle a,u\rangle$, $u\in\SS^{d-1}$, for
  some $a\in\R^d$.

  Recall the ideal $\fK_0$ generated by the modulus
  $\suptau(K)=\|K\|$. Then $[\fK_0]$ consists of families of convex
  compact sets with diameters bounded below by some $\eps>0$. This
  ideal is actually the Hausdorff metric exclusion ideal $\fK_{(0)}$
  obtained by excluding from $\sK_c^d$ the family of all singletons.
  We have that
  \begin{displaymath}
    \ttau([K])=\inf\{\|K+a\|\colons a\in\R^d\}
  \end{displaymath}
  is the radius of the smallest circumscribed ball around $K$. This
  ball is unique by \cite[Lemma~3.1.5]{schn2} and
  \cite[Proposition~9.2.24]{burago2001}, and the map from $K$ to 
  $\ttau([K])$ is continuous (even Lipschitz) by the definition
  of the Hausdorff metric; see also
  \cite[Exercise~9.2.25]{burago2001}. Hence, $\ttau([x])$ is
  continuous in $x$ and generates the ideal $q\fK_0$.
  
  By Theorem~\ref{rv-XXT}, if a random set $X$ is regularly varying on
  $\fK_0$ with the tail measure not exclusively supported by
  singletons, then the equivalence class $qX$ is a regularly varying
  element in $\tilde{\sK}_c^d$ with the ideal $q\fK_0$.  For
  instance, let $X=B_\eta(\xi)$ be the closed ball in $\R^d$ of radius
  $\eta$ centred at $\xi$, see Example~\ref{ex:r-ball}. The regular
  variation property of $\tilde{X}=qX$ depends only on the
  distribution of $\eta$, namely, $\tilde{X}$ is regularly varying if
  and only if $\eta$ is a regularly varying random variable. If $\xi$
  is regularly varying with a tail heavier than that of $\eta$, then
  $X$ is regularly varying, however, the pushforward of the tail
  measure is concentrated on singletons, meaning that $qX$ is not regularly
  varying in the quotient space.
\end{example}

\begin{example}[Equivalence up to linear transforms]
  Consider the family $\sK_{c,0}^d$ of convex compact sets with
  non-empty interior (also called convex bodies).
  \index{convex body}
  Assume that $K\sim L$ if $L=AK+a$, where $A\in\mathrm{SL}_d$ is a positive definite
  matrix with determinant one and $a\in\R^d$. Let $\fK_V$ be the ideal
  generated by the modulus
  \begin{displaymath}
    \modulus_V(K)=V_d(K)^{1/d},
  \end{displaymath}
  where $V_d$ denotes the volume. Since this ideal is invariant under
  the transformation $K\mapsto AK+a$, we have that
  $[\fK_V]=\fK_V$. Consequently, 
  the regular variation property of a random convex body $X$ implies
  the same property for its equivalence class $\tilde{X}$.

  Now consider another ideal $\fK$ on $\sK_{c,0}^d$ generated by the modulus
  \begin{displaymath}
    \modulus(K)=V_d(E_K)^{1/d}+\|K\|,
  \end{displaymath}
  where $E_K$ is the L\"owner ellipsoid of $K$, which is the uniquely
  defined ellipsoid of minimal volume containing $K$; see
  \cite[Section~10.12]{schn2}. The map $K\mapsto E_K$ is continuous in
  the Hausdorff metric. Indeed, if $K_n\to K$, then $(E_{K_n})_{n\in\NN}$ is a
  bounded sequence of convex bodies that has a convergent
  subsequence; see \cite[Theorem~1.8.4]{schn2}. Assuming that
  $E_{K_n}\to E$, we see that $E$ is an ellipsoid of minimal volume
  containing $K$ and therefore coincides with $E_K$ by 
  uniqueness. Indeed, assume for contradiction that
  $V_d(E_K)<V_d(E)$. Then the
  $\eps$-envelope of $E_K$ contains $K_n$ for all sufficiently
  large $n$. Hence, for each $\delta>0$ we have
  $V_d(E_{K_n})\leq V_d(E_K)+\delta$, implying
  $V_d(E)\leq V_d(E_K)+\delta$, which is a contradiction.  Then
  \begin{displaymath}
    \ttau([K])=\inf\big\{V_d(E_{AK})^{1/d}+\|AK+a\|\colons 
    A\in\mathrm{SL}_d,a\in\R^d\big\}
  \end{displaymath}
  is continuous in $K$ and thus it determines the ideal $q\fK$, see
  Lemma~\ref{lemma:q-B}. By Theorem~\ref{rv-XXT}, if the tail measure
  of a random convex body $X$ is not exclusively supported by the
  origin, then the equivalence class $qX$ is regularly varying on the
  ideal $q\fK$.
\end{example}

\begin{example}[Pairs of sets]
  Let $\XX=\sK_c^d\times\sK_c^d$ be the family of pairs of convex
  compact sets $(K_1,K_2)$ in $\R^d$ with linear scaling applied
  componentwise, so that the only zero element is $(\{0\},\{0\})$.
  Define an ideal $\fK$ generated by the modulus $\modulus(K_1,K_2)$ being the
  maximum of the mean widths $\bar{b}(K_1)$ and $\bar{b}(K_2)$, see
  \eqref{eq:mean-width}. We say that $(K_1,K_2)\sim (L_1,L_2)$
  if $K_1+L_2=K_2+L_1$.  This is the standard construction that turns
  the semigroup into a group via taking the quotient space build from
  pairs of its elements. 
  Then the
  equivalence class $q(K_1,K_2)$ may then be considered as the difference
  $K_1-K_2$. In particular, $(K_1,K_2)\sim (\{0\},\{0\})$ if and only
  if $K_1=K_2$, implying that $\tilde\zero$ is the equivalence class of
  $(\{0\},\{0\})$, which corresponds to pairs with identical
  components. Condition~\hyperref[condS]{(S)} obviously holds.
  Then
  \begin{displaymath}
    \ttau\big([(K_1,K_2)]\big)
    =\inf\big\{\max\big(\bar{b}(L_1),\bar{b}(L_2)\big)\colons 
    (L_1,L_2)\sim (K_1,K_2)\big\}
    \geq \big|\bar{b}(K_1)-\bar{b}(K_2)\big|. 
  \end{displaymath}
  where we use the fact that the mean width is additive under
  Minkowski addition.
  If a pair $(X_1,X_2)$ is regularly varying on $\fK$ and its tail
  measure assigns positive mass to the family of pairs $(K_1,K_2)$
  with $|\bar{b}(K_1)-\bar{b}(K_2)|>0$, then its equivalence
  class is regularly varying in the quotient space with the ideal
  $[\fK]$.
\end{example}

\chapter*{Acknowledgements}

This work was supported by the Swiss Enlargement Contribution in the
framework of the Croatian-Swiss Research Programme (project number
IZHRZ0\_180549). The work of Bojan Basrak was partially supported by
the Croatian Science Foundation (project number IP-2022-10-2277) and
the European Union institutional grant of the University of Zagreb,
Faculty of Science (project number IK IA 1.1.3. Impact4Math).

The authors are grateful to Fabien Baeriswyl, 
Andrii Ilienko, Hrvoje
Planini\'c, Stilian Stoev and Olivier Wintenberger for helpful
discussions and encouragement and to Nick Bingham and Daniela
Ivankovi\'c for 
providing remarks on the manuscript.

Furthermore, the authors are grateful to the anonymous
referees for their positive assessment of this work, careful reading, and
numerous constructive remarks.


\printindex

\cleardoublepage

\addcontentsline{toc}{chapter}{List of Notation}

\begin{center}
  {\sc List of notation}
\end{center}

\noindent

\begin{tabular}{ll}
  \multicolumn{2}{l}{\textsf{Topological spaces:}}\\
  $\XX$  & a topological space,\\
  $B^c$ & the complement of set $B$,\\
  $\cl B$, $\Int B$ & closure and interior of $B$,\\
  $\dmet$ & a metric on $\XX$,\\
  $\sB(\XX)$, $\sBA(\XX)$ & Borel and Baire $\sigma$-algebras,\\
  $\gamma\in\Gamma$ & parameter of a net. \\
  \multicolumn{2}{l}{\textsf{Scaling and cones:}}\\
  $\R_+$ & the positive half-line $[0,\infty)$,\\
  $T_tx$ & generic scaling of $x\in\XX$ by $t>0$,\\
  $T_tA$ & scaling of a sset $A$,\\
  $tx$ and $\mydot$ & linear scaling,\\
  $\zero$ & set of scaling-invariant elements,\\
  $\cone$ & a cone in $\XX$, \\
  $\VV$ & a semicone in $\XX$.\\
  \multicolumn{2}{l}{\textsf{Moduli and ideals:}}\\
  $\modulus$ & a generic modulus on $\XX$,\\
  $\ttau(A)$ & infimum of $\modulus(x)$ over $x\in A$,\\
  $\suptau(A)$ & supremum of $\modulus(x)$ over $x\in A$,\\
  $\rho$ & polar decomposition,\\
  $\sX$, $\sR$, $\sC$ & generic ideals on $\XX$, $\R^d$ and
                               $\Cont$,  respectively, \\
  $\sS_\cone$, $\sS_0$ & the ideal obtained by excluding $\cone$ or $\zero$,\\
  $\sS_\modulus$ & ideal generated by $\modulus$, see
                   Definition~\ref{def:modulus},\\
  $\sS_\cone$ & metric exclusion ideal obtained by excluding cone $\cone$,\\
  $\SS_\modulus$ & the set of $x\in\XX$ with $\modulus(x)=1$,\\
  $\sX^m(k)$ & product of ideals,\\
  $\sR_0^d$, $\sR_0$ & metric exclusion ideals on $\R^d$ or $\R$
                       obtained by \\
         &\qquad excluding the origin,\\
  $\sR_0^d(k)$ & metric exclusion ideal on $\R^d$ obtained by excluding\\
         & \qquad  coordinate subspaces of dimension at
           most $k-1$,\\ 
  $\psi$ & bornologically consistent morphism.\\
  \multicolumn{2}{l}{\textsf{Vague convergence:}}\\
  $\Msigma$ & Baire measures which are finite of the ideal $\sS$,\\
  $\Mb$ & Borel measures which are finite of the ideal $\sS$,\\
  $\vto$ & vague convergence on the ideal $\sS$,\\
  $\wto$, $\dto$ & weak convergence / convergence in distribution,\\
  $\fC(\XX,\sS)$ & continuous bounded functions with support in
                   $\sS$,\\
  $\psi\mu$ & pushforward of measure $\mu$ under $\psi$.

\end{tabular}

\noindent

\begin{tabular}{ll}
  \multicolumn{2}{l}{\textsf{Regular variation:}}\\
  $\RV_\alpha$ & functions on the real line, which are \\
               &\qquad regularly
                 varying of order $\alpha$,\\
  $\RV(\XX,T,\sS,g,\mu)$ & regularly varying measure (or random element),\\
  $\alpha$ & tail index,\\
  $\theta_\alpha$ & measure on $(0,\infty)$ with
                    $\theta_\alpha((t,\infty))=t^{-\alpha}$,\\ 
  $\mu$, $\mu^*$, $\tilde{\mu}$ & tail measure,\\  
  $g$ & normalising function, usually from $\RV_\alpha$.\\
  \multicolumn{2}{l}{\textsf{Quotient spaces:}}\\
  $\sim$ & equivalence relation,\\
  $\XXT$, $\tilde{x}$ & quotient space and its elements,\\
  $q$ & quotient map,\\
  $[x]$, $[B]$  & equivalence class of $x$ and saturation of $B$,\\
  $\tilde{T}$ & scaling on the quotient space.\\
  \multicolumn{2}{l}{\textsf{Spaces of sequences:}}\\
  $\R^\infty$ & space of infinite real-valued sequences $(x_1,x_2,\ldots)$,\\
  $\ell_p$ & space of $p$-summable sequences, $p\in[1,\infty)$,\\
  $\ell_\infty$ & space of bounded sequences.\\
  $\sR_{\sup}$, $\sR_{\inf}$, $\sR^\infty(k)$ & ideals on the space of
                                              sequences.\\
  $\psi_m(x)$ & sequence $(x_1,\dots,x_m,0,0,\ldots)$ obtained from
                $x\in\R^\infty$,\\ 
  $\pi_m x$ & projection of $x\in\R^\infty$ to $(x_1,\dots,x_m)$.\\
  \multicolumn{2}{l}{\textsf{Spaces of functions:}}\\
  $\Cont([0,1])$, $\Cont(\R)$ & space of continuous functions on
                                $[0,1]$ or on $\R$,\\
  $\Cp$ & space of continuous functions with the topology,\\
    &\qquad of pointwise convergence,\\
  $\Cont_{\text{loc}}(\R)$ & continuous functions with locally
                             uniform topology,\\
  $\Dfun([0,1],\R)$ & real-valued c\`adl\`ag functions on $[0,1]$,\\
  $\USC([0,1])$ & family of upper semicontinuous functions on $[0,1]$,\\
  $\sC_0$ & ideal generated by the uniform norm on $\Cont([0,1])$,\\
  $\sC_{\inf}$ & ideal generated by $\modulus(x)=\inf x$,\\
  $\sD_0$ & metric exclusion ideal on $\Dfun([0,1],\R)$,\\
  $\pi_\gamma x$ & values $(x(u_1),\dots,x(u_m))$ of $x$
                  at $\gamma=(u_1,\dots,u_m)$.\\ 
  \multicolumn{2}{l}{\textsf{Space of measures:}}\\
  $\Mp$ & counting measures which are finite on the ideal $\sX$,\\
  $\delta_x$ & Dirac measure at $x\in\XX$,\\
  $\Mpk$ & the ideal on $\Mp$ obtained by excluding measures\\
  &\qquad  with mass at most $k-1$,\\
  $\me$, $\me^{(k)}$ & counting measure and its factorial product,\\
  $\etap$, $\etap^{(k)}$ & point process and its factorial product,\\
  $\supp\me$, $\supp\etap$ & support of a counting measure or of a
                             point process,\\
  $\lambda$ & intensity measure of a point process,\\
  $J_{k,B}(A)$ & the Janossy measure of order $k$.
\end{tabular}

\begin{tabular}{ll}
  \multicolumn{2}{l}{\textsf{Space of closed sets:}}\\
  $\sF=\sF(\XX)$ & family of closed subsets of $\XX$,\\
  $\sK$, $\sK^d$ & family of compact sets in $\XX$ / in $\R^d$,\\
  $\dmet_{\mathrm{H}}$ & the Hausdorff metric,\\
  $\dmet(x,F)$ & distance from $x$ to $F$,\\
  $\sK_c^d$ & family of convex compact sets in $\R^d$,\\
  $\conv(K)$ & the closed convex hull of $K\subset\R^d$,\\
  $\fF$, $\fK$ & generic ideals on $\sF$ and $\sK$,\\
  $\fK_0$ & ideal on $\sK$ generated by the modulus $\suptau(K)=\|K\|$,\\
  $\fK_{[k]}$ & ideal on $\sK^d$ obtained by excluding linear
                subspaces\\
              & \qquad of dimension $k$,\\
  $V_j(K)$, $V_d(K)$ & $j$-th intrinsic volume / if $j=d$ the volume,\\
  $B_{r}(x)$ & the closed ball of radius $r>0$ centred at $x\in\R^d$,\\
 \multicolumn{2}{l}{\textsf{Probability and measures:}}\\
  $\P$, $\E$ & probability and expectation corresponding to the
               \\
         & \qquad choice of a probability space
           $(\Omega,\mathfrak{F},\P)$.
\end{tabular}

\end{document}